\documentclass[12pt]{amsart}
\usepackage{amscd}
\usepackage{amssymb}
\usepackage[all]{xy}

\newtheorem{theorem}{Theorem}[section]
\newtheorem{lemma}[theorem]{Lemma}

\newtheorem{cor}[theorem]{Corollary}
\theoremstyle{definition}
\newtheorem{definition}[theorem]{Definition}
\newtheorem{example}[theorem]{Example}

\theoremstyle{remark}
\newtheorem{remark}[theorem]{Remark}
\numberwithin{equation}{section}

\newcommand\B{\mathbb{B}}
\newcommand\C{\mathbb{C}}
\newcommand\F{\mathbb{F}}
\newcommand\Z{\mathbb{Z}}
\newcommand\Q{\mathbb{Q}}
\newcommand\R{\mathbb{R}}

\newcommand\T{\mathbb{T}}

\newcommand\fS{\mathfrak{S}}

\newcommand\cB{\mathcal{B}}
\newcommand\cC{\mathcal{C}}
\newcommand\cD{\mathcal{D}}
\newcommand\cG{\mathcal{G}}
\newcommand\cF{\mathcal{F}}
\newcommand\cH{\mathcal{H}}
\newcommand\cK{\mathcal{K}}
\newcommand\cJ{\mathcal{J}}

\newcommand\cM{\mathcal{M}}
\newcommand\cN{\mathcal{N}}
\newcommand\cO{\mathcal{O}}

\newcommand\cR{\mathcal{R}}
\newcommand\cU{\mathcal{U}}
\newcommand\cZ{\mathcal{Z}}
\newcommand\ii{\mathrm{i}}

\newcommand{\Ad}{\operatorname{Ad}}
\newcommand{\Aut}{\operatorname{Aut}}
\newcommand{\Out}{\operatorname{Out}}
\newcommand{\Hom}{\operatorname{Hom}}
\newcommand{\End}{\operatorname{End}}
\newcommand{\Irr}{\operatorname{Irr}}

\newcommand\Span{\operatorname{span}}
\newcommand\Tr{\operatorname{Tr}}

\newcommand\id{\mathrm{id}}

\newcommand\inpr[2]{\langle{#1,#2}\rangle}
\newcommand\hG{\hat{G}}
\newcommand\hH{\hat{H}}

\newcommand\biota{\overline{\iota}}
\newcommand\bpi{\overline{\pi}}
\newcommand\brho{\overline{\rho}}
\newcommand\talpha{\tilde{\alpha}}
\newcommand\tbeta{\tilde{\beta}}
\newcommand\trho{\tilde{\rho}}
\newcommand\tgamma{\tilde{\gamma}}
\newcommand\hgamma{\hat{\gamma}}

\title[The classification of near-group categories]
{A Cuntz algebra approach to the classification of near-group categories} 

\author{Masaki Izumi}
\address{Department of Mathematics\\ Graduate School of Science\\
Kyoto University\\ Sakyo-ku, Kyoto 606-8502\\ Japan}
\email{izumi@math.kyoto-u.ac.jp}


\thanks{Supported in part by the Grant-in-Aid for Scientific Research (B) 22340032, and 15H03623, JSPS}

\begin{document} 

\dedicatory{Dedicated to Professor Vaughan Jones for his sixtieth birthday.} 

\begin{abstract} We classify C$^*$ near-group categories by using Vaughan Jones theory of subfactors and the Cuntz algebra endomorphisms. 
Our results show that there is a sharp contrast between two essentially different cases, integral and irrational cases. 
When the dimension of the unique non-invertible object is an integer, we obtain a complete classification list, and it turns out that 
such categories are always group theoretical.   
When it is irrational, we obtain explicit polynomial equations whose solutions completely classify the C$^*$ near-group categories in this class. 
\end{abstract}

\maketitle

\tableofcontents

\section{Introduction} 
In his celebrated paper \cite{J83}, Vaughan Jones introduced the notion of index for subfactors, and opened up a totally new subject 
related to various fields in mathematics and mathematical physics, such as low dimensional topology, conformal field theory, quantum groups etc.
The theory of subfactors is often compared with the classical Galois theory of field extensions, and in fact a subfactor 
gives rise to tensor categories, which play an analogous role of the Galois group. 
More precisely, to a subfactor $N\subset M$ of finite index and finite depth, we can associate  
two kinds of fusion categories, consisting of $M-M$ bimodules and $N-N$ bimodules respectively, 
and two kinds of module categories of them consisting of $N-M$ bimodules and $M-N$ bimodules. 

Fusion categories may be considered as a generalization of the representation categories 
of finite groups, and commonly appear in the above mentioned fields related to subfactors. 
Their axioms were formulated by Etingof, Nikshych, and Ostrik \cite{ENO05} as follows: 
a fusion category over an algebraically closed field $k$ is a rigid semisimple 
$k$-linear tensor category with finitely many simple objects and finite dimensional 
morphism spaces such that the unit object $1$ is simple. 
Other than the representation category of a finite group $G$, a typical example of 
a fusion category, which may not be commutative in general, 
is given by $G$ itself. 
Namely, if every simple object of a fusion category $\cC$ is invertible, 
the set of equivalence classes of simple objects in $\cC$ is identified with a finite group 
$G$ obeying the fusion rules 
\begin{equation}\label{pointed}
g\otimes h\cong gh,\quad g,h\in G.
\end{equation}
It is known that such fusion categories, called pointed categories, 
are parametrized by the third cohomology group $H^3(G,k^\times)$ 
arising from the associativity constraint.    

A near-group category, formally introduced by Siehler \cite{S03}, is one step beyond 
the pointed categories, and has only one non-invertible simple object.  
Let $\cC$ be a near-group category with the group $G$ of the invertible simple objects,  
and let $\rho$ be the unique non-invertible object in $\cC$. 
Then the only possible fusion rules, apart from Eq.(\ref{pointed}), are 
$$g\otimes \rho\cong \rho\otimes g\cong \rho,\quad g\in G,$$
$$\rho\otimes \rho\cong \bigoplus_{g\in G}g\oplus \overbrace{\rho\oplus \rho\oplus \cdots\oplus \rho}^{m\; \mathrm{times}}.$$
Thus in the level of fusion rules, the multiplicity $m$ of $\rho$ in $\rho\otimes \rho$ is the only free parameter, which we call 
the multiplicity parameter of $\cC$.   
The based ring $\Z G+\Z\rho$ obeying the above relations is denoted by $K(G,m)$ in \cite{O13}. 
The classification of the near-group categories having a prescribed based ring $K(G,m)$ is a fundamental question 
in the subject. 

Many near-group categories appear in nature. 
For example, the Ising model in conformal field theory is a near-group category 
with $(G,m)=(\Z_2,0)$, and the representation category of the symmetric group $\fS_3$ 
is a near-group category with $(G,m)=(\Z_2,1)$. 
The $E_6$ subfactor provides a more exotic example with $(G,m)=(\Z_2,2)$. 
The recent developments of the classification of subfactors \cite{JMS14}, \cite{AMP15} show that quadratic categories, 
slight generalization of near-group categories, are rather typical (see \cite{GIS15}), and it is hard to construct 
subfactors neither related to (quantum) groups nor quadratic categories. 
So far the only known exception is the extended Haagerup subfactor constructed in \cite{BPMS12}. 

The first systematic classification result of a class of near-group categories was obtained by Tambara-Yamagami \cite{TY98}, 
and they completely classified the near-group categories with $m=0$, now called Tambara-Yamagami categories. 
They showed that such categories exist if and only if $G$ is abelian, and obtained complete classification invariants in terms of $G$. 
Another extreme case was treated by Ostrik \cite{O03}, where he showed that a near-group category with trivial $G$ 
exists if and only if $m=1$. 
The proof requires a much more sophisticated argument than one would think at first sight, 
and it uses braiding in a crucial way. 
Since then, there have been several attempts to classify various classes of near-group categories 
(see \cite{S03}, \cite{EGO04}, \cite{HH09}, \cite{O13}, \cite{NO}, \cite{EG14}, \cite{L14}).   

Fusion categories arising from subfactors form a special class, called C$^*$ fusion categories,  
whose morphism spaces are complex Banach spaces with $*$-structure obeying the C$^*$ condition. 
Throughout this note, we focus on near-group categories satisfying this condition, which we call 
C$^*$ near-group categories. 
It is known that every C$^*$ fusion category is uniquely realized in the category 
of bimodules of the hyperfinite II$_1$ factor, and similarly in the category of 
endomorphisms $\End(M)$ of the hyperfinite type III$_1$ factor $M$ (see \cite{HY00}, \cite{P95}). 
In the latter category, invertible objects are nothing but automorphisms of $M$, 
and one can utilize well-known results on group actions on operator algebras 
to analyse the structure of C$^*$ near-group categories.  
In fact, the third cohomology class of a pointed category for a finite group $G$ mentioned above 
can be identified with the Connes obstruction of a $G$-kernel (see \cite{C77}, \cite{J80}, \cite{S80}). 
Moreover, the author pointed out in \cite{I93} and \cite{I01} that C$^*$ near group categories realized in 
$\End(M)$ can be reconstructed from Cuntz algebra endomorphisms, which provides us a handy way 
to construct new examples. 
The main purpose in \cite{I01} was to compute the Drinfeld centers for concrete examples of tensor categories,  
and for that we deduced polynomial equations whose solutions give C$^*$ near group categories via 
Cuntz algebra endomorphisms. 
One of the purposes of this paper is to further pursue this approach, and to classify the C$^*$ near-group categories 
by using operator algebra techniques.

We summarize the main results of this paper now, which we prove in Section 3-6 
(see Theorem \ref{representation} and Theorem \ref{noncommutative}).  

\begin{theorem}\label{Main} 
Let $\cC$ be a C$^*$ near-group category with $G,\rho,$ and $m\neq 0$ as above, and  
let $d=\frac{m+\sqrt{m^2+4|G|}}{2}$ be the dimension of the object $\rho$. 
Then the following hold. 

If $d$ is rational, either of the following two cases occurs:
\begin{itemize} 
\item[(1)] $G$ is abelian and $m=|G|-1$.  
\item[(2)] $G$ is an extra-special 2-group of order $2^{2a+1}$ and $m=2^a$ with a natural number $a$. 
For each extra-special 2-group, there exist exactly 3 C$^*$ near group-categories. 
\end{itemize}

If $d$ is irrational, $G$ is abelian and $m$ is a multiple of $|G|$. 
\end{theorem} 

Siehler \cite[Theorem 1.2]{S03} showed, under the implicit assumption of $G$ being abelian, 
that if $m<|G|$, the group $G$ is cyclic, $m=|G|-1$, and $|G|+1$ is a prime power. 
On the other hand, Etingof-Gelaki-Ostrik \cite[Corollary 7.4]{EGO04} completely classified such fusion categories,  
and showed that there exist exactly three categories for $G=\Z_2$,  two categories for each of $\Z_3$ and $\Z_7$, 
and there exists a unique category for every other cyclic group $\Z_{q-1}$ with a prime power $q$. 
These results hold without the C$^*$ condition. 
The case (2) was overlooked in \cite{S03} for the reason stated above.

The proof of the statement for the irrational case can be found in \cite{NO} 
(and in \cite[Theorem 2(a)]{EG14} with an extra assumption).   
In the case of $m=|G|$, Evans-Gannon \cite[Theorem 4]{EG14} showed that the solutions of the polynomial equations 
obtained in \cite{I01} completely classify the C$^*$ near-group categories in this class. 
Moreover, they obtained a number of new solutions and computed the Drinfeld centers of the corresponding 
C$^*$ near-group categories following the approach established in \cite{I00}, \cite{I01}. 
We will deduce the polynomial equations for the general irrational case in Section 7-8, 
and show that their solutions completely classify such categories in Theorem \ref{irrcl}. 
We will treat the special case of $m=|G|$ in Section 9, and the case of $m=2|G|$ in Section 10. 
Recently an example with $(G,m)=(\Z_3,6)$ was discovered, and the latter case also draws attention of specialists 
(see \cite{EP12}, \cite{L14}, \cite{L15}). 
Theorem \ref{Z36} shows that there exist exactly two C$^*$ near group categories with $(G,m)=(\Z_3,6)$, 
and they are complex conjugate to each other. 

This paper is an extended version of author's personal note written around 2008, which roughly corresponds to 
the contents in Section 3-9 in the present version. 
Since then there have been several new developments about near-group categories, and some results in the original note 
are no longer new as already mentioned above. 
Nevertheless, the author believes that it is still worth publishing the old part because 
it lays the foundation for the other part of this note, and it still contains new results, for example, 
the case (2) in Theorem \ref{Main}. 
Section 2 and Section 10-13 are newly written in 2015. 

As already mentioned above, it is a folklore result that any C$^*$ fusion category uniquely embeds into $\End(M)$ 
for the hyperfinite type III$_1$ factor, and the uniqueness part is based on Popa's classification result for subfactors \cite{P95}.  
In Section 2, we clarify in what sense the uniqueness statement holds. 
The author would like to thank Roberto Longo for pointing out the author's vague understanding of the statement before, 
and Luca Giorgetti for pointing out a flaw in Theorem \ref{uniqueness} in the first version.  

In \cite{I93}, \cite{I01}, we showed that a C$^*$ near-group category with $m=|G|$ gives rise to so called 
a $2^G1$ subfactor. 
Generalizing this observation, we show in Section 11 that there exists a one-to-one correspondence between 
the C$^*$ near-group categories in the irrational case and a certain class of subfactors with specific principal graphs. 
We show that these subfactors are self-dual under a mild assumption. 

Section 12 is devoted to de-equivariantization of C$^*$ near group categories, which is a systematic account inspired by 
Ostrik's wonderful observation that the Haagerup category is related to a near-group category with $(G,m)=(\Z_3\times \Z_3,9)$ via 
de-equivariantization (see Example \ref{Z3Z3}). 
As a byproduct, we find new C$^*$ near-group categories with $(G,m)=(\Z_2\times \Z_2\times \Z_3,12)$ 
missing in \cite[Table 2]{EG14} (see Example \ref{Z2Z2Z3}). 

In subsection 13, we discuss equivariantization. 
For that purpose, we determined the structure of the automorphism groups of $C^*$ near-group categories in the irrational case 
in subsection 13.1.

The author would like thank David Evans, Terry Gannon, Pinhas Grossman, Vaughan Jones, Zhengwei Liu, 
Scott Morrison, Sebastien Palcoux, David Penneys, Victor Ostrik, and Noah Snyder for stimulating discussions and encouragement. 

\section{Preliminaries}
Our basic references are \cite{EGNO15} for tensor categories, \cite{EK98} for operator algebras and subfactors, 
and \cite{BKLR15} for the category of endomorphisms of von Neumann algebras. 
Every von Neumann algebra in this note is assumed to have separable predual. 

For a Hilbert space $\cH$, we denote by $\B(\cH)$ the set of bounded operators on $\cH$, and by $\cU(\cH)$ the set of unitaries on $\cH$. 
The identity operator of $\cH$ is denoted by $I_{\cH}$ or simply by $I$. 
For a unital C$^*$-algebra $A$, we denote by $\cU(A)$ the set of unitaries in $A$. 
The unit of $A$ is denoted by $I_A$ or simply by $I$. 

Let $M$ be a properly infinite factor. 
Then the set of unital endomorphisms $\End(M)$ forms a tensor category with the monoidal product $\rho\otimes \sigma$ of two objects 
$\rho, \sigma\in \End(M)$ given by the composition $\rho\circ \sigma$, and the morphism space from $\rho$ to $\sigma$ 
given by 
$$\Hom_{\End(M)}(\rho,\sigma)=\{T\in M;\; T\rho(x)=\sigma(x)T,\;\forall x\in M\}.$$
For simplicity, we denote $(\rho,\sigma)=\Hom_{\End(M)}(\rho,\sigma)$. 
In this tensor category, the monoidal product $T_1\otimes T_2$ of two morphisms $T_i\in (\rho_i,\sigma_i)$, $i=1,2$, are given by 
$$T_1\rho_1(T_2)=\sigma_1(T_2)T_1.$$
This is graphically expressed as
$$
\begin{xy}(0,0)*+[F]{T_1},(10,7)*+[F]{T_2} 
\ar(0,15);(0,3)^<{\rho_1}
\ar(10,15);(10,10)^<{\rho_2}
\ar(0,-3);(0,-8)^>{\sigma_1}
\ar(10,4);(10,-8)^>{\sigma_2}
\end{xy}
=
\begin{xy}(0,7)*+[F]{T_1},(10,0)*+[F]{T_2} 
\ar(0,15);(0,10)^<{\rho_1}
\ar(10,15);(10,3)^<{\rho_2}
\ar(0,4);(0,-8)^>{\sigma_1}
\ar(10,-3);(10,-8)^>{\sigma_2}
\end{xy}
.$$
By definition, two objects $\rho,\sigma$ are equivalent if and only if there exists a unitary $U\in \cU(M)$ satisfying 
$\rho=\Ad U\circ \sigma$, where $\Ad U$ is the inner automorphism of $M$ given by $\Ad U (x)=UxU^{-1}$. 
We denote by $[\rho]$ the equivalence class of $\rho$. 
The self-morphism space $(\rho,\rho)$ is nothing but the relative commutant $M\cap \rho(M)'$, 
and when this space consists of only scalars, we say that $\rho$ is irreducible (or simple). 

The morphism space $(\rho,\sigma)$ inherits the Banach space structure from $M$, and the $*$-operation 
of $M$ sends $(\rho,\sigma)$ to $(\sigma,\rho)$, which makes $\End(M)$ a C$^*$ tensor category 
(see \cite[Section 1]{BKLR15}). 
Moreover, if $\rho$ is irreducible, the space $(\rho,\sigma)$ is a Hilbert space with an inner product 
given by $T_1^*T_2=\inpr{T_1}{T_2}I_M$ for $T_1,T_2\in (\rho,\sigma)$. 
Throughout the paper, we assume that any functor between C$^*$ fusion categories preserves  
the $*$-structure. 

For $\rho\in \End(M)$, its dimension $d(\rho)$ is defined by $[M:\rho(M)]_0^{1/2}$, 
where $[M:\rho(M)]_0$ is the minimal index of $\rho(M)$ in $M$. 
We denote by $\End_0(M)$ the set of $\rho\in \End(M)$ with finite $d(\rho)$. 
The dimension function $\End_0(M)\ni \rho\mapsto d(\rho)$ is additive with respect to 
the direct sum operation and multiplicative with respect to the monoidal product operation. 
The tensor category $\End_0(M)$ is rigid in the following sense: for any $\rho\in \End_0(M)$, 
there exist $\brho\in \End_0(M)$, called the conjugate endomorphism of $\rho$, and two isometries 
$R_\rho\in (\id,\brho\circ \rho)$, $\overline{R}_\rho\in (\id,\rho\circ \rho)$ satisfying 
$$\overline{R}_\rho^*\rho(R_\rho)=R_\rho^*\brho(\overline{R}_\rho)=\frac{1}{d(\rho)}.$$
If $\rho$ is self-conjugate, either $\overline{R}_\rho=R_\rho$ 
or $\overline{R}_\rho=-R_\rho$ occurs. 
We say that $\rho$ is real in the former case, and $\rho$ is pseudo-real in the latter case. 

If we replace $\End(M)$ with the set of unital homomorphisms between two type III factors, 
the dimension function and conjugate morphisms still make sense, and we use the same notation as above (see \cite{I98}, \cite{BKLR15}).  

Every C$^*$ fusion category is realized as a category of bimodules of the hyperfinite II$_1$ factor 
(see \cite{HY00}), which implies the following statement by a tensor product trick. 

\begin{theorem}\label{embedding} Every C$^*$ fusion category is realized as a subcategory of $\End_0(M)$ for any 
hyperfinite type III factor $M$.  
\end{theorem} 

For uniqueness, we have the following statement, which is a consequence of Popa's classification 
theorem for amenable subfactors. 
Recall that a monoidal functor from a strict fusion category $\cC$ to another strict fusion category $\cD$ is a pair $(F,L)$ 
consisting of a functor $F:\cC\to \cD$ and natural isomorphisms 
$$L_{\rho,\sigma}\in \Hom_\cD(F(\rho)\otimes F(\sigma), F(\rho\otimes \sigma))$$ 
satisfying 
$$L_{\rho\otimes \sigma,\tau}\circ (L_{\rho,\sigma}\otimes I_{F(\tau)})
=L_{\rho,\sigma\otimes \tau}\circ (I_{F(\rho)}\otimes L_{\sigma,\tau})$$
for any $\rho,\sigma,\tau\in \cC$ (see \cite[Definition 2.4.1]{EGNO15}).  
We may and do assume $F(\mathbf{1}_{\cC})=\mathbf{1}_{\cD}$ and $L_{\mathbf{1}_\cC,\rho}=L_{\rho,\mathbf{1}_\C}=I_{F(\rho)}$.   
When $\cC$ and $\cD$ are C$^*$ categories, we further assume that $L_{\rho,\sigma}$ is a unitary.

\begin{theorem}\label{uniqueness} Let $M$ and $P$ be hyperfinite type III$_1$ factors, 
and let $\cC$ and $\cD$ be C$^*$ fusion categories embedded in $\End(M)$ and $\End(P)$ respectively. 
Let $(F,L)$ be a monoidal functor from $\cC$ to $\cD$ that is an equivalence of the two C$^*$ fusion categories 
$\cC$ and $\cD$. 
Then there exists a surjective isomorphism $\Phi:M\to P$ and unitaries $U_\rho\in P$ for each object $\rho \in \cC$ 
satisfying  
$$F(\rho)=\Ad U_\rho \circ \Phi \circ\rho\circ\Phi^{-1},$$
$$F(X)=U_\sigma\Phi(X)U_\rho^*,\quad X\in (\rho,\sigma),$$
$$L_{\rho,\sigma}=U_{\rho\circ\sigma}\Phi\circ\rho\circ\Phi^{-1}(U_\sigma^*)U_\rho^*=
U_{\rho\circ\sigma}U_\rho^*F(\rho)(U_\sigma^*).$$
\end{theorem}

Before proving the statement, let us recall Popa's classification theorem in the special 
case of finite depth subfactors of the hyperfinite type III$_1$ factor (see  \cite{P95}, \cite{M05}). 
Note that hyperfinite type III$_1$ factors are mutually isomorphic due to Haagerup \cite{H}.  
Let $N\subset M$ be an inclusion of hyperfinite type III$_1$ factors of finite index and finite depth, 
and let $l$ be an integer larger than the depth of the inclusion. 
Let 
$$M\supset N=N_0\supset N_1\supset \cdots \supset N_l$$
be the downward basic construction. 
In each step, the subfactor $N_{k+1}$ is uniquely determined up to inner conjugacy in $N_k$. 
The standard invariant of $N\subset M$ is determined by the following nested system of 
finite dimensional von Neumann algebras: 
$$\begin{array}{ccccccc}
M\cap N' &\subset &M\cap N_1' &\subset &\cdots &\subset &M\cap N_l'  \\
\cup & &\cup & & & &\cup  \\
\C &\subset &N\cap N_1' &\subset &\cdots &\subset &N\cap N_l' 
\end{array}.$$
Popa showed that there is a continuation of the downward basic construction 
$$N_l\supset N_{l+1}\supset N_{l+2}\supset \cdots$$
so that $M=M^{\mathrm{st}}\otimes R$ and $N=N^{\mathrm{st}}\otimes R$ hold, where 
$$M^{\mathrm{st}}=\bigvee_{k=0}^\infty (M\cap N_k'),\quad N^{\mathrm{st}}=\bigvee_{k=0}^\infty (N\cap N_k'),$$
and $R$ is the relative commutant of $M^{\mathrm{st}}$ in $M$, which is hyperfinite of type III$_1$. 
Since the core inclusion $N^{\mathrm{st}}\subset M^{\mathrm{st}}$ is completely determined by the standard invariant, 
it classifies $N\subset M$. 

Now we recall the precise statement we need in the proof of Theorem \ref{uniqueness}. 
Assume that $Q\subset P$ is another inclusion of hyperfinite type III$_1$ factors with the same standard invariant 
as that of $N\subset M$. 
Let 
$$P\supset Q\supset Q_1\supset \cdots\supset Q_l$$
be an arbitrary downward basic construction up to $l$ step. 
Popa's theorem implies that if $\Phi$ is an isomorphism from $M\cap N_l'$ onto $P\cap Q_l'$ with 
$\Phi(M\cap N_k')=P\cap Q_k'$ and $\Phi(N\cap N_k')=Q\cap Q_k'$ for any $1\leq k\leq l$, 
it extends to an isomorphism from $M$ onto $P$ mapping $N$ onto $Q$.  

\begin{proof}[Proof of Theorem \ref{uniqueness}] 
In the following arguments, whenever we apply the functor $F$ to morphisms, we need a special care 
because the same operator may belong to different morphism spaces. 
For example, an operator $X\in (\rho,\sigma)$ at the same time belongs to $(\rho\circ\mu,\sigma\circ \mu)$ 
as it is identified with $X\otimes I_\mu$. 
On the other hand, we have 
$$F(X\otimes I_\mu)=L_{\sigma,\mu}\circ(F(X)\otimes I_{F(\mu)})\circ L_{\rho,\mu}^{-1}.$$ 

Let $\Irr(\cC)$ be a complete system of representatives of the set of equivalence classes 
of irreducible endomorphisms in $\cC$. 
We may assume $\id\in \Irr(\cC)$, and $F(\id)=\id$. 
We choose an object
$$\gamma=\bigoplus_{\xi\in \Irr(\cC)}\xi\in \cC.$$
Then $\gamma(M)\subset M$ is a finite depth subfactor. 
Since $\gamma$ is self-conjugate, 
$$M\supset \gamma(M)\supset \gamma^2(M)\supset \gamma^3(M)$$
is a downward basic construction, and the standard invariant of the inclusion $\gamma(M)\subset M$ is determined by 
$$\begin{array}{ccccc}
(\gamma,\gamma) &\subset&(\gamma^2,\gamma^2) &\subset &(\gamma^3,\gamma^3)  \\
\cup &&\cup & &\cup  \\
\C&\subset &\gamma((\gamma,\gamma)) &\subset &\gamma((\gamma^2,\gamma^2)) 
\end{array}.$$
In the tensor category language, this is expressed as 
$$\begin{array}{ccccc}
\End_\cC(\gamma)\otimes I_\gamma\otimes I_\gamma &\subset&\End_\cC(\gamma\otimes\gamma)\otimes I_\gamma &\subset &
\End_\cC(\gamma\otimes \gamma\otimes \gamma)  \\
\cup &&\cup & &\cup  \\
\C I_\gamma\otimes I_\gamma\otimes I_\gamma&\subset &I_\gamma\otimes\End_\cC( \gamma)\otimes I_\gamma 
&\subset &I_\gamma\otimes \End_\cC(\gamma\otimes \gamma) 
\end{array}.$$
Let 
$$L_{\gamma,\gamma,\gamma}:=L_{\gamma^2,\gamma}L_{\gamma,\gamma}=L_{\gamma,\gamma^2}F(\gamma)(L_{\gamma,\gamma}),$$
or in the tensor category language, 
$$L_{\gamma,\gamma,\gamma}=L_{\gamma\otimes\gamma, \gamma}\circ (L_{\gamma,\gamma}\otimes I_{F(\gamma)})=
L_{\gamma,\gamma\otimes\gamma}\circ (I_{F(\gamma)}\otimes L_{\gamma,\gamma}).$$
Then we have $\Ad L_{\gamma,\gamma,\gamma}\circ F(\gamma)^3=F(\gamma^3)$, and $\Ad L_{\gamma,\gamma,\gamma}^*$ induces an isomorphism from 
$(F(\gamma^3),F(\gamma^3))$ onto $(F(\gamma)^3,F(\gamma)^3)$. 
We denote by $\Phi_0$ the composition of $F$ restricted to $(\gamma^3,\gamma^3)$ and $\Ad L_{\gamma,\gamma,\gamma}^*$, 
which is an isomorphism from $(\gamma^3,\gamma^3)$ onto 
$(F(\gamma)^3,F(\gamma)^3)$. 
We claim that $\Phi_0$ induces an isomorphism of the standard invariants of $\gamma(M)\subset M$ and 
$F(\gamma)(P)\subset P$. 
Indeed, for $Y\in (\gamma^2,\gamma^2)$, we have 
\begin{align*}
\Phi_0(Y\otimes I_\gamma)&=(L_{\gamma,\gamma}^*\otimes I_{F(\gamma)})
L_{\gamma\otimes\gamma,\gamma}^*F(Y\otimes I_\gamma)L_{\gamma\otimes\gamma,\gamma}(L_{\gamma,\gamma}\otimes I_{F(\gamma)})\\
 &=L_{\gamma,\gamma}^*F(Y)L_{\gamma,\gamma}\otimes I_{F(\gamma)},
\end{align*}
\begin{align*}
\Phi_0(I_\gamma\otimes Y)&=(I_{F(\gamma)}\otimes L_{\gamma,\gamma}^*)
L_{\gamma,\gamma\otimes\gamma}^*F(I_\gamma\otimes Y)L_{\gamma,\gamma\otimes\gamma}
(I_{F(\gamma)}\otimes L_{\gamma,\gamma})\\
 &=I_{F(\gamma)}\otimes L_{\gamma,\gamma}^*F(Y)L_{\gamma,\gamma}.
\end{align*}
In the same way, for $X\in \End_\cC(\gamma)$, we have 
$$\Phi_0(X\otimes I_\gamma\otimes I_\gamma)=F(X)\otimes I_{F(\gamma)}\otimes I_{F(\gamma)},$$
$$\Phi_0(I_\gamma\otimes X \otimes I_\gamma)=I_{F(\gamma)}\otimes F(X)\otimes I_{F(\gamma)},$$
$$\Phi_0(I_\gamma \otimes I_\gamma\otimes X)=I_{F(\gamma)}\otimes I_{F(\gamma)}\otimes F(X),$$
which shows the claim. 
Thus $\Phi_0$ extends to an isomorphism $\Phi$ from 
$M$ onto $P$ satisfying $\Phi(\gamma(M))=F(\gamma)(P)$. 

By construction and the above computation, we have $\Phi(Y)=L_{\gamma,\gamma}^*F(Y)L_{\gamma,\gamma}$ 
and $\Phi(\gamma(Y))=F(\gamma)(\Phi(Y))$ for $Y\in (\gamma^2,\gamma^2)$, 
and we have $\Phi(X)=F(X)$ and $\Phi(\gamma(X))=F(\gamma)(\Phi(X))$ for $X\in (\gamma,\gamma)$.

We choose an isometry $V_\xi\in (\xi,\gamma)$ for each $\xi\in \Irr(\cC)$. 
Then we have 
$$\gamma(x)=\sum_{\xi\in \Irr(\cC)}V_\xi\xi(x)V_\xi^*,$$
Since $V_\xi V_\xi^*\in (\gamma,\gamma)$, we get
$$\Phi(V_\xi V_\xi^*)=F(V_\xi V_\xi^*)=F(V_\xi)F(V_\xi)^*,$$
and $\Phi(V_\xi)$ and $F(V_\xi)$ are isometries with the same range projection. 
Thus there exists a unitary $W_\xi\in \cU(P)$ for each $\xi\in \Irr(\cC)$ satisfying $\Phi(V_\xi)=F(V_\xi)W_\xi$. 

Since $\Phi(\gamma(M))=F(\gamma)(P)$, there exists an isomorphism $\varphi$ from $M$ onto $P$ satisfying 
$\Phi\circ \gamma=F(\gamma)\circ \varphi$. 
On one hand, 
$$F(\gamma)\circ \varphi(x)=\sum_{\xi\in \Irr(\cC)}F(V_\xi)F(\xi)(\varphi(x))F(V_\xi)^*,$$
and on the other hand, 
$$\Phi\circ \gamma(x)=\sum_{\xi\in \Irr(\cC)}\Phi(V_\xi)\Phi \circ \xi(x)\Phi(V_\xi)^*
=\sum_{\xi\in \Irr(\cC)}F(V_\xi)W_\xi\Phi \circ \xi(x)W_\xi^*F(V_\xi)^*.$$
This implies $F(\xi)\circ \varphi=\Ad W_\xi \circ \Phi\circ \xi$, 
and in particular $\varphi=\Ad W_{\id}\circ \Phi$ in the case with $\xi=\id$. 
Thus we obtain $F(\xi)=\Ad U_\xi\circ \Phi\circ \xi\circ \Phi^{-1}$ with 
$$U_\xi=W_\xi\Phi\circ \xi\circ \Phi^{-1}(W_\id^*)=F(\xi)(W_\id^*)W_\xi.$$

Let $Z\in (\gamma,\gamma^2)$. 
Then $Y_1=ZV_\id^*$ belongs to $(\gamma^2,\gamma^2)$, and $F(Y_1)=L_{\gamma,\gamma}\Phi(Y_1)L_{\gamma,\gamma}^*$. 
Since $Y_1$ should be interpreted as $Z\circ (V_\id^*\otimes I_\gamma)$, the left-hand side is 
$F(Z)F(V_\id)^*L_{\gamma,\gamma}^*$. 
Thus we get 
$$F(Z)=L_{\gamma,\gamma}\Phi(Z)\Phi(V_\id)^*F(V_\id)=L_{\gamma,\gamma}\Phi(Z)W_\id^*.$$
On the other hand, the operator $Y_2=Z\gamma(V_\id^*)$ belongs to $(\gamma^2,\gamma^2)$, and 
we have $F(Y_2)=L_{\gamma,\gamma}\Phi(Y_2)L_{\gamma,\gamma}^*$. 
Since $Y_2$ should be interpreted as $Z\circ (I_\gamma\otimes V_\id^*)$, the left hand side is 
$F(Z)F(\gamma)(V_\id)^*L_{\gamma,\gamma}^*$, and we get 
\begin{align*}
\lefteqn{F(Z)=L_{\gamma,\gamma}\Phi(Z)\Phi(\gamma(V_\id)^*)F(\gamma(V_\id))} \\
 &=L_{\gamma,\gamma}\Phi(Z)F(\gamma)(W_\id\Phi(V_\id)^*W_\id^*)F(\gamma)(F(V_\id)) \\
 &=L_{\gamma,\gamma}\Phi(Z)F(\gamma)(F(V_\id)^*W_\id^*F(V_\id)),
\end{align*}
and 
$$\Phi(Z)F(\gamma)(F(V_\id)^*W_\id^*F(V_\id))=\Phi(Z)W_\id^*.$$
Setting $Z=V_\id$ and multiplying the both sides by $\Phi(V_\id)^*$ from left, 
we get 
$$F(\gamma)(F(V_\id)^*W_\id F(V_\id))=W_\id,$$ which implies
\begin{equation}\label{unique1} F(V_\eta)^*W_\id  F(V_\eta)=F(\eta)(F(V_\id)^*W_\id F(V_\id)).
\end{equation}

Let $\xi, \eta, \zeta\in \Irr(\cC)$, and let $X\in (\zeta,\xi\circ \eta)$. 
Since 
$\Phi(X)\in (\Phi\circ \zeta\circ \Phi^{-1},\Phi\circ \xi\circ \eta\circ\Phi^{-1})$, 
$\Phi\circ \zeta\circ \Phi^{-1}=\Ad U_\zeta^*\circ F(\zeta)$ and 
$$\Phi\circ \xi\circ \eta\circ \Phi^{-1}=\Ad(U(\xi)^*F(\xi)(U(\eta))^*)\circ F(\xi)\circ F(\eta),$$
we have 
$$F(\xi)(U_\eta)U_\xi\Phi(X)U_\zeta^*=U_\xi\Phi\circ \xi\circ \Phi^{-1}(U_\eta)\Phi(X)U_\zeta^*\in 
(F(\zeta),F(\xi)\circ F(\eta)).$$
We claim that this coincides with $L_{\xi,\eta}^*F(X)$. 
Indeed, since $Z=\gamma(V_\eta)V_\xi XV_\zeta^*\in (\gamma,\gamma^2)$, we have 
$F(Z)=L_{\gamma,\gamma}\Phi(Z)F(\gamma)(F(V_\id)^*W_\id^*F(V_\id))$. 
Since $Z$ should be understood as $(I_\gamma\otimes V_\eta)\circ (V_\xi\otimes I_\eta)\circ X\circ V_\zeta^*$, 
the left-hand side is 
\begin{align*}
 &(L_{\gamma,\gamma}F(\gamma)(F(V_\eta))L_{\gamma,\eta}^*)(L_{\gamma,\eta}F(V_\xi)L_{\xi,\eta}^*)F(X)F(V_\zeta)^* \\
 &=L_{\gamma,\gamma}F(\gamma)(F(V_\eta))F(V_\xi)L_{\xi,\eta}^*F(X)F(V_\zeta)^*.
\end{align*}
The right-hand side is 
\begin{align*}
\lefteqn{L_{\gamma,\gamma}\Phi(\gamma(V_\eta))\Phi(V_\xi)\Phi(X)\Phi(V_\zeta)^*F(\gamma)(F(V_\id)^*W_\id^*F(V_\id))} \\
 &=L_{\gamma,\gamma}F(\gamma)(W_\id \Phi(V_\eta)W_\id^*)F(V_\xi)W_\xi\Phi(X)W_\zeta^*F(V_\zeta)^* F(\gamma)(F(V_\id)^*W_\id^*F(V_\id))\\
 &=L_{\gamma,\gamma}F(\gamma)(W_\id F(V_\eta)W_\eta W_\id^*)F(V_\xi)W_\xi\Phi(X)W_\zeta^*F(V_\zeta)^* F(\gamma)(F(V_\id)^*W_\id^*F(V_\id))\\
\end{align*}
Thus
\begin{align*}
\lefteqn{L_{\xi,\eta}^*F(X)} \\
 &=F(\xi)(F(V_\eta)^*W_\id F(V_\eta)W_\eta W_\id^*)W_\xi\Phi(X)W_\zeta^*
 F(\zeta)(F(V_\id)^*W_\id^*F(V_\id))\\
&=F(\xi)(F(\eta)(F(V_\id)^*W_\id F(V_\id))W_\eta W_\id^*)W_\xi\Phi(X)W_\zeta^*F(\zeta)(F(V_\id)^*W_\id^*F(V_\id))\\
&=F(\xi)\circ F(\eta)(F(V_\id)^*W_\id F(V_\id)W_\id)  F(\xi)(U_\eta)U_\xi\Phi(X)U_\zeta^*\\
&\times F(\zeta)(W_\id^*F(V_\id)^*W_\id^*F(V_\id))\\
&=F(\xi)(U_\eta)U_\xi\Phi(X)U_\zeta^*,
\end{align*}
where we used Eq.(\ref{unique1}), and $F(\xi)(U_\eta)U_\xi\Phi(X)U_\zeta^*\in (F(\zeta),F(\xi)\circ F(\eta))$. 
This shows that the claim holds. 

For any object $\rho\in \cC$, we choose an orthonormal basis $\{V(\xi,\rho)_i\}_{i=1}^{\dim (\xi,\rho)}$ 
of $(\xi,\rho)$ and define a unitary $U_\rho\in \cU(P)$ by 
$$U_\rho=\sum_{\xi,i}F(V(\xi,\rho)_i)U_\xi \Phi(V(\xi,\rho)_i)^*.$$
Then it does not depend on the choice of the orthonormal basis, and coincides with the previous 
definition if $\rho\in \Irr(\cC)$.  
It is straightforward to show 
$$F(\rho)=\Ad U_\rho\circ \Phi\circ \rho\circ \Phi^{-1}.$$ 

Let $\rho, \sigma\in \cC$, and let $X\in (\rho,\sigma)$. 
For $\xi,\eta\in \Irr(\cC)$, we have $V(\xi,\sigma)_i^*XV(\eta,\rho)_j\in (\eta,\xi)$, 
which vanishes if $\xi\neq \eta$. 
Since $\xi$ is irreducible, the restriction of $\Phi$ on $(\xi,\xi)=\C I_\rho$ coincides with that of $F$ on $(\xi,\xi)$, 
and 
$$F(V(\xi,\sigma)_i)^*F(X)F(V(\eta,\rho)_j)=\delta_{\xi,\eta}\Phi(V(\xi,\sigma)_i^*)\Phi(X)\Phi(V(\eta,\rho)_j).$$
Thus
\begin{align*}
\lefteqn{F(X)=\sum_{\xi,\eta\in \Irr(\cC)}\sum_{i,j} F(V(\xi,\sigma)_i) F(V(\xi,\sigma)_i)^*F(X)
F(V(\eta,\rho)_j) F(V(\eta,\rho)_j)^*  }\\
 &=\sum_{\xi\in \Irr(\cC)}\sum_{i,j}
 F(V(\xi,\sigma)_i) \Phi(V(\xi,\sigma)_i^*)\Phi(X)\Phi(V(\xi,\rho)_j)F(V(\xi,\rho)_j) \\
 &=\sum_{\xi\in \Irr(\cC)}\sum_{i,j}
 F(V(\xi,\sigma)_i)U_\xi \Phi(V(\xi,\sigma)_i^*)\Phi(X)\Phi(V(\xi,\rho)_j)U_\xi^*F(V(\xi,\rho)_j) \\
 &=U_\sigma \Phi(X)U_\rho^*.
\end{align*}

Let $\zeta\in \Irr(\cC)$. Then
$\{V(\xi,\rho)_i\xi(V(\eta,\sigma)_j)V(\zeta,\xi\circ \eta)_k\}_{\xi,\eta,i,j,k}$ 
is an orthonormal basis of $(\zeta,\rho\circ \sigma)$, and we have 
\begin{align*}
U_{\rho\circ \sigma}
 &=\sum_{\xi,\eta,\zeta\in \Irr(\cC)}\sum_{i,j,k}F(V(\xi,\rho)_i\xi(V(\eta,\sigma)_j)V(\zeta,\xi\circ \eta)_k)\\
 &\times U_\zeta\Phi(V(\zeta,\xi\circ \eta)_k^*\xi(V(\eta,\sigma)_j)^*V(\xi,\rho)_i^*). \\
\end{align*}
Here $V(\xi,\rho)_i\xi(V(\eta,\sigma)_j)V(\zeta,\xi\circ \eta)_k$ should be interpreted as 
$$(V(\xi,\rho)_i\otimes I_\sigma)\circ (I_\xi\otimes V(\eta,\sigma)_j)\circ V(\zeta,\xi\circ\eta)_k,$$
and we have  
\begin{align*}
\lefteqn{F(V(\xi,\rho)_i\xi(V(\eta,\sigma)_j)V(\zeta,\xi\circ \eta)_k)} \\
 &=(L_{\rho,\sigma}F(V(\xi,\rho)_i)L_{\xi,\sigma}^*)(L_{\xi,\sigma}F(\xi)(F(V(\eta,\sigma)_j))L_{\xi,\eta}^*)
 F(V(\zeta,\xi\circ \eta)_k)\\
 &=L_{\rho,\sigma}F(V(\xi,\rho)_i)F(\xi)(F(V(\eta,\sigma)_j))L_{\xi,\eta}^*F(V(\zeta,\xi\circ \eta)_k).
\end{align*}
Note that we have already seen 
$$L_{\xi,\eta}^*F(V(\zeta,\xi\circ \eta)_k)=F(\xi)(U_\eta)U_\xi\Phi(V(\zeta,\xi\circ \eta)_k)U_\zeta^*.$$
Thus
\begin{align*}
\lefteqn{U_{\rho\circ \sigma}} \\
 &=L_{\rho,\sigma}\sum_{\xi,\eta,\zeta\in \Irr(\cC)}\sum_{i,j,k}F(V(\xi,\rho)_i)F(\xi)(F(V(\eta,\sigma)_j))
 F(\xi)(U_\eta)U_\xi\Phi(V(\zeta,\xi\circ \eta)_k)U_\zeta^*\\
 &\times U_\zeta\Phi(V(\zeta,\xi\circ \eta)_k)^*U_\xi^*F(\xi)(\Phi(V(\eta,\sigma)_j))^*U_\xi\Phi(V(\xi,\rho)_i^*) \\
 &=L_{\rho,\sigma}\sum_{\xi,\eta\in \Irr(\cC)}\sum_{i,j}F(V(\xi,\rho)_i)F(\xi)(F(V(\eta,\sigma)_j)U_\eta\Phi(V(\eta,\sigma)_j)^*)U_\xi\Phi(V(\xi,\rho)_i^*) \\
&=L_{\rho,\sigma}\sum_{\xi\in \Irr(\cC)}\sum_{i}F(\rho)(U_\sigma)F(V(\xi,\rho)_i)U_\xi\Phi(V(\xi,\rho)_i^*) \\
&=L_{\rho,\sigma}F(\rho)(U_\sigma)U_\rho. 
\end{align*}
This finishes the proof. 
\end{proof}

Applying the above theorem to the case with $M=P$ and $\cC=\cD$, we obtain the following statement.  

\begin{cor} Let $M$ be a hyperfinite type III$_1$ factor, let $\cC\subset \End_0(M)$ be a 
C$^*$ fusion category. 
Then up to a natural isomorphism, every automorphism of $\cC$ is induced by an automorphism $\Phi$ of $M$ 
in the following sense: it is given by $\rho\mapsto \Phi\circ \rho\circ \Phi^{-1}$ for an object $\rho$ and by 
$X\mapsto \Phi(X)$ for a morphism $X$. 
\end{cor}

One of the main tools in this note is the Cuntz algebra $\cO_n$, and we summarize the main feature of it here. 
Let $n$ be an integer larger than 1. 
The Cuntz algebra $\cO_n$ is the universal C$^*$-algebra with generators $\{S_i\}_{i=1}^n$ and relations 
$$S_i^*S_j=\delta_{i,j}I,$$
$$\sum_{i=1}^nS_iS_i^*=I.$$
The most peculiar property of the Cuntz algebra is that it is at the same time universal and simple (see \cite{Cu77}). 
Therefore if $\{T_i\}_{i=1}^n$ are noncommutative polynomials of the generators obeying the same relation 
as the defining relation, then there exists a unique endomorphism $\sigma\in \End(\cO_n)$ satisfying $\sigma(S_i)=T_i$.  

\section{Basic ingredients}\label{basic}
Let $G$ be a finite group of order $n$. 
We would like to classify a C$^*$ near group category $\cC$ with group $G$ and the multiplicity parameter $m$. 
Throughout this note, we assume that $G$ is not trivial and $m\neq 0$ because the two cases are completely understood 
as we mentioned in Introduction. 
In this section, we deduce basic ingredients to determine the structure of such $\cC$. 

Thanks to Theorem \ref{embedding} and Theorem \ref{uniqueness}, we may assume that $\cC$ is a subcategory of 
$\End(M)$ where $M$ is the hyperfinite type III$_1$ factor, 
and $\cC$ is generated by a single irreducible endomorphism $\rho\in \End_0(M)$ satisfying the following fusion rules:        
$$[\rho]^2=\bigoplus_{g\in G}[\alpha_g]\oplus m[\rho],$$
$$[\alpha_g][\alpha_h]=[\alpha_{gh}],$$
$$[\alpha_g][\rho]=[\rho][\alpha_g]=[\rho],$$
where the map $\alpha:G\rightarrow \Aut(M)$ induces an injective homomorphism from $G$ into $\Out(M)$. 
Since $[\alpha_g][\rho]=[\rho]$, we can arrange $\alpha$ so that $\alpha_g\circ \rho=\rho$ holds for all $g\in G$. 
Then we have $\alpha_g\circ \alpha_h=\alpha_{gh}$, that is, the map $\alpha$ is an action of $G$ on $M$, and in particular, 
the pointed subcategory generated by $\alpha_G$ has trivial third cohomology. 
Indeed, from the fusion rules, there exists a unitary $U(g,h)\in M$ satisfying $\alpha_g\circ \alpha_h=\Ad U(g,h)\circ \alpha_{gh}$. 
On the other hand, we have 
$$\rho=\alpha_g\circ \alpha_h\circ \rho=\Ad U(g,h)\circ \alpha_{gh}\circ \rho=\Ad U(g,h)\circ \rho.$$
Since $\rho$ is irreducible, the unitary $U(g,h)$ is a scalar, and we get the claim. 

We set 
$$d=d(\rho)=\frac{m+\sqrt{m^2+4n}}{2},$$
which is the dimension of $\rho$ satisfying $d^2=n+md$. 
Throughout this note, we keep using the symbols $G$, $m$, $n$, $d$ in this sense. 

We fix an isometry $S_e\in (\id,\rho^2)$. 
Since $\rho$ is self-conjugate, we have 
\begin{equation}
S_e^*\rho(S_e)=\frac{\epsilon}{d},\quad \epsilon \in \{1,-1\}.
\end{equation}
When $\epsilon=1$ (resp. $\epsilon=-1$), we say that $\rho$ is a real (resp. pseudo-real) sector. 
Graphically, we have 
$$\sqrt{d}S_e=
\begin{xy}\ar@(u,u)(10,-5);(0,-5)^>\rho
\end{xy}\quad ,
$$
$$\epsilon \sqrt{d}S_e^*=
\begin{xy}\ar@(d,d)(10,5);(0,5)^<\rho
\end{xy}\quad .
$$

We set $S_g=\alpha_g(S_e)$. 
Then $(\alpha_g,\rho^2)=\C S_g$. 
Let $\cK=(\rho,\rho^2)$, and let $\{T_i\}_{i=1}^m$ be an orthonormal basis of $\cK$. 
Then $\{S_g\}_{g\in G}\cup \{T_i\}_{i=1}^m$ satisfies the Cuntz algebra $\cO_{n+m}$ relation, and in particular, 
\begin{equation}\label{support}
\sum_{g\in G}S_gS_g^*+\sum_{i=1}^mT_iT_i^*=I. 
\end{equation}
We set $P=\sum_{g\in G}S_gS_g^*$ and $Q=\sum_{i=1}^mT_iT_i^*$, which are projections. 

Let $\cK^s\cK^{*t}$ be the linear span of $\{T_{i_1}T_{i_2}\cdots T_{i_s}T_{j_t}^*T_{j_{t-1}}^*\cdots T_{j_1}^*\}$. 
We identify $\cK^s$ with $\cK^{\otimes s}$ via the identification of $T_{i_1}T_{i_2}\cdots T_{i_s}$ with 
$T_{i_1}\otimes T_{i_2}\otimes \cdots \otimes T_{i_s}$. 
We identify $\cK^s\cK^{*t}$ with $\B(\cK^t,\cK^s)$ by left multiplication, and $\cK^2\cK^{2*}$ with $\B(\cK)\otimes \B(\cK)$. 
For example, we denote $T_{i_1}T_{j_1}^*\otimes T_{i_2}T_{j_2}^*=T_{i_1}T_{i_2}T_{j_2}^*T_{j_1}^*$. 
We abuse this notation and denote $T_iT_j^*\otimes x=T_ixT_j^*$ for any $x\in \cO_{n+m}$. 
With this notation, we have 
\begin{equation}
(\rho^2,\rho^2)=\bigoplus_{g\in G}\C S_gS_g^*\oplus \B(\cK). 
\end{equation}

Since $(\rho,\rho\circ \alpha_g)\subset (\rho^2,\rho^2)$, 
we can choose a (a priori projective) unitary representation $\{U(g)\}_{g\in G}$ in $(\rho^2,\rho^2)$ such that 
$(\rho,\rho\circ\alpha_g)=\C U(g)$. 
Since $U(g)S_e\in (\id,\rho^2)$, we can normalize $U(g)$ so that $U(g)S_e=S_e$ holds. 
Then $\{U(g)\}$ is a genuine representation of the form 
\begin{equation}\label{U}
U(g)=\sum_{h\in G}\chi_h(g)S_hS_h^*+U_{\cK}(g),
\end{equation}
where $\chi_h\in \Hom(G,\T)$ and $\{U_{\cK}(g)\}_{g\in G}$ is a unitary representation of $G$ in $\B(\cK)$.  
Since $\alpha_g\circ \rho=\rho$, we have $\alpha_g(\cK)=\cK$, and there exists a unitary representation $\{V(g)\}_{g\in G}$ 
in $\B(\cK)$ such that $\alpha_g(T)=V(g)T$ for all $T\in \cK$ and $g\in G$.

For $T\in \cK$, we set 
\begin{equation}\label{defj_1} j_1(T)=\sqrt{d}T^*\rho(S_e)\in \cK,\end{equation}
\begin{equation}\label{defj_2} j_2(T)=\sqrt{d}\rho(T)^*S_e\in \cK.\end{equation}
Then the Frobenius reciprocity (\cite{I98}) implies 

\begin{lemma}\label{j} The maps $j_1$ and $j_2$ are anti-linear isometries of $\cK$ satisfying 
\begin{equation}\label{involution} j_1^2=j_2^2=\epsilon,\end{equation} 
\begin{equation}\label{j1} V(g)j_1=j_1V(g),\quad g\in G,\end{equation} 
\begin{equation}\label{j2} U_{\cK}(g)j_2=j_2V(g),\quad g\in G.\end{equation}
In particular, the two unitary representations $U_{\cK}$ and $V$ are unitarily equivalent with an intertwining unitary 
$j_2\circ j_1^{-1}$. 
\end{lemma}

\begin{remark} For those readers who would like to reproduce our arguments in this paper without assuming 
the C$^*$ condition on near-group categories, 
we briefly give graphical expressions of the intertwiners appearing so far. 
Let $\cC$ be a pivotal near-group category whose simple objects are $G\cup \{\rho\}$. 
We assume that $\cC$ is strict. 
We first choose and fix a non-zero homomorphism $S_e\in \Hom(\id,\rho\otimes \rho)$ and isomorphisms $f_g\in \Hom(g\otimes \rho,\rho)$ for $g\in G$. 
Then there exist unique isomorphisms $m_{g,h}\in \Hom(g\otimes h,gh)$ for $g,h\in G$ to make the following diagrams commutative:
$$
\begin{CD}
g\otimes h\otimes \rho @>{I_g\otimes f_h}>> g\otimes \rho \\
@V{m_{g,h}\otimes I_\rho}VV  @VV{f_g}V\\
gh\otimes \rho @>>{f_{gh}}>  \rho
\end{CD}. 
$$
With this family $\{m_{g,h}\}_{g,h\in G}$, we can show that the following diagrams are commutative, 
$$
\begin{CD}
g\otimes h\otimes k @>{I_g\otimes m_{h,k}}>> g\otimes gh\\
@V{m_{g,h}\otimes I_k}VV @VV{m_{g,hk}}V\\
gh\otimes k @>>{m_{gh,k}}> ghk
\end{CD}\quad ,
$$ 
and in consequence, we can see that the group part has trivial third cohomology. 

The homomorphism $S_g\in \Hom(g,\rho\otimes \rho)$ is given by 
$$S_g=(f_g\otimes I_\rho)\circ (I_g\otimes S_e)=
\begin{xy}(0,0)*+[F]{f_g},(10,10)*+[F]{S_e} 
\ar(-2,15);(-2,3)_<g
\ar(-2,-3);(-2,-10)^>\rho
\ar(8,7);(2,3)_>\rho
\ar(12,7);(12,-10)^>\rho
\end{xy}
.$$ 
We choose $S_e^*\in \Hom (\rho\otimes \rho,\id)$ satisfying $S_e^*\circ S_e=1$, and set 
$$S_g^*=(I_g\otimes S_e^*)\circ (f_g^{-1}\otimes I_\rho)\in \Hom(\rho\otimes \rho,g).$$
Then we have $S_g^*\circ S_g=I_g$. 

Setting 
$$V(g):\Hom (\rho,\rho\otimes \rho)\ni T\mapsto (f_g\otimes I_\rho)\circ(I_g\otimes T)\circ f_g^{-1}=
\begin{xy}(7,0)*+[F]{T}, (5,15)*+[F]{f_g^{-1}},(0,-13)*+[F]{f_g}
\ar(5,25);(5,20)^(0.3)\rho
\ar(7,11);(7,3)^(0.5)\rho
\ar@/_/(3,11);(-2,-10)_(0.7)g
\ar(5,-3);(2,-10)^(0.5)\rho
\ar(0,-16);(0,-22)_>\rho
\ar(9,-3);(14,-22)^>\rho
\end{xy}
\quad \in \Hom(\rho,\rho\otimes \rho),$$ 
we get a representation $\{V(g)\}_{g\in G}$ of $G$ on $\Hom(\rho,\rho\otimes \rho)$. 

The homomorphism $U(g)\in \Hom(\rho,\rho\otimes g)$ is determined by the normalization condition 
$$S_e=(I_\rho\otimes f_g)\circ (U(g)\otimes I_\rho)\circ S_e=
\begin{xy}
(0,0)*+[F]{U(g)},(10,-10)*+[F]{f_g},(5,10)*+[F]{S_e}
\ar(-2,-3);(-2,-17)_>\rho
\ar(2,-3);(8,-7)^(.6)g
\ar(3,7);(0,3)_<<<\rho
\ar(7,7);(12,-7)^<<<\rho
\ar(10,-13);(10,-17)^>\rho
\end{xy}
.$$
Then they satisfy the following relation 
$$(I_\rho\otimes m_{g,h})\circ (U(g)\otimes I_h)\circ U(h)=U(gh),$$
$$\begin{xy}(5,10)*+[F]{U(h)},(0,0)*+[F]{U(g)},(10,-10)*+[F]{m_{g,h}}
\ar(5,17);(5,13)^\rho
\ar(3,7);(0,3)_>\rho
\ar(7,7);(12,-7)^>h
\ar(-2,-3);(-2,-17)^>\rho
\ar(2,-3);(8,-7)_g
\ar(10,-13);(10,-17)^>{gh}
\end{xy}
=\begin{xy}
(0,0)*+[F]{U(gh)}
\ar(0,17);(0,3)^\rho
\ar(-2,-3);(-2,-17)^>\rho
\ar(2,-3);(10,-17)^>{gh}
\end{xy}.
$$
Using this, we can see that 
$$G\ni g\mapsto (I_\rho\otimes f_g)\circ (U(g)\otimes I_\rho)=
\begin{xy}(0,6)*+[F]{U(g)},(10,-6)*+[F]{f_g}
\ar(0,15);(0,9)^\rho
\ar(-2,3);(-2,-15)_>\rho
\ar(2,3);(8,-3)^(0.5)g
\ar(12,15);(12,-3)^\rho
\ar(10,-9);(10,-15)^>\rho
\end{xy}
\quad \in \Hom(\rho\otimes\rho,\rho\otimes \rho)$$ 
gives a representation of $G$. 
It is easy to show 
$$(1\otimes S_e^*)\circ(I_\rho\otimes f_g\otimes I_\rho)\circ (U(g)\otimes I_\rho\otimes I_\rho)\circ (I_\rho\otimes S_e)=\delta_{g,e}I_\rho,
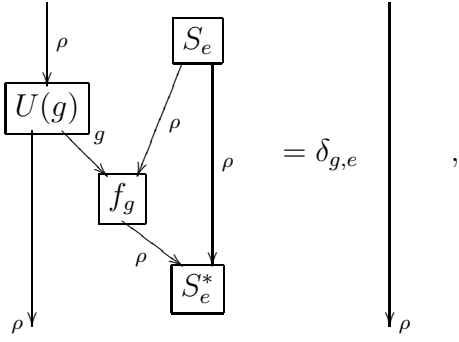
$$
\begin{equation}\label{trace}
\begin{xy}(0,6)*+[F]{U(g)},(10,-6)*+[F]{f_g},(20,15)*+[F]{S_e},(20,-18)*+[F]{S_e^*}
\ar(0,20);(0,9)^\rho
\ar(-2,3);(-2,-23)_>\rho
\ar(2,3);(8,-3)^(0.5)g
\ar(18,12);(12,-3)^\rho
\ar(10,-9);(18,-15)_(0.5)\rho
\ar(22,12);(22,-15)^(0.5)\rho
\end{xy}
\quad =\delta_{g,e}\quad
\begin{xy}
\ar(0,20);(0,-23)^>\rho
\end{xy}\quad, \nonumber
\end{equation}
which is essentially the right categorical trace of $(1\otimes f_g)\circ (U(g)\otimes I_\rho)$. 
Choosing a basis $\{T_i\}_{i=1}^m$ of 
$\Hom(\rho,\rho\otimes \rho)$ and a basis $\{T_i^*\}_{i=1}^m$ of $\Hom(\rho\otimes \rho,\rho)$ satisfying 
$T_i^*\circ T_j=\delta_{i,j}I_\rho$, we have 
$$I_{\rho\otimes \rho}=\sum_{g\in G}S_g\circ S_g^*+\sum_{i=1}^m T_i\circ T_i^*.$$
Although neither $j_1$ nor $j_2$ can be defined without C$^*$ condition, they can be replaced by linear maps 
from $\Hom(\rho\otimes\rho,\rho)$ to $\Hom(\rho,\rho\otimes \rho)$ given by the Frobenius reciprocity. 
\end{remark}

We get back to our original situation with a C$^*$ near-group category $\cC$ realized inside $\End(M)$. 

\begin{lemma}\label{rhoS}For any $g\in G$, we have 
\begin{equation}\rho(S_e)=\frac{\epsilon}{d}\sum_{h\in G}S_h+\frac{1}{\sqrt{d}}\sum_{i=1}^mT_ij_1(T_i),
\end{equation}
\begin{equation}\label{rhoSg}
\rho(S_g)=U(g)\rho(S_e)U(g)^*.
\end{equation}

\end{lemma}

\begin{proof} The second statement follows from $\Ad U(g)\circ \rho=\rho\circ \alpha_g$ and $S_g=\alpha_g(S_e)$, 
From Eq.(\ref{support}), we obtain 
\begin{align*}\rho(S_e) 
&=\sum_{h\in G}S_hS_h^*\rho(S_e)+\sum_{i=1}^m T_iT_i^*\rho(S_e)=
\sum_{h\in G}\alpha_h(S_eS_e^*\rho(S_e))+\frac{1}{\sqrt{d}}\sum_{i=1}^m T_ij_1(T_i)\\  
&=\frac{\epsilon}{d}\sum_{h\in G}\alpha_h(S_e)+\frac{1}{\sqrt{d}}\sum_{i=1}^m T_ij_1(T_i)
=\frac{\epsilon}{d}\sum_{h\in G}S_h+\frac{1}{\sqrt{d}}\sum_{i=1}^m T_ij_1(T_i).  
\end{align*}
\end{proof}

\begin{lemma} 
There exists a unique linear map $l:\cK\rightarrow \cK^2\cK^*$ such that for any $T\in \cK$ 
\begin{equation}\label{rhoT}
\rho(T)=\frac{1}{\sqrt{d}}\sum_{h\in G}S_h\alpha_h(j_2(T)^*)+\sum_{h\in G}\alpha_h(j_2\circ j_1^{-1}(T))S_hS_h^*+l(T).
\end{equation}
The maps $j_1$, $j_2$ and $l$ satisfy the following for all $T,T'\in \cK$:
\begin{equation}\label{invariance}
\alpha_g(l(T))=l(T),\quad  g\in G, 
\end{equation}
\begin{equation}\label{orthogonality1}
\frac{\epsilon}{d}\sum_{g\in G}j_2(T)^*V(g)^*+\sum_{i=1}^mj_1(T_i)^*T_i^*l(T)=0,
\end{equation}
\begin{equation}\label{orthogonality2}
\frac{1}{d}\sum_{h\in G}V(h)j_2(T')j_2(T)^*V(h)^*+l(T')^*l(T)=\inpr{T}{T'}\sum_{i=1}^mT_iT_i^*,
\end{equation}
\begin{equation}\label{Frobenius1}
l(j_1(T))=\sum_{i,j=1}^ml(T)^*T_iT_jj_1(T_j)T_i^*,
\end{equation}
\begin{equation}\label{Frobenius2}
l(j_2(T))=\sum_{i=1}^mT_il(T)^*j_1(T_i),
\end{equation}
\begin{equation}\label{period3}(j_2\circ j_1)^3=I.
\end{equation}
\end{lemma}

\begin{proof} To show that $\rho$ is of the above form, we determine $P\rho(T)$ and $Q\rho(T)$ separately. 
For $P\rho(T)$, we have 
\begin{align*}
P\rho(T) &=\sum_{h\in G}S_hS_h^*\rho(T)=\sum_{h\in G}S_h\alpha_h(S_e^*\rho(T))=\frac{1}{\sqrt{d}}\sum_{h\in G}S_h\alpha_h(j_2(T)^*)\\
&=\frac{1}{\sqrt{d}}\sum_{h\in G}S_hj_2(T)^*V(h)^*.   
\end{align*}
Since $\cK^*\rho(\cK)\subset (\rho^2,\rho^2)$, there exist linear maps $l_h:\cK\rightarrow \cK$ for $h\in G$ and $l:\cK\rightarrow \cK^2\cK^*$ such that 
$$Q\rho(T)=\sum_{h\in G}l_h(T)S_hS_h^*+l(T).$$
Since $\alpha_g\circ \rho=\rho$, we obtain $l_h(T)=\alpha_h(l_e(T))$ and $\alpha_g(l(T))=l(T)$. 

We compute $\rho(j_1(T))$ by using the definition of $j_1$: 
\begin{align*}\lefteqn{
\rho(j_1(T)) =\sqrt{d}\rho(T^*\rho(S_e))=\sqrt{d}\rho(T)^*\sum_{h\in G}S_hS_hS_h^*+\sqrt{d}\rho(T)^*\sum_{i=1}^mT_i\rho(S_e)T_i^*} \\
 &=\sqrt{d}\sum_{h\in G}\alpha_h(\rho(T)^*S_e)S_hS_h^*+\sqrt{d}\sum_{i=1}^m\sum_{h\in G}\inpr{T_i}{V(h)l_e(T)}S_hS_h^*\rho(S_e)T_i^*\\
 &+ \sqrt{d}\sum_{i=1}^ml(T)^*T_i\rho(S_e)T_i^* \\
 &=\sum_{h\in G}\alpha_h(j_2(T))S_hS_h^*+\frac{\epsilon}{\sqrt{d}}\sum_{h\in G}S_h l_e(T)^*V(h)^*\\
 &+ \sum_{i,j=1}^ml(T)^*T_iT_jj_1(T_j)T_i^*.
\end{align*}
This shows $l_e=j_2\circ j_1^{-1}$ and Eq.(\ref{Frobenius1}). 

We compute $\rho(j_2(T))$ by using the definition of $j_2$: 
\begin{align*}
\lefteqn{\rho(j_2(T))=\sqrt{d}\rho(\rho(T)^*S_e)=\frac{\epsilon}{\sqrt{d}}\sum_{h\in G}\rho^2(T^*)S_h+\sum_{i=1}^m\rho^2(T^*)T_ij_1(T_i)} \\
 &=\frac{\epsilon}{\sqrt{d}}\sum_{h\in G}S_h\alpha_h(T^*)+\sum_{i=1}^mT_i\rho(T)^*j_1(T_i) \\
 &=\frac{\epsilon}{\sqrt{d}}\sum_{h\in G}S_h\alpha_h(T^*)+\sum_{i=1}^m\sum_{h\in G}\inpr{j_1(T_i)}{V(h)j_2\circ j_1^{-1}(T)}T_iS_hS_h^* \\
 &+\sum_{i=1}^mT_il(T)^*j_1(T_i)\\
 &=\frac{\epsilon}{\sqrt{d}}\sum_{h\in G}S_h\alpha_h(T^*)+\sum_{h\in G}V(h)j_1^{-1}\circ j_2\circ j_1^{-1}(T)S_hS_h^* \\
 &+\sum_{i=1}^mT_il(T)^*j_1(T_i).
\end{align*}
This implies Eq.(\ref{Frobenius2}), and $j_2\circ j_1^{-1}\circ j_2=j_1^{-1}\circ j_2\circ j_1^{-1}$, which shows Eq.(\ref{period3}). 

Eq.(\ref{orthogonality1}) and (\ref{orthogonality2}) follow from $\rho(S_e)^*\rho(T)=0$ and $\rho(T')^*\rho(T)=\inpr{T}{T'}I$. 
\end{proof}

\begin{remark}\label{remark1}
Let $l^{(1)}_{ij}, l^{(2)}_{ij}\in \B(\cK)$ be linear maps defined by 
$l(T)=\sum_{i,j=1}^ml^{(1)}_{ij}(T)T_iT_j^*$ and $l(T)=\sum_{i,j=1}^mT_il^{(2)}_{ij}(T)T_j^*$ respectively. 
Then Eq.(\ref{Frobenius1}) and Eq.(\ref{Frobenius2}) are equivalent to 
$l^{(1)}_{ij}(j_2(T))=j_1^{-1}(l^{(1)}_{ji}(T))$ and 
$l^{(2)}_{ij}(j_1(T))=j_1(l^{(2)}_{ji}(T))$ 
respectively. 
\end{remark}

\begin{remark}\label{6j}
The above lemma shows that every morphism between objects generated by $\rho$ and $\alpha_g$ are noncommutative polynomials 
of the Cuntz algebra generators $\{S_g\}_{g\in G}\cup\{T_i\}_{i=1}^m$ and their adjoints, and the structure of $\cC$, 
or more precisely the 6$j$ symbols of $\cC$, are completely determined by the tuple $(\cK,j_1,j_2,V,U_\cK,\chi,l)$. 
To obtain this tuple from $\cC$, we made the following choices: 
\begin{itemize}
\item[(1)] the representative $\rho$ from the class $[\rho]$, 
\item[(2)] the parametrization of the group $\{[\alpha_g]\}_{g\in G}$, 
\item[(3)] the choice of $S_e$ from $\T S_e$. 
\end{itemize}
A different choice in (1) only ends up with a unitarily equivalent tuple in an appropriate sense, and 
that in (2) allows us to insert an group automorphism of $G$ in the variables of $V,U_K,\chi$. 
Replacing $S_e$ with $\omega S_e$ results in replacing $j_1$ and $j_2$ with $\omega j_1$ and $\omega j_2$ respectively, 
which is a special case of unitary equivalence. 
\end{remark}

In terms of $U(g)$, $V(g)$, and $l$, the relation $\rho(\alpha_g(T))=U(g)\rho(T)U(g)^*$ is expressed as follows: 
\begin{lemma} Let the notation be as above. Then 
\begin{equation}\label{Weyl}U_{\cK}(g)V(h)=\chi_h(g)V(h)U_{\cK}(g),
\end{equation}
\begin{equation}\label{equivariance}
l(V(g)T)=U(g)l(T)U(g)^*,\quad g\in G,\;T\in \cK. 
\end{equation}
\end{lemma}

\begin{proof} We compute $\rho(\alpha_g(T))$ first:
\begin{align*}
\rho(\alpha_g(T)) &= 
\frac{1}{\sqrt{d}}\sum_{h\in G}S_hj_2(V(g)T)^*V(h)^*+\sum_{h\in G}V(h)j_2\circ j_1^{-1}(V(g)T)S_hS_h^*\\
&+l(V(g)T)\\
 &=\frac{1}{\sqrt{d}}\sum_{h\in G}S_hj_2(T)^*U_{\cK}(g)^*V(h)^*+\sum_{h\in G}V(h)U_{\cK}(g)j_2\circ j_1^{-1}(T)S_hS_h^*\\
 &+l(V(g)T).
\end{align*}
On the other hand, 
\begin{align*}
U(g)\rho(T)U(g)^* &= 
\frac{1}{\sqrt{d}}\sum_{h\in G}U(g)S_hj_2(T)^*V(h)^*U_{\cK}(g)^*\\
&+\sum_{h\in G}U_{\cK}(g)V(h)j_2\circ j_1^{-1}(T)S_hS_h^*U(g)^*
+U_{\cK}(g)l(T)U_{\cK}(g)^*\\
 &=\frac{1}{\sqrt{d}}\sum_{h\in G}\chi_h(g)S_hj_2(T)^*V(h)^*U_{\cK}(g)^*\\
 &+\sum_{h\in G}\overline{\chi_h(g)}U_{\cK}(g)V(h)j_2\circ j_1^{-1}(T)S_hS_h^*+U_{\cK}(g)l(T)U_{\cK}(g)^*,
\end{align*}
which shows the statement. 
\end{proof}

\begin{cor} For $g,h\in G$, the following hold: 
\begin{equation}\label{symmetric} \chi_{h}(g)=\chi_g(h),\end{equation}
\begin{equation}\label{alphaU} \alpha_h(U(g))=\overline{\chi_h(g)}U(g).\end{equation} 
\end{cor}

\begin{proof} From Lemma \ref{j}, we have $j_2V(g)j_2^*=j_2^*V(g)j_2=U_{\cK}(g)$. 
Thus Eq.(\ref{Weyl}) implies 
$$j_2U_{\cK}(g)j_2^*j_2V(h)j_2^*=\overline{\chi_h(g)}j_2V(h)j_2^*j_2U_{\cK}(g)j_2^*,$$
and so 
$$V(g)U_{\cK}(h)=\overline{\chi_h(g)}U_{\cK}(h)V(g),$$
which shows $\chi_h(g)=\chi_g(h)$. 
Using this, we obtain the second statement from 
\begin{align*}
\alpha_h(U(g)) &=\sum_{k\in G}\chi_k(g)\alpha_h(S_kS_k^*)+\alpha_h(U_{\cK}(g)) \\
&=\sum_{k\in G}\chi_k(g)S_{hk}S_{hk}^*+V(h)U_{\cK}(g)V(h)^* \\
 &=\sum_{k\in G}\chi_{h^{-1}k}(g)S_kS_k^*+\overline{\chi_h(g)}U_{\cK}(g).
\end{align*}
\end{proof}

We denote by $\lambda$ the left regular representation of $G$. 

\begin{theorem}\label{representation} 
When $d=\frac{m+\sqrt{m^2+4n}}{2}$ is irrational, then the group $G$ is always abelian, $m$ is a multiple of $n$, and 
\begin{equation}\bigoplus_{h\in G}\chi_h\cong \lambda,
\end{equation}
\begin{equation}
V\cong U_{\cK}\cong \frac{m}{n}\lambda.
\end{equation}
When $d$ is rational, there exist two natural numbers $s$ and $t$ such that 
$n=st^2$, $m=(s-1)t$, and $d=st$. Moreover, 
\begin{itemize} 
\item [(i)] When $t=1$, the group $G$ is abelian. 
In this case,  the character $\chi_h$ is trivial for all $h\in G$ and 
\begin{equation}1\oplus V\cong 1\oplus U_{\cK}\cong \lambda.
\end{equation} 
\item [(ii)] When $t>1$, the group $G$ is non-abelian. 
In this case, the order of $\Hom(G,\T)$ is $t^2$ and 
\begin{equation}\bigoplus_{h\in G}\chi_h\equiv s\bigoplus_{\chi\in \Hom(G,\T)}\chi.\end{equation}   
Let $\hG^\dagger$ be the set of equivalence classes of irreducible unitary representations of $G$ whose 
dimensions are greater than 1. 
Then $t$ divides $\dim \pi$ for all $\pi\in \hG^\dagger$, and 
\begin{equation}
V\cong U_{\cK}\cong \bigoplus_{\pi\in \hG^\dagger}\frac{\dim \pi}{t}\pi.
\end{equation}
\end{itemize}
\end{theorem}

\begin{proof} Since we have $\delta_{g,e}=\rho(S_e)^*\rho(S_g)=\rho(S_e)^*U(g)\rho(S_e)U(g)^*$ for $g\in G$, we have 
\begin{align*}
\delta_{g,e} &=\rho(S_e)^*U(g)\rho(S_e)=\frac{1}{d^2}\sum_{h\in G}\chi_h(g)+\frac{1}{d}\sum_{i,j=1}^mj_1(T_j)^*T_j^*U_{\cK}(g)T_ij_1(T_i) \\
 &=\frac{1}{d^2}\sum_{h\in G}\chi_h(g)+\frac{1}{d}\sum_{i,j=1}^m \inpr{U_{\cK}(g)T_i}{T_j}\inpr{j_1(T_i)}{j_1(T_j)}\\
 &=\frac{1}{d^2}\sum_{h\in G}\chi_h(g)+\frac{1}{d}\Tr{U_{\cK}(g)}.\\
\end{align*}
This implies 
\begin{equation}\label{character}
\Tr \lambda(g)=\frac{n}{d^2}\sum_{h\in G}\chi_h(g)+\frac{n}{d}\Tr U_{\cK}(g). 
\end{equation}
For $\chi\in \Hom(G,\T)$, we denote by $a_\chi$ and $b_\chi$ the multiplicities of $\chi$ in $\bigoplus_{h\in G}\chi_h$ and 
$U_{\cK}$ respectively. 

First, we assume that $d$ is irrational. 
Then Eq.(\ref{character}) implies $d^2=na_\chi+nb_\chi d$ for $\chi\in \Hom(G,\T)$. 
Since $d^2=n+md$, we obtain $a_\chi=1$ and $b_\chi n=m$. 
This implies 
$$\bigoplus_{h\in G}\chi_h\cong \bigoplus_{\chi\in \Hom(G,\T)}\chi,$$
and so $G$ is abelian. 

Assume now that $d$ is rational. Then $d$ is an integer. 
The inclusion $M^G\supset \rho(M)$ is of finite depth, and its index is $d^2/n=1+md/n$. 
Since the index of a finite depth inclusion is an algebraic integer, the number $1+md/n$ is indeed an integer, which we denote by $s$. 
Then $d^2=n+md$ implies $n(s-1)^2=sm^2$, and so $m$ is a multiple of $s-1$. 
Letting $t=m/(s-1)$, we get $n=st^2$ and $d=st$. 
Now Eq.(\ref{character}) is of the form 
\begin{equation}\label{character2}
\Tr \lambda(g)=\frac{1}{s}\sum_{h\in G}\chi_h(g)+t\Tr U_{\cK}(g). 
\end{equation}

Assume that $t=1$. Then $n=d=s$ and $m=s-1$. 
Let $\chi\in \Hom(G,\T)$. 
Since $0\leq a_\chi\leq s$ and  Eq.(\ref{character}) implies $1=a_\chi/n+b_\chi$, either 
$a_\chi=0$ and $b_\chi=1$, or $a_\chi=n$ and $b\chi=0$ hold. 
Since $\chi_e=1$, we get $a_1=s$, which implies that $\chi_h=1$ for all $h\in G$. 
Thus we get $\Tr \lambda(g)=1+\Tr U_{\cK}(g)$, and 
$$1\oplus V\cong 1\oplus U_\cK\cong \lambda.$$ 
Let $v_0(g)=V(g)$, $v_1(g)=U_\cK(g)$, and $w=j_1j_2$. 
Since $\chi_h=1$ for any $h\in G$, Eq(\ref{Weyl}) implies that $v_0$ and $v_1$ commute with each other. 
Moreover, since $w^3=1$ and $v_1(g)=w^*v_0(g)w$, if we define $v_2(g)$ by $w^*v_1(g)w$, the three representations 
$v_0$, $v_1$, and $v_2$ commute with each other. 
Since $1\oplus v_i\cong \lambda$ for $i=0,1,2$, the group $G$ is abelian. 
Indeed, since the dimension of the commutant $v_0(G)'$ of $v_0(G)$ is $n-1$, we can see that $v_1(G)''$ is the commutant of $v_0(G)''$. 
Thus $v_0(G)'\cap v_1(G)'$ coincides with the center of $v_0(G)''$. 
Since $v_2$ is a faithful representation of $G$ in $v_0(G)'\cap v_1(G)'$, we conclude that $G$ is abelian.

Assume $t>1$ now. 
Let $\pi$ be an irreducible representation of $G$ contained in $U_{\cK}$. 
Then Eq.(\ref{character2}) implies $\dim \pi\geq t$, and so $G$ is non-abelian. 
The rest of the statements in (ii) follow from a similar reasoning as above. 
\end{proof}

\begin{remark}
Under naive identification of $\rho$ and $\overline{\rho}$, 
the map $j_2\circ j_1$ would be graphically expressed as 
$$j_2\circ j_1:\quad 
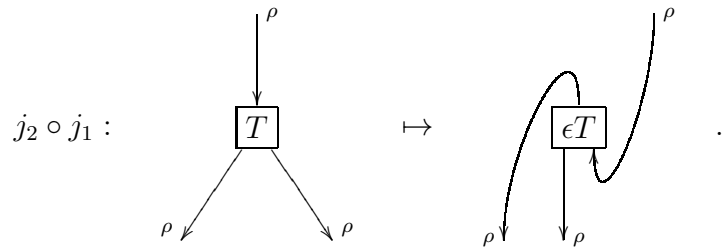
\begin{xy}(20,0)*+[F]{T}="T"
\ar(20,15);(20,3)^<\rho
\ar(18,-3);(10,-15)_>\rho
\ar(22,-3);(30,-15)^>\rho
\end{xy}
\quad \mapsto \quad 
\begin{xy}(20,0)*+[F]{\epsilon T}="{T}"
\ar@(u,u)(20,3) ;(10,-15)_>\rho
\ar(18,-3);(18,-15)^>\rho
\ar@(d,d)(30,15);(22,-3)^<\rho
\end{xy}\quad. 
$$
This means that $360^\circ$ rotation in this picture ends up with multiplying by $\epsilon$. 
To avoid this awkward convention in the case of $\epsilon=-1$, we need to properly take the pivotal structure into account 
to define rotation by $120^\circ$, which should be $j_2\circ j_1$ instead of $j_2\circ j_1^{-1}=\epsilon j_2\circ j_1$. 
In fact, we can deduce Eq.(\ref{period3}) from \cite[Theorem 5.1]{NS07} by identifying $(\rho,\rho^2)$ with $(\id,\rho^3)$. 
Since we don't not rely on graphical calculus at all in this paper, instead of seriously pursuing it, 
we give a short and general argument, based on Longo's observation in \cite{L90}, 
giving another proof of $(j_2\circ j_1)^3=I$ here. 
Let $\phi_\rho$ be the left inverse of $\rho$ defined by $\phi_\rho(x)=S_e^*\rho(x)S_e$. 
Following the notation in \cite{NS07}, we denote by $E_\rho^{(n)}:(\id,\rho^n)\rightarrow (\id,\rho^n)$ the restriction of $d\phi_\rho$. 
We first claim that $E_\rho^{(n)}$ is a unitary. 
Indeed, for $W_1,W_2\in (\id,\rho^n)$, we have 
$$\inpr{E_\rho^{(n)}W_1}{E_\rho^{(n)}W_2}=d^2S_e^*\rho(W_2)^*S_eS_e^*\rho(W_1)S_e.$$
Since $\rho(W_2)^*S_eS_e^*\rho(W_1)\in (\rho,\rho)$ is already a scalar, we have 
$$\inpr{E_\rho^{(n)}W_1}{E_\rho^{(n)}W_2}=d^2\phi_\rho(\rho(W_2)^*S_eS_e^*\rho(W_1))
=d^2W_2^*\phi_\rho(S_eS_e^*)W_1=\inpr{W_1}{W_2}.$$
As was already observed in \cite{L14}, the $n$-th power of $E_\rho^{(n)}$ is positive for
$$\inpr{(E_\rho^{(n)})^nW}{W}=d^nW^*\phi_\rho^n(W)=d^n\phi_\rho(\rho^n(W^*)W)=d^n\phi_\rho^n(WW^*)\geq 0.$$
Thus $(E_\rho^{(n)})^n=\id$. 
Let $\Phi: (\rho,\rho^2)\rightarrow (\id,\rho^3)$ be a unitary defined by $\Phi(T)=\rho(T)S_e$. 
Then $\Phi^{-1}=dS_e^*\rho(\cdot)$, and 
\begin{align*}
\lefteqn{\Phi^{-1}\circ E_\rho^{(3)}\circ \Phi(T)=dS_e^*\rho(d\phi_\rho(\rho(T)S))} \\
 &=d^2S_e^*\rho(T\phi_\rho(S_e))=dS_e^*\rho(TS)=\sqrt{d}j_2(T)^*\rho(S)=j_1\circ j_2(T), \\
\end{align*}
which proves $(j_2\circ j_1)^3=(j_1\circ j_2)^{-3}=I$ again. 
\end{remark}

\begin{remark}
Our explicit formula for $E_\rho^{(n)}$ as above allows us to compute easily the higher Frobenius-Schur indicators $\nu_{n,r}(\rho)
=\mathrm{tr}((E_\rho^{(n)})^r)$. 
Indeed, we have $\nu_{2,1}(\rho)=\epsilon$, which should be treated as a given datum.  
Since $E_\rho^{(3)}$ is unitarily equivalent to $j_1\circ j_2$, we have $\nu_{3,1}(\rho)=\mathrm{tr}(j_1\circ j_2)$. 
Since $\{S_gS_e\}_{g\in G}\cup\{T_iT_jS_e\}_{1\leq i,j\leq m}$ is an orthonormal basis of $(\id, \rho^4)$, we have 
\begin{align*}
\lefteqn{\nu_{4,1}(\rho)=d\sum_{g\in G}S_e^*S_g^*\phi_\rho(S_gS_e)+d\sum_{i,j}S_e^*T_j^*T_i^*\phi_\rho(T_iT_jS_e)} \\
 &=d\sum_{g\in G} S_e^*S_g^*S_e^*\rho(S_gS_e)S_e+d\sum_{i,j}S_e^*T_j^*T_i^*S_e^*\rho(T_iT_jS_e)S_e\\
 &=d\sum_{g\in G}S_e^*S_g^*S_e^*U(g)\rho(S_e)U(g)^*\rho(S_e)S_e+\sqrt{d}\sum_{i,j}S_e^*T_j^*T_i^*j_2(T_i)^*\rho(T_jS_e)S_e\\
 &=\sum_{g\in G}S_e^*S_g^*U(g)^*\rho(S_e)S_e+\sum_{i,j}S_e^*T_j^*j_1((T_i^*j_2(T_i)^*\rho(T_j))^*)S_e\\
 &=\frac{1}{d}\sum_{g\in G}\overline{\chi_g(g)}+\sum_{i,j}\inpr{j_1((T_i^*j_2(T_i)^*\rho(T_j))^*)}{T_j}\\
 &=\frac{1}{d}\sum_{g\in G}\overline{\chi_g(g)}+\sum_{i,j}\inpr{j_1(T_j)}{(T_i^*j_2(T_i)^*\rho(T_j))^*}\\
 &=\frac{1}{d}\sum_{g\in G}\overline{\chi_g(g)}+\sum_{i,j}T_i^*j_2(T_i)^*\rho(T_j)j_1(T_j).
\end{align*} 
It is easy to evaluate these for concrete examples. 
\end{remark}

We get back to the original setting where our near-group categories live in $\End(M)$. 
About $\rho(U(g))$, we have 

\begin{lemma}
\begin{equation}\label{rhoU}
\rho(U(g))=\sum_{h\in G}S_hS_{hg}^*+j_2\circ j_1^{-1}U_{\cK}(g)(j_2\circ j_1^{-1})^*\otimes U(g). 
\end{equation}
\end{lemma}

\begin{proof} Since $\rho(U(g))S_e\in (\alpha_{g^{-1}},\rho^2)$ there exists a complex number $c$ of modulus 1 such that 
$\rho(U(g))S_e=cS_{g^{-1}}$. 
To show that $c=1$, we compute both sides of 
$S_e^*\rho^2(U(g))\rho(S_e)=cS_e^*\rho(S_{g^{-1}})$. 
The left-hand side is 
$$S_e^*\rho^2(U(g))\rho(S_e)=U(g)S_e^*\rho(S_e)=\frac{\epsilon}{d} U(g).$$ 
On the other hand, the right-hand side is 
$$cS_e^*\rho(S_{g^{-1}})=cS_e^*U(g)^*\rho(S_e)U(g)=cS_e^*\rho(S_e)U(g)=c\frac{\epsilon}{n} U(g),$$ 
showing $c=1$. 
Thus we have $S_{hg^{-1}}=\alpha_h(\rho(U(g))S_e)=\rho(U(g))S_h$ and 
$$\rho(U(g))P=\sum_{h\in G}\rho(U(g))S_hS_h^*=\sum_{h\in G}S_{hg^{-1}}S_h^*.$$
This implies that $P$ and $Q$ commute with $\rho(U(g))$ and 
$$\rho(U(g))=\sum_{h\in G}S_{hg^{-1}}S_h^*+Q\rho(U(g))Q.$$
Note that $T_i^*\rho(U(g))T_j\in (\rho,\rho\circ \alpha_g)=\C U(g)$. 
This means that there exists a unitary representation $U'$ of $G$ in $\B(\cK)$ such that $Q\rho(U(g))Q=U'(g)\otimes U(g)$, 
and it is determined by $S_e^*T_i^*\rho(U(g))T_jS_e=T_i^*U'(g)T_j$. 
Indeed, expanding $U(g)$ as 
$$U(g)=\sum_{h\in G}\chi_h(g)S_hS_h^*+\sum_{p,q=1}^mU_\cK(g)_{pq}T_pT_q^*,$$
we get 
\begin{align*}\lefteqn{S_e^*T_i^*\rho(U(g))T_jS_e}\\ 
&=\sum_{h\in G}\chi_h(g)S_e^*T_i^*\rho(S_hS_h^*)T_jS_e
+\sum_{p,q=1}^mU_{\cK}(g)_{pq} S_e^*T_i^*\rho(T_pT_q^*)T_jS_e \\
 &=\sum_{p,q=1}^mU_{\cK}(g)_{pq} \inpr{j_2\circ j_1^{-1}(T_p)}{T_i}\inpr{T_j}{j_2\circ j_1^{-1}(T_q)},
\end{align*}
and 
\begin{align*}
\lefteqn{U'(g)=\sum_{i,j=1}^mS_e^*T_i^*\rho(U(g))T_jS_e T_iT_j^*}\\
&=\sum_{i,j=1}^m\sum_{p,q=1}^mU_{\cK}(g)_{pq} \inpr{j_2\circ j_1^{-1}(T_p)}{T_i}\inpr{T_j}{j_2\circ j_1^{-1}(T_q)}T_iT_j^*\\
&=\sum_{p,q=1}^mU_{\cK}(g)_{pq} j_2\circ j_1^{-1}(T_p)j_2\circ j_1^{-1}(T_q)^*\\
&=j_2\circ j_1^{-1}U_{\cK}(g)(j_2\circ j_1^{-1})^*.
\end{align*}
\end{proof}

\begin{remark}
We can graphically express the three unitary representations of $G$ on $\cK$ as follows: 
$$V(g)T:=\alpha_g(T)=
\begin{xy}(20,0)*+[F]{T} 
\ar(20,15);(20,3)^<\rho 
\ar(18,-3);(6,-15)_>\rho
\ar(22,-3);(34,-15)^>\rho
\ar@/_4mm/(20,11);(12,-9)_{\alpha_g}
\end{xy}
\quad,$$
$$U_{\cK}(g)T=
\begin{xy}(20,0)*+[F]{T} 
\ar(20,15);(20,3)^<\rho 
\ar(18,-3);(6,-15)_>\rho
\ar(22,-3);(34,-15)^>\rho
\ar@/_4mm/(12,-9);(28,-9)_{\alpha_g}
\end{xy}
\quad,$$
$$(j_2\circ j_1^{-1})U_K(g)(j_2\circ j_1^{-1})^{-1}T=\rho(U(g))TU(g)^*=
\begin{xy}(20,0)*+[F]{T} 
\ar(20,15);(20,3)^<\rho 
\ar(18,-3);(6,-15)_>\rho
\ar(22,-3);(34,-15)^>\rho
\ar@/_4mm/(28,-9);(20,11)_{\alpha_g}
\end{xy}
\quad.$$
This explains why they are related by the rotation map $j_2\circ j_1$. 
\end{remark}

The relation 
$$I=\sum_{g\in G}\rho(S_g)\rho(S_g)^*+\sum_{i=1}^m\rho(T_i)\rho(T_i)^*$$
implies 
\begin{lemma}\label{complete} Let $\hG^V$ be the set of irreducible representations of $G$ contained in $V$, and 
let 
$$(U_{\cK},\cK)=\bigoplus_{\pi\in G^V}\bigoplus_{a=1}^{m_\pi}(\pi,\cK^a_\pi)$$
be the irreducible decomposition. 
We choose an orthonormal basis $\{T^a_{\pi,i}\}_{i=1}^{\dim \pi}$ of $\cK^a_\pi$ so that 
$$U_{\cK}(g)=\sum_{\pi\in \hG^V}\sum_{a=1}^{m_\pi}\pi(g)_{ij}T^a_{\pi,i}{T^a_{\pi,j}}^*.$$
Then 
\begin{equation}\label{complete1}
\frac{n}{d}\sum_{\pi,a,b,i,j}\frac{1}{\dim \pi}T^a_{\pi,i}j_1(T^a_{\pi,j})j_1(T^b_{\pi,j})^*{T^b_{\pi,i}}^*+\sum_{i=1}^ml(T_i)l(T_i)^*=
Q\otimes Q,
\end{equation}
\begin{equation}\label{complete2}
\sum_{i=1}^ml(T_i)V(h)j_2(T_i)+\frac{\epsilon n}{d}\sum_{\pi,a}\delta_{\chi_h,\pi}T^a_{\pi}j_1(T^a_\pi)=0.
\end{equation}
\end{lemma}

\begin{proof} By using the Peter-Weyl theorem, we obtain
\begin{align*}
\lefteqn{ \sum_{g\in G}\rho(S_g)\rho(S_g)^*=\sum_{g\in G}U(g)\rho(S_e)\rho(S_e^*)U(g)^*} \\ 
 &=\sum_{g,h,k\in G}\chi_h(g)\overline{\chi_k(g)}S_hS_h^*\rho(S_e)\rho(S_e)^*S_kS_k^* \\
 &+\sum_{g,k\in G}\sum_{\pi,a,i,j}\pi(g)_{ij}\overline{\chi_k(g)}T^a_{\pi,i}{T^a_{\pi,j}}^*\rho(S_e)\rho(S_e)^*S_kS_k^*+(\cdots)^*\\
 &+\sum_{g\in G}\sum_{\pi,\sigma,i,j,p,q}\pi(g)_{ij}\overline{\sigma(g)_{pq}}
 T^a_{\pi,i}{T^a_{\pi,j}}^*\rho(S_e)\rho(S_e)^*T^b_{\sigma,b}{T^b_{\sigma,p}}^*\\
&=\frac{n}{d^2}\sum_{h,k\in G}\delta_{\chi_h,\chi_k}S_hS_k^*
+\frac{\epsilon n}{d\sqrt{d}}\sum_{k,\pi,a}\delta_{\chi_k,\pi}T^a_\pi j_1(T^a_\pi)S_k^*+(\cdots)^*\\
&+\frac{n}{d}\sum_{\pi,a,b,i,j}\frac{1}{\dim \pi}T^a_{\pi,i}j_1(T^a_{\pi,j})j_1(T^b_{\pi,j})^*{T^b_{\pi,i}}^*. 
\end{align*}
On the other hand, 
\begin{align*}
\lefteqn{\sum_{i=1}^m\rho(T_i)\rho(T_i)^*
=\frac{1}{d}\sum_{i=1}^m\sum_{h,k\in G}\inpr{V(h^{-1}k)j_2(T_i)}{j_2(T_i)}S_hS_k^*} \\
 &+\sum_{i=1}^m\sum_{h\in G}V(h)j_2\circ j_1^{-1}(T_i)S_hS_h^*j_2\circ j_1^{-1}(T_i)^*V(h)^*+\sum_{i=1}^ml(T_i)l(T_i)^* \\
 &+\frac{1}{\sqrt{d}}\sum_{i=1}^m\sum_{h\in G}l(T_i)V(h)j_2(T_i)S_h^*+(\cdots)^* \\
 &= \frac{1}{d}\sum_{h,k\in G}\Tr V(h^{-1}k)S_hS_k^*+\sum_{i=1}^m\sum_{h\in G} T_iS_hS_h^*T_i^*\\
 &+\sum_{i=1}^ml(T_i)l(T_i)^*+\frac{1}{\sqrt{d}}\sum_{i=1}^m\sum_{h\in G}l(T_i)V(h)j_2(T_i)S_h^*+(\cdots)^* .\\
\end{align*}
Since $V$ is unitarily equivalent to $U_\cK$, Eq.(\ref{character2}) and Eq.(\ref{symmetric}) imply  
\begin{align*}
\lefteqn{\Tr V(h^{-1}k)=\Tr U_\cK(h^{-1}k)} \\
 &=\frac{d}{n}\Tr(\lambda(h^{-1}k))-\frac{1}{d}\sum_{g\in G}\chi_k(g)\overline{\chi_{h}(g)}=d\delta_{h,k}-\frac{n}{d}\delta_{\chi_h,\chi_k},
\end{align*}
and 
\begin{align*}
\lefteqn{\sum_{i=1}^m\rho(T_i)\rho(T_i)^*=P-\frac{n}{d^2}\sum_{h,k\in G}\delta_{\chi_h,\chi_k}S_hS_k^*} \\
 &+\sum_{i=1}^m\sum_{h\in G} T_iS_hS_h^*T_i^*+\sum_{i=1}^ml(T_i)l(T_i)^*\\
 &+\frac{1}{\sqrt{d}}\sum_{i=1}^m\sum_{h\in G}l(T_i)V(h)j_2(T_i)S_h^*+(\cdots)^* .\\
\end{align*}
Since 
$$I=\sum_{g\in G}\rho(S_g)\rho(S_g)^*+\sum_{i=1}^m\rho(T_i)\rho(T_i)^*,$$
we get
\begin{align*}
\lefteqn{Q=\sum_{i=1}^m\sum_{h\in G} T_iS_hS_h^*T_i^*} \\
 &+\frac{n}{d}\sum_{\pi,a,b,i,j}\frac{1}{\dim \pi}T^a_{\pi,i}j_1(T^a_{\pi,j})j_1(T^b_{\pi,j})^*{T^b_{\pi,i}}^* +\sum_{i=1}^ml(T_i)l(T_i)^*\\
 &+\frac{1}{\sqrt{d}}\sum_{i=1}^m\sum_{h\in G}l(T_i)V(h)j_2(T_i)S_h^*+(\cdots)^* \\
 &+\frac{\epsilon n}{d\sqrt{d}}\sum_{h,\pi,a}\delta_{\chi_h,\pi}T^a_\pi j_1(T^a_\pi)S_h^*+(\cdots)^*,\\
\end{align*}
which shows the statement. 
\end{proof}

\begin{lemma}\label{rhoU3} In terms of $l(T)$, Eq.(\ref{rhoU}) is expressed as 
\begin{align}\label{rhoU2}\lefteqn{}\\
&j_2\circ j_1^{-1}U_{\cK}(g)(j_2\circ j_1^{-1})^*\otimes U_{\cK}(g)\nonumber\\
&=\frac{1}{d}\sum_{\pi,\sigma,a,b,i,j,p,q}\inpr{\chi_g\pi_{ij}}{\sigma_{pq}}
T^a_{\pi,i}j_1(T^a_{\pi,j})j_1(T^b_{\sigma,q})^*{T^b_{\sigma,p}}^*
+\sum_{\pi,a,i,j}\pi(g)_{ij}l(T^a_{\pi,i})l(T^a_{\pi,j})^*,\nonumber
\end{align}
where 
$$\inpr{\chi_g\pi_{ij}}{\sigma_{pq}}=\sum_{h\in G}\chi_g(h)\pi_{ij}(h)\overline{\sigma_{pq}(h)}.$$
\end{lemma}

\begin{proof} Since $\rho(U(g))$ is a unitary,  Eq.(\ref{rhoU}) holds if and only if the following two hold: 
$$P\rho(U(g))P=\sum_{h}S_hS_{hg}^*,$$ 
$$Q\rho(U(g))Q=j_2\circ j_1^{-1}U_{\cK}(g)(j_2\circ j_1^{-1})^*\otimes U(g).$$ 
The first one does not give any condition on $l$ as 
\begin{align*}\lefteqn{
P\rho(U(g))P =\sum_{h\in G}\chi_h(g)P\rho(S_hS_h^*)P+\sum_{i,j=1}^m U_{\cK}(g)_{ij}P\rho(T_i)\rho(T_j)^*P} \\
 &= \sum_{h,s,t}\chi_h(g)PU(h)\rho(S_eS_e^*)U(h)^*P\\
 &+\frac{1}{d}\sum_{i,j=1}^m\sum_{s,t\in G}U_{\cK}(g)_{ij}\inpr{V(t)j_2(T_j)}{V(s)j_2(T_i)}S_sS_t^*\\
 &=\frac{1}{d^2}\sum_{h,s,t}\chi_h(t^{-1}sg)S_sS_t^*
 +\frac{1}{d}\sum_{i,j=1}^m\sum_{s,t\in G}U_{\cK}(g)_{ij}\inpr{j_2^*V(t^{-1}s)j_2T_i}{T_j}S_sS_t^*\\
 &=\frac{1}{d^2}\sum_{h,s,t}\chi_h(t^{-1}sg)S_sS_t^*
 +\frac{1}{d}\sum_{i,j=1}^m\Tr U_{\cK}(t^{-1}sg) S_sS_t^*\\
 &=\sum_{s\in G}S_sS_{sg}^*. 
\end{align*}
For the second, we have 
$$Q\rho(U(g))Q=\sum_{h\in G}\chi_h(g)QU(h)\rho(S_eS_e^*)U(h)^*Q+\sum_{\pi,a,i,j}\pi(g)_{ij}Q\rho(T^a_{\pi,i}{T^a_{\pi,j}}^*)Q.$$
The first term above is 
\begin{align*}
\lefteqn{\sum_{h\in G}\chi_h(g)QU(h)\rho(S_eS_e^*)U(h)^*Q}\\
&=\sum_{\pi,\sigma,a,b,i,j,p,q}\sum_{h\in G}\chi_g(h)\pi(h)_{ij}\overline{\sigma(h)_{pq}}
T^a_{\pi,i}{T^a_{\pi,j}}^*\rho(S_eS_e^*)T^b_{\sigma,q}{T^b_{\sigma,p}}^*\\
 &= \frac{1}{d}\sum_{\pi,\sigma,a,b,i,j,p,q}\inpr{\chi_g\pi_{ij}}{\sigma_{pq}}
T^a_{\pi,i}j_1(T^a_{\pi,j})j_1(T^b_{\sigma,q})^*{T^b_{\sigma,p}}^*
\end{align*}
The second term is 
\begin{align*}
\lefteqn{\sum_{\pi,a,i,j}\pi(g)_{ij}Q\rho(T^a_{\pi,i}{T^a_{\pi,j}}^*)Q} \\
 &=\sum_{\pi,a,i,j}\sum_{h\in G}\pi(g)_{ij}V(h)j_2\circ j_1^{-1}(T^a_{\pi,i})S_hS_h^*j_2\circ j_1^{-1}(T^a_{\pi,j})^*V(h)^*\\
 &+\sum_{\pi,a,i,j}\pi(g)_{ij}l(T^a_{\pi,i})l(T^a_{\pi,j})^*\\
 &=\sum_{h\in G}V(h)j_2\circ j_1^{-1}U_{\cK}(g)(j_2\circ j_1^{-1})^*V(h)^*\otimes S_hS_h^*+\sum_{\pi,a,i,j}\pi(g)_{ij}l(T^a_{\pi,i})l(T^a_{\pi,j})^*\\
 &= \sum_{h\in G}j_2\circ j_1^{-1}U_{\cK}(g)(j_2\circ j_1^{-1})^*\otimes \chi_{h}(g)S_hS_h^*
 +\sum_{\pi,a,i,j}\pi(g)_{ij}l(T^a_{\pi,i})l(T^a_{\pi,j})^*,
\end{align*}
where we use the fact that $j_2\circ j_1$ has period three. 
Thus the statement follows. 
\end{proof}

\section{Reconstruction}\label{RCS} 
In the previous section, we showed that every information of a C$^*$ near-group category $\cC$ with a finite group $G$ is encoded in 
two anti-linear isometries $j_1,j_2$ of a finite dimensional Hilbert space $\cK$, 
two unitary representations $V,U_\cK$ of $G$ on $\cK$, characters $\{\chi_h\}_{h\in G}$ of $G$, and a linear map 
$l:\cK\to \B(\cK,\cK\otimes \cK)$.  
In this section, we deduce a necessary and sufficient condition for the tuple $(\cK,j_1,j_2,V,U_\cK,\chi,l)$ to recover 
the original C$^*$ near-group category $\cC$.  

Let $G$ be a finite group of order $n$, and let $m$ be a natural number satisfying the condition in the conclusion of 
Theorem \ref{representation}. 
Let $\cO_{n+m}$ be the Cuntz algebra with the canonical generators $\{S_g\}_{g\in G}\cup \{T_i\}_{i=1}^m$. 
We set $P=\sum_{g\in G}S_gS_g^*$, $Q=\sum_{i=1}^mT_iT_i^*$, and $\cK=\Span\{T_i\}_{i=1}^m$. 
Let $\epsilon \in \{1,-1\}$. 
We choose anti-linear isometries $j_1$ and $j_2$ of $\cK$ and unitaries representations $V$ and $U_{\cK}$ in $\B(\cK)$ satisfying 
Eq.(\ref{involution}),(\ref{j1}),(\ref{j2}),(\ref{period3}). 
We assume that $\chi_h\in \Hom(G,\T)$, $h\in G$, satisfy the condition in the conclusion of Theorem \ref{representation} and 
Eq.(\ref{Weyl}),(\ref{symmetric}). 
Under the above assumption, we can introduce a unitary representation $U$ of $G$ in $\cO_{n+m}$ by 
$$U(g)=\sum_{h\in G}\chi_h(g)S_hS_h^*+U_{\cK}(g), $$
and an action $\alpha$ of $G$ on $\cO_{n+m}$ by $\alpha_g(S_h)=S_{gh}$ and $\alpha_g(T)=V(g)T$ for $g\in G$ and $T\in \cK$. 
Choosing a linear map $l:\cK\rightarrow \B(\cK,\cK\otimes \cK)=\cK^2\cK^*$ satisfying Eq.(\ref{invariance}),(\ref{orthogonality1}),(\ref{orthogonality2}),
(\ref{Frobenius1}),(\ref{Frobenius2}),(\ref{equivariance}), and (\ref{rhoU2}),  
we can introduce a unital endomorphism $\rho$ of $\cO_{n+m}$ by 
$$\rho(S_e)=\frac{\epsilon}{d}\sum_{h\in G}S_h+\frac{1}{\sqrt{d}}\sum_{i=1}^mT_ij_1(T_i),$$
$$\rho(S_g)=U(g)\rho(S_e)U(g)^*,$$
$$\rho(T)=\frac{1}{\sqrt{d}}\sum_{h\in G}S_h\alpha_h(j_2(T)^*)+\sum_{h\in G}\alpha_h(j_2\circ j_1^{-1}(T))S_hS_h^*+l(T).$$
Indeed, we first define $\rho$ on the canonical generators $\{S_g,T_i\}$ of the Cuntz algebra $\cO_{n+m}$. 
Then $\{\rho(S_g),\rho(T_i)\}$ are isometries with mutual orthogonal ranges, and so 
$$E=\sum_{g\in G}\rho(S_g)\rho(S_g)^*+\sum_{i=1}^m\rho(T_i)\rho(T_i)^*$$ 
is a projection. 
The proof of Lemma \ref{rhoU3} in the case of $g=0$ implies $PEP=P$ and $QEQ=Q$, which shows that $E=I$. 
Thus $\rho$ extends to a unital endomorphism of $\cO_{n+m}$, and the proof of Lemma \ref{complete} implies that 
Eq.(\ref{complete2}) holds (note that Eq.(\ref{complete1}) is a special case of Eq.(\ref{rhoU2})). 
Now it is easy to show Eq.(\ref{defj_1}), Eq.(\ref{defj_2}), $\alpha_g\circ\rho=\rho$, and $\Ad U(g)\circ\rho=\rho\circ \alpha_g$. 
The proof of Lemma \ref{rhoU3} shows that Eq.(\ref{rhoU}) holds.

\begin{lemma}\label{technical} Let $X$ be an isometry of $\cO_{n+m}$. 
If $S_e^*\rho^2(X)S_e=X$ and $T_i^*\rho^2(X)T_i=\rho(X)$ for all $i$, then 
$$\rho^2(X)=\sum_{h\in G}S_h\alpha_h(X)S_h^*+\sum_{i=1}^m T_i\rho(X)T_i^*.$$
\end{lemma}

\begin{proof} Applying $\alpha_g$ to $S_e^*\rho^2(X)S_e=X$, we get $S_g^*\rho^2(X)S_g=\alpha_g(X)$. 
Since $I=\rho^2(X)^*\rho^2(X)$, we have $I=S_g^*\rho^2(X)^*\rho^2(X)S_g$ too, and 
\begin{align*}
\lefteqn{I=\sum_{h\in G}S_g^*\rho^2(X)^*S_hS_h^*\rho^2(X)S_g+\sum_{i=1}S_g^*\rho^2(X)^*T_iT_i^*\rho^2(X)S_g} \\
 &=S_g^*\rho^2(X)^*S_gS_g^*\rho^2(X)S_g+\sum_{h\in G\setminus\{g\}}S_g^*\rho^2(X)^*S_hS_h^*\rho^2(X)S_g\\
 &+\sum_{i=1}S_g^*\rho^2(X)^*T_iT_i^*\rho^2(X)S_g \\
 &=I+\sum_{h\in G\setminus\{g\}}S_g^*\rho^2(X)^*S_hS_h^*\rho^2(X)S_g
 +\sum_{i=1}S_g^*\rho^2(X)^*T_iT_i^*\rho^2(X)S_g.
\end{align*}
Thus $S_h^*\rho^2(X)S_g=0$ for $g\neq h$ and $T_i^*\rho^2(X)S_g=0$. 
In a similar way, we have $S_g\rho^2(X)T_i=0$ and $T_j^*\rho^2(X)T_i=0$ for $j\neq i$.
These relations imply $P\rho^2(X)Q=Q\rho^2(X)P=0$ and 
\begin{align*}
\rho^2(X)&=(P+Q)\rho^2(X)(P+Q)=P\rho^2(X)P+Q\rho^2(X)Q \\
 &=\sum_{h\in G}S_h\alpha_h(X)S_h^*+\sum_{i=1}^m T_i\rho(X)T_i^*. 
\end{align*}
\end{proof}

\begin{lemma}For $g,h\in G$, the equation $S_h^*\rho^2(S_g)S_h=S_{hg}$ holds. 
\end{lemma}

\begin{proof} It suffices to show the statements for $h=e$ because the others follow from them by applying $\alpha_h$ to them. 
\begin{align*}
S_e^*\rho^2(S_e)S_e &=\frac{\epsilon}{d}\sum_{k\in G}S_e^*\rho(S_k)S_e+\frac{1}{\sqrt{d}}\sum_{i=1}^mS_e^*\rho(T_i)\rho(j_1(T_i))S_e \\
 &=\frac{\epsilon}{d}\sum_{k\in G}S_e^*U(k)\rho(S_e)U(k)^*S_e+\frac{1}{d}\sum_{i=1}^mj_2(T_i)^*\rho(j_1(T_i))S_e \\
 &=\frac{1}{d^2}\sum_{k\in G}S_e+\frac{1}{d}\sum_{i=1}^mS_e=S_e. 
\end{align*}
Now  Eq.(\ref{rhoU}) implies 
$$S_e^*\rho(S_g)S_e=S_e^*\rho(U(g))\rho^2(S_e)\rho(U(g))^*S_e=S_g^*\rho^2(S_e)S_g=S_g.$$
\end{proof}

\begin{lemma}
$$\rho^2(S_g)=\sum_{h\in G}S_h\alpha_h(S_g)S_h^*+\sum_{i=1}^mT_i\rho(S_g)T_i^*.$$
\end{lemma}

\begin{proof} We first show the statement for $g=e$. 
Thanks to the above two lemmas, it suffices to show $Q\rho^2(S_e)Q=\sum_{i=1}^mT_i\rho(S_e)T_i^*$. 
Instead of the orthogonal basis $\{T_i\}_{i=1}^m$, we use $\{T^a_{\pi,p}\}$ sometimes. 
By the definition of $\rho$, we have 
$$Q\rho^2(S_e)Q=\frac{\epsilon}{d}\sum_{g\in G}Q\rho(S_g)Q+\frac{1}{\sqrt{d}}\sum_{i=1}^mQ\rho(T_i)\rho(j_1(T_i))Q.$$
The first term is 
\begin{align*}
\frac{\epsilon}{d}\sum_{g\in G}U_{\cK}(g)\rho(S_e)U_{\cK}(g)^*
&=\frac{\epsilon}{d}\sum_{g\in G}\sum_{\pi,\sigma,a,b,i,j,p,q}\pi(g)_{ij}\overline{\sigma(g)_{pq}}
 T^a_{\pi,i}{T^a_{\pi,j}}^*\rho(S_e)T^b_{\sigma,q}{T^b_{\sigma,p}}^*\\
 &=\frac{n\epsilon }{d}\sum_{\pi,a,b,i,j}\frac{1}{\dim \pi}
 T^a_{\pi,i}{T^a_{\pi,j}}^*\rho(S_e)T^b_{\pi,j}{T^b_{\pi,i}}^*\\
 &=\frac{n\epsilon }{d\sqrt{d}}\sum_{\pi,a,b,i,j}\frac{1}{\dim \pi}
 T^a_{\pi,i}j_1(T^a_{\pi,j})T^b_{\pi,j}{T^b_{\pi,i}}^*.\\
\end{align*}
The second term is \begin{align*}
\textrm{Second term}&=\frac{1}{\sqrt{d}}\sum_{i=1}^m \big\{
\frac{1}{\sqrt{d}}\sum_{h\in G}\alpha_h(j_2\circ j_1^{-1}(T_i))S_h\alpha_h(j_2\circ j_1(T_i)^*)+l(T_i)l(j_1(T_i))\big\}
 \\
 &=\frac{\epsilon}{d}\sum_{j=1}^m\sum_{h\in G}T_jS_hT_j^*+\frac{1}{\sqrt{d}}\sum_{i=1}^ml(T_i)l(j_1(T_i)).
\end{align*}
Using Eq.(\ref{Frobenius1}),(\ref{complete1}), we have 
\begin{align*}
\lefteqn{\frac{1}{\sqrt{d}}\sum_{i=1}^ml(T_i)l(j_1(T_i))=\frac{1}{\sqrt{d}}\sum_{i,p,q}l(T_i)l(T_i)^*T_pT_qj_1(T_q)T_p^* }\\
 &=\frac{1}{\sqrt{d}} \sum_{p,q}T_pT_qj_1(T_q)T_p^* \\
 &-\frac{n}{d\sqrt{d}}\sum_{p,q}\sum_{\pi,a,b,i,j}\frac{1}{\dim \pi}T^a_{\pi,i}j_1(T^a_{\pi,j})j_1(T^b_{\pi,j})^*{T^b_{\pi,i}}^*T_pT_qj_1(T_q)T_p^* \\
 &=\frac{1}{\sqrt{d}} \sum_{p,q}T_pT_qj_1(T_q)T_p^* 
 -\frac{n}{d\sqrt{d}}\sum_{\pi,a,b,i,j}\frac{1}{\dim \pi}T^a_{\pi,i}j_1(T^a_{\pi,j})j_1^2(T^b_{\pi,j}){T^b_{\pi,i}}^* \\
 &=\frac{1}{\sqrt{d}} \sum_{p,q}T_pT_qj_1(T_q)T_p^* 
 -\frac{n\epsilon}{d\sqrt{d}}\sum_{\pi,a,b,i,j}\frac{1}{\dim \pi}T^a_{\pi,i}j_1(T^a_{\pi,j})T^b_{\pi,j}{T^b_{\pi,i}}^*.
\end{align*}
Thus we obtain the statement for $g=e$. 
The general statement follows from  this and Eq.(\ref{rhoU}). 
\end{proof}

\begin{lemma} The following conditions are equivalent: 
\begin{itemize}
\item [(1)] $S_g\in (\alpha_g,\rho^2)$ for all $g\in G$. 
\item [(2)] $S_e^*\rho(l(T))S_e=(1-2n/d^2)T$ for all $T\in \cK$. 
\item [(3)] Let $l^{(2)}_{ij}:\cK\rightarrow \cK$ be as in Remark \ref{remark1}. Then 
\begin{equation}\label{S*rho2S}
\frac{1}{d}\sum_{i,j=1}^m j_2(T_i)^*l(l^{(2)}_{ij}(T))j_2(T_j)=(1-\frac{2n}{d^2})T,\quad T\in \cK.
\end{equation}
\end{itemize}
\end{lemma}

\begin{proof} It is easy to show that (2) and (3) are equivalent. 
Thanks to the above lemmas, it suffices to show that (2) is equivalent to $S_e^*\rho^2(T)S_e=T$. 
Indeed, 
\begin{align*}\lefteqn{
S_e^*\rho^2(T)S_e}\\ 
&=\frac{1}{\sqrt{d}}\sum_{h\in G}S_e^*\rho(\alpha_h(S_ej_2(T)^*))S_e
+\sum_{h\in G}S_e^*\rho(\alpha_h(j_2\circ j_1^{-1}(T)S_eS_e^*))S_e+S_e^*\rho(l(T))S_e. \\
\end{align*}
The first term is 
\begin{align*}\lefteqn{
\frac{1}{\sqrt{d}}\sum_{h\in G}S_e^*U(h)\rho(S_ej_2(T)^*)U(h)^*S_e}\\
&=\frac{n}{\sqrt{d}}S_e^*\rho(S_ej_2(T)^*)S_e =\frac{n\epsilon}{d^2}j_2^2(T)=\frac{n}{d^2}T.
\end{align*}
The second term is 
\begin{align*}\lefteqn{
\sum_{h\in G}S_e^*U(h)\rho(j_2\circ j_1^{-1}(T)S_eS_e^*)U(h)^*S_e}\\
&=nS_e^*\rho(j_2\circ j_1^{-1}(T)S_eS_e^*)S_e=\frac{n\epsilon}{d\sqrt{d}}j_2^2\circ j_1^{-1}(T)^*\rho(S_e)
=\frac{n}{d\sqrt{d}}j_1^{-1}(T)^*\rho(S_e)\\
&=\frac{n}{d^2}T. 
\end{align*}
Thus the statement is proved. 
\end{proof}

\begin{lemma} Assume that Eq.(\ref{S*rho2S}) holds. 
Then the following conditions are equivalent:
\begin{itemize}
\item [(1)] $\cK=(\rho,\rho^2)$. 
\item [(2)] For any $T,T',T''\in \cK$, the equality $\rho(T'')^*\rho^2(T)T'=\rho(T'')^*T'\rho(T)$ holds. 
\end{itemize}
\end{lemma}

\begin{proof} We only show that (2) implies (1). 
Assume that (2) holds. 
Thanks to the above lemmas, it suffices to show $\rho^2(T)T'=T'\rho(T)$. 
For this, it suffices to show $\rho(S_g)^*\rho^2(T)T'=\rho(S_g^*)T'\rho(T)$ for we have $\rho(P)+\rho(Q)=I$. 
Since we have 
$$\rho(S_g)^*\rho^2(T)T'=U(g)\rho(S_e)^*\rho^2(T)U(g)^*T',$$ 
$$\rho(S_g^*)T'\rho(T)=U(g)\rho(S_e^*)U(g)^*T'\rho(T),$$ 
we only take care of the case with $g=e$. 

For $\rho(S_e)^*\rho^2(T)T'$, we have 
\begin{align*}
\rho(S_e)^*\rho^2(T)T' &=\rho(S_e^*\rho(T))T'=\frac{1}{\sqrt{d}}\rho(j_2(T))^*T' \\
 &=\frac{1}{\sqrt{d}}\sum_{h\in G}\inpr{T'}{V(h)j_2\circ j_1^{-1}\circ j_2(T)}S_hS_h^*+\frac{1}{\sqrt{d}}l(j_2(T))^*T'. 
\end{align*}
For $\rho(S_e)^*T'\rho(T)$, we have 
\begin{align*}
 \rho(S_e)^*T'\rho(T)&= \frac{1}{\sqrt{d}}j_1(T')^*\rho(T)\\
 &=\frac{1}{\sqrt{d}}\sum_{h\in G}\inpr{V(h)j_2\circ j_1^{-1}(T)}{j_1(T')}S_hS_h^*+\frac{1}{\sqrt{d}}j_1(T')^*l(T)\\
 &=\frac{1}{\sqrt{d}}\sum_{h\in G}\inpr{T'}{V(h)j_1\circ j_2\circ j_1(T)}S_hS_h^*+\frac{1}{\sqrt{d}}j_1(T')^*l(T).
\end{align*}
Now the statement follows from Eq.(\ref{Frobenius2}),(\ref{period3}). 
\end{proof}

\begin{theorem}\label{classification} The endomorphism $\rho$ satisfies 
$$\rho^2(x)=\sum_{g\in G}S_g\alpha_g(x)S_g^*+\sum_{i=1}^m T_i\rho(x)T_i^*$$
if and only if Eq.(\ref{S*rho2S}) and the following three equations hold 
for all $h,k\in G$ and $T,T',T''\in \cK$: 
\begin{align}\label{l1}
\lefteqn{\sum_{i=1}^m{T'}^*l(T_i)j_2(l(T)^*j_2(T'')T_i)} \\
 &= \epsilon\inpr{T''}{T'}T-\frac{1}{d}\sum_{h\in G}\inpr{U_{\cK}(h)T''}{j_1(T)}j_1(U_{\cK}(h)T'),\nonumber
\end{align}
\begin{equation}\label{l2}
l(T'')^*j_2\circ j_1^{-1}(T')j_2\circ j_1^{-1}(T)=j_2\circ j_1^{-1}({T''}^*l(T)T'), 
\end{equation}
\begin{align}\label{l3}\lefteqn{
l(T'')^*T'l(T)=\sum_{i=1}^m l({T''}^*l(T)T_i)l(T_i)^*T'}\\
&+\frac{1}{d}\sum_{h\in G}\sum_{i=1}^m\inpr{V(h)j_2\circ j_1^{-1}(T)}{T''}U_{\cK}(h)T_ij_1(T_i) j_1(U_{\cK}(h)^*T')^* \nonumber.
\end{align}
\end{theorem}

\begin{proof} It suffices to show that $\rho({T''}^*\rho(T))T'=\rho(T'')^*T'\rho(T)$ is equivalent to the above three. 
We compute $\rho(T'')^*\rho^2(T)T'$ first: 
$$\rho({T''}^*\rho(T))T'=\sum_{h\in G}\inpr{V(h)j_2\circ j_1^{-1}(T)}{T''}\rho(S_hS_h^*)T'+\rho({T''}^*l(T))T'.$$
The first term is \begin{align*}
\lefteqn{\sum_{h\in G}\inpr{V(h)j_2\circ j_1^{-1}(T)}{T''}U(h)\rho(S_eS_e^*)U(h)^*T'} \\
 &=\frac{\epsilon}{d\sqrt{d}}\sum_{h,k\in G}\inpr{V(h)j_2\circ j_1^{-1}(T)}{T''}\chi_{k}(h)S_k j_1(U_{\cK}(h)^*T')^*\\
 &+ \frac{1}{d}\sum_{h\in G}\sum_{i=1}^m\inpr{V(h)j_2\circ j_1^{-1}(T)}{T''}U(h)T_ij_1(T_i) j_1(U_{\cK}(h)^*T')^*. 
\end{align*}
The second term is 
\begin{align*}
\lefteqn{\rho({T''}^*l(T))T'=\sum_{i=1}^m \rho({T''}^*l(T)T_i)\rho(T_i)^*T' } \\
 &=\sum_{i=1}^m \sum_{h\in G}\inpr{T'}{V(h)j_2\circ j_1^{-1}(T_i)}\rho({T''}^*l(T)T_i) S_hS_h^*+\sum_{i=1}^m \rho({T''}^*l(T)T_i)l(T_i)^*T'\\
 &=\sum_{i=1}^m \sum_{h\in G}\inpr{j_1\circ j_2^{-1}V(h)^*T'}{T_i}V(h)j_2\circ j_1^{-1}({T''}^*l(T)T_i) S_hS_h^*\\
 &+\frac{1}{\sqrt{d}}\sum_{i=1}^m\sum_{k\in G}S_k j_2({T''}^*l(T)T_i)^*V(k)^* l(T_i)^*T'+\sum_{i=1}^m l({T''}^*l(T)T_i)l(T_i)^*T'\\
 &=\sum_{h\in G}V(h)j_2\circ j_1^{-1}({T''}^*l(T)j_1\circ j_2^{-1}(V(h)^*T')) S_hS_h^*\\
 &+\frac{1}{\sqrt{d}}\sum_{i=1}^m\sum_{k\in G}S_k j_2({T''}^*l(T)T_i)^*l(T_i)^*V(k)^*T'V(k)^*+\sum_{i=1}^m l({T''}^*l(T)T_i)l(T_i)^*T'.
\end{align*}
On the other hand, 
\begin{align*}
\lefteqn{\rho(T'')^*T'\rho(T)=\sum_{k\in G} \inpr{T'}{V(k)j_2\circ j_1^{-1}(T'')}S_kS_k^*\rho(T)+l(T'')^*T'\rho(T)} \\
 &=\frac{1}{\sqrt{d}}\sum_{k\in G} \inpr{T'}{V(k)j_2\circ j_1^{-1}(T'')}S_kj_2(T)^*V(k)^*\\
 &+\sum_{h\in G}l(T'')^*T'V(h)j_2\circ j_1^{-1}(T)S_hS_h^*+ l(T'')^*T'l(T).
\end{align*}
Therefore we get 
\begin{align}\label{l4}
\lefteqn{\inpr{V(k)j_2\circ j_1^{-1}(T'')}{T'}j_2(T)} \\
 &= \sum_{i=1}^m{T'}^*V(k)l(T_i)j_2({T''}^*l(T)T_i)\nonumber\\
 &+\frac{\epsilon}{d}\sum_{h\in G}\overline{\chi_{k}(h)}
 \inpr{T''}{V(h)j_2\circ j_1^{-1}(T)}V(k)^*j_1(U_{\cK}(h)^*T'),\nonumber 
\end{align}
\begin{equation}\label{l5}
l(T'')^*T'V(h)j_2\circ j_1^{-1}(T)=V(h)j_2\circ j_1^{-1}({T''}^*l(T)j_1\circ j_2^{-1}(V(h)^*T')), 
\end{equation}
and Eq.(\ref{l3}). 
Thanks to Eq.(\ref{invariance}), we have 
$$V(h)^*l(T'')^*=l(T'')(V(h)^*\otimes V(h)^*),$$ 
and Eq.(\ref{l5}) 
is equivalent to Eq.(\ref{l2}) if $T'$ is replaced with $V(h)j_2\circ j_1^{-1}(T')$.  
Since 
$$V(k)^*j_1(U_{\cK}(h)^*T')=j_1(V(k)^*U_{\cK}(h)^*T')=j_1(\overline{\chi_{k}(h)}U_{\cK}(h)^*V(k)^*T'),$$
if $T'$ is replaced with $V(k)T'$, Eq.(\ref{l4}) is equivalent to 
\begin{align*}
\lefteqn{\inpr{j_2\circ j_1^{-1}(T'')}{T'}j_2(T)} \\
 &= \sum_{i=1}^m{T'}^*l(T_i)j_2({T''}^*l(T)T_i)
 +\frac{\epsilon}{d}\sum_{h\in G}
 \inpr{T''}{V(h)j_2\circ j_1^{-1}(T)}j_1(U_{\cK}(h)^*T'),
\end{align*}
If we replace $T''$ with $j_1\circ j_2^{-1}(T'')$ and $T$ with $j_2^{-1}(T)$, this is equivalent to \begin{align*}
\lefteqn{\inpr{T''}{T'}T} \\
 &=\epsilon \sum_{i=1}^m{T'}^*l(T_i)j_2(j_1\circ j_2^{-1}(T'')^*l(j_2(T))T_i)\\
 &+\frac{\epsilon}{d}\sum_{h\in G}
 \inpr{j_1\circ j_2^{-1}(T'')}{V(h)j_2\circ j_1^{-1}\circ j_2^{-1}(T)}j_1(U_{\cK}(h)^*T')\\
 &=\epsilon \sum_{i=1}^m{T'}^*l(T_i)j_2(l(T)^*j_2(T'')T_i)\\
 &+\frac{1}{d}\sum_{h\in G}
 \inpr{j_1\circ j_2^{-1}(T'')}{V(h)j_1\circ j_2\circ j_1(T)}j_1(U_{\cK}(h)^*T')\\
 &=\epsilon \sum_{i=1}^m{T'}^*l(T_i)j_2(l(T)^*j_2(T'')T_i)
 +\frac{\epsilon}{d}\sum_{h\in G}
 \inpr{U_{\cK}(h)^*T''}{j_1(T)}j_1(U_{\cK}(h)^*T'),\\
\end{align*}
where we use Eq.(\ref{Frobenius2}),(\ref{period3}). 
Thus we are done. 
\end{proof}

\begin{remark} Replacing $T''$ with $j_2(T'')$ and $T'$ with $j_1(T')$, we see that in Eq.(\ref{l3}) is equivalent to 
\begin{align}\label{l6}\lefteqn{
{T'}^*l(T'')l(T)=\sum_{i=1}^m l(j_2(T'')^*l(T)j_2(T_i)){T'}^*l(T_i)}\\
&+\frac{1}{d}\sum_{h\in G}\sum_{i=1}^m\inpr{T''}{U_{\cK}(h)j_1(T)}U_{\cK}(h)T_ij_1(T_i) j_1(U_{\cK}(h)^*j_1(T'))^* \nonumber.
\end{align}
Likewise, replacing $T$ with $j_1(T)$ in Eq.(\ref{l3}), it is equivalent to 
\begin{align}\label{l7}\lefteqn{
\sum_{i,j=1}^ml(T'')^*T'l(T)^*T_iT_jj_1(T_j)T_i^*=\sum_{i,j=1}^m l({T''}^*l(T)^*T_iT_jj_1(T_j))l(T_i)^*T'}\\
&+\frac{1}{d}\sum_{h\in G}\sum_{i=1}^m\inpr{V(h)j_2(T)}{T''}U_{\cK}(h)T_ij_1(T_i) j_1(U_{\cK}(h)^*T')^* \nonumber.
\end{align}
\end{remark}

\begin{definition}\label{tuple} We say that a tuple $(\cK,j_1,j_2,V,U_\cK,\chi,l)$ is admissible if it satisfies 
Eq.(\ref{involution}),(\ref{j1}),(\ref{j2}),(\ref{period3}), 
the condition in the conclusion of Theorem \ref{representation}, and 
Eq.(\ref{Weyl}),(\ref{symmetric}),(\ref{invariance}),(\ref{orthogonality1}),(\ref{orthogonality2}),
(\ref{Frobenius1}),(\ref{Frobenius2}),(\ref{equivariance}),(\ref{rhoU2}),(\ref{S*rho2S}),(\ref{l1}),
(\ref{l2}),(\ref{l3}). 
We say that two tuples $(\cK,j_1,j_2,V,U_\cK,l)$ and $(\cK',j_1',j_2',V',U_\cK',l')$ are equivalent 
if there exist a unitary $W: \cK\to \cK'$ and a group automorphism $\varphi\in \Aut(G)$ satisfying 
$j_1'W=Wj_1$, $j_2'W=Wj_2$, $U_{\cK'}(g)W=WU_\cK(\varphi(g))$, 
$V'(g)W=WV(\varphi(g))$, $\chi'_{h}(g)=\chi_{\varphi(h)}(\varphi(g))$, and 
$$l'(WT)W=(W\otimes W)l(T),\quad T\in \cK.$$
\end{definition}

We have seen that starting from an admissible tuple $(\cK,j_1,j_2,V,U_\cK,\chi,l)$, we can construct the 
Cuntz algebra endomorphism $\rho\in \End(\cO_{m+n})$ and the $G$-action $\alpha$ satisfying relevant properties. 
As in \cite{I93}, we can choose an appropriate representation of $\cO_{m+n}$ so that 
$\rho$ and $\alpha$ extend to the weak closure of $\cO_{m+n}$, which is a hyperfinite type III$_\lambda$ factor, 
without changing morphism spaces (see Appendix and \cite{I93}). 
This finishes the reconstruction process from $(\cK,j_1,j_2,V,U_\cK,\chi,l)$ to $\cC$. 

Let $\cC$ and $\cD$ be C$^*$ near-group categories with finite group $G$ and a multiplicity parameter $m$ realized 
in $\End_0(M)$, which give rise to two admissible tuples $(\cK,j_1,j_2,V,U_\cK,l)$ and $(\cK',j_1',j_2',V',U_\cK',l')$ respectively. 
In view of Theorem \ref{uniqueness} and Remark \ref{6j}, we see that the two C$^*$ near-group categories $\cC$ and 
$\cD$ are equivalent if and only if the two corresponding tuples are equivalent. 
In conclusion, we get the following result. 

\begin{theorem}\label{unique} The association $\cC\mapsto (\cK,j_1,j_2,V,U_\cK,l)$ gives 
a one-to-one correspondence between the set of 
equivalence classes of C$^*$ near-group categories with finite group $G$ and a multiplicity 
parameter $m$, and the set of equivalence classes of admissible tuples. 
\end{theorem}


\section{The case of $m=|G|-1$}

In this section, we briefly give an account of the classification of near-group categories with 
a finite group $G$ and the multiplicity parameter $m=|G|-1$. 
We have seen in Theorem \ref{representation} that $G$ is abelian under the C$^*$ condition. 
In fact, we have the following classification result without this additional assumption. 

A fusion category is said to be group theoretical if it is categorically Morita equivalent to a pointed category 
(see \cite[Definition 9.7.1]{EGNO15}). 

\begin{theorem} Let $\cC$ be a near group category with a finite group $G$ and the multiplicity parameter $m=|G|-1$. 
Then the group $G$ is cyclic and $|G|+1$ is a prime power $q$. 
If $G=\Z_2$, there are three such categories, if $G=\Z_3$ or $G=\Z_7$, there are two such categories, and 
if $G=\Z_{q-1}$ in the other cases, there is one such category. 
All these fusion categories are group theoretical. 
\end{theorem}

\begin{example} For $G=\Z_2$, we have $\dim \cK=1$, and Eq.(\ref{complete1}) implies $l=0$. 
We may choose a basis $\{T\}$ of $\cK$ so that $j_1(T)=T$, and so $\epsilon=1$. 
In this case, the only choices of $j_2$ are $j_2(T)=\zeta T$ with $\zeta^3=1$. 
All the conditions in Definition \ref{tuple} are easy to verify in this case, 
and there are certainly three near-group categories. 
This example was already discussed in \cite{I93}.  
\end{example}

For the proof of the above theorem, Siehler \cite[Theorem 1.2]{S03} was the first to show that $G$ is cyclic and $|G|+1$ is a prime power 
under an additional assumption of $G$ being abelian (though this assumption is not explicitly written in \cite[Theorem 1.2]{S03}). 
Etingof-Gelaki-Ostrik \cite[Corollary 7.4]{EGO04} showed the statement under the assumption 
that $G$ is cyclic. 
Recently Nikshych-Ostrik \cite{NO} showed that $G$ is cyclic and the classification was completed.   

Let $\F_q$ be the finite field of order $q$, and let $\F_q^\times$ be its multiplicative group. 
We regard $\F_q$ as an additive group on which $\F_q^\times$ acts by multiplication.  
Etingof-Gelaki-Ostrik reduced the classification to a counting argument of the group $H^3(\F_q,\T)^{\F_q^\times}$ 
of $\F_q^\times$-invariant cohomology classes in $H^3(\F_q,\T)$. 
We reproduce their reduction argument from the view point of operator algebras now.  
The following argument was developed in \cite{I97}, \cite[Theorem 9.8,(i)]{IK02}, 
which could serve as an introduction to more complicated arguments in the next section. 

Let $N=\rho(M)$, which is an irreducible subfactor of index $d^2=|G|^2$. 
Since $\alpha_g\circ \rho=\rho$, the fixed point algebra 
$$M^G=\{x\in M;\; \alpha_g(x)=x,\;\forall g\in G\},$$
is an intermediate subfactor between $M$ and $N$ with 
$$[M:M^G]=[M^G:N]=|G|.$$ 
Let $\iota :M^G\hookrightarrow M$ and $\kappa:N\hookrightarrow M^G$ be inclusion maps. 
Then we have the decomposition $\rho=\iota\circ\kappa\circ \rho_0$, where $\rho_0$ is $\rho$ regarded as an 
isomorphism from $M$ onto $N$. 
Let $\alpha'_g=\rho_0\circ \alpha_g\circ \rho_0^{-1}$, which is an outer action of $G$ on $N$, 
and let $N^G$ be the fixed point algebra 
$$N^G=\{x\in N;\; \alpha'_g(x)=x,\;\forall g\in G\}.$$
Since $\Ad U(g)\circ \rho=\rho\circ \alpha_g$, we have $\alpha'_g(y)=\Ad U(g)(y)$ for any $y\in N$, 
and the von Neumann algebra generated by $N$ and $\{U(g)\}_{g\in G}$ is identified with the crossed product 
$N\rtimes_{\alpha'}G$. 
Eq.(\ref{alphaU}) and Theorem \ref{representation} show $\alpha_g(U(h))=U(h)$ for any $g,h\in G$, 
and we get $N\rtimes_{\alpha'}G\subset M^G$. 
Since 
$$[M^G:N]=|G|=[N\rtimes_{\alpha'}G:N],$$ 
we get $M^G=N\rtimes_{\alpha'}G$. 

Since the image of $\rho_0\circ \iota$ is $N^G$, the duality between the fixed point inclusion 
$N^G\subset N$ and the crossed product inclusion $N\subset N\rtimes_{\alpha'}G$ implies that 
the endomorphism $\kappa\circ \rho_0\circ\iota\in \End_0(M^G)$ contains an automorphism, which we denote 
$\theta\in \Aut(M^G)$. 
Then the Frobenius reciprocity implies $[\kappa\rho_0]=[\theta\overline{\iota}]$, and we get 
$$[\rho]=[\iota\kappa\rho_0]=[\iota\theta\overline{\iota}].$$

Since $G$ is abelian, there exists an outer action $\beta:\hG\rightarrow \Aut(M^G)$ of the dual group 
$\hG$ of $G$ such that $M=M^G\rtimes_\beta \hG$ and $\alpha$ is the dual action of $\beta$. 
Thus 
$$[\biota\iota]=\bigoplus_{\chi\in \hG}[\beta_\chi].$$
We denote by $L$ the group generated by $[\theta]$ and $[\beta_{\hG}]$ in $\Out(M^G)$. 

\begin{lemma} Let the notation be as above. 
\begin{itemize} 
\item[(1)] For $\chi_1,\chi_2,\tau_1,\tau_2\in \hG$, we have $[\beta_{\chi_1}\theta\beta_{\chi_2}]=[\beta_{\tau_1}\theta\beta_{\tau_2}]$ if and only if 
$\chi_i=\tau_i$ for $i=1,2$. 
\item[(2)] $L=[\beta_{\hG}]\sqcup [\beta_{\hG}][\theta][\beta_{\hG}]$. 
\end{itemize}
\end{lemma}

\begin{proof} Since $\rho$ is irreducible, we have 
$$1=\dim(\rho,\rho)=\dim(\iota\theta\biota,\iota\theta\biota)=\dim(\biota\iota\theta,\theta\biota\iota)
=\sum_{\chi,\tau\in \hG}(\beta_\chi\theta,\theta,\beta_\tau),$$
which shows (1). 

Since $\rho$ is self-conjugate, we $[\iota\theta\biota]=[\iota\theta^{-1}\biota]$. 
This implies that $[\theta^{-1}]$ is contained in $[\biota\iota\theta\biota\iota]$, 
and it is an element in $[\beta_{\hG}][\theta][\beta_{\hG}]$. 
Let $L_1$ be the right-hand side of (2). 
The fusion rules of the category implies  
$$\sum_{g\in G}[\alpha_g]+(|G|-1)[\rho]=[\rho^2]=[\iota\theta\biota\iota\theta\biota]
=\sum_{\chi\in \hG}[\iota\theta\beta_\chi\theta\biota].$$
This means that $[\theta\beta_\chi\theta]$ belongs to $L_1$, which shows that $L_1$ is already a group. 
Therefore $L=L_1$.  
\end{proof}

In what follows, we identify $\hG$ with $[\beta_{\hG}]$ for simplicity. 
The above double coset decomposition implies that $L$ acting on $L/\hG$ is a sharply 2-transitive permutation group 
with the abelian point stabilizer $\hG$, which allows us to identify the pair $(L,\hG)$ 
with $(\F_q\rtimes \F_q^\times,\F^\times)$ (see \cite[Chapter XII, \S 9]{HB82}). 
The embedding $L\subset \Out(L)$ carries a cohomology class in $\omega\in H^3(L,\T)$ whose restriction to 
$\hG$ is trivial as it comes from the genuine group action $\beta$. 
Such a class $\omega$ is identified with a class in $H^3(\F_q,\T)^{\F_q^\times}$ by restriction 
thanks to the Lyndon-Hochschild-Serre spectral sequence. 

Now we reverse the above process. 
Assume that we are given a cohomology class $\omega\in H^3(\F_q\rtimes \F_q^\times,\T)$ 
whose restriction to $\F_q^\times$ is trivial. 
Let $P$ be a hyperfinite type III$_1$ factor. 
Then there exists an embedding of $\F_q\rtimes \F_q^\times$ into $\Out(P)$ carrying the class $\omega$, 
which is unique up to conjugacy in $\Out(P)$ thanks to Theorem \ref{uniqueness} (or see \cite{KT07}). 
We choose a lifting $\gamma:\F_q\rtimes \F_q^\rtimes\to \Aut(P)$ of it. 
Since the restriction of $\omega$ to $\F_q^\times$ is trivial, we may assume that $\gamma$ restricted to 
$\F_q^\times$ is an action, which we denote by $\beta$. 
Since $\F_q^\times$ is cyclic, the second cohomology $H^2(\F_q^\times,\T)$ is trivial, and such $\beta$ 
is unique on $\F_q^\times$ up to 1-cocycle perturbation. 
We denote by $G$ the dual group of $\F_q^\times$. 
Let $M=P\rtimes_\beta \F_q^\times$, and let $\iota:P\hookrightarrow M$ be the inclusion map. 
We choose an arbitrary $h\in \F_q\rtimes \F_q^\times \setminus \F_q^\times$, and set 
$\rho=\iota\gamma_h\biota$, whose equivalence class does not depend on the choice of $h$. 
The same computation as above shows that $\rho$ is irreducible and 
$$[\rho^2]=\sum_{g\in G}[\hat{\beta}_g]+(q-2)[\rho],$$
where $\hat{\beta}$ is the dual action of $\beta$. 
Therefore we get a C$^*$ near-group category with the group $G\cong \Z_{q-1}$ and the multiplicity 
parameter $m=q-2$.

\section{The noncommutative case}
In this section, we classify C$^*$ near group categories with noncommutative $G$. 
In this case, Theorem \ref{representation} implies that $d$ is an integer and 
there exist natural numbers $s,t$ with $t>1$ satisfying $n=|G|=st^2$, $m=(s-1)t$, and $d=st$, 
where we use the same notation as in Section \ref{basic}. 
Let $\hG$ be the set of equivalence classes of unitary representations of $G$, 
and let $\hG^\dagger=\hG\setminus \Hom(G,\T)$. 
Then Theorem \ref{representation} implies that $t$ divides $\dim \pi$ for every $\pi\in \hG^\dagger$, and $\#\Hom(G,\T)=t^2$. 
We denote by $[G,G]$ the commutator subgroup of $G$. 
Since $\Hom(G,\T)$ is identified with the dual group of $G/[G,G]$, 
we have $\#[G,G]=s$. 

Let $p$ be a prime number. 
A $p$-group $G$ is said to be extra-special if $Z(G)=[G,G]\cong \Z_p$, where $Z(G)$ is the center of $G$. 
Our goal in this section is to prove the following theorem. 

\begin{theorem}\label{noncommutative} Let $G$ be a noncommutative finite group. 
Then a C$^*$ near-group category with $G$ exists if and only if $s=2$ and $G$ is an extra-special 2-group.   
In particular $t$ is a power of 2, $n=|G|=2t^2$, $m=t$, and $d=2t$. 
For each extra-special 2-group, there exist exactly three different C$^*$ near-group categories. 
\end{theorem}

\begin{remark} As we will see in the proof below, the three fusion categories for a given extra-special 2-group 
$G$ are distinguished by the third Frobenius-Schur indicator $\nu_{3,1}(\rho)$. 
\end{remark}

\begin{example}The representation category of the bicrossed product Hopf algebra arising from the symmetric group 
$\mathfrak{S}_4=\mathfrak{S_3}\cdot \Z/4\Z$ is an example of such a fusion category with the dihedral group 
$G=D_8$ of order 8 (see \cite[Theorem 14.40,II]{IK02}).  
\end{example}

We will prove Theorem \ref{noncommutative} in several steps. 
Assume that $\cC$ is a C$^*$ near group category with a noncommutative $G$ and we use the same notation as in Section \ref{basic} 
and Theorem \ref{representation}. 
Let $N=\rho(M)$, which is an irreducible subfactor of index $d^2=s^2t^2$. 
Let $\alpha'$ be the outer action of $G$ on $N$ defined by $\alpha'_g=\rho\circ\alpha_g\circ \rho^{-1}$. 
Since $\rho\circ \alpha_g=\Ad U(g)\circ \rho$, we can identify the von Neumann algebra 
generated by $N$ and $U(G)$ with the crossed product $N\rtimes_{\alpha'} G$. 
We denote by $M^G$ the fixed point subalgebra of $M$ under the $G$-action $\alpha$, that is, 
$$M^G=\{x\in M;\; \alpha_g(x)=x,\; \forall g\in G\}.$$  
Since $\alpha_g\circ \rho=\rho$, the fixed point algebra $M^G$ is an intermediate subfactor between 
$M$ and $N$ with index $[M:M^G]=|G|=st^2$. 

\begin{lemma} Let the notation be as above. Then $N\rtimes_{\alpha'} G=M^{[G,G]}$ and $N\rtimes_{\alpha'} [G,G]=M^G$.  
\end{lemma}

\begin{proof}
Let $\iota_1:N\rtimes_{\alpha'} G\hookrightarrow M$, 
$\iota_2:N\rtimes_{\alpha'} [G,G]\hookrightarrow N\rtimes_{\alpha'} G$, and 
$\iota_3:N \hookrightarrow N\rtimes_{\alpha'} [G,G]$ be the inclusion maps. 
Then we have $\rho=\iota_1\iota_2\iota_3\rho_0$, where $\rho_0$ is $\rho$ regarded as 
an isomorphism from $M$ onto $N$.  
We have $d(\iota_1)=\sqrt{s}$, $d(\iota_2)=t$, $d(\iota_3)=\sqrt{s}$. 

Since $N\rtimes_{\alpha'}[G,G]\subset N\rtimes_{\alpha'}G$ is a crossed product inclusion by a $G/[G:G]$-action, 
and $G/[G,G]$ is an abelian group of order $t^2$, the endomorphism $\iota_2\overline{\iota_2}$ 
is decomposed into $t^2$ automorphisms. 
Since $[\rho^2]=[\rho\brho]$ 
contains $[\iota_1(\iota_2\overline{\iota_2})\overline{\iota_1}]$, 
if $\gamma$ is an automorphism contained in $\iota_2\overline{\iota_2}$, then $\iota_1\gamma\overline{\iota_1}$ is contained in 
$\rho^2$. 
Since $d(\iota_1\gamma\overline{\iota_1})=s<d(\rho)$ and 
$$[\rho^2]=\sum_{g\in G}[\alpha_g]+(s-1)t[\rho],$$ 
the endomorphism $\iota_1\gamma\overline{\iota_1}$ is decomposed into automorphisms, 
and it is the case for $\iota_1\iota_2\overline{\iota_2}\:\overline{\iota_1}$ as well. 
Since $d(\iota_1\iota_2)=\sqrt{s}t=\sqrt{\# G}$, we get 
$$[\iota_1\iota_2\overline{\iota_2}\;\overline{\iota_1}]=\sum_{g\in G}[\alpha_g].$$
This means that $N\rtimes [G,G]$ coincides with 
the fixed point algebra $M^G$. 

Since $\rho$ is self-conjugate, we can show $M^{[G,G]}=N\rtimes_{\alpha'} G$ too 
switching the roles of the crossed products and fixed point algebras.  
\end{proof}

Let $N^{[G,G]}$ be the fixed point algebra 
$$N^{[G,G]}=\{x\in N;\;\alpha'_g(x)=x,\;\forall g\in [G,G]\}.$$
Since $\rho_0(M^{[G,G]})=N^{[G,G]}$, the image of $\rho_0\iota_1$ is noting but $N^{[G,G]}$. 
By the duality between the crossed product subfactor $N\subset N\rtimes_{\alpha'}[G,G]=M^G$ and 
the fixed point algebra subfactor $N^{[G,G]}\subset N$, 
the homomorphism $\iota_3\rho_0\iota_1$ from $M^{[G,G]}$ to $M^G$ contains an isomorphism from 
$M^{[G:G]}$ onto $M^G$, say $\varphi$, and we have $\iota_3\rho_0=\varphi\overline{\iota_1}$ by the Frobenius reciprocity. 
Thus 
\begin{equation}
[\rho]=[\iota_1\iota_2\varphi\overline{\iota_1}]=[\iota_1\varphi^{-1}\overline{\iota_2}\:\overline{\iota_1}].
\end{equation}

Let $\widehat{[G,G]}$ be the set of equivalence classes of unitary representations of $[G,G]$. 
Then we have the following irreducible decomposition 
\begin{equation}
[\overline{\iota_2}\:\overline{\iota_1}\iota_1\iota_2]=\bigoplus_{\tau\in \Hom(G,\T)}[\beta_\tau]\oplus 
\bigoplus_{\pi\in \hG^\dagger}\dim \pi[\beta_\pi],
\end{equation}
\begin{equation}
[\overline{\iota_1}\iota_1]=\bigoplus_{\sigma\in \widehat{[G,G]}}\dim \sigma [\gamma_\sigma].
\end{equation}

\begin{lemma} \label{distinct} With the above notation, the following hold. 
\begin{itemize}
\item [(1)] the homomorphism $\beta_\pi\varphi\gamma_\sigma$ is irreducible for all $\pi \in \hG$ and $\sigma\in \widehat{[G,G]}$. 
\item [(2)] for $\pi,\pi'\in \hG$ and $\sigma,\sigma'\in \widehat{[G,G]}$, the two homomorphisms $\beta_\pi\varphi\gamma_\sigma$ and 
$\beta_{\pi'}\varphi\gamma_{\sigma'}$ are equivalent if and only if $\pi=\pi'$ and $\sigma=\sigma'$. 
\item [(3)] for $\pi,\pi'\in \hG$, we have $\dim (\iota_1\varphi^{-1}\beta_\pi,\iota_1\varphi^{-1}\beta_{\pi'})=\delta_{\pi,\pi'}$. 
\item [(4)] for $\sigma,\sigma'\in \widehat{[G,G]}$, we have $\dim (\iota_1\iota_2\varphi\gamma_\sigma,\iota_1\iota_2\varphi\gamma_{\sigma'})
=\delta_{\sigma,\sigma'}$. 
\end{itemize}
\end{lemma}

\begin{proof} Since $\rho$ is irreducible, 
\begin{align*}
1 &=\dim (\iota_1\iota_2\varphi\overline{\iota_1},\iota_1\iota_2\varphi\overline{\iota_1})
=\dim (\overline{\iota_2}\:\overline{\iota_1}\iota_1\iota_2\varphi,\varphi\iota_1\overline{\iota_1}) \\
&=\sum_{\pi\in \hG}\sum_{\sigma\in \widehat{[G,G]}}\dim \pi \dim \sigma \dim (\beta_\pi\varphi,\varphi\gamma_\sigma),
\end{align*}
which shows that $[\beta_\pi\varphi]=[\varphi\gamma_\sigma]$ if and only if $\pi=1$ and $\sigma=1$. 
Since the right-hand side of
$$\dim (\beta_\pi\varphi\gamma_\sigma,\beta_{\pi'}\varphi\gamma_{\sigma'})
= \dim (\beta_{\overline{\pi}'}\beta_\pi\varphi,\varphi\gamma_{\sigma'}\gamma_{\overline{\sigma}}),$$
is 1 if $\pi=\pi'$ and $\sigma=\sigma'$, and it is $0$ otherwise, we obtain (1) and (2). 
(3) and (4) follow from (1),(2) and the Frobenius reciprocity. 
\end{proof}

Let $\widehat{[G,G]}^\dagger=\widehat{[G,G]}\setminus\{1\}$. 

\begin{lemma} There exists a unique bijection $\Phi:\widehat{[G,G]}^\dagger\rightarrow \hG^\dagger$ such that 
$[\iota_1\iota_2\varphi\gamma_\sigma]=[\iota_1\varphi^{-1}\beta_{\Phi(\sigma)}\varphi]$ and $\dim \Phi(\sigma)=t\dim \sigma$. 
Moreover, there exists $\Psi(\sigma)\in \Hom([G,G],\T)$ for $\sigma\in \widehat{[G,G]}^\dagger$ satisfying  
$[\iota_2\varphi \gamma_\sigma]=[\gamma_{\Psi(\sigma)}\varphi^{-1}\beta_{\Phi(\sigma)}\varphi]$. 
\end{lemma}

\begin{proof} On one hand, we have 
$$[\rho^2]=\sum_{g\in G}[\alpha_g]+(s-1)t[\rho],$$
and on the other hand, 
$$[\rho^2]=[\iota_1\iota_2\varphi\overline{\iota_1}\iota_1\varphi^{-1}\overline{\iota_2}\:\overline{\iota_1}]
=\sum_{\sigma\in \widehat{[G,G]}}\dim \sigma [\iota_1\iota_2\varphi\gamma_\sigma\varphi^{-1}\overline{\iota_2}\:\overline{\iota_1}].$$
Since 
$$[\iota_1\iota_2\varphi\varphi^{-1}\overline{\iota_2}\:\overline{\iota_1}]=\sum_{g\in G}[\alpha_g],$$
for $\sigma\neq 1$, we have $[\iota_1\iota_2\varphi\gamma_\sigma\varphi^{-1}\overline{\iota_2}\:\overline{\iota_1}]=t\dim \sigma [\rho]$. 
Thus for $\sigma\neq 1$,
\begin{align*}
t\dim \sigma&=\dim (\iota_1\iota_2\varphi\gamma_\sigma\varphi^{-1}\overline{\iota_2}\:\overline{\iota_1},\rho)
=\dim (\iota_1\iota_2\varphi\gamma_\sigma,\iota_1\varphi^{-1}\overline{\iota_2}\:\overline{\iota_1}\iota_1\iota_2\varphi)\\
 &=\sum_{\pi\in \hG}\dim \pi \dim (\iota_1\iota_2\varphi\gamma_\sigma,\iota_1\varphi^{-1}\beta_\pi\varphi).
\end{align*}
Thanks to Lemma \ref{distinct}(3),(4), there exists unique $\pi$, which we call $\Phi(\sigma)$ such that 
$[\iota_1\iota_2\varphi\gamma_\sigma]=[\iota_1\varphi^{-1}\beta_\pi\varphi]$ and $\dim \pi=t\dim \sigma$. 
Moreover $\Phi$ is an injection. 
Since 
$$\sum_{\sigma\in \widehat{[G,G]}^\dagger}(\dim \sigma)^2=|[G,G]|-1=s-1,$$
we have 
$$\sum_{\sigma\in \widehat{[G,G]}^\dagger}(\dim \Phi(\sigma))^2=t^2(s-1)=\# G-\# \Hom(G,\T)=\sum_{\pi\in \hG^\dagger}(\dim \pi)^2,$$
and we see that $\Phi$ is a surjection. 
Since 
\begin{align*}1&=\dim (\iota_1\iota_2\varphi\gamma_\sigma,\iota_1\varphi^{-1}\beta_{\Phi(\sigma)}\varphi)
=\dim (\iota_2\varphi\gamma_\sigma,\overline{\iota_1}\iota_1\varphi^{-1}\beta_{\Phi(\sigma)}\varphi)\\
&= \sum_{\sigma'\in \widehat{[G,G]}}\dim \sigma'\dim (\iota_2\varphi\gamma_\sigma,\gamma_{\sigma'}\varphi^{-1}\beta_{\Phi(\sigma)}\varphi),
\end{align*}
there exists a unique $\sigma'\in \Hom([G,G],\T)$ such that 
$[\iota_2\varphi\gamma_\sigma]=[\gamma_{\sigma'}\varphi^{-1}\beta_{\Phi(\sigma)}\varphi]$. 
\end{proof}

\begin{lemma} The commutator subgroup $[G,G]$ is abelian, and $\dim \pi=t$ for all $\pi \in \hG^\dagger$.  
\end{lemma}

\begin{proof} Suppose that $[G,G]$ is non-abelian. 
Then there exists an irreducible representation $\sigma\in \widehat{[G,G]}$ such that $\sigma\otimes \overline{\sigma}$ 
contains a non-trivial irreducible representation $\mu$. 
Since  $[\iota_2\varphi\gamma_\sigma]=[\gamma_{\Psi(\sigma)}\varphi^{-1}\beta_{\Phi(\sigma)}\varphi]$, 
the endomorphism $\gamma_\mu$ is contained in $\varphi^{-1}\beta_{\overline{\Phi(\sigma)}}\beta_{\Phi(\sigma)}\varphi$, 
and there exists $\mu\in \hG$ such that $[\gamma_\mu]=[\varphi^{-1}\beta_\mu\varphi]$. 
However, this implies $[\varphi\gamma_\mu]=[\beta_\mu\varphi]$, which contradicts Lemma \ref{distinct} because 
$\mu\neq 1$. 
\end{proof}

We recall a well-known fact that the dimension of any irreducible projective representation 
of a finite group does not exceed the square root of the order of the group (see \cite[Problem 11.7]{Is94}). 

\begin{lemma} The commutator subgroup $[G,G]$ coincides with the center $Z(G)$ of $G$. 
\end{lemma}

\begin{proof} Let $\sigma\in \widehat{[G,G]}^\dagger$. Since $[\iota_2\varphi \gamma_\sigma]=[\gamma_{\Psi(\sigma)}\varphi^{-1}\beta_{\Phi(\sigma)}\varphi]$ 
and $\iota_2\overline{\iota_2}$ is decomposed into $t^2$ automorphisms, the endomorphism $\beta_{\Phi(\sigma)}\beta_{\overline{\Phi(\sigma)}}$ contains $t^2$ automorphisms. 
Since $\dim \Phi(\sigma)=t$, this means that 
$$\pi\otimes \overline{\pi}=\bigoplus_{\tau\in \Hom(G,\T)}\tau,$$
for all $\pi\in \hG^\dagger$. 
Let $g\in [G,G]$. 
Then $\pi(g)$ is a scalar for all $\pi\in \hG^\dagger$ because $\tau(g)=1$ for all $\tau\in \Hom(G,\T)$. 
This implies $g\in Z(G)$, and $[G,G]\subset Z(G)$. 

Since $\pi\in \hG^\dagger$ can be regarded as an irreducible projective representation of $G/Z(G)$, the dimension of $\pi$ does not exceed the square root of 
the order of $G/Z(G)$. 
Since $\dim \pi=t$, we have $t^2\leq \#G/Z(G)\leq \#G/[G,G]=t^2$, and $Z(G)=[G,G]$.   
\end{proof}

Recall that $\widehat{Z(G)}^\dagger=\widehat{Z(G)}\setminus \{1\}$. 

\begin{lemma}\label{tensorproduct} For each character $\sigma\in \widehat{Z(G)}^\dagger$, 
the induced representation $\mathrm{Ind}_{Z(G)}^G\sigma$ is decomposed as
$$\mathrm{Ind}_{Z(G)}^G\sigma\cong \overbrace{\pi_\sigma\oplus \cdots \oplus \pi_\sigma}^{t},$$
where $\pi_\sigma$ is an irreducible representation of dimension $t$. 
For $\sigma_1,\sigma_2\in \widehat{Z(G)}^\dagger$ with $\sigma_1\neq \sigma_2$, the corresponding two induced representations are mutually disjoint. 
Moreover, we have 
$\hG^\dagger=\{\pi_{\sigma}\}_{\sigma\in \widehat{Z(G)}^\dagger}$, and
$$\pi_{\sigma_1}\otimes \pi_{\sigma_2}\cong \overbrace{\pi_{\sigma_1\sigma_2}\oplus \cdots \oplus \pi_{\sigma_1\sigma_2}}^{t}, 
\; \mathrm{for}\; \sigma_1\neq \overline{\sigma_2},$$ 
$$\pi_\sigma\otimes \tau\cong \pi_\sigma,\quad \tau\in \Hom(G,\T),$$
$$\pi_{\sigma}\otimes \pi_{\overline{\sigma}}\cong \bigoplus_{\tau\in \Hom(G,\T)}\tau,$$
$$\overline{\pi_\sigma}\cong \pi_{\overline{\sigma}}.$$
\end{lemma} 

\begin{proof} For $\pi\in \hG^\dagger$ and $g\in Z(G)=[G,G]$, we have seen in the proof of the previous lemma that $\pi(g)$ is a scalar. 
Thus the restriction of $\pi$ to $Z(G)$ is decomposed as $t$ copies of some $\sigma\in \widehat{Z(G)}$. 
If $\sigma$ were trivial, the representation $\pi$ would reduce to a representation of $G/Z(G)=G/[G,G]$. 
Since $\dim \pi>1$ and $G/[G,G]$ is abelian, this is impossible, and $\sigma\in \widehat{Z(G)}^\dagger$. 
By the Frobenius reciprocity, the induced representation $\mathrm{Ind}_{Z(G)}^G\sigma$ contains $t$ copies of $\pi$. 
Since $\#G/Z(G)=t^2$ and $\dim \pi=t$, we obtain 
$$\mathrm{Ind}_{Z(G)}^G\sigma\cong \overbrace{\pi\oplus \cdots \oplus \pi}^{t}.$$
This means that $\pi\in \hG^\dagger$ is completely determined by its restriction to $Z(G)$ and the other statements follow.
\end{proof}

\begin{lemma} The group $Z(G)$ is cyclic, and $s+1$ is a prime power. 
\end{lemma}

\begin{proof} Since $t=\dim \pi_\sigma$, Theorem \ref{representation} implies that 
the representation $(V,\cK)$ has the following irreducible decomposition:  
$$(V,\cK)=\bigoplus_{\sigma\in \widehat{Z(G)}^\dagger}(\pi_\sigma,\cK_\sigma).$$
Eq.(\ref{j1}) implies $j_1\cK_\sigma=\cK_{\overline{\sigma}}$. 

Thanks to Eq.(\ref{Weyl}), the operator $U_{\cK}(h)$ for $h\in G$ is an intertwiner between $V$ and $\chi_h\otimes V$. 
Since the above lemma shows that $\chi_h\otimes \pi_\sigma$ is equivalent to $\pi_\sigma$, we have 
$U_{\cK}(h)\cK_\sigma=\cK_\sigma$. 
Thus Eq.(\ref{j2}) implies that there exists a permutation $\theta$ of order two of $\widehat{Z(G)}^\dagger$ such that $j_2\cK_\sigma=\cK_{\theta(\sigma)}$. 
Eq.(\ref{period3}) implies 
$$\theta(\overline{\theta(\overline{\theta(\overline{\sigma})})})=\sigma,\quad \sigma\in \widehat{Z(G)}^\dagger.$$

Eq.(\ref{orthogonality2}) shows that if $T\in \cK_\sigma$ is an isometry, we have 
\begin{equation}\label{l=0} I_{\cK_{\theta(\sigma)}}+l(T)^*l(T)=Q.\end{equation}
Thus if $s=2$, we have $l(T)=0$. In this case $Z(G)\cong \Z_2$ and the statement holds. 
Assume $s\neq 2$ now. 
Then $l(T)\neq 0$ for any $T\in \cK_\sigma\setminus \{0\}$. 
Eq(\ref{invariance}) and Eq.(\ref{equivariance}) implies 
$$l(\cK_\sigma) \subset \bigoplus_{\mu,\nu\in \widehat{Z(G)}^\dagger,\; \sigma=\theta(\nu)\overline{\theta(\mu)}
}\cK_\mu \cK_{\nu\overline{\mu}}\cK_\nu^*.$$ 
Eq.(\ref{Frobenius2}) implies $\theta(\theta(\nu)\overline{\theta(\mu)})=\theta(\nu\overline{\mu})\overline{\theta(\overline{\mu})}$. 
Thus \cite[Theorem 6.1]{S03} implies that $Z(G)$ is cyclic, and $s+1$ is a prime power. 
\end{proof} 

\begin{lemma} The number $s$ is a prime and $G/Z(G)$ is an elementary abelian $s$-group.  
In consequence, the group $G$ is an extra-special $s$-group.  
\end{lemma}

\begin{proof} Since $\dim \pi=t$ and $\# G/Z(G)=t^2$, if we regard $\pi_\sigma$ for $\sigma\in \widehat{Z(G)}^\dagger$ as a projective representation 
of $G/Z(G)$, it is faithful. 
Indeed, if it were not the case, there would be a normal subgroup $H$ of $G$ strictly containing $Z(G)$ such that we can regard $\pi_\sigma$ as 
an irreducible projective representation of $G/H$. 
However, this is impossible because $t=\dim \pi_\sigma$ cannot exceed the square root of the order of $G/H$. 

For each $g\in G/Z(G)$, we choose a lift $\tilde{g}$, and set $\omega(g,h)=\tilde{g}\tilde{h}\tilde{g}^{-1}\tilde{h}^{-1}\in Z(G)$ for $g,h\in G/Z(G)$. 
Then $\omega:(G/Z(G))^2\rightarrow Z(G)$ is an anti-symmetric bihomomorphism, which is independent of the choice of the lifts. 
Let $p$ be a prime dividing $t$. 
We choose $g\in G/Z(G)$ whose order is $p$. 
Assume that $s$ is not equal to $p$. 
Then there exists a character $\sigma\in \widehat{Z(G)}$ such that $\sigma^p\neq 1$.  
Since $\sigma^p(\omega(g,h))=1$ for all $h\in G/Z(G)$, the matrix $\pi_{\sigma^p}(\tilde{g})$ is a scalar. 
However, this would imply $g=e$, which is contradiction, and we obtain $s=p$. 
Since this is the case for all prime $p$ dividing $t$, we see that $G/Z(G)$ is a $p$-group. 
If there existed $g\in G/Z(G)$ satisfying $g^p\neq e$, for any $\sigma\in \widehat{Z(G)}^\dagger$ and $h\in G/Z(G)$, 
we would have $\sigma(\omega(g^p,h))=1$, and $\omega(g^p,h)=e$. 
This means $\tilde{g}^p\in Z(G)$, and is again contradiction. 
Thus we conclude that $G/Z(G)$ is an elementary abelian $p$-group. 
\end{proof}

\begin{lemma} The number $s$ is $2$, and $j_2j_1$ is a scalar. 
In particular, the group $G$ is an extra-special 2-group. 
\end{lemma}

\begin{proof} We choose an orthonormal basis $\{T(\sigma)_{i}\}_{i=1}^t$ of $\cK_\sigma$ for $\sigma\in \widehat{Z(G)}^\dagger$, and express $\pi_\sigma$ as 
$$\pi_\sigma(g)=\sum_{i,j=1}^t\pi_\sigma(g)_{ij}T(\sigma)_iT(\sigma)_j^*.$$
Since $j_2V(g)j_2^*=U_{\cK}(g)$, we have the irreducible decomposition of $U_{\cK}$ as 
$$U_{\cK}(g)=\sum_{\sigma\in \widehat{Z(G)}^\dagger}\overline{\pi_\sigma(g)_{ij}}j_2(T(\sigma)_i)j_2(T(\sigma)_j)^*.$$
Since $\chi_h\otimes \pi_\sigma\cong \pi_\sigma$, Eq.(\ref{rhoU2}) takes the following form now: 
\begin{align*}\lefteqn{
j_2j_1U_{\cK}(g)j_1^*j_2^*\otimes U_{\cK}(g)}\\
&=\frac{1}{st}\sum_{h,\sigma,i,j,p,q}\chi_g(h)\overline{\pi_\sigma(h)_{ij}}\pi_\sigma(h)_{pq}j_2T(\sigma)_iT(\sigma)_p^*j_2^*
\otimes j_1j_2T(\sigma)_jT(\sigma)_q^*j_2^*j_1^*\\
&+\sum_{\sigma,ij}\overline{\pi_\sigma(g)_{ij}}l(j_2T(\sigma)_i)l(j_2T(\sigma)_j)^*,
\end{align*}
which is equivalent to 
\begin{align*}\lefteqn{
j_1j_2V(g)j_2^*j_1^*\otimes V(g)}\\
&=\frac{1}{st}\sum_{h,\sigma,i,j,p,q}\overline{\chi_g(h)}\pi_\sigma(h)_{ij}\overline{\pi_\sigma(h)_{pq}}T(\sigma)_iT(\sigma)_p^*
\otimes j_2j_1j_2T(\sigma)_jT(\sigma)_q^*j_2^*j_1^*j_2\\
&+\sum_{\sigma,ij}\pi_\sigma(g)_{ij}(j_2\otimes j_2)l(j_2T(\sigma)_i)l(j_2T(\sigma)_j)^*(j_2\otimes j_2)^*.
\end{align*}
This and Lemma \ref{tensorproduct} imply
\begin{align*}\lefteqn{
\sum_{\sigma}j_1j_2\pi_\sigma(g)j_2^*j_1^*\otimes \pi_{\overline{\sigma}}(g)}\\
&=\frac{1}{st}\sum_{h,\sigma,i,j,p,q}\overline{\chi_g(h)}\pi_\sigma(h)_{ij}\overline{\pi_\sigma(h)_{pq}}T(\sigma)_iT(\sigma)_p^*
\otimes j_2j_1j_2T(\sigma)_jT(\sigma)_q^*j_2^*j_1^*j_2, 
\end{align*}
and 
\begin{align*}\lefteqn{
\sum_{\sigma\neq \overline{\tau}}j_1j_2\pi_\sigma(g)j_2^*j_1^*\otimes \pi_{\tau}(g)}\\
&=\sum_{\sigma,ij}\pi_\sigma(g)_{ij}(j_2\otimes j_2)l(j_2T(\sigma)_i)l(j_2T(\sigma)_j)^*(j_2\otimes j_2)^*.
\end{align*}

Recall that for $\sigma \in \widehat{Z(G)}^\dagger$, the map $G\times G\ni(g,h)\mapsto \sigma(ghg^{-1}h^{-1})\in \T$ induces a non-degenerate 
anti-symmetric bicharacter of $G/[G,G]$. 
Thus for each $g\in G$, there exists a map $\varphi_\sigma:G\rightarrow G$ satisfying $\overline{\chi_g(h)}
=\sigma(\varphi_\sigma(g)h\varphi_\sigma(g)^{-1}h^{-1})$. 
Note that the element $\varphi_\sigma(g)$ is unique up to a multiple of a central element of $G$. 
Using this, we have 
\begin{align*}
\lefteqn{\sum_{h\in G}\overline{\chi_g(h)}\pi_\sigma(h)_{ij}\overline{\pi_\sigma(h)_{pq}}} \\
&=\sum_{h\in G}\sigma(\varphi_\sigma(g)h\varphi_\sigma(g)^{-1}h^{-1})\pi_\sigma(h)_{ij}\overline{\pi_\sigma(h)_{pq}}\\
&=\sum_{h\in G}\pi_\sigma(\varphi_\sigma(g)h\varphi_\sigma(g)^{-1}h^{-1}h)_{ij}\overline{\pi_\sigma(h)_{pq}}\\
&=\sum_{h\in G}\sum_{a,b=1}^t\pi_\sigma(\varphi_\sigma(g))_{ia}\pi_\sigma(h)_{ab}\overline{\pi_\sigma(\varphi_\sigma(g))_{jb}\pi_\sigma(h)_{pq}}\\
&=\frac{n}{\dim \pi_\sigma}\pi_\sigma(\varphi_\sigma(g))_{ip}\overline{\pi_\sigma(\varphi_\sigma(g))_{jq}}.
\end{align*}
Thus we obtain
$$\sum_{\sigma\in \widehat{Z(G)}^\dagger}j_1j_2\pi_\sigma(g)j_2^*j_1^*\otimes \pi_{\overline{\sigma}}(g)
=\sum_{\sigma\in \widehat{Z(G)}^\dagger}\pi_\sigma(\varphi_\sigma(g))\otimes j_2j_1j_2\pi_\sigma(\varphi_\sigma(g))j_2^*j_1^*j_2^*.$$
This implies that there exists a scalar $c_{\sigma}(g)\in \T$ such that 
$$\pi_\sigma(\varphi_\sigma(g))=c_\sigma(g)j_1j_2\pi_{\theta(\overline{\sigma})}(g)j_2^*j_1^*,$$
$$j_2j_1j_2\pi_\sigma(\varphi_\sigma(g))j_2^*j_1^*j_2^*=\overline{c_\sigma(g)}\pi_{\overline{\theta(\overline{\sigma})}}(g).$$
Since $(j_2j_1)^3=I$ and $j_1\pi_\sigma(g)j_1^*=\pi_{\overline{\sigma}}(g)$, the above two equations are equivalent. 

In the above argument, we have seen that $\overline{\chi_g(h)}\pi_\sigma(h)=\pi_\sigma(\varphi_\sigma(g)h\varphi_{\sigma}(g)^{-1})$ holds. 
On the other hand, Eq.(\ref{Weyl}) implies that 
$$\overline{\chi_g(h)}\pi_\sigma(h)=j_2\pi_{\theta(\sigma)}(g^{-1})j_2^*\pi_\sigma(h)j_2\pi_{\theta(\sigma)}(g)j_2^*$$ 
holds, and so 
$$j_2\pi_{\theta(\sigma)}(g^{-1})j_2^*\sim j_1j_2\pi_{\theta(\overline{\sigma})}(g)j_2^*j_1^*,$$
where $\sim$ means that two matrices are proportional.  
Let $R=j_2j_1$, which is a unitary of period three. 
Then if $\sigma$ is replaced with $\theta(\sigma)$ and $g$ is replaced with $g^{-1}$ in the above, we obtain
$$\pi_\sigma(g)\sim Rj_2\pi_{\theta(\overline{\theta(\sigma)})}(g^{-1})j_2^*R^*= Rj_2\pi_{\overline{\theta({\overline{\sigma})}}}(g^{-1})j_2^*R^*
=R^*\pi_{\theta({\overline{\sigma})}}(g^{-1})R.$$
Since $R^3=I$ and $\theta(\overline{\theta(\overline{\theta(\overline{\sigma})})})=\sigma$, this implies $\pi_\sigma(g)\sim \pi_\sigma(g^{-1})$, and 
$g^2\in Z(G)$ for all $g\in G$. 
Therefore $s=2$. 

Let $\sigma$ be the unique element in $\widehat{Z(G)}^\dagger$, and let $\pi=\pi_{\sigma}$. 
The above argument shows that there exists $\tau\in \Hom(G,\T)$ such that $R^*\pi(g)R=\tau(g)\pi(g)$. 
Since $R$ has period three, we have $\tau^3=1$, and $\tau=1$ as $\Hom(G,\T)=\widehat{G/Z(G)}$ is an elementary 2-group. 
This implies that $R$ is a scalar. 
\end{proof}

We have seen that $G$ is an extra-special $2$-group of order $st^2=2t^2$.
Let $\sigma$ be a unique non-trivial character of $Z(G)\cong \Z_2$. 
Then $\pi_\sigma$ is a unique irreducible representation of dimension $t$, and we denote $(\pi_\sigma,\cK_\sigma)=(\pi,\cK_\pi)$ for simplicity. 
The set $\hG^\dagger$ is a singleton $\{\pi\}$, and we can identify $(V,\cK)$ with $(\pi,\cK_\pi)$. 
The period three unitary $j_2j_1$ is a scalar, which we denote by $\zeta\in \T$,  
and we have $j_2=\epsilon \zeta j_1$. 
This implies that $U_\cK=V$, and $\chi_h(g)$ is determined by $\pi(g)\pi(h)=\chi_h(g)\pi(h)\pi(g)$.  
Eq.(\ref{l=0}) implies $l=0$. 
In summary, every information of $\cC$ is encoded in the tuple 
$$(\cK=\cK_\pi,V=\pi, U_\cK=\pi, j_1,j_2=\epsilon \zeta j_1,\chi, l=0).$$  

Now we proceed to the reconstruction part. 
Recall that there are exactly two isomorphism classes of extra-special 2-groups for a given order $2^{2k+1}$, 
where $k$ is a natural number (see \cite[Exercise 5.3.7]{R93}). 
In the both cases, there exists only one irreducible representation whose dimension is larger than one, 
and its dimension is necessarily $2^k$.  

\begin{lemma} Let $G$ be an extra-special 2-group of order $2^{2k+1}$, and let $(\pi,\cK_\pi)$ be the unique 
irreducible representation $G$ of dimension $2^k$. 
Let $j$ be an anti-unitary on $\cK_\pi$ satisfying $j^2=\epsilon\in \{1,-1\}$ and $j\pi(g)=\pi(g)j$ for all $g\in G$, 
and let $\zeta$ be a third root of unity. 
Let $\chi$ be a bicharacter of $G$ determined by $\pi(g)\pi(h)=\chi_h(g)\pi(h)\pi(g)$. 
Then 
$$(\cK=\cK_\pi,V=\pi, U_\cK=\pi, j_1=j,j_2=\epsilon \zeta j,\chi, l=0),$$
is an admissible tuple.  
\end{lemma}

\begin{proof} We will show only Eq.(\ref{symmetric}),(\ref{rhoU2}),(\ref{S*rho2S}),(\ref{l1}),(\ref{l2}),(\ref{l3}) 
because the other conditions are easy to verify. 
Recall that we have $n=|G|=2^{2k+1}$, $m=2^k$, $d=2^{k+1}$. 

Since $ghg^{-1}h^{-1}\in Z(G)\cong \Z_2$, Eq.(\ref{symmetric}) holds. 
Since $d^2=2n$, Eq.(\ref{S*rho2S}) is automatically satisfied. 
Since $l=0$, Eq.(\ref{l2}) is automatically satisfied, and the left-hand sides of Eq.(\ref{l1}) and Eq.(\ref{l3}) are $0$. 
Since $\pi\otimes \pi\otimes \overline{\pi}$ contains no trivial representation, the right-hand side of 
Eq.(\ref{l3}) is $0$. 
For an orthonormal basis $\{T_i\}_{i=1}^{2^k}$ of $\cK=\cK_\pi$, let 
$$\pi(g)=\sum_{i,j=1}^{2^k}\pi(g)_{ij}T_iT_j^*.$$
Then 
\begin{align*}
\lefteqn{\frac{1}{d}\sum_{h\in G}\inpr{U_{\cK}(h)T''}{j_1(T)}j_1(U_{\cK}(h)T')} \\
 &=\frac{1}{d}\sum_{i,j,a,b=1}^{2^k}\sum_{h\in G}\pi(h)_{ij}\overline{\pi(h)_{ab}}
 \inpr{T_i}{j_1(T)}\inpr{T''}{T_j}\inpr{T_b}{T'}j_1(T_a) \\
 &=\frac{n}{d\dim \pi}\sum_{i,j,a,b=1}^{2^k}\delta_{i,a}\delta_{j,b} \inpr{T_i}{j_1(T)}\inpr{T''}{T_j}\inpr{T_b}{T'}j_1(T_a) \\
 &=\inpr{T''}{T'}j_1^2(T)=\epsilon \inpr{T''}{T'}T,
\end{align*}
and the right-hand side of Eq.(\ref{l1}) is $0$. 

The left-hand side of Eq.(\ref{rhoU2}) is $\pi(g)\otimes \pi(g)$. 
On the other hand, the right-hand side of Eq.(\ref{rhoU2}) is 
\begin{align*}
\lefteqn{\frac{1}{d}\sum_{i,j,x,y=1}^{2^k}\sum_{h\in G}\chi_h(g)\pi(h)_{ij}\overline{\pi(h)_{xy}}T_iT_x^*\otimes j_1T_jT_y^*j_1^*} \\
 &=\frac{1}{d}\sum_{i,j,x,y=1}^{2^k}\sum_{h\in G}\pi(ghg^{-1})_{ij}\overline{\pi(h)_{xy}}T_iT_x^*\otimes j_1T_jT_y^*j_1^*\\
 &=\frac{1}{d}\sum_{i,j,x,y,a,b=1}^{2^k}\sum_{h\in G}\pi(g)_{ia}\pi(h)_{ab}\overline{\pi(g)_{jb}}\overline{\pi(h)_{xy}}T_iT_x^*\otimes j_1T_jT_y^*j_1^*\\
 &=\frac{n}{d\dim \pi}\sum_{i,j,x,y,a,b=1}^{2^k}\delta_{a,x}\delta_{b,y}\pi(g)_{ia}\overline{\pi(g)_{jb}}T_iT_x^*\otimes j_1T_jT_y^*j_1^*\\
 &=\pi(g)\otimes j_1\pi(g)j_1^*\\
 &=\pi(g)\otimes \pi(g),
\end{align*}
which finishes the proof. 
\end{proof}

\begin{proof}[Proof of Theorem \ref{noncommutative}]
We use the same notation as in the above lemma. 
Since $\pi$ is irreducible, if $j'$ is another anti-unitary on $\cK_\pi$ satisfying $j'\pi(g)=\pi(g)j'$ for any $g\in G$, 
there exists $c\in \T$ with $j'=cj$. 
Choosing a square root $c^{1/2}$ of $c$, we get $j'=c^{1/2}jc^{-1/2}$. 
Therefore the equivalence class of the tuple in the above lemma is determined by the third root of unity $\zeta$. 
Theorem \ref{unique} shows that there are exactly three C$^*$ near-group categories for a given extra-special 2-group 
$G$ as far as $j$ exists. 
Therefore to finish the proof, it suffices to show the existence of $j$ satisfying the condition 
of the above lemma. 

When $|G|=8$, we have $\dim \pi=2$, and the anti-unitary $j$ with $j^2=1$ is unitarily equivalent to the complex conjugation of $\C^2$. 
Thus the condition $j\pi(g)=\pi(g)j$ is equivalent to the existence of a real representation equivalent to $\pi$. 
The dihedral group $D_8$ of order 8 satisfies this condition. 
When $\dim \pi=2$ and $j^2=-1$, the condition $j\pi(g)=\pi(g)j$ is equivalent to $\pi(G)\subset SU(2)$. 
The quaternion group $Q_8$ satisfies this condition. 

Now we consider the general case. 
Note that every extra-special 2-group of order $2^{2k+1}$ is a central product of either $k$ copies of 
$D_8$, or $k-1$ copies of $D_8$ and 1 copy of $Q_8$ (see \cite[Exercise 5.3.7]{R93}). 
Moreover, the irreducible representation $\pi$ can be constructed from 2-dimensional irreducible representations of 
$D_8$ and $Q_8$ by tensor product. 
Thus there exists $j$ as in the previous lemma with $\epsilon=1$ in the former case, and $\epsilon=-1$ in the latter case. 
This finishes the proof.
\end{proof}

We denote by $\cC_{G,\zeta}$ the C$^*$ near-group category arising from the tuple 
$$(\cK=\cK_\pi,V=\pi, U_\cK=\pi, j_1=j,j_2=\epsilon \zeta j,\chi, l=0).$$

\begin{cor} The fusion category $\cC_{G,\zeta}$ is group theoretical. 
\end{cor}

\begin{proof} We may assume $\cC_{G,\omega}\subset \End_0(M)$ and we use the same notation as before. 
We claim that there exists an abelian normal subgroup $H\triangleleft G$ of order $2^{k+1}$ with $Z(G)\subset H$ and $\chi_{h_1}(h_2)=1$ 
for any $h_1,h_2\in H$. 
Indeed, since $G$ is a central product of copies of $D_8$ and $Q_8$, it suffices to verify the claim for $G=D_8$ and $G=Q_8$, 
which is straightforward. 

Let $M^H$ and $N^H$ be the fixed point algebras 
$$\{x\in M;\;\alpha_h(x)=x,\; \forall h\in H\},$$
$$\{x\in N;\;\alpha'_h(x)=x,\; \forall h\in H\}.$$
We claim $M^H=N\rtimes_{\alpha'}H$. 
Indeed, we have $M^H\subset M^{Z(G)}=N\rtimes_{\alpha'}G$, and every element in $N\rtimes_{\alpha'}G$ is uniquely expands as 
$$\sum_{g\in G}U(g)\rho(x_g).$$
Thus Eq.(\ref{alphaU}) implies $N\rtimes_{\alpha'}H\subset M^H$. 
Since, 
$$[M^H:N]=[M:N]/|H|=2^{k+1}=|H|=[N\rtimes_{\alpha'}H:N],$$
we get the claim. 

Let $\nu:M^H\hookrightarrow M$ and $\mu:N\hookrightarrow M^H$ be the inclusion maps. 
Then we have $\rho=\nu\mu\rho_0$, and the image of $\rho_0\nu$ is $N^H$. 
By the duality between the crossed product inclusion $N\subset N\rtimes_{\alpha'}H=M^H$ and the fixed point algebra inclusion $N^H\subset N$, 
the endomorphism $\mu\rho_0\nu\in \End_0(M^H)$ contains an automorphism, say $\theta\in \Aut(M^H)$. 
Thus the Frobenius reciprocity implies that $[\mu\rho_0]=[\theta\overline{\nu}]$ and we get 
$[\rho]=[\nu\theta\overline{\nu}]$. 

Since $H$ is abelian, the endomorphism $\overline{\nu}\nu$ is decomposed into automorphisms, and so is $\overline{\nu}\rho\nu$. 
Since $\alpha_g$ normalizes $\alpha_H$, it globally preserves $M^H$, and there exists $\beta_g\in \Aut(M^H)$ satisfying 
$\alpha_g\nu=\nu\beta_g$. 
This implies that $\overline{\nu}\alpha_g\nu=\overline{\nu}\nu\beta_g$ is decomposed into automorphisms too. 
Thus the fusion category generated by $\overline{\nu}\rho\nu$ and $\overline{\nu}\alpha_g\nu$ is a pointed category, 
which is categorically Morita equivalent to $\cC_{G,\zeta}$.  
\end{proof}

\begin{remark} Let the notation be as in the above proof. 
\begin{itemize}
\item[(1)] The subgroup $H$ is not unique. 
For $D_8$, there are three possibilities; two of them are isomorphic to $\Z_2\times \Z_2$ and 
the other is isomorphic to $\Z_4$. 
For $Q_8$ there are three possibilities, and they are all isomorphic to $\Z_4$. 
\item[(2)] Let $L$ be the group generated by the automorphisms contained in $\overline{\nu}\rho\nu$ and 
$\overline{\nu}\alpha_g\nu$ in $\Out(M^H)$. 
Since the global dimension of $\cC_{G,\zeta}$ is $|G|+d^2=2^{2k+1}\times 3$, the order of $L$ is 
$2^{2k+1}\times 3$. 
The group $\Z_2^{2k}\rtimes SL(2,2)\cong \Z_2^{2k}\rtimes \fS_3$ is a canonical candidate of $L$ 
as we see below. 
When $k=1$, this group is isomorphic to $\fS_4$. 
\end{itemize}
\end{remark}

\begin{example} Let $\F_2$ be the finite field of order 2, and we consider the following subgroups of 
$SL(2+k,2)$: 
$$L=\{\left(
\begin{array}{ccc}
x &y &a  \\
z &u &b  \\
0_k^T&0_k^T &I_k 
\end{array}
\right)\in SL(2+k,2);\; a,b\in \F_2^k,\;\left(
\begin{array}{cc}
x &y  \\
z &u 
\end{array}
\right)\in SL(2,2)\},$$
$$A=\{\left(
\begin{array}{ccc}
1 &y &a  \\
0 &1 &0_k  \\
0_k^T &0_k^T &I_k 
\end{array}
\right)\in SL(2+k,2);\; y\in \F_2,\; a\in \F_2^k\}.$$
Then $L\cong \Z_2^{2k}\rtimes SL(2,2)\cong \Z_2^{2k}\rtimes \fS_3$ and $A\cong \Z_2^{k+1}$. 
The normalizer of $A$ in $L$ is 
$$N_L(A)=\{\left(
\begin{array}{ccc}
1 &y &a  \\
0 &1 &b  \\
0_k^T &0_k^T &I_k 
\end{array}
\right)\in L;\; a,b\in \F_2^k,\; y\in \F_2
\},$$
and $N_L(A)=A\rtimes K$ with 
$$K=\{\left(
\begin{array}{ccc}
1 &0 &0_k  \\
0 &1 &b  \\
0_k^T &0_k^T &I_k 
\end{array}
\right)\in L;\; b\in \F_2^k
\}\cong \Z_2^k.$$
Let 
$$B=\{\left(
\begin{array}{ccc}
1 &0 &a  \\
0 &1 &0_k  \\
0_k^T &0_k^T &I_k 
\end{array}
\right)\in SL(2+k,2);\; a\in \F_2^k\}.$$
Let $P$ be a type III factor and let $\beta:L \to \Aut(P)$ be a map that induces an injective homomorphism 
from $L$ into $\Out(P)$. 
We assume that the restriction of $\beta$ to $A$ is an action and $\beta_k$ normalizes $\beta_A$ for any $k\in K$. 
Let $M=P\rtimes_\beta A$, and let $\nu:P\hookrightarrow M$ be the inclusion map, let 
$$f=\left(
\begin{array}{ccc}
0 &1 &0_k  \\
1 &0 &0_k  \\
0_k^T &0_k^T &I_k 
\end{array}
\right)\in L,
$$ 
and let $\theta=\beta_f$. 
We claim that $\rho=\nu\theta\overline{\nu}$ generates a C$^*$ near-group category with 
a noncommutative $G$ with $|G|=2^{2k+1}$. 
Indeed, since $f^{-1}=f$, the endomorphism $\rho$ is self-conjugate. 
Since 
$$\dim(\rho,\rho)=\dim(\overline{\nu}\nu\beta_f,\beta_f\overline{\nu}\nu)=\sum_{g,h\in A}(\beta_{fgf},\beta_h)=1,$$
$\rho$ is irreducible. 
Since $\beta_k$ for $k\in K$ normalizes $\beta_A$, it extends to an automorphism $\tilde{\beta}_k\in \Aut(M)$. 
We denote by $\hat{\beta}$ the dual action of $\beta$ restricted to $A$. 
Then since 
$$f\left(
\begin{array}{ccc}
1&0 &a  \\
0&1 &0_k  \\
0_k^T &0_k^T &I_k 
\end{array}
\right)f=\left(
\begin{array}{ccc}
1&0 &0_k  \\
0&1 &a  \\
0_k^T &0_k^T &I_k 
\end{array}
\right)\in K,$$
$$f\left(
\begin{array}{ccc}
1&1 &a  \\
0&1 &0_k  \\
0_k^T &0_k^T &I_k 
\end{array}
\right)f=\left(
\begin{array}{ccc}
1&0 &0_k  \\
1&1 &a  \\
0_k^T &0_k^T &I_k 
\end{array}
\right)
=\left(
\begin{array}{ccc}
1&1 &a  \\
0&1 &0_k  \\
0_k^T &0_k^T &I_k 
\end{array}
\right)f
\left(
\begin{array}{ccc}
1&1 &a  \\
0&1 &0_k  \\
0_k^T &0_k^T &I_k 
\end{array}
\right),
$$
we get 
\begin{align*}
\lefteqn{[\rho^2]=[\nu\theta\overline{\nu}\nu\theta\overline{\nu}]=\sum_{g\in A}[\nu\beta_{fgf}\overline{\nu}]
=\sum_{k\in K}[\nu\beta_k\overline{\nu}] +\sum_{h\in B}[\nu\beta_{hfh}\overline{\nu}]} \\
 &=\sum_{k\in K}[\tilde{\beta}_k\nu\overline{\nu}] +\sum_{h\in B}[\nu\theta\overline{\nu}]
 =\sum_{\tau\in \hat{A},\;k\in K}[\tilde{\beta}_k\hat{\beta}_\tau]+2^k[\rho]. 
\end{align*}
This shows the claim with $G$ equal to $\{[\tilde{\beta}_k\hat{\beta}_\tau]\}_{\tau\in \hat{A},\;k\in K}\subset \Out(M)$, 
which is an extension 
$$0\to \hat{A}\to G\to K\to 0.$$ 
It is an interesting problem to compute a cohomological invariant for $\beta$, 
and identify it with the third root of unity $\zeta$ in $\cC_{G,\zeta}$. 

The extra-special 2-group $G$ obtained in this way is always a central product of $k$ copies of $D_8$. 
To obtain the other type of an extra-special 2-group of the same order, we should replace $A$ with a subgroup 
$A_1$ isomorphic to $\Z_4\times \Z_2^{k-1}$. 
For example, we could choose $A_1$ to be the group generated by $B$ and 
$$\left(
\begin{array}{ccc}
1 &1 &0_k  \\
0 &1 &e_1  \\
0_k^T &0_k^T &I_k 
\end{array}
\right),\quad e_1=(1,0,\cdots,0)\in \F_2^k.$$ 
Note that the group $G$ is not uniquely determined by the pair $(L,A_1)$, but it also depends on 
the choice of $\beta$. 
The pointed category generated by $\beta_{L}$ carries a cohomology class in $H^3(L,\T)$. 
On the other hand, there is a restriction map from $H^3(L,\T)$ to $H^3(SL(2,2),\T)\cong \Z_6$. 
We conjecture that $H^3(SL(2,2),\T)$ parameterizes 6 different $C^*$ near-group categories 
$\cC_{G,\zeta}$ with $|G|=2^{2k+1}$. 
\end{example}

\section{Irrational case} 
In the rest of this note, we investigate the structure of C$^*$ near-group categories in the irrational case. 
This section is devoted to laying the basis for the classification of such categories. 

Assume that we have a C$^*$ near group category as in Section \ref{basic} with irrational $d$. 
Then Theorem \ref{representation} implies that $G$ is abelian, the multiplicity parameter $m$ is a multiple of $n=|G|$, and 
the symmetric bicharacter $\inpr{\cdot}{\cdot}:G\times G\rightarrow \T$ defined by $\inpr{g}{h}=\chi_g(h)$ is non-degenerate. 
In what follows, we use additive notation for $G$. 
We have two unitary representations $V(g), U_{\cK}(g)$ in $\cK$, equivalent to the regular representation of $G$, 
and two anti-unitaries $j_1,j_2$ satisfying the conditions stated in Section \ref{basic}. 
Let $v_0(g)=V(g)$, $v_1(g)=U_{\cK}(g)$, $v_2(g)=(j_1j_2)^*U_{\cK}(g)j_1j_2$, and $w=j_1j_2$. 
Then they satisfy the following relations: $w^3=I$, $j_1^2=\epsilon\in \{1,-1\}$, $j_1v_i(g)j_1^{-1}=v_{-i}(g)$, $j_1wj_1^{-1}=w^*$, 
$w^*v_i(g)w=v_{i+1}(g)$, 
$$v_{i+1}(g)v_i(h)=\inpr{h}{g}v_{i}(h)v_{i+1}(g),$$
where $i$ is understood as an element of $\Z/3\Z$.

\begin{lemma}\label{3nreps} Let $G$ be a finite abelian group of order $n$, and let $\inpr{\cdot}{\cdot}:G\times G\rightarrow \T$ be 
a non-degenerate symmetric bicharacter. 
Let $\cH(G)$ be the universal C$^*$-algebra generated by three unitary representations $v_0,v_1,v_2$ of $G$ 
and a unitary $w$ satisfying the following commutation relations:
\begin{equation}
v_{i+1}(g)v_i(h)=\inpr{h}{g}v_{i}(h)v_{i+1}(g),
\end{equation}
\begin{equation}
w^*v_i(g)w=v_{i+1}(g),
\end{equation}
\begin{equation}
w^3=1,
\end{equation}
where $i$ is understood as an element of $\Z/3\Z$. 
Then there exist exactly $3n$ irreducible representations of $\cH(G)$, and they are of the following form: 
The representation space is $\ell^2(G)$ and 
\begin{equation}
\pi_{a,c}(v_0(g))f(h)=\inpr{g}{h}f(h),
\end{equation}
\begin{equation}\pi_{a,c}(v_1(g))f(h)=f(h+g),
\end{equation}
\begin{equation}
\pi_{a,c}(v_2(g))f(h)=a(h)\overline{a(h-g)}f(h-g),
\end{equation}
\begin{equation}
\pi_{a,c}(w)f(h)=\frac{c}{\sqrt{n}}\sum_{k}a(h)\overline{\inpr{h}{k}}f(k),
\end{equation}
where $a:G\rightarrow \T$ and $c\in \T$ satisfy 
\begin{equation}\label{coboundary}
a(g+h)\inpr{g}{h}=a(g)a(h),
\end{equation}\label{c}
\begin{equation}
c^3\sum_{g\in G}a(g)=\sqrt{n}. 
\end{equation}
\end{lemma}

\begin{proof} Note that the dimension of the $*$-algebra generated by $v_0$,$v_1$, $v_2$, and $w$ is $3n^3$. 
Thus it suffices to show that there exist $3n$ irreducible representations of $\cH(G)$ of dimension $n$ as above. 

Since $\inpr{\cdot}{\cdot}$ is a symmetric 2-cocycle, it is  a coboundary and a function $a$ satisfying Eq.(\ref{coboundary}) certainly exists. 
We choose one of them and denote it by $a$. 
Let $a':G\rightarrow \T$. 
Then $a'$ satisfies Eq.(\ref{coboundary}) if and only if $a'/a$ is a character, and so there are exactly $n$ such functions. 

For a function $f$ on $G$, we define its Fourier transform $\cF f(h)=\hat{f}(h)$ by 
$$\hat{f}(g)=\frac{1}{\sqrt{n}}\sum_{h\in G}\overline{\inpr{g}{h}}f(h).$$
Then Eq.(\ref{coboundary}) implies 
\begin{equation}\label{Fourier-a}
\hat{a}(g)=\hat{a}(0)\overline{a(g)}.
\end{equation}
Since the Fourier transform preserves the $\ell^2$-norm, we have $|\hat{a}(0)|=1$. 
Thus for a given $a$, there exists exactly three $c\in \T$ satisfying Eq.(\ref{c}). 
We choose one of them. 

Let $\mu(g)f(h)=\inpr{g}{h}f(h)$ and $\varrho(g)f(h)=f(h+g)$. 
Then the commutation relations $\varrho(g)\mu(h)=\inpr{g}{h}\mu(h)\varrho(g)$ and $\mu(g)\cF=\cF\varrho(g)$ hold. 
Let $Af(h)=a(h)f(h)$. 
Then we have $\mu(g)A\cF=A\cF\varrho(g)$. 
To show that $\pi_{a,c}$ is a well-defined irreducible representation, it suffices to show $c^3(A\cF)^3=1$ because 
if it is the case, we can put $\pi_{a,c}(w)=cA\cF$, $\pi_{a,c}(v_0(g))=\mu(g)$, $\pi_{a,c}(v_1(g))=\varrho(g)$, and 
$$\pi_{a,c}(v_2(g))=\pi_{a,c}(w)^*\pi_{a,c}(v_1(g))\pi_{a,c}(w).$$ 
Indeed, 
\begin{align*}
\cF A\cF f(g) &=\frac{1}{n}\sum_{h,k\in G}\overline{\inpr{g}{h}}a(h)\overline{\inpr{h}{k}}f(k)
=\frac{1}{\sqrt{n}}\sum_{k\in G}\hat{a}(g+k)f(k) \\
 &=\frac{\hat{a}(0)}{\sqrt{n}}\sum_{k\in G}\overline{a(g+k)}f(k)
 =\frac{\hat{a}(0)}{\sqrt{n}}\sum_{k\in G}\overline{a(g)}\inpr{g}{k}\overline{a(k)}f(k)\\
 &=\hat{a}(0)A^*\cF^*A^*f(g).
\end{align*}
This shows $(A\cF)^3=\hat{a}(0)$. 

We now show that if $\pi_{a,c}$ and $\pi_{a',c'}$ are equivalent then $(a,c)=(a',c')$. 
Note that they coincide on $C^*\{v_0,v_1\}$, which are already irreducible. 
Thus if they are equivalent, they must coincide on $\cH(G)$, which implies $(a,c)=(a',c')$. 
\end{proof}

\begin{remark}\label{r3nrep} We always have the following relation:
$$\pi_{a,c}(v_2(g))=a(g)\pi_{a,c}(v_1(-g))\pi_{a,c}(v_0(-g)).$$
\end{remark}

We introduce an anti-linear involutive $*$-isomorphism $\kappa$ of $\cH(G)$ by setting 
$$\kappa(v_i(g))=v_{-i}(g),\quad  \kappa(w)=w^{-1}.$$ 
We say that $(\pi,j,\cK)$ is a covariant representation of $(\cH(G),\kappa)$ if 
$(\pi,\cK)$ is a representation of $\cH(G)$ and $j$ is an anti-unitary of $\cK$ such that 
$\pi(\kappa(x))=j\pi(x)j^{-1}$ holds for all $x\in \cH(G)$ and $j^2=\epsilon$ is a scalar. 
We say that a covariant representation is even (resp. odd) if $\epsilon=1$ (resp. $\epsilon=-1$). 
We would like to classify all even and odd covariant representations. 

We fix $a:G\rightarrow \T$ satisfying Eq.(\ref{coboundary}) and $a(g)=a(-g)$ for all $g\in G$, 
and set $a_\chi(g)=a_1(g)\chi(g)$ for all $\chi\in \hG$. 
Let $c_{\chi,i}$, $i=0,1,2$ be the three solutions of Eq.(\ref{c}) for $a_\chi$. 
Note that $\widehat{a_\chi}(0)=\widehat{a_{\chi^{-1}}}(0)$ holds. 
Indeed, choosing $h\in G$ satisfying $\chi(g)=\inpr{g}{h}$ for all $g\in G$, we obtain 
$$\widehat{a_\chi}(0)=\hat{a}(-h)=\hat{a}(0)\overline{a(-h)}=\hat{a}(0)\overline{a(h)}=\hat{a}(h)=\widehat{a_{\chi^{-1}}}(0).$$
Thus we may and do assume $c_{\chi,i}=c_{\chi^{-1},i}$. 

Let $(\pi,\cK)$ be a representation of $\cH(G)$. 
Then Lemma \ref{3nreps} shows that up to unitary equivalence, we may assume the following: 
$\cK=\ell^2(G)\otimes \cK_0$, $\pi(v_0(g))=\mu(g)\otimes I$, $\pi(v_1(g))=\varrho(g)\otimes I$, 
$$\pi(w)=\sum_{g\in G}(e_g\otimes CA(g))(\cF\otimes I),$$
where $e_g$ is the projection from $\ell^2(G)$ onto $\C\delta_g$. 
$\cK_0$ is decomposed as 
$$\cK_0=\bigoplus_{\chi\in \hG,\;i=1,2,3}\cK_{\chi,i}$$
and $A(g)|_{\cK_{\chi,i}}=a_\chi(g)I$, $C|_{\cK_{\chi,i}}=c_{\chi,i}I$. 

Let $j$ be an anti-linear unitary of $\cK$ satisfying $j^2=\epsilon$. 
Then $(\pi,K,j)$ is a covariant representation of $(\cH(G),j)$ if and only if 
\begin{equation}\label{covariant1}j\pi(v_0(g))=\pi(v_0(g))j,
\end{equation}  
\begin{equation}\label{covariant2}j\pi(w)=\pi(w)^{-1}j.
\end{equation} 
Indeed, it is clear that the two conditions Eq.(\ref{covariant1}),(\ref{covariant2}) are necessary. 
Assume conversely that they are satisfied. 
Then 
\begin{align*}j\pi(v_1(g))j^{-1}&=j\pi(w)^{-1}\pi(v_0(g))\pi(w)j^{-1}=\pi(w)j\pi(v_0(g))j^{-1}\pi(w)^{-1}\\
&=\pi(w)\pi(v_0(g))\pi(w)^{-1}=\pi(v_2(g)),
\end{align*}
which shows that  $(\pi,H,j)$ is a covariant representation. 

The equation Eq.(\ref{covariant1}) is equivalent to the condition that there exists a family of anti-unitaries $j(g)$ on $H$ 
such that $j\delta_h\otimes \xi=\delta_{-h}\otimes j(h)\xi$. 
The condition $j^2=\epsilon$ is equivalent to $j(-g)j(g)=\epsilon$. 
Under these conditions, we have 
\begin{equation}
j\pi(w)\delta_h\otimes \xi=\frac{1}{\sqrt{n}}j\sum_{k\in G}\overline{\inpr{k}{h}}\delta_k\otimes CA(k)\xi
=\frac{1}{\sqrt{n}}\sum_{k\in G}\inpr{k}{h}\delta_{-k}\otimes j(k)CA(k)\xi,
\end{equation}
\begin{equation}
\pi(w)^{-1}j\delta_h\otimes \xi=\pi(w)^{-1}\delta_{-h}\otimes j(-h)\xi=\frac{1}{\sqrt{n}}\sum_{k\in G}\inpr{k}{h}\delta_{-k}\otimes A(-h)^*C^*j(h)\xi. 
\end{equation}
Thus Eq.(\ref{covariant2}) is equivalent to 
\begin{equation}\label{covariant2'}
j(k)CA(k)=A(-h)^*C^*j(h).
\end{equation}

We claim that Eq.(\ref{covariant2'}) is equivalent to the following two:
\begin{equation}\label{j(0)C}
j(0)C=C^*j(0),
\end{equation}
\begin{equation}\label{j(h)}
j(h)=A(-h)j(0)=j(0)A(h)^*.
\end{equation}
Assume Eq.(\ref{covariant2'}) first. 
Then for $h=k=0$, we get Eq.(\ref{j(0)C}). 
This together with Eq.(\ref{covariant2'}) for the case $k=0$ and the case $h=0$ implies Eq.(\ref{j(h)}). 
It is straightforward to show that Eq.(\ref{j(0)C}) and (\ref{j(h)}) imply Eq.(\ref{covariant2'}), 
and the claim is shown. 
Assuming these equivalent conditions, we known that $j(-h)j(h)=\epsilon$ is equivalent to $j(0)^2=\epsilon$. 

Summing up the above argument, now we have the following lemma:

\begin{lemma}\label{clacov} Every even covariant representation of $(\cH(G),\kappa)$ is a direct sum of the following covariant representations 
$(\pi,H,j)$: 
\begin{itemize} 
\item [$(1)$] $H=\ell^2(G)$, $\pi=\pi_{\chi,c}$ with $\chi^2=1$,  and $j\delta_h=a_\chi(h)\delta_{-h}$. 
\item [$(2)$] $H=\ell^2(G)\oplus \ell^2(G)$, $\pi=\pi_{\chi,c}\oplus \pi_{\chi^{-1},c}$ with $\chi^2\neq 1$ and 
$$j(x\delta_h\oplus y\delta_k)=\overline{y}a_{\chi^{-1}}(k)\delta_{-k}\oplus \overline{x}a_\chi(h)\delta_{-h}.$$ 
\end{itemize} 

Every odd covariant representation of $(\cH(G),\kappa)$ is a direct sum of the following covariant representations 
$(\pi,H,j)$: $H=\ell^2(G)\oplus \ell^2(G)$, $\pi=\pi_{\chi,c}\oplus \pi_{\chi^{-1},c}$ and 
$$j(x\delta_h\oplus y\delta_k)=-\overline{y}a_{\chi^{-1}}(k)\delta_{-k}\oplus \overline{x}a_\chi(h)\delta_{-h}.$$ 
\end{lemma}

We get back to the C$^*$ near-group category $\cC$ with irrational $d$.  
We may identify $\cK$ with $\ell^2(G)\otimes \cK_0$ as above. 
We denote by $T_h(\xi)\in \cK$ the element corresponding to $\delta_h\otimes \xi$. 

\begin{lemma}\label{ACJ} For a given C$^*$ near group category with a finite abelian group $G$ and irrational $d$ 
as in Section \ref{basic}, there exist a non-degenerate symmetric bicharacter $\inpr{\cdot}{\cdot}:G\times G\to \T$, 
an anti-unitary $J$, and mutually commuting unitaries $A(g)$, $C$  acting on $\cK_0$ satisfying 
$J^2=\epsilon$, $JC=C^*J$, $JA(g)=A(-g)^*J$, 
$$A(g)A(h)=\inpr{g}{h}A(g+h),$$
$$\frac{1}{\sqrt{n}}\sum_{g\in G}A(g)=C^{-3},$$
$$V(g)T_h(\xi)=\inpr{g}{h}T_h(\xi),$$
$$U_{\cK}(g)T_h(\xi)=T_{h-g}(\xi),$$
$$j_1(T_h(\xi))=T_{-h}(A(-h)J\xi),$$
$$j_2(T_h(\xi))=\frac{\epsilon}{\sqrt{n}}\sum_{k\in G}\overline{\inpr{h}{k}}T_k(C^*J\xi),$$
$$j_1j_2(T_h(\xi))=\frac{1}{\sqrt{n}}\sum_{k\in G}\overline{\inpr{h}{k}}T_k(CA(k)\xi),$$
$$j_2j_1(T_h(\xi))=\frac{1}{\sqrt{n}}\sum_{k\in G}\inpr{h}{k}T_k(C^*A(h)^*\xi),$$
\begin{align*}
\lefteqn{(j_2\circ j_1^{-1})U_{\cK}(g)(j_2\circ j_1^{-1})^*T_h(\xi)} \\
 &=T_{h+g}(A(g+h)A(h)^*\xi)=U_{\cK}(-g)V(-g)T_h(A(g)\xi).
\end{align*}

Moreover we can choose an orthonormal basis $\{e_t\}_{t\in \Lambda}$ of $\cK_0$ with an involution $\Lambda\ni t\mapsto \overline{t}\in \Lambda$ 
satisfying $A(g)e_t=a(g)\chi_t(g)e_t$, and $Ce_t=c_te_t$, $Je_t=\epsilon_{t}e_{\overline{t}}$ with 
\begin{itemize}
\item [$(1)$] $a(-g)=a(g)$, $a(g)a(h)=\inpr{g}{h}a(g+h)$, 
\item [$(2)$] $\chi_t\in \hat{G}$, $\chi_{\overline{t}}=\chi_t^{-1}$, 
\item [$(3)$] $c_{\overline{t}}=c_{t}$, $\sum_{g\in G}a(g)\chi_t(g)=\sqrt{n}c_t^{-3}$,   
\item [$(4)$] $\epsilon_{t}\in \{1,-1\}$, $\epsilon_t\epsilon_{\overline{t}}=\epsilon$.
\end{itemize}
When $\epsilon=1$, we can arrange the basis so that $\epsilon_t=1$ for all $t\in \Lambda$. 
\end{lemma}

Now we determine the form of $\rho(T_g(\xi))$. 
Since 
$\rho(T_g(\xi))=\rho(U(-g))\rho(T_0(\xi))$, we have 
\begin{align*}
l(T_g(\xi))&=((j_2\circ j_1^{-1})U_{\cK}(-g)(j_2\circ j_1^{-1})^*\otimes U_{\cK}(-g))l(T_0(\xi)) \\
 &=(AU_{\cK}(g)A^*\otimes U_{\cK}(-g))l(T_0(\xi)),
\end{align*}
where $AT_h(\xi)=T_h(A(h)\xi)$. Thus $l(T_g(\xi))$ is determined by $l(T_0(\xi))$. 
Moreover, $l(T_0(\xi))$ satisfies $\alpha_g(l(T_0(\xi)))=l(T_0(\xi))$ and $U_{\cK}(g)l(T_0(\xi))U_{\cK}(g)^*=l(T_0(\xi))$. 
Let $\cK_h$ be the linear span of $\{T_h(\xi);\;\xi\in \cK_0\}$, and let $Q_h$ be the projection from $\cK$ onto $\cK_h$. 
This notation for $h=0$ is consistent with the previous one as we identify $\xi\in \cK_0$ in the previous sense 
with $T_0(\xi)$ in what follows. 
Then
\begin{align*}
l(T_0(\xi))&=\sum_{k\in G}l(T_0(\xi))Q_k=\sum_{k\in G}l(T_0(\xi))U_{\cK}(-k)Q_0U_{\cK}(k)\\
&=\sum_{k\in G}U_{\cK}(-k)l(T_0(\xi))Q_0U_{\cK}(k).
\end{align*}
Therefore $\rho(T_g(\xi))$ is determined by 
$l(T_0(\xi))Q_0\in \cK^2\cK_0^*$. 

Since $\alpha_g(l(T_0(\xi))Q_0)=l(T_0(\xi))Q_0$, we have 
$$l(T_0(\xi))Q_0\in \bigoplus_{h\in G} \cK_h\cK_{-h}\cK_0^*,$$
and there exists a family of linear maps $B_h:\cK_0\rightarrow \cK_0 \cK_0\cK_0^*$ satisfying 
$$l(T_0(\xi))Q_0=\sum_{h\in G}(U_{\cK}(-h)\otimes U_{\cK}(h))(A(h)\otimes I)B_h(\xi).$$
Now we can write down $l(T_g(\xi))$ in terms of $B_h(\xi)$. 

\begin{lemma} There exists a family of linear maps $B_h:\cK_0\rightarrow \cK_0^2\cK_0^*$ satisfying 
\begin{equation}\label{lB} 
l(T_g(\xi))=\sum_{h,k}\inpr{g}{k}(U_{\cK}(-h-k)\otimes U_{\cK}(h))(A(h)\otimes I)B_{g+h}(\xi)U_{\cK}(k).
\end{equation}
\end{lemma}

\begin{proof} This follows from 
\begin{align*}
\lefteqn{l(T_g(\xi))=\sum_{h,k}\big(AU_{\cK}(g)A^*U_{\cK}(-k)\otimes U_{\cK}(-g)\big)} \\
 &\times (U_{\cK}(-h)\otimes U_{\cK}(h))(A(h)\otimes I)B_h(\xi)U_{\cK}(k)\\
 &=\sum_{h,k}(U_{\cK}(-h-k)\otimes U_{\cK}(h))(A(h+k)A(g+h+k)^*A(h+g)\otimes I)B_{g+h}(\xi)U_{\cK}(k) \\
 &=\sum_{h,k}\inpr{g+h}{k}(U_{\cK}(-h-k)\otimes U_{\cK}(h))(A(h+k)A(k)^*\otimes I)B_{g+h}(\xi)U_{\cK}(k)\\
 &=\sum_{h,k}\inpr{g}{k}(U_{\cK}(-h-k)\otimes U_{\cK}(h))(A(h)\otimes I)B_{g+h}(\xi)U_{\cK}(k).
\end{align*}
\end{proof}

\section{Polynomial equations for the irrational case}\label{PEIC}
Let $G$ be a finite abelian group of order $n$, let $m$ be a multiple of $n$, and let $\epsilon\in \{1,-1\}$. 
Let $\inpr{\cdot}{\cdot}$, $\cK$, $\cK_0$, $A(g)$, $C$, and $J$ satisfy the conditions in the statement of Lemma \ref{ACJ},  
and let $B_g:\cK_0\rightarrow \cK_0\cK_0\cK_0^*$ be a linear map. 
We defined $l\rightarrow \cK\cK\cK^*$ by Eq.(\ref{lB}). 
We will deuce the polynomial equations classifying the corresponding C$^*$ near-group categories in terms of the above data. 
We set 
$$\hat{B}_g(\xi)=\frac{1}{\sqrt{n}} \sum_{h\in G}\overline{\inpr{g}{h}} B_h(\xi).$$

\begin{lemma}
Eq.(\ref{orthogonality1}), (\ref{orthogonality2}), (\ref{Frobenius1}), and (\ref{Frobenius2}) are equivalent to 
\end{lemma}
\begin{equation}\label{orthogonality1'}
\sum_{s,h}T_0(Je_s)^*T_0(e_s)^*B_h(\xi)=-\frac{\sqrt{n}}{d}T_0(C^*J\xi)^*.
\end{equation}
\begin{equation}\label{orthogonality2'}
\sum_{h}B_{h+q}(\eta)^*B_{h+p}(\xi)=\delta_{p,q}\inpr{\xi}{\eta}Q_0-\frac{1}{d}T_0(C^*J\eta)T_0(C^*J\xi)^*.
\end{equation}
\begin{align}\label{Frobenius1'}
\lefteqn{(X(h)\otimes X(h))B_{-g}(X(h)^*A(-g)J\xi)X(h)^*} \\
 &=\sum_{s,t}B_g(\xi)^*T_0(e_s)T_0(e_t)T_0(Je_t)T_0(e_s)^*. \nonumber
\end{align}
\begin{equation}\label{Frobenius2'}
\hat{B}_g(\xi)=\sum_{s}T_0(e_s)B_g(C^*J\xi)^*T_0(Je_s)
\end{equation}

\begin{proof} Eq.(\ref{orthogonality1}) is 
\begin{align*}
0 &=\frac{\epsilon n}{d}j_2(T_g(\xi))^*Q_0+\sum_{l,s}j_1(T_l(e_s))^*T_l(e_s)^*l(T_g(\xi)) \\
 &= \frac{\sqrt{n}}{d}T_0(C^*J\xi)^*+\sum_{l,s}T_{-l}(A(-l)Je_s)^*T_l(e_s)^*l(T_g(\xi)) \\
 &= \frac{\sqrt{n}}{d}T_0(C^*J\xi)^*+\sum_{l,s}T_{0}(A(-l)Je_s)^*T_0(e_s)^*(A(l)\otimes I)B_{g+l}(\xi)  \\
 &= \frac{\sqrt{n}}{d}T_0(C^*J\xi)^*+\sum_{l,s}T_{0}(JA(l)^*e_s)^*T_0(A(l^*)e_s)^*B_{g+l}(\xi)  \\
 &=\frac{\sqrt{n}}{d}T_0(C^*J\xi)^*+\sum_{l,s}T_{0}(Je_s)^*T_0(e_s)^*B_{g+l}(\xi).
\end{align*}

Eq.(\ref{orthogonality2}) is 
\begin{align*}
\lefteqn{\delta_{p,q}\inpr{\xi}{\eta}Q=l(T_q(\eta))^*l(T_p(\xi))+\frac{n}{d}\sum_{k}Q_kj_2(T_q(\eta))j_2(T_p(\xi))^*Q_k} \\
 &=\sum_{h,k}\inpr{p-q}{k}U_{\cK}(k)^*B_{h+q}(\eta)^*B_{h+p}(\xi)U_{\cK}(k)\\
 &+\frac{1}{d}\sum_{k}\inpr{p-q}{k}T_k(C^*J\eta)T_k(C^*J\xi)^*, \\
\end{align*}
which is equivalent to 
$$\sum_hB_{h+q}(\eta)^*B_{h+p}(\xi)=\delta_{p,q}\inpr{\xi}{\eta}Q_0-\frac{1}{d}T_0(C^*J\eta)T_0(C^*J\xi)^*.$$

Eq.(\ref{Frobenius1}) is 
\begin{align*}
\lefteqn{l(T_{-g}(A(-g)J\xi))=\sum_{h,k,s,t}l(T_g(\xi))^*T_{h+k}(e_s)T_{-h}(e_t)T_h(A(h)Je_t)T_{h+k}(e_s)^*} \\
 &=\sum_{h,k,s,t}\overline{\inpr{g}{k}}U_{\cK}(k)^* B_{g+h}(\xi)^*T_0(A(h)^*e_s)T_0(e_t)T_h(A(h)Je_t)T_0(e_s)^*U_{\cK}(h+k)\\
 &=\sum_{h,k,s,t}\overline{\inpr{g}{h+k}}U_{\cK}(h+k)^* B_{g-h}(\xi)^*\\
 &\times T_0(A(-h)^*e_s)T_0(e_t)T_{-h}(A(-h)Je_t)T_0(e_s)^*U_{\cK}(k)\\
 &=\sum_{h,k,s,t}\overline{\inpr{g}{h+k}}(U_{\cK}(-h-k)\otimes U_{\cK}(h))(I\otimes A(-h))B_{g-h}(\xi)^*\\
 &\times  T_0(e_s)T_0(e_t)T_0(Je_t)T_0(e_s)^*A(-h)^*U_{\cK}(k),\\
\end{align*}
which is equivalent to 
\begin{align*}\lefteqn{(A(h)\otimes I)B_{h-g}(A(-g)J\xi)}\\
&=\overline{\inpr{g}{h}}(I\otimes A(-h))B_{g-h}(\xi)^*\sum_{s,t}T_0(e_s)T_0(e_t)T_0(Je_t)T(e_s)^*A(-h)^*.
\end{align*}
Replacing $g$ with $g+h$, we get 
\begin{align*}
\lefteqn{\inpr{h}{h}(A(h)\otimes A(-h)^*)B_{-g}(A(-h)A(-g)J\xi)A(-h)} \\
 &=B_g(\xi)^*\sum_{s,t}T_0(e_s)T_0(e_t)T_0(Je_t)T_0(e_s)^*.
\end{align*}

Eq.(\ref{Frobenius2}) is 
\begin{align*}
\lefteqn{l(j_2(T_g(\xi)))=\sum_{l,s}T_l(e_s)l(T_g(\xi)^*)T_{-l}(A(-l)Je_s)} \\
 &=\sum_{h,k,s}\overline{\inpr{g+h}{k}} T_{-h-k}(e_s)U_{\cK}(-k)B_{g+h}(\xi)^*(A(k)A(h+k)^*)T_0(A(h+k)Je_s)U_{\cK}(-h)\\
 &=\sum_{h,k,s}\overline{\inpr{g+h}{k}} T_{-h-k}(e_s)U_{\cK}(-k)B_{g+h}(\xi)^*T_0(A(k)Je_s)U_{\cK}(-h) \\
 &=\sum_{h,k,s}\inpr{g-k}{h}(U_{\cK}(-h-k)\otimes U_{\cK}(h)) T_0(e_s)B_{g-k}(\xi)^*T_0(A(-h)Je_s)U_{\cK}(k)\\ 
 &=\epsilon \sum_{h,k,s}\inpr{g-k}{h}(U_{\cK}(-h-k)\otimes U_{\cK}(h)) T_0(A(h)Je_s)B_{g-k}(\xi)^*T_0(e_s)U_{\cK}(k). 
\end{align*}
The left-hand side is 
\begin{align*}
\lefteqn{\frac{\epsilon}{\sqrt{n}}\sum_{l}\overline{\inpr{g}{l}}l(T_l(C^*J\xi))} \\
 &=\frac{\epsilon}{\sqrt{n}}\sum_{l}\overline{\inpr{g}{l}}\sum_{h,k}\inpr{l}{k}(U_k(-h-k)\otimes U_{\cK}(h))(A(h)\otimes I)B_{h+l}(C^*J\xi)U_{\cK}(k) \\
 &=\frac{\epsilon}{\sqrt{n}}\sum_{h,k,l}\overline{\inpr{g-k}{l-h}}(U_k(-h-k)\otimes U_{\cK}(h))(A(h)\otimes I)B_l(C^*J\xi)U_{\cK}(k) \\
 &=\epsilon\sum_{h,k}\inpr{g-k}{h}(U_k(-h-k)\otimes U_{\cK}(h))(A(h)\otimes I)\hat{B}_{g-k}(C^*J\xi)U_{\cK}(k), \\
\end{align*}
which shows that Eq.(\ref{Frobenius2}) is equivalent to 
$$\hat{B}_g(C^*J\xi)=\sum_{s}T_0(Je_s)B_g(\xi)^*T_0(e_s).$$
\end{proof}

\begin{remark}
By Fourier transform, Eq.(\ref{orthogonality1'}) and (\ref{orthogonality2'}) are equivalent to the following two:
\begin{equation}\label{orthogonality1''}
\sum_{s}T_0(Je_s)^*T_0(e_s)^*\hat{B}_0(\xi)=-\frac{1}{d}T_0(C^*J\xi)^*,
\end{equation}
\begin{equation}\label{orthogonality2''}
\hat{B}_h(\eta)^*\hat{B}_h(\xi)=\frac{\inpr{\xi}{\eta}}{n}Q_0-\frac{\delta_{h,0}}{d}T_0(C^*J\eta)T_0(C^*J\xi)^*.
\end{equation}
Eq.(\ref{Frobenius1'}) is equivalent to the following two: 
\begin{equation}\label{Frobenius1''1}
B_g(A(h)\xi)=(A(h)\otimes A(h))B_g(\xi)A(h)^*, 
\end{equation}
\begin{equation}\label{Frobenius1''2}
B_{g}(\xi)^*=\epsilon \sum_{s,t}T_0(e_s)T_0(Je_t)^*T_0(e_t)^*T_0(e_s)^*B_{-g}(A(-g)J\xi). 
\end{equation}
Eq.(\ref{Frobenius2'}) is the equivalent to the following under the presence of Eq.(\ref{Frobenius1''2}):
\begin{align}\label{Frobenius2''}
\lefteqn{\hat{B}_g(\xi)} \\
 &=\epsilon\sum_{s,t,r}\big( T_0(e_t)^*T_0(e_s)^*B_{-g}(A(-g)C\xi)T_0(e_r)\big)T_0(Je_r)T_0(e_s)T_0(Je_t)^*.\nonumber
\end{align}
Note that $T_0(e_t)^*T_0(e_s)^*B_{-g}(A(-g)C\xi)T_0(e_r)$ is already a scaler. 
\end{remark}

\begin{lemma} Eq.(\ref{rhoU2}) is equivalent to 
\begin{equation}\label{rehoU'}
\sum_{r}\hat{B}_g(e_r)\hat{B}_g(e_r)^*=\frac{1}{n}Q_0\otimes Q_0-\frac{\delta_{g,0}}{d}\sum_{s,t}T_0(e_s)T_0(Je_s)T_0(Je_t)^*T_0(e_t)^*.
\end{equation}
\end{lemma}

\begin{proof} Note that we have 
$$U(g)=\sum_{h}\inpr{g}{h}S_hS_h^*+\sum_{h,t}T_{h-g}(e_t)T_h(e_t)^*.$$
In view of the proof of Lemma \ref{rhoU3}, it suffices to show that the above equality is equivalent to 
$$j_2\circ j_1^{-1}U_{\cK}(g)(j_2\circ j_1^{-1})^*\otimes U_{\cK}(g)=\sum_{l}\inpr{g}{l}Q\rho(S_lS_l^*)Q+\sum_{l,s}l(T_{l-g}(e_s))l(T_l(e_s))^*.$$

The left-hand side is 
\begin{align*}
\lefteqn{\sum_{h,k,s,t}T_{g+h}(A(g+h)A(h)^*e_s)T_h(e_s)^*\otimes T_{k-g}(e_t)T_k(e_t)^*} \\
 &=\sum_{h,k,s,t}\overline{\inpr{g}{h}}T_{g+h}(A(g)e_s)T_h(e_s)^*\otimes T_{k-g}(e_t)T_k(e_t)^*.
\end{align*}

The first term of the right-hand side is 
\begin{align*}
\lefteqn{\sum_{l}\inpr{g}{l}U_{\cK}(l)\rho(S_0S_0^*)U_{\cK}(l)^*}\\
 &=\frac{1}{d} \sum_{l,p,q,s,t}\inpr{g}{l}U_{\cK}(l)T_p(e_s)T_{-p}(A(-p)Je_s)T_{-q}(A(-q)Je_t)^*T_q(e_t)^*U_{\cK}(l)^*\\
 &=\frac{1}{d} \sum_{l,p,q,s,t}\inpr{g}{l}T_{p-l}(e_s)T_{q-l}(e_t)^*\otimes T_{-p}(A(-p)Je_s)T_{-q}(A(-q)Je_t)^*\\
 &=\frac{1}{d} \sum_{h,k,p,s,t}\overline{\inpr{g}{h+k}}T_{h+k+p}(e_s)T_h(e_t)^*\otimes T_{-p}(A(-p)Je_s)T_k(A(k)Je_t)^*\\
 &=\frac{1}{d} \sum_{h,k,p,s,t}\overline{\inpr{g}{h+k}}T_{h+p}(e_s)T_h(e_t)^*\otimes T_{k-p}(A(k-p)Je_s)T_k(A(k)Je_t)^*\\
\end{align*}
The second term of the right-hand side is \begin{align*}
\lefteqn{
\sum_{h,k,k',l,s} (U_{\cK}(-h-k')\otimes U_{\cK}(k'))(\inpr{l-g}{h}A(k')\otimes I)B_{k'+l-g}(e_s)}\\
&\times  B_{l+k}(e_s)^*(\inpr{l}{-h}A(k)^*\otimes I)(U_{\cK}(h+k)\otimes U_{\cK}(-k))\\
 &=
\sum_{h,k,k',l,s}\overline{\inpr{g}{h}}(U_{\cK}(-h-k')\otimes U_{\cK}(k'))(A(k')\otimes I)\\
&\times B_{l+k'-k-g}(e_s)B_l(e_s)^*(A(k)^*\otimes I)
(U_{\cK}(h+k)\otimes U_{\cK}(-k)) \\
 &=\sum_{h,k,p,l,s}\overline{\inpr{g}{h+k}}(U_{\cK}(-h-p)\otimes U_{\cK}(-k+p))(A(p-k)\otimes I)\\
&\times B_{l+p-g}(e_s)B_l(e_s)^*(A(-k)^*\otimes I)
(U_{\cK}(h)\otimes U_{\cK}(k)). \\
\end{align*}
Thus Eq.(\ref{rhoU2}) is equivalent to 
\begin{align*}
\lefteqn{\sum_{l,s}B_{l+p-g}(e_s)B_l(e_s)^*} \\
 &= \delta_{g,p}\inpr{g}{k}A(g-k)^*A(g)A(-k)\otimes Q_0\\
 &-\frac{1}{d}\sum_{s,t}T_0(A(p-k)^*e_s)T_0(A(-k)^*e_t)^*\otimes T_0(A(k-p)Je_s)T_0(A(k)Je_t)^*\\
  &= \delta_{g,p}Q_0\otimes Q_0-\frac{1}{d}\sum_{s,t}T_0(e_s)T_0(e_t)^*\otimes T_0(Je_s)T_0(Je_t)^*.
\end{align*}
\end{proof}

\begin{lemma} Eq.(\ref{S*rho2S}) is equivalent to 
\begin{align}\label{S*rho2S'}
\lefteqn{(\frac{m}{n}-\frac{1}{d})T_0(\xi)} \\
 &=\sum_{h,s,t}T_0(A(-h)^*C^*Je_s)^*B_{-h}(T_0(e_s)^*B_h(\xi)T_0(e_t))T_0(C^*Je_t).\nonumber
\end{align}
\end{lemma}

\begin{proof} The left-hand side of Eq.(\ref{S*rho2S}) for $T_i=T_{h+k}(e_s)$ and $T_j=T_k(e_t)$ is
\begin{align*}
\lefteqn{\frac{1}{d}\sum_{h,k,s,t}j_2(T_{h+k}(e_s))^*l(T_{h+k}(e_s)^*l(T_g(\xi))T_{k}(e_t))j_2(T_{k}(e_t))} \\
 &= \frac{1}{nd}\sum_{h,k,p,q,s,t}\inpr{h+k}{p}\inpr{g}{k}\overline{\inpr{k}{q}}\\
 &\times T_p(C^*Je_s)^*l(U_{\cK}(h)^*T_0(A(h)^*e_s)^*B_{g+h}(\xi)T_0(e_t))T_q(C^*Je_t)\\
 &= \frac{1}{d}\sum_{h,q,s,t}\inpr{h}{q-g}
 T_{q-g}(C^*Je_s)^*l(T_{-h}(T_0(A(h)^*e_s)^*B_{g+h}(\xi)T_0(e_t)))T_q(C^*Je_t)\\
 &=\frac{1}{d}\sum_{h,q,s,t}\overline{\inpr{h}{g}}U_{\cK}(-g)\\
 &\times  T_0(A(-g)^*C^*Je_s)^*B_{-h-g}(T_0(A(h)^*e_s)^*B_{g+h}(\xi)T_0(e_t))T_0(C^*Je_t)\\
 &=\frac{n}{d}\sum_{h,s,t}\inpr{g-h}{g}U_{\cK}(-g)\\
 &\times  T_0(A(-g)^*C^*Je_s)^*B_{-h}(T_0(A(h-g)^*e_s)^*B_h(\xi)T_0(e_t))T_0(C^*Je_t)\\
 &=\frac{n}{d}\sum_{h,s,t}\inpr{g-h}{g}U_{\cK}(-g)\\
 &\times  T_0(A(-g)^*A(g-h)^*C^*Je_s)^*B_{-h}(T_0(e_s)^*B_h(\xi)T_0(e_t))T_0(C^*Je_t)\\
 &=\frac{n}{d}\sum_{h,s,t}U_{\cK}(-g)T_0(A(-h)^*C^*Je_s)^*B_{-h}(T_0(e_s)^*B_h(\xi)T_0(e_t))T_0(C^*Je_t).
\end{align*}
Since 
$$\frac{d}{n}(1-\frac{2n}{d^2})=\frac{d^2-2n}{nd}=\frac{md-n}{nd}=\frac{m}{n}-\frac{1}{d},$$
we get the statement. 
\end{proof}

\begin{lemma} Eq.(\ref{l1}) is equivalent to 

\begin{align}\label{l1'}
\lefteqn{\delta_{g,0}\inpr{\eta}{\zeta}T_0(\xi)-\frac{1}{d}\inpr{\xi}{\zeta}T_0(\eta)} \\
 &=\sum_{h,s}T_0(J\eta)^*B_{g-h}(e_s)T_0(C^*JB_h(\xi)^*T_0(A(h)^*C^*\zeta)T_0(e_s)).\nonumber
\end{align}

Eq.(\ref{l2}) is equivalent to 
\begin{align}\label{l2'}
\lefteqn{\epsilon C^*A(k)^*T_0(A(-k)^*\zeta)^*B_{g-k}(\xi)T_0(\eta)} \\
 &=\frac{1}{\sqrt{n}}\sum_{h}\overline{\inpr{g}{h}}B_{h-k}(\zeta)^*T_0(C^*A(h)^*\eta)T_0(C^*A(g)^*\xi).\nonumber
\end{align}

Under the presence of Eq.(\ref{Frobenius1''1}), Eq.(\ref{l3}) is equivalent to 
\begin{align}\label{l3'}
\lefteqn{\sum_{g,t}B_g(T_0(A(-g-h)^*\zeta)^*B_{-g-h}(\xi)T_0(e_t))B_{g+k}(e_t)^*T_0(\eta)} \\
 &=\inpr{h}{k}B_{k}(A(-h)^*\zeta)^*T_0(\eta)B_{-h}(\xi)
 -\frac{\epsilon}{d\sqrt{n}}\inpr{C^*\xi}{\zeta}_{\cK_0}\sum_{t}T_0(e_t)T_0(Je_t)T_0(J\eta)^*.\nonumber
\end{align}
\end{lemma}

\begin{proof} We set $T=T_g(\xi)$, $T'=T_{g'}(\eta)$, and $T_{g''}(\zeta)$ in Eq.(\ref{l1}), (\ref{l2}, (\ref{l3})). 
Then the left-hand side of Eq.(\ref{l1})is 
\begin{align*}
\lefteqn{\sum_{i}T_{g'}(\eta)^*l(T_i)j_2(l(T_g(\xi))^* j_2(T_{g''}(\zeta))T_i)} \\
 &=\sum_{h,k,s} T_{g'}(\eta)^*l(T_{-h}(e_s))j_2(l(T_g(\xi))^* \frac{\epsilon}{\sqrt{n}}\overline{\inpr{g''}{h+k}}T_{h+k}(C^*J\zeta)T_{-h}(e_s))\\
 &= \frac{\epsilon}{\sqrt{n}} \sum_{h,k,s}\inpr{g''}{h+k}\inpr{k}{g} T_{g'}(\eta)^*l(T_{-h}(e_s))j_2(T_k(B_{g+h}(\xi)^*T_{0}(A(h)^*C^*J\zeta)T_0(e_s)))\\
 &=\frac{1}{n} \sum_{h,k,f,s}\inpr{g+g''-f}{k}\inpr{h}{g''}\\
 &\times T_{g'}(\eta)^*l(T_{-h}(e_s))T_f(C^*JB_{g+h}(\xi)^*T_{0}(A(h)^*C^*J\zeta)T_0(e_s)) \\
 &=\sum_{h,k,s}\inpr{h}{g''} T_{g'}(\eta)^*l(T_{-h}(e_s))T_{g+g''}(C^*JB_{g+h}(\xi)^*T_{0}(A(h)^*C^*J\zeta)T_0(e_s))\\
 &=\sum_{h,k,s}\overline{\inpr{h}{g}} U_{\cK}(g'-g''-g)T_0(A(g'-g''-g)^*\eta)^*B_{g'-g''-g-h}(e_s)\\
 &\times T_0(C^*JB_{g+h}(\xi)^*T_{0}(A(h)^*C^*J\zeta)T_0(e_s))\\
 &=\sum_{h,k,s} U_{\cK}(g'-g''-g)T_0(A(g'-g''-g)^*\eta)^*B_{g'-g''-g-h}(e_s)\\
 &\times T_0(C^*JB_{g+h}(\xi)^*T_{0}(A(h+g)^*A(g)C^*J\zeta)T_0(e_s))\\
 &=\sum_{h,k,s} U_{\cK}(g'-g''-g)T_0(A(g'-g''-g)^*\eta)^*B_{g'-g''-h}(e_s)\\
 &\times T_0(C^*JB_h(\xi)^*T_{0}(A(h)^*A(g)C^*J\zeta)T_0(e_s)).
\end{align*}
The right-hand side is 
\begin{align*}
\lefteqn{
\epsilon\delta_{g',g''} \inpr{\zeta}{\eta}T_g(\xi)-\frac{1}{d}\sum_h\inpr{T_{g''-h}(\zeta)}{T_{-g}(A(-g)J\xi)} T_{h-g'}(A(h-g')J\eta) 
}\\
&=\epsilon\delta_{g',g''}\inpr{\zeta}{\eta}T_g(\xi)-\frac{1}{d}\sum_h\inpr{\zeta}{A(-g)J\xi}\delta_{h,g+g''}T_{h-g'}(A(h-g')J\eta) \\
 &=\epsilon\delta_{g',g''}\inpr{\zeta}{\eta}T_{g+g''-g'}(\xi)-\frac{\epsilon}{d}\inpr{\xi}{A(g)J\zeta}T_{g+g''-g'}(A(g+g''-g')J\eta). \\
\end{align*}
Thus we get 
\begin{align*}
\lefteqn{\delta_{g',0}\inpr{\eta}{\zeta}T_0(\xi)-\frac{1}{d}\inpr{\xi}{A(g)\zeta}
T_0(A(g-g')\eta)
} \\
 &=\sum_{h,s}T_0(A(g'-g)^*J\eta)^*B_{g'-h}(e_s)T_0(C^*JB_h(\xi)^*T_0(A(h)^*A(g)C^*\zeta)T_0(e_s)).
\end{align*}
Replacing $\eta$ with $A(g-g')^*\eta$ and $\zeta$ with $A(g)^*\zeta$, we get Eq.(\ref{l1'}). 

Note that we have 
\begin{align*}
j_2\circ j_1^{-1}(T_g(\xi)) &=\epsilon j_2(T_{-g}(A(-g)J\xi))=\frac{1}{\sqrt{n}}\sum_h\inpr{g}{h}T_h(C^*JA(-g)J\xi) \\
 &=\frac{\epsilon}{\sqrt{n}}\sum_h\inpr{g}{h}T_h(C^*A(g)^*\xi).
\end{align*}
The left-hand side of Eq.(\ref{l2}) is
\begin{align*}
\lefteqn{\frac{1}{n}\sum_{h,k}l(T_{g''}(\zeta))^*\inpr{g'}{h+k}\inpr{g}{-h}T_{h+k}(C^*A(g')^*\eta)T_{-h}(C^*A(g)^*\xi)
}\\
&=\frac{1}{n}\sum_{h,k}\inpr{g'-g}{h}\inpr{g'-g''}{k}U_{\cK}(-k)B_{g''+h}(\zeta)^*T_0(C^*A(h)^*A(g')^*\eta)T_0(C^*A(g)^*\xi)\\
&=\frac{1}{n}\sum_{h,k}\overline{\inpr{g}{h}}\inpr{g'-g''}{k}U_{\cK}(-k)B_{g''+h}(\zeta)^*T_0(C^*A(h+g')^*\eta)T_0(C^*A(g)^*\xi)\\
&=\frac{1}{n}\sum_{h,k}\overline{\inpr{g}{h-g'}}\inpr{g'-g''}{k}U_{\cK}(-k)B_{h+g''-g'}(\zeta)^*T_0(C^*A(h)^*\eta)T_0(C^*A(g)^*\xi)\\
\end{align*}
The right-hand side of Eq.(\ref{l2}) is 
\begin{align*}
\lefteqn{j_2\circ j_1^{-1}(T_{g''}(\zeta)^*l(T_g(\xi))T_{g'}(\eta))} \\
 &=\inpr{g'}{g}j_2\circ j_1^{-1}(T_{g'-g''}(T_0(A(g''-g')^*\zeta)^*B_{g+g''-g'}(\xi)T_0(\eta))) \\
 &=\frac{\epsilon\inpr{g'}{g}}{\sqrt{n}}\sum_{k}\inpr{k}{g'-g''}T_k(C^*A(g'-g'')^*T_0(A(g''-g')^*\zeta)^*B_{g+g''-g'}(\xi)T_0(\eta)).
\end{align*}
Let $f=g'-g''$. 
Then Eq.(\ref{l2}) is equivalent to 
\begin{align*}
\lefteqn{\epsilon C^*A(f)^*T_0(A(-f)^*\zeta)^*B_{g-f}(\xi)T_0(\eta)} \\
 &=\frac{1}{\sqrt{n}}\sum_{h}\overline{\inpr{g}{h}}B_{h-f}(\zeta)^*T_0(C^*A(h)^*\eta)T_0(C^*A(g)^*\xi).
\end{align*}

The left-hand side of Eq.(\ref{l3}) is 
\begin{align*}
\lefteqn{l(T_{g''}(\zeta))^*T_{g'}(\eta)l(T_g(\xi))} \\
 &=\sum_{k,h}\overline{\inpr{g''}{k}}U_{\cK}(-k)B_{g''+h}(\zeta)^*(A(h)^*U_{\cK}(h+k)\otimes U_{\cK}(-h))T_{g'}(\eta)l(T_g(\xi)) \\
 &=\sum_{k,h}\overline{\inpr{g''}{k}}U_{\cK}(-k)B_{g''+h}(\zeta)^*T_{g'-h-k}(A(h)^*\eta)U_{\cK}(-h)l(T_g(\xi))  \\
 &=\sum_{k}\overline{\inpr{g''}{k}}U_{\cK}(-k)B_{g'+g''-k}(\zeta)^*T_0(A(g'-k)^*\eta)U_{\cK}(k-g')l(T_g(\xi))  \\
 &=\sum_{k,k',h'}\inpr{g}{k'}\overline{\inpr{g''}{k}}U_{\cK}(-k)B_{g'+g''-k}(\zeta)^*T_0(A(g'-k)^*\eta)U_{\cK}(k-g')\\
 &\times (U_{\cK}(-h'-k')A(h')\otimes U_{\cK}(h'))B_{g+h'}(\xi)U_{\cK}(k')\\
 &=\sum_{k,k'}\inpr{g}{k'}\overline{\inpr{g''}{k}}U_{\cK}(-k)B_{g'+g''-k}(\zeta)^*T_0(A(g'-k)^*\eta)\\
 &\times (A(k-k'-g')\otimes U_{\cK}(k-k'-g'))B_{g-g'+k-k'}(\xi)U_{\cK}(k')\\ 
\end{align*}
The first term of the right-hand side is 
\begin{align*}
\lefteqn{\sum_{f,t}l(T_{g''}(\zeta)^*l(T_g(\xi))T_f(e_t))l(T_f(e_t))^*T_{g'}(\eta)}\\
 &=\sum_{f,t,h}\inpr{g}{f}l(T_{g''}(\zeta)^*   (U_{\cK}(-h-f)A(h)\otimes U_{\cK}(h))B_{g+h}(\xi)T_0(e_t))l(T_f(e_t))^*T_{g'}(\eta)  \\
 &=\sum_{f,t}\inpr{g}{f}l(U_{\cK}(g''-f)T_0(A(g''-f)^*\zeta)^*B_{g+g''-f}(\xi)T_0(e_t))l(T_f(e_t))^*T_{g'}(\eta)  \\
 &=\sum_{f,t,p,q}\inpr{g}{f}\inpr{q}{f-g''} (U_{\cK}(-p-q)A(p)\otimes U_{\cK}(p))\\
 &\times B_{f-g''+p}(T_0(A(g''-f)^*\zeta)^*B_{g+g''-f}(\xi)T_0(e_t))U_{\cK}(q)l(T_f(e_t))^*T_{g'}(\eta) \\
 &=\sum_{f,t,p,q,r}\inpr{g}{f}\overline{\inpr{q}{g''}}(U_{\cK}(-p-q)A(p)\otimes U_{\cK}(p))\\
 &\times B_{f-g''+p}(T_0(A(g''-f)^*\zeta)^*B_{g+g''-f}(\xi)T_0(e_t))U_{\cK}(q)\\
 &\times U_{\cK}(-q)B_{f+r}(e_t)^*(A(r)^*U_{\cK}(r+q)\otimes U_{\cK}(-r))T_{g'}(\eta) \\
 &=\sum_{f,t,p,r}\inpr{g}{f}\inpr{r-g'}{g''} 
 (U_{\cK}(-p+r-g')A(p)\otimes U_{\cK}(p))\\
 &\times B_{f-g''+p}(T_0(A(g''-f)^*\zeta)^*B_{g+g''-f}(\xi)T_0(e_t)) B_{f+r}(e_t)^*T_0(A(r)^*\eta)U_{\cK}(-r) \\
 &=\sum_{f,t,k,r}\inpr{g}{f}\inpr{r-g'}{g''} (U_{\cK}(-k)A(k+r-g')\otimes U_{\cK}(k+r-g'))\\
 &\times B_{f+k+r-g'-g''}(T_0(A(g''-f)^*\zeta)^*B_{g+g''-f}(\xi)T_0(e_t)) B_{f+r}(e_t)^*T_0(A(r)^*\eta)U_{\cK}(-r) \\
&=\sum_{f,t,k,k'}\inpr{g}{f}\overline{\inpr{k'+g'}{g''}}(U_{\cK}(-k)A(k-k'-g')\otimes U_{\cK}(k-k'-g'))\\
 &\times B_{f+k-k'-g'-g''}(T_0(A(g''-f)^*\zeta)^*B_{g+g''-f}(\xi)T_0(e_t)) B_{f-k'}(e_t)^*T_0(A(-k')^*\eta)U_{\cK}(k') \\
\end{align*}
The second term of the right-hand side is 
\begin{align*}
\lefteqn{\frac{1}{d}\sum_{h,k,t}\inpr{V(h)j_2\circ j_1^{-1}(T_g(\xi))}{T_{g''}(\zeta)}_\cK} \\
&\times U_{\cK}(h)T_{h+k}(e_t)j_1(T_{h+k}(e_t))j_1(U_{\cK}(h)^*T_{g'}(\eta))^*\\
 &= \frac{\epsilon}{d\sqrt{n}}\sum_{h,k,t}\inpr{h}{g''}\inpr{g}{g''}\inpr{C^*A(g)^*\xi}{\zeta}_{\cK_0}\\
 &\times T_{k}(e_t)T_{-h-k}(A(-h-k)Je_t)T_{-h-g'}(A(-h-g')J\eta)^*\\
 &= \frac{\epsilon}{d\sqrt{n}}\sum_{k,k',t}\overline{\inpr{k'+g'}{g''}}\inpr{g}{g''}\inpr{C^*A(g)^*\xi}{\zeta}_{\cK_0}\\
 &\times T_{k}(e_t)T_{g'-k+k'}(A(g'-k+k')Je_t)T_{k'}(A(k')J\eta)^*.\\
\end{align*}
Thus we get 
\begin{align*}
\lefteqn{\sum_{f,t}\inpr{g}{f}  B_{f+k-k'-g'-g''}(T_0(A(g''-f)^*\zeta)^*B_{g+g''-f}(\xi)T_0(e_t))} \\
 &\times B_{f-k'}(e_t)^*T_0(A(-k')^*\eta)\\
 &\frac{\epsilon}{d\sqrt{n}}\inpr{g}{g''}\inpr{C^*A(g)^*\xi}{\zeta}_{\cK_0}\sum_{t}T_0(e_t)T_0(Je_t)T_0(A(k')J\eta)^*\\
 &=\inpr{g'}{g''}\inpr{g+g''}{k'}\overline{\inpr{g''}{k}}A(k-k'-g')^*B_{g'+g''-k}(\zeta)^*T_0(A(g'-k)^*\eta)\\
 &\times A(k-k'-g')B_{g-g'+k-k'}(\xi)\\ 
 &=\inpr{g'}{g''}\inpr{g+g''}{k'}\overline{\inpr{g''}{k}}B_{g'+g''-k}(A(k-k'-g')^*\zeta)^*\\
 &\times T_0(A(k-k'-g')^*A(g'-k)^*\eta)B_{g-g'+k-k'}(\xi)\\ 
 &=\inpr{g''}{g'-k+k'}\inpr{g}{k'}\inpr{k-k'-g'}{-g'+k} \\
 &\times B_{g'+g''-k}(A(k-k'-g')^*\zeta)^*T_0(A(-k')^*\eta)B_{g-g'+k-k'}(\xi),
\end{align*}
where we used Eq.(\ref{Frobenius1''1}). 
Replacing $k$, $\eta$, and $\zeta$ with $k''+k'+g'$, $A(-k')\eta$, and $A(g)^*\zeta$ respectively, and multiplying $\overline{\inpr{g}{g''}}$ 
make the first term of the left-hand side of this as 
\begin{align*}
\lefteqn{\sum_{f,t}B_{f+k''-g''}(T_0(A(g+g''-f)^*\zeta)^*B_{g+g''-f}(\xi)T_0(e_t))B_{f-k'}(e_t)^*T_0(\eta)} \\
 &=\sum_{x,t}B_x(T_0(A(g+k''-x)^*\zeta)^*B_{g+k''-x}(\xi)T_0(e_t))B_{x-k''-k'+g''}(e_t)^*T_0(\eta).
\end{align*}
The second term becomes 
$$\frac{\epsilon}{d\sqrt{n}}\inpr{C^*\xi}{\zeta}_{\cK_0}\sum_{t}T_0(e_t)T_0(Je_t)T_0(J\eta)^*.$$
The right-hand side becomes 
\begin{align*}
\lefteqn{\overline{\inpr{g''}{k''}}\inpr{g}{k'-g''}\inpr{k''}{k'+k''}B_{g''-k''-k'}(A(k'')^*A(g)^*\zeta)^*T_0(\eta)B_{g+k''}(\xi) } \\
 &=\inpr{g}{k''} \overline{\inpr{g''}{k''}}\inpr{g}{k'-g''}\inpr{k''}{k'+k''}B_{g''-k''-k'}(A(g+k'')^*\zeta)^*T_0(\eta)B_{g+k''}(\xi)\\
 &=\inpr{g''-k'-k''}{-g-k'}B_{g''-k''-k''}(A(g+k'')^*\zeta)^*T_0(\eta)B_{g+k''}(\xi). 
\end{align*}
This is equivalent to Eq.(\ref{l3'}). 
\end{proof}

Theorem \ref{unique} and our computation so far show the following classification theorem. 

\begin{theorem}\label{irrcl} C$^*$ near-group categories with a finite abelian group $G$ and irrational $d$ is completely classified by 
$$(\epsilon,\inpr{\cdot}{\cdot},\cK_0,A(g),B_g,C,J)$$ 
satisfying the conditions in the statement of Lemma \ref{ACJ} and Eq.(\ref{orthogonality1''})-(\ref{l3'}) up to equivalence in the following sense: 
we say that two tuples  
$$(\epsilon,\inpr{\cdot}{\cdot},\cK_0,A(g),B_g,C,J),\quad (\epsilon',\inpr{\cdot}{\cdot}',\cK_0',A'(g),B'_g,C',J')$$ 
are equivalent if there exist a unitary $W:\cK_0\to \cK_0'$ and a group automorphism $\varphi\in \Aut(G)$ satisfying 
$\inpr{g}{h}'=\inpr{\varphi(g)}{\varphi(h)}$, $A'(g)=WA(\varphi(g))W^*$, 
$B'_{\varphi(g)}(W\xi)W=(W\otimes W)B_g(\xi)$, $C'=WCW^*$, and $J'=WJW^*$. 
\end{theorem}

Now we assume that an orthonormal basis $\{e_t\}_{t\in \Lambda}$ of $\cK_0$ satisfies the conditions in Lemma \ref{ACJ}, 
and define $b^{s,t}_{r,p}(g)\in \C$ by 
$$b^{r,s}_{t,u}(g)=T(e_s)^*T(e_r)^*B_g(e_u)T(e_t).$$ 
Then we have
$$B_g(e_u)=\sum_{r,s,t}b^{r,s}_{t,u}(g)T_0(e_r)T_0(e_s)T_0(e_t)^*.$$

\begin{theorem}\label{irrpe} In terms of $b^{r.s}_{t,u}(g)$, Eq.(\ref{orthogonality1''})-(\ref{l3'}) altogether are equivalent to 
\begin{equation}\label{p1}
\frac{1}{\sqrt{n}}\sum_h\inpr{g}{h}{b^{r,s}_{t,u}}(h)=\epsilon \epsilon_r\epsilon_t c_ua(g)\chi_u(g)b^{s,\overline{t}}_{\overline{r},u}(g),
\end{equation}
\begin{equation}\label{p2}
\sum_{r}b^{r,s}_{r,u}(0)=-\frac{\delta_{s,u}}{d},
\end{equation}
\begin{equation}\label{p3}
\sum_{s}b^{r,s}_{t,s}(0)=-\frac{\delta_{r,t}}{d},
\end{equation}
\begin{equation}\label{p4}
\sum_{r,t}\overline{b^{r,s'}_{t,u'}(g)}b^{r,s}_{t,u}(g)=\frac{\delta_{s,s'}\delta_{u,u'}}{n}
-\frac{\delta_{g,0}\delta_{s,u}\delta_{s',u'}}{d},
\end{equation}
\begin{equation}\label{p5}
\sum_{s,u}b^{r,s}_{t,u}(g)\overline{b^{r',s}_{t',u}(g)}=\frac{\delta_{r,r'}\delta_{t,t'}}{n}
-\frac{\delta_{g,0}\delta_{r,t}\delta_{r',t'}}{d},
\end{equation}
\begin{equation}\label{p6}
\chi_r\chi_s\neq \chi_t\chi_u \Longrightarrow b^{r,s}_{t,u}(g)=0,\quad \forall g\in G,
\end{equation}
\begin{equation}\label{p7}
\overline{b^{r,s}_{t,u}(g)}=\epsilon_{s}\epsilon_u a(g)\chi_u(g)b^{t,\overline{s}}_{r,\overline{u}}(-g),
\end{equation}
\begin{equation}\label{p8}
\overline{b^{r,s}_{t,u}(g)}=\epsilon_{t}\epsilon_r c_r\overline{c_t}a(g)\chi_r(g)b^{\overline{r},u}_{\overline{t},s}(-g),
\end{equation}
\begin{equation}\label{p9}
b^{r,s}_{t,u}(g)=\epsilon_r\epsilon_s\epsilon_t\epsilon_uc_t\overline{c_r}\overline{\chi_r(g)\chi_s(g)}
b^{\overline{t},\overline{u}}_{\overline{r},\overline{s}}(g), 
\end{equation}
\begin{align}\label{p10}
\lefteqn{c_r\epsilon_r\sum_{g,q,s,t}\epsilon_t\overline{c_t}
\overline{b^{v,w}_{q,s}(g)}
b^{\overline{r},u}_{\overline{t},s}(g+h)
b^{p,x}_{q,t}(g+k)} \\
 &=\epsilon_u\epsilon_w\chi_r(h)\overline{\chi_u(h)\inpr{h}{k}}
 \sum_{y} b^{p,y}_{v,r}(k)b^{x,\overline{w}}_{y,\overline{u}}(h)
 -\frac{c_u\epsilon_{\overline{p}}\epsilon_v}{d\sqrt{n}}\delta_{r,u}\delta_{w,\overline{v}}\delta_{x,\overline{p}}.\nonumber 
\end{align}
\end{theorem}

\begin{proof} Eq.(\ref{Frobenius2''}) is equivalent to (\ref{p1}). 
Under the presence of this condition, we have equivalence of 
Eq.(\ref{orthogonality1''}),(\ref{orthogonality2''}),(\ref{Frobenius1''1}),(\ref{Frobenius1''2}),(\ref{rehoU'}) 
and Eq.(\ref{p2}),(\ref{p4})-(\ref{p7}). 

Eq.(\ref{S*rho2S'}) is equivalent to 
$$(\frac{m}{n}-\frac{1}{d})\delta_{p,u}=\sum_{g,r,s,t}\epsilon_r\epsilon_tc_r\overline{c_t}a(g)\chi_r(g)
b^{\overline{r},p}_{\overline{t},s}(-g)b^{r,s}_{t,u}(g),$$
which follows from Eq.(\ref{p4}),(\ref{p7}),(\ref{p9}). 
We show, on the other hand, that this together with Eq.(\ref{p4}),(\ref{p7}) implies Eq.(\ref{p9}). 
Eq.(\ref{p4}) implies 
$$\sum_s\sum_h\sum_{r,r}|b^{r,s}_{t,u}(g)|^2=\sum_s\sum_h(\frac{1}{n}-\frac{\delta_{g,0}\delta_{s,u}}{d})
=\sum_s(1-\frac{\delta_{s,u}}{d})=\frac{n}{m}-\frac{1}{d}.$$
In the same way, we have 
$$\sum_s\sum_h\sum_{r,r}|\epsilon_r\epsilon_tc_r\overline{c_t}a(g)\chi_r(g)
b^{\overline{r},p}_{\overline{t},s}(-g)|^2=\frac{n}{m}-\frac{1}{d}.$$
Thus the Cauchy-Schwartz inequality implies  
$$b^{\overline{r},p}_{\overline{t},s}(-g)=\epsilon_r\epsilon_tc_t\overline{c_ra(g)\chi_r(g)b^{r,s}_{t,u}(g)},$$
which together with Eq.(\ref{p7}) implies Eq.(\ref{p9}).

Although Eq.(\ref{p3}),(\ref{p8}) are redundant, we put them in the statement 
to emphasize that the equations are symmetric in left and right variables. 
In fact, Eq.(\ref{p3}) follows from Eq.(\ref{p2}),(\ref{p5}), and  
Eq.(\ref{p8}) follows from Eq.(\ref{p7}),(\ref{p9}). 
In the rest of the proof, we assume the conditions we have obtained so far, and show that 
Eq.(\ref{l1'}),(\ref{l2'}) follow from them, and Eq.(\ref{l3'}) is equivalent to Eq.(\ref{p10}). 

Eq.(\ref{l1'}) is equivalent to 
$$\delta_{g,0}\delta_{r,v}\delta_{p,u}-\frac{\delta_{u,r}\delta_{p,v}}{d}
=\sum_{h,s,t}\epsilon_t\epsilon_vc_r\overline{c_t}a(h)\chi_r(h)b^{r,s}_{t,u}(h)b^{\overline{v},p}_{\overline{t},s}(g-h),$$
and by Eq.(\ref{p8}) and the Fourier transform, the right-hand side is equal to 
$$\sum_{h,s,t}\epsilon_r\epsilon_v\overline{b^{\overline{r},u}_{\overline{t},s}(-h)}b^{\overline{r},p}_{\overline{t},s}(g-h)
=\sum_{k,s,t}\epsilon_r\epsilon_v\overline{\widehat{b^{\overline{r},u}_{\overline{t},s}}(k)}
\widehat{b^{\overline{r},p}_{\overline{t},s}}(k).$$
Thanks to Eq.(\ref{p1}),(\ref{p5}), this is equal to the left-hand side. 

Eq.(\ref{l2'}) is equivalent to 
\begin{align*}
\lefteqn{\epsilon\overline{c_s\chi_r(k)\chi_s(k)}b^{r,s}_{t,u}(g)} \\
 &=\overline{c_tc_ua(g+k)a(k)\chi_u(g+k)\chi_t(k)\inpr{g+k}{k}} \frac{1}{\sqrt{n}} 
 \sum_{h}\overline{\inpr{g}{h}a(h)\chi_t(h)b^{t,u}_{s,r}(h)}.
\end{align*}
Thanks to $a(g+k)a(k)\inpr{g+k}{k}=a(g)a(k)^2\inpr{k}{k}=a(g)$ and Eq.(\ref{p8}), the right-hand side is 
\begin{align*}
\lefteqn{\overline{c_tc_ua(g)\chi_u(g+k)\chi_t(k)} \frac{1}{\sqrt{n}} 
 \sum_{h}\overline{\inpr{g}{h}a(h)\chi_t(h)b^{t,u}_{s,r}(h)} } \\
 &=\epsilon_s\epsilon_t\overline{c_sc_ua(g)\chi_u(g+k)\chi_t(k)} \frac{1}{\sqrt{n}} 
 \sum_{h}\overline{\inpr{g}{h}}b^{\overline{t},r}_{\overline{s},u}(-h). 
\end{align*}
This coincides with the left-hand side thanks to Eq.(\ref{p1}),(\ref{p6}). 

We set $\xi=e_u$, $\eta=e_p$, $\zeta=e_r$ in Eq.(\ref{l3'}). 
Then left-hand side is 
$$\sum_{g,s,t}a(-g-h)\chi_r(-g-h)b^{r,s}_{t,u}(-g-h)B_g(T_0(e_s))B_{g+k}(e_t)^*T_0(e_p).$$
Thanks to Eq.(\ref{p8}), this is equal to 
\begin{align*}
\lefteqn{\sum_{g,s,t}\epsilon_t\epsilon_rc_t\overline{c_r}\overline{b^{\overline{r},u}_{\overline{t},s}(g+h)}B_g(T_0(e_s))B_{g+k}(e_t)^*T_0(e_p)} \\
 &=\sum_{g,s,t,v,w,x}\epsilon_t\epsilon_rc_t\overline{c_r}b^{v,w}_{q,s}(g)\overline{b^{\overline{r},u}_{\overline{t},s}(g+h)}
 \overline{b^{p,x}_{q,t}(g+k)} T_0(e_v)T_0(e_w)T(e_x)^*.
\end{align*}
The first term of the right-hand side is 
\begin{align*}
\lefteqn{\inpr{h}{k}a(h)\overline{\chi_r(h)}B_k(e_r)^*T(e_p)B_{-h}(e_u)} \\
 &=\inpr{h}{k}a(h)\overline{\chi_r(h)}\sum_{v,w,x,y} \overline{b^{p,y}_{v,r}(k)}b^{y,w}_{x,u}(-h)T_0(e_v)T_0(e_w)T_0(e_x)^*\\
 &=\epsilon_u\inpr{h}{k}\chi_u(h)\overline{\chi_r(h)}\sum_{v,w,x,y}\epsilon_w\overline{b^{p,y}_{v,r}(k)b^{x,\overline{w}}_{y,\overline{u}}(h)}
 T_0(e_v)T_0(e_w)T_0(e_x)^*,
\end{align*}
where we used Eq.(\ref{p7}).
The second term is 
$$-\frac{\overline{c_u}\delta_{r,u}\epsilon_{\overline{p}}}{d\sqrt{n}}\sum_{v,w,x}\epsilon_v\delta_{w,\overline{v}}\delta_{x,\overline{p}}
T_0(e_v)T_0(e_w)T_0(e_x)^*.$$
Thus Eq.(\ref{l3'}) is equivalent to Eq.(\ref{p10}). 
\end{proof}

Since $\inpr{\cdot}{\cdot}$ is non-degenerate, there exists unique $g_r\in G$ for each $r\in \Lambda$ satisfying 
$\chi_r(g)=\inpr{g}{g_r}$ for any $g\in G$. 

\begin{lemma}
Eq.(\ref{p1}) and (\ref{p9}) imply
\begin{equation}\label{p11} b^{r,s}_{t,u}(g)=c_rc_u\overline{c_sc_t}b^{s,r}_{u,t}(g+g_s-g_u).
\end{equation}
\end{lemma} 

\begin{proof} Assuming Eq.(\ref{p1}),(\ref{p6}), and (\ref{p9}), we get \begin{align*}
\lefteqn{\cF^{-1}(b^{r,s}_{t,u})(h)=\epsilon\epsilon_r\epsilon_tc_u a(h)\chi_u(h)b^{s,\overline{t}}_{\overline{r},u}} \\
 &=\epsilon\epsilon_r\epsilon_tc_ua(h)\chi_u(h)
 \epsilon_s\epsilon_{\overline{t}}\epsilon_{\overline{r}}\epsilon_u c_{\overline{r}}\overline{c_s}\overline{\chi_s(h)\chi_{\overline{t}}(h)}
 b^{r,\overline{u}}_{\overline{s},t}(h)\\
 &=\epsilon\epsilon_s\epsilon_u  c_uc_{r}\overline{c_s}a(h)\chi_u(h)\chi_{t}(h)\overline{\chi_s(h)}
 b^{r,\overline{u}}_{\overline{s},t}(h) \\
 &=c_uc_{r}\overline{c_s c_t}\chi_u(h)\overline{\chi_s(h)}\cF^{-1}(b^{s,r}_{u,t})(h)\\
 &=c_uc_{r}\overline{c_s c_t}\inpr{h}{g_u-g_s}\cF^{-1}(b^{s,r}_{u,t})(h),
\end{align*}
which show the statement. 
\end{proof}

We denote by $\cG(A,C,J)$ the set of all unitaries acting on $\cK_0$ commuting with $A(g),C,J$, and 
call it the gauge group. 
An element $v$ in the gauge group acts on the solution of the above polynomial equations as 
$B'_g=(v\otimes v)B_g(v^*\cdot)v^*$, or equivalently, 
$$b'^{r',s'}_{t',u'}(g)=\sum_{r,s,t,u}v_{r'r}v_{s's}\overline{v_{t't}v_{u'u}} b^{r,s}_{t,u}(g).$$

It is convenient to introduce a matrix $\cB(g)$ for each $g\in G$ with an index set $\Lambda\times \Lambda$ whose 
$\big((r,t),(s,u)\big)$-entry is $b^{r,s}_{t,u}(g)$. 
Then 
$$\cB(g)^*\cB(g)=\cB(g)\cB(g)^*=\frac{1}{n}I-\frac{\delta_{g,0}}{d}\delta\otimes \delta,$$
where $\delta\in \C^{\Lambda\times \Lambda}$ is a vector whose $(r,t)$-component is $\delta_{r,t}$. 
In fact $\delta$ is a common eigenvector of $\cB(0)$ and $\cB(0)^*$ with an eigenvalue $-1/d$. 
Thus $\sqrt{n}\cB(g)$ is a unitary for $g\neq 0$, and $\sqrt{n}\cB(0)$ is a unitary on $\{\delta\}^{\perp}$.  
An element $v$ of the gauge group acts on $\cB(g)$ as 
$$\cB'(g)=(v\otimes \overline{v})\cB(g)(v^{T}\otimes v^*).$$

\section{The case of $m=|G|$}\label{m=n} 
In this section, we give a brief account of the most tractable case $m=n$ among the irrational case, 
which was essentially done in \cite{I01} and \cite{EG14}.  
Since $\dim \cK_0=1$, we can choose $e\in \cK_0$ with $\|e\|=1$ and $Je=e$, and so $\epsilon=1$. 
Such $e$ is unique up to sign.
We denote $T_g=T_g(e)$.  
Then $V(g)T_h=\inpr{g}{h}T_h$ and $U_{\cK}(g)T_h=T_{h-g}$. 
We can arrange $a(g)$ so that $A(g)=a(g)I$, and $C$ is a scaler, which we denote by $c$.  
Thus 
$$j_1(T_h)=a(h)T_{-h},$$
$$j_2(T_h)=\frac{\overline{c}}{\sqrt{n}}\sum_{k\in G}\overline{\inpr{h}{k}}T_k.$$
There exists $b:G\rightarrow \C$ satisfying $B_g(e)=b(g)T_0T_0T_0^*$, and 
$$l(T_g)=\sum_{h,k}a(h)b(h+g)\inpr{g}{k}T_{h+k}T_{-h}T_k^*.$$

\begin{theorem}\label{m=nTh} The C$^*$ near-group categories with a finite abelian group $G$ with $m=|G|$ are completely 
classified by $(\inpr{\cdot}{\cdot},a,b,c)$ satisfying the conditions in the statement of Lemma \ref{ACJ} and 
the following equations up to the group automorphisms of $G$. 
\begin{equation}\label{m=n1}
\hat{b}(g)=ca(g)b(-g),
\end{equation}\label{m=n2}
\begin{equation}
b(0)=-\frac{1}{d},
\end{equation}\label{m=n3}
\begin{equation}
|b(g)|^2=\frac{1}{n}-\frac{\delta_{g,0}}{d},
\end{equation}
\begin{equation}\label{m=n4}
\overline{b(g)}=a(g)b(-g),
\end{equation}
\begin{equation}\label{m=n5}
\sum_{g\in G}b(g+h)b(g+k)\overline{b(g)}=\overline{\inpr{h}{k}}b(h)b(k)-\frac{c}{d\sqrt{n}}.
\end{equation}
\end{theorem}

\begin{proof} 
Note that replacing $e$ with $-e$ does not change the above equations. 
Thus the statement follows from Theorem \ref{irrcl} and Theorem \ref{irrpe}. 
\end{proof}

\begin{remark} We rephrase the equations in the above theorem so that we can easily guess 
Galois conjugate solutions in the non-unitary case. 
We use a parameter $c'=\overline{c}/\sqrt{n}$ instead of $c$. 
For a given pair $(\inpr{\cdot}{\cdot},a(g))$, the polynomial equations are equivalent to 
\begin{equation}\label{Gal1}
{c'}^3=\frac{1}{n^2}\sum_{g\in G}a(g),
\end{equation}
\begin{equation}\label{Gal2}
d^2=dn+n,
\end{equation}
\begin{equation}\label{Gal3}
\cR_{c'}b=b,
\end{equation}
\begin{equation}\label{Gal4}
b(0)=-\frac{1}{d}
\end{equation}
\begin{equation}\label{Gal5}
a(g)b(g)b(-g)=\frac{1}{n}-\frac{\delta_{g,0}}{d},
\end{equation}
\begin{equation}\label{Gal6}
\sum_{g\in G}a(g)b(-g)b(g+h)b(g+k)=\inpr{h}{k}^{-1}b(h)b(k)-\frac{c'^{-1}}{dn},
\end{equation}
\begin{equation}\label{Gal7}
Jb=b,
\end{equation}
\begin{equation}\label{Gal8}
d>0,
\end{equation}
where $\cR_{c'}$ is a period 3 unitary on $\ell^2(G)$ given by 
$$\cR_{c'}f(g)=c'a(g)^{-1}\sum_{h\in G}\inpr{g}{h}f(h),$$
and $\cJ$ is a period 2 anti-unitary on $\ell^2(G)$ given by 
$$\cJ f(g)=\overline{a(g)f(-g)}.$$
They satisfy $\cJ\cR_{c'}=\cR_{c'}^2\cJ$, and the eigenspaces of $\cR_{c'}$ are preserved by $\cJ$. 
We suspect that the last two equations come from the unitarity of the category, 
and the other equations are already good enough to give near-group categories in the general case. 
We show that known solutions of the whole equations have Galois conjugate solutions that satisfy 
all the equations except for the last two. 
\end{remark}

\begin{example} For $G=\Z_2$, there is a unique non-degenerate symmetric bicharacter $\inpr{g}{h}=(-1)^{gh}$, 
and there is unique $a(g)$ up to complex conjugate given by $a(1)=\ii$. 
Then there is a unique solution of Eq.(\ref{Gal1})-(\ref{Gal8}): 
$$c'=\frac{e^{-\frac{7\pi\ii}{12}}}{\sqrt{2}}=\frac{1-\sqrt{3}-(1+\sqrt{3})\ii}{4},\quad d=1+\sqrt{3},$$
$$b(1)=\frac{e^{-\frac{\pi\ii}{4}}}{\sqrt{2}}=\frac{1-\ii}{2}.$$
For only  Eq.(\ref{Gal1})-(\ref{Gal6}), there is another solution 
$$c'=\frac{e^{\frac{\pi\ii}{12}}}{\sqrt{2}}=\frac{1+\sqrt{3}+(\sqrt{3}-1)\ii}{4},\quad d=1-\sqrt{3},$$
$$b(1)=\frac{e^{-\frac{\pi\ii}{4}}}{\sqrt{2}}=\frac{1-\ii}{2}.$$
\end{example} 

\begin{example} For $G=\Z_3$, there is a unique non-degenerate symmetric bicharacter up to complex conjugate, 
and we consider $\inpr{g}{h}=\zeta_3^{gh}$, where $\zeta_n=e^{\frac{2\pi\ii}{n}}$. 
For this there is a unique $a(g)$ given by $a(1)=a(2)=\zeta_3$. 
For this pair, there exists a unique solution of Eq.(\ref{Gal1})-(\ref{Gal8}) up to group automorphism given by 
$$c'=\frac{e^{\frac{\pi\ii}{6}}}{\sqrt{3}}=\frac{1}{2}+\frac{\sqrt{3}\ii}{6}\in \Q(\zeta_3),\quad d=\frac{3+\sqrt{21}}{2},$$
$$b(1)=\zeta_3\big(\frac{-3+\sqrt{21}}{12}+\sqrt{\frac{3+\sqrt{21}}{2}}\frac{\ii}{2\sqrt{3}}\big), \;
b(2)=\zeta_3\big(\frac{-3+\sqrt{21}}{12}-\sqrt{\frac{3+\sqrt{21}}{2}}\frac{\ii}{2\sqrt{3}}\big).$$
For only  Eq.(\ref{Gal1})-(\ref{Gal6}), there is another solution 
$$c'=\frac{e^{\frac{\pi\ii}{6}}}{\sqrt{3}}=\frac{1}{2}+\frac{\sqrt{3}\ii}{6},\quad d=\frac{3-\sqrt{21}}{2},$$
$$b(1)=\zeta_3\big(\frac{-3-\sqrt{21}}{12}+\sqrt{\frac{-3+\sqrt{21}}{2}}\frac{1}{2\sqrt{3}}\big),\;
b(3)=\zeta_3\big(\frac{-3-\sqrt{21}}{12}-\sqrt{\frac{-3+\sqrt{21}}{2}}\frac{1}{2\sqrt{3}}\big).$$
\end{example}

\begin{example} For $\Z_4$, there is a unique non-degenerate symmetric bicharacter up to complex conjugate, 
and we consider $\inpr{g}{h}=\ii^{gh}$. 
Then there exists a unique solution of Eq.(\ref{Gal1})-(\ref{Gal8}) up to group automorphism given by 
$$a(1)=a(3)=e^{-\frac{\pi\ii}{4}},\quad a(2)=-1,$$
$$c'=\frac{e^{-\frac{3\pi\ii}{4}}}{2},\quad d=2+2\sqrt{2},$$
$$b(1)=\zeta_{16}(\frac{\sqrt{4-2\sqrt{2}}}{4}+\frac{\ii}{2^{5/4}}),\quad b(2)=\zeta_{16}(\frac{\sqrt{4-2\sqrt{2}}}{4}-\frac{\ii}{2^{5/4}}),$$
$$b(2)=\frac{-\ii}{2}.$$
For only  Eq.(\ref{Gal1})-(\ref{Gal6}), there is another solution 
$$a(1)=a(3)=-e^{-\frac{\pi\ii}{4}},\quad a(2)=-1,$$
$$c'=\frac{e^{\frac{\pi\ii}{4}}}{2},\quad d=2-2\sqrt{2},$$
$$b(1)=\zeta_{16}^5\big(\frac{\sqrt{4+2\sqrt{2}}}{4}+\frac{1}{2^{5/4}}\big),
\quad b(3)=\zeta_{16}^5\big(\frac{\sqrt{4+2\sqrt{2}}}{4}-\frac{1}{2^{5/4}}\big),$$
$$b(2)=\frac{-\ii}{2}.$$
\end{example}

\begin{example}\label{Z2Z2} For $G=\Z_2\times \Z_2=\{0,g_1,g_2,g_3\}$, there is a unique C$^*$ near-group category.  
There exist exactly two non-degenerate symmetric bicharacters on $\Z_2\times \Z_2$ up to group automorphisms, 
and we denote by $\inpr{\cdot}{\cdot}_1$ the one given by the following table. 
\begin{table}[ht]
\begin{tabular}{|c|c|c|c|}\hline
&$g_0$&$g_1$&$g_2$\\ \hline
$g_0$&-1&1&-1\\ \hline
$g_1$&1&-1&-1\\ \hline
$g_2$&-1&-1&1\\ \hline
\end{tabular}
\end{table}
Up to flipping $g_1$ and $g_2$, there are exactly two $a(g)$ for $\inpr{\cdot}{\cdot}_1$. 
For $a(g_0)=-a(g_1)=\ii$, $a(g_2)=1$, there exists a unique solution for Eq.(\ref{Gal1})-(\ref{Gal8}): 
$$c'=\frac{1}{2}, \quad d=2+2\sqrt{2},$$
$$b(g_1)=\frac{e^{\frac{3\pi\ii}{4}}}{2},\quad b(g_2)=\frac{e^{\frac{-3\pi\ii}{4}}}{2},\quad b(g_3)=\frac{1}{2}.$$
For only Eq.(\ref{Gal1})-(\ref{Gal6}) there is another solution 
$$c'=\frac{1}{2}, \quad d=2-2\sqrt{2},$$
$$b(g_1)=\frac{e^{\frac{-\pi\ii}{4}}}{2},\quad b(g_2)=\frac{e^{\frac{\pi\ii}{4}}}{2},\quad b(g_3)=\frac{1}{2}.$$
For $a(g_0)=a(g_1)=\ii$, $a(g_2)=-1$, there is no solution. 

The other non-degenerate symmetric bicharacter $\inpr{\cdot}{\cdot}_2$ is given by the following table 
\begin{table}[ht]
\begin{tabular}{|c|c|c|c|}\hline
&$g_0$&$g_1$&$g_2$\\ \hline
$g_0$&1&-1&-1\\ \hline
$g_1$&-1&1&-1\\ \hline
$g_2$&-1&-1&1\\ \hline
\end{tabular}
\end{table}
though there is no solution for it. 
\end{example}

For the complete list of the solutions of Eq.(\ref{Gal1})-(\ref{Gal8}) for $G$ with $|G|\leq 13$, 
see \cite[Table 2]{EG14} (see Example \ref{Z2Z2Z3} below for two additional solutions for $G=\Z_2\times \Z_2\times \Z_3$). 
\section{The case of $m=2|G|$}\label{2n}
In this section we assume $m=2n$ and write down the polynomial equations Eq(\ref{p1})-(\ref{p9}) 
in a more accessible form. 
The author's experience tells that there are only finitely many solutions for Eq(\ref{p1})-(\ref{p9}) 
up to equivalence, and they are likely to satisfy Eq.(\ref{p10}) automatically, which in practice could be 
verified by computer. 
We show that there are exactly two solutions, up to equivalence, of the polynomial equations for $G=\Z_3$. 

\subsection{Possible cases}
We choose the index set $\Lambda$ of the orthonormal basis of $\cK_0$ as $\Lambda=\{1,2\}$. 
When $\chi_1^2=1$, we may replace $a(g)$ with $a(g)\chi_1(g)$, and we may and do assume $\chi_1=1$. 
Lemma \ref{clacov} shows that the only possible cases are the following: 
\begin{itemize}
\item Case I. $\chi_1=\chi_2=1$, $\epsilon=1$, $Je_1=e_1$, $Je_2=e_2$. 

\item Case II. $\chi_1=\chi_2=1$, $\epsilon=-1$, $Je_1=e_2$, $Je_2=-e_2$, $c_1=c_2$. 

\item Case III. $\chi_1=1$, $\chi_2\neq 1$, $\chi_2^2=1$, $\epsilon=1$, $Je_1=e_1$, $Je_2=e_2$. 

\item Case IV. $\chi_2=\chi_1^{-1}$, $\chi_1^2\neq 1$. $Je_1=e_2$, $Je_2=\epsilon e_1$, $c_1=c_2$. 
\end{itemize}

We use the lexicographic order of the set $\Lambda^2=\{1,2\}^2$ to express $\cB(g)$ as a matrix.

\begin{lemma} Case IV never occurs. 
\end{lemma}

\begin{proof} 
We assume that $b^{r,s}_{t,u}(g)$ is a solution for Eq.(\ref{p1})-Eq.(\ref{p10}) in Case IV. 
We introduce a unitary operator $\cR_1$ of period 3 on $\ell^2(G)$ by 
$$\cR_1f(g)=\frac{\overline{c_1a(g)}}{\sqrt{n}}\sum_{h\in G}\inpr{g}{h}f(h).$$
Then Eq.(\ref{p1}) implies $\cR_1 b^{1,2}_{2,1}(g)=b^{2,1}_{2,1}(g)$ and 
$\cR_1^2 b^{1,2}_{2,1}(g)=\epsilon b^{1,1}_{1,1}(g)$, 
which shows 
$$\|b^{2,1}_{2,1}\|=\|b^{1,1}_{1,1}\|=\|b^{1,2}_{2,1}\|,$$
where $\|f\|$ denotes the $\ell^2$-norm of $f\in \ell^2(G)$.  

Eq.(\ref{p6}) implies that the matrix $\cB(g)$ is of the form 
$$\cB(g)=\left(
\begin{array}{cccc}
* &0 &0 &*  \\
0 &* &* &0  \\
0 &* &* &0  \\
* &0 &0 &* 
\end{array}
\right),
$$
and Eq.(\ref{p4}) implies 
$$|b^{1,2}_{2,1}(g)|^2+|b^{2,2}_{1,1}(g)|^2=\frac{1}{n},$$
$$|b^{1,1}_{1,1}(g)|^2+|b^{2,1}_{2,1}(g)|^2=\frac{1}{n}-\frac{\delta_{g,0}}{d}.$$
Thus 
$$2\|b^{1,2}_{2,1}\|^2=\|b^{1,1}_{1,1}\|^2+\|b^{2,1}_{2,1}\|^2=\sum_{g\in G}(\frac{1}{n}-\frac{\delta_{g,0}}{d})=1-\frac{1}{d},$$
and $\|b^{1,2}_{2,1}\|^2<1/2$. 
If $\chi_1^4\neq 1$, Eq.(\ref{p6}) implies $b^{2,2}_{1,1}(g)=0$ and $|b^{2,1}_{1,2}(g)|^2=1/n$, which is contradiction. 
If $\chi_1^4=1$, Eq.(\ref{p9}) implies $b^{2,2}_{1,1}(g)=\chi_1(g)^2b^{2,2}_{1,1}(g)$, and $b^{2,2}_{1,1}$ is supported by 
$H=\{g\in G;\; \chi_1(g)^2=1\}$. 
Note that $H$ is an index 2 subgroup of $G$. 
For $g\in G\setminus H$, we have $|b^{1,2}_{2,1}(g)|^2=1/n$, and 
$$\|b^{1,2}_{2,1}\|^2\geq \sum_{g\in G\setminus H}|b^{1,2}_{2,1}(g)|^2=\frac{1}{2},$$
which is contradiction too. 
\end{proof}

\subsection{Case I}
We first assume only $\chi_1=\chi_2=1$. 

We choose $c\in \T$ with $c^3\hat{a}(0)=1$, and set $\omega_r=c_t/c$. 
Then $\omega_t^3=1$. 
We introduce a unitary $\cR\in \B(\ell^2(G))$ of period 3 by 
$$\cR f(g)=\frac{\overline{ca(g)}}{\sqrt{n}}\sum_{h}\inpr{g}{h}f(h),$$
and anti-unitary $\cJ$ of period 2 by 
$$\cJ f(g)=\overline{a(g)f(-g)}.$$
Then they satisfy $\cR\cJ=\cJ\cR^2$. 
Eq.(\ref{p1}) and Eq.(\ref{p2}) now become  
$$\cR b^{r,s}_{t,u}(g)=\omega_ub^{s,\overline{t}}_{\overline{r},u}(g),$$
$$\cJ b^{r,s}_{t,u}(g)=\epsilon_s\epsilon_u b^{t,\overline{s}}_{r,\overline{u}}(g).$$
In particular, the function $b^{r,r}_{\overline{r},u}$ is an eigenvector of $\cR $ for 
the eigenvalue $\omega_u$. 

Now we assume $\epsilon=1$, $Je_1=e_1$, $Je_2=e_2$. 
When $c_1=c_2$, the gauge group $\cG(A,C,J)$ is $O(2)$. 
When $c_1\neq c_2$, it is $\Z_2\times \Z_2$. 

It is straightforward to show the following lemma. 

\begin{lemma} \label{CaseI1} 
Assume $\chi_1=\chi_2=1$, $\epsilon=1$, $Je_1=e_1$, $Je_2=e_2$. 
Let $\xi_1(g)=b^{1,1}_{1,1}(g)$, $\xi_2(g)=b^{2,2}_{2,2}(g)$, $\eta_1(g)=b^{2,2}_{2,1}(g)$, $\eta_2(g)=b^{1,1}_{1,2}(g)$, 
$\mu (g)=b^{1,2}_{1,2}(g)$. 
Then Eq.(\ref{p1})-(\ref{p9}) are equivalent to the following:
\begin{itemize}
\item[(1)] The matrix $\cB(g)=(b^{r,s}_{t,u}(g))_{(r,t),(s,u)}$ is expressed as 
$$\cB(g)=\left(
\begin{array}{cccc}
\xi_1(g) &\eta_2(g) &\eta_2(g) &\mu (g)  \\
\omega_1\omega_2^2\eta_2(g) &\omega_2\cR ^2\mu (g) &\omega_1^2\cR \mu (g) &\omega_1\omega_2^2\eta_1(g)  \\
\omega_1^2\omega_2\eta_2(g) &\omega_2^2\cR \mu (g) &\omega_1\cR ^2\mu (g) &\omega_1^2\omega_2\eta_1(g)  \\
\mu (g) &\eta_1(g) &\eta_1(g) &\xi_2(g) 
\end{array}
\right).
$$

\item[(2)] For any $g\in G$, 
\begin{equation}
\cR \xi_1(g)=\omega_1\xi_1(g),\quad \cJ\xi_1(g)=\xi_1(g), 
\end{equation}
\begin{equation}
\cR \xi_2(g)=\omega_2\xi_2(g),\quad \cJ\xi_2(g)=\xi_2(g), 
\end{equation}
\begin{equation}
\cR \eta_1(g)=\omega_1\eta_1(g),\quad \cJ\eta_1(g)=\eta_1(g), 
\end{equation}
\begin{equation}
\cR \eta_2(g)=\omega_2\eta_2(g),\quad \cJ\eta_2(g)=\eta_2(g), 
\end{equation}
\begin{equation}
\cJ\mu (g)=\mu (g),\quad \cJ\cR \mu (g)=\cR ^2\mu (g), 
\end{equation}
(where the second one follows from the first one in the last equation though). 

\item[(3)] $\xi_1(0)+\mu(0)=\xi_2(0)+\mu(0)=-1/d$, $\eta_1(0)+\eta_2(0)=0$. 

\item[(4)] For any $g \in G$.
\begin{equation}
|\xi_1(g)|^2+2|\eta_2(g)|^2+|\mu(g)|^2=\frac{1}{n}-\frac{\delta_{g,0}}{d},
\end{equation}
\begin{equation}
|\xi_2(g)|^2+2|\eta_1(g)|^2+|\mu(g)|^2=\frac{1}{n}-\frac{\delta_{g,0}}{d},
\end{equation}
\begin{equation}
|\eta_1(g)|^2+|\eta_2(g)|^2+|\cR \mu(g)|^2+|\cR ^2\mu(g)|^2=\frac{1}{n},
\end{equation}
\begin{equation}
2\eta_2(g)\overline{\eta_1(g)}+\mu(g)\overline{\xi_1(g)}+\overline{\mu(g)}\xi_2(g)=-\frac{\delta_{g,0}}{d}, 
\end{equation}
\begin{equation}
\eta_2(g)\overline{\xi_1(g)}+\eta_1(g)\overline{\mu(g)}+(\omega_1\omega_2\cR \mu(g)+\overline{\omega_1\omega_2}\cR ^2\mu(g))\overline{\eta_2(g)}=0,
\end{equation}
\begin{equation}
\eta_1(g)\overline{\xi_2(g)}+\eta_2(g)\overline{\mu(g)}+(\omega_1\omega_2\cR \mu(g)+\overline{\omega_1\omega_2}\cR ^2\mu(g))\overline{\eta_1(g)}=0,
\end{equation}
\begin{equation}
|\eta_1(g)|^2+|\eta_2(g)|^2+\overline{\omega_1\omega_2}\cR \mu(g)\overline{\cR ^2\mu(g)}+\omega_1\omega_2\overline{\cR \mu(g)}\cR ^2\mu(g)=0. 
\end{equation}
\end{itemize}
\end{lemma}

\begin{remark}\label{rCaseI} The condition (4) above imply 
\begin{equation}\label{Ir1}
\|\xi_1\|^2+2\|\eta_2\|^2+\|\mu\|^2=1-\frac{1}{d},
\end{equation}
\begin{equation}\label{Ir2}
\|\xi_2\|^2+2\|\eta_1\|^2+\|\mu\|^2=1-\frac{1}{d},
\end{equation}
\begin{equation}\label{Ir3}
\|\eta_1\|^2+\|\eta_2\|^2+2\|\mu\|^2=1,
\end{equation}
\begin{equation}\label{Ir4}
2\inpr{\eta_2}{\eta_1}_{\ell^2(G)}+\inpr{\mu}{\xi_1}_{\ell^2(G)}+\inpr{\xi_2}{\mu}_{\ell^2(G)}=-\frac{1}{d}, 
\end{equation}
\begin{equation}\label{Ir5}
\inpr{\eta_2}{\xi_1}_{\ell^2(G)}+\inpr{\eta_1}{\mu}_{\ell^2(G)}+(\omega_1\overline{\omega_2}+\overline{\omega_1}\omega_2)\inpr{\mu}{\eta_2}_{\ell^2(G)}=0,
\end{equation}
\begin{equation}\label{Ir6}
\inpr{\eta_1}{\xi_2}_{\ell^2(G)}+\inpr{\eta_2}{\mu}_{\ell^2(G)}+(\omega_1\overline{\omega_2}+\overline{\omega_1}\omega_2)\inpr{\mu}{\eta_1}_{\ell^2(G)}=0,
\end{equation}
\begin{equation}\label{Ir7}
\|\eta_1\|^2+\|\eta_2\|^2+\overline{\omega_1\omega_2}\inpr{\mu}{\cR\mu}_{\ell^2(G)}+\omega_1\omega_2\inpr{\cR \mu}{\mu}_{\ell^2(G)}=0. 
\end{equation}
Eq.(\ref{Ir3}) and Eq.(\ref{Ir7}) imply
\begin{equation}\label{Ir8}
2\|\mu\|^2-\omega_1\omega_2\inpr{\cR \mu}{\mu}_{\ell^2(G)}-\overline{\omega_1\omega_2}\inpr{\cR^2\mu}{\mu}_{\ell^2(G)}=1. 
\end{equation}
\end{remark}

\begin{lemma} \label{CaseI2} 
The solutions of the equations in Lemma \ref{CaseI1} satisfy either of the following two: 
\begin{itemize}
\item[(1)] $\eta_1(0)=-\eta_2(0)\in \R$ with $|\eta_1(0)|\leq 1/(2\sqrt{n})$, and there exist 
$\kappa_1,\kappa_2\in \{1,-1\}$ satisfying  
$$\xi_1(0)=\xi_2(0)=-\frac{1}{2d}-\frac{\kappa_1\sqrt{1-4n\eta_1(0)^2}}{2\sqrt{n}},$$
$$\mu(0)=-\frac{1}{2d}+\frac{\kappa_1\sqrt{1-4n\eta_1(0)^2}}{2\sqrt{n}},$$
$$\cR\mu(0)=\overline{\cR^2\mu(0)}=\overline{\omega_1\omega_2}\frac{\kappa_1\sqrt{1-4n\eta_1(0)^2}+\kappa_2\ii }{2\sqrt{n}}.$$

\item[(2)] There exist $\kappa_1,\kappa_2 \in \{1,-1\}$ satisfying 
$$\xi_1(0)=\xi_2(0)=-\frac{1}{2d}-\frac{\kappa_1}{2\sqrt{n}},$$
$$\eta_1(0)=\eta_2(0)=0,$$
$$\mu(0)=-\frac{1}{2d}+\frac{\kappa_1}{2\sqrt{n}},$$
$$\cR\mu(0)=\overline{\cR^2\mu(0)}=\overline{\omega_1\omega_2}\frac{-\kappa_1+\kappa_2\ii }{2\sqrt{n}}.$$
\end{itemize}
\end{lemma}

\begin{proof} Lemma \ref{CaseI1} (2) shows that $\xi_r(0)$, $\eta_r(0)$, and $\mu(0)$ are all real and 
$\cR^2\mu(0)=\overline{\cR\mu(0)}$. 
Lemma \ref{CaseI1} (3) shows that $\eta_1(0)+\eta_2(0)=0$ and there exist real numbers $x$ and $y$ satisfying 
$$\xi_1(0)=\xi_2(0)=-\frac{1}{2d}-x,\quad \mu(0)=-\frac{1}{2d}+x.$$
Now Lemma \ref{CaseI1} (4) with $g=0$ is equivalent to 
$$x^2+\eta_1(0)^2=\frac{1}{4n},$$
$$\eta_1(0)^2+|\cR\mu(0)|^2=\frac{1}{2n},$$
$$\eta_1(0)(\omega_1\omega_2\cR\mu(0)+\overline{\omega_1\omega_2\cR\mu(0)}-2x)=0,$$
$$2\eta_1(0)^2+(\omega_1\omega_2\cR\mu(0))^2+(\overline{\omega_1\omega_2\cR\mu(0)})^2=0.$$
Solving these, we get the statement. 
\end{proof}

Let $\zeta_3=e^{2\pi\ii /3}$. 
We can expand $\mu$ satisfying $\cJ\mu=\mu$ as $\mu=\mu_0+\mu_1+\mu_2$ with $\cR\mu_i=\zeta_3^i\mu_i$ 
and $\cJ\mu_i=\mu_i$ because we have $\cJ\cR=\cR^2\cJ$. 
The index $i$ in $\mu_i$ will be understood as an element of $\Z_3$. 

\begin{lemma}\label{CaseI3}  
If $\omega_1=\omega_2=1$, there exists a solution of the equations in Lemma \ref{CaseI1} 
only if $\dim\ker(\cR-1)\geq 2$. 
The solutions satisfy $\|\mu_1\|^2+\|\mu_2\|^2=1/3$, and either of the following two:
\begin{itemize}
\item[(1)]
$\eta_1(0)=-\eta_2(0)\in \R$ with $ |\eta_1(0)|\leq 1/(2\sqrt{n})$, and there exist 
$\kappa_1,\kappa_2\in \{1,-1\}$ satisfying  
$$\xi_1(0)=\xi_2(0)=-\frac{1}{2d}-\frac{\kappa_1\sqrt{1-4n\eta_1(0)^2}}{2\sqrt{n}},$$
$$\mu_0(0)=-\frac{1}{6d}+\frac{\kappa_1\sqrt{1-4n\eta_1(0)^2}}{2\sqrt{n}},$$
$$\mu_1(0)=-\frac{1}{6d}+\frac{\kappa_2}{2\sqrt{3n}},$$
$$\mu_2(0)=-\frac{1}{6d}-\frac{\kappa_2}{2\sqrt{3n}}.$$
\item[(2)] There exist $\kappa_1,\kappa_2 \in \{1,-1\}$ satisfying 
$$\xi_1(0)=\xi_2(0)=-\frac{1}{2d}-\frac{\kappa_1}{2\sqrt{n}},$$
$$\eta_1(0)=\eta_2(0)=0,$$
$$\mu_0(0)=-\frac{1}{6d}-\frac{\kappa_1}{6\sqrt{n}},$$
$$\mu_1(0)=-\frac{1}{6d}+\frac{\kappa_1}{3\sqrt{n}}+\frac{\kappa_2}{2\sqrt{3n}},$$
$$\mu_2(0)=-\frac{1}{6d}+\frac{\kappa_1}{3\sqrt{n}}-\frac{\kappa_2}{2\sqrt{3n}}.$$
\end{itemize}
\end{lemma}

\begin{proof} The equality $\|\mu_1\|^2+\|\mu_2\|^2=1/3$ follows from Eq.(\ref{Ir8}). 
Assume that (1) in Lemma \ref{CaseI2} occurs. 
Then we have 
$$\mu_0(0)+\mu_1(0)+\mu_2(0)=-\frac{1}{2d}+\frac{\kappa_1\sqrt{1-4n\eta_1(0)^2}}{2\sqrt{n}},$$
$$\mu_0(0)+\zeta_3\mu_1(0)+\overline{\zeta_3}\mu_2(0)
=\frac{\kappa_1\sqrt{1-4n\eta_1(0)^2}+\kappa_2\ii }{2\sqrt{n}},$$
$$\mu_0(0)+\overline{\zeta_3}\mu_1(0)+\zeta_3\mu_2(0)
=\frac{\kappa_1\sqrt{1-4n\eta_1(0)^2}-\kappa_2\ii }{2\sqrt{n}},$$
which imply (1). 

We further assume $\dim (\cR-1)=1$ and get contradiction. 
Since $(\xi_1(0),\mu_0(0))\neq (0,0)$, there exists a unique $f\in \ker(\cR-1)$ satisfying $\cJ f=f$ and $f(0)=1$. 
Then we have $\xi_1=\xi_2=\xi_1(0)f$, $\eta_1=-\eta_2=\eta_1(0)f$, and $\mu_0=\mu_0(0)f$. 
Eq.(\ref{Ir6}) implies $\eta_1(0)(\xi_1(0)+\mu_0(0))=0$, and we get $\eta_1(0)=0$. 
Thus $\eta_1=\eta_2=0$. 
Eq.(\ref{Ir3}) implies $\|\mu\|^2=1/2$.  
Since $\|\mu_1\|^2+\|\mu_2\|^2=1/3$, 
we get $\|\mu_0\|^2=1/6$ and $\mu_0(0)^2\|f\|^2=1/6$. 
Eq.(\ref{Ir1}) and Eq.(\ref{Ir4}) imply 
$$\xi_1(0)^2\|f\|^2=\frac{1}{2}-\frac{1}{d},$$
$$\mu_0(0)\xi_1(0)\|f\|^2=-\frac{1}{2d}.$$
and 
$$(\frac{\xi_1(0)}{\mu_0(0)})^2=3(1-\frac{2}{d}),$$
$$\frac{\xi_1(0)}{\mu_0(0)}=-\frac{3}{d},$$
which is contradiction. 

Now assume that (2) in Lemma \ref{CaseI2} occurs. 
Then we have 
$$\mu_0(0)+\mu_1(0)+\mu_2(0)=-\frac{1}{2d}+\frac{\kappa_1}{2\sqrt{n}},$$
$$\mu_0(0)+\zeta_3\mu_1(0)+\overline{\zeta_3}\mu_2(0)=\frac{-\kappa_1+\kappa_2\ii }{2\sqrt{n}},$$
$$\mu_0(0)+\overline{\zeta_3}\mu_1(0)+\zeta_3\mu_2(0)=\frac{-\kappa_1-\kappa_2\ii }{2\sqrt{n}},$$
which implies (2). 

We further assume $\dim (\cR-1)=1$ and get contradiction. 
In this case, we would have $\xi_1=\xi_2=3\mu_0$, $\eta_1=\eta_2=0$, 
which contradicts Eq.(\ref{Ir4}).
\end{proof}

\begin{remark} If $\omega_1=\omega_2=\zeta_3^i$, we can apply the same argument to $\zeta_3^{-i}\cR$ instead of $\cR$, 
and the same statement replacing $(\mu_0,\mu_1,\mu_2)$ with $(\mu_i,\mu_{1+i},\mu_{2+i})$ holds. 
\end{remark}

In the same way, we can show 

\begin{lemma}\label{CaseI4} 
If $\omega_1=\zeta_3$ and $\omega_2=\zeta_3^{-1}$, the solutions of the equations in Lemma \ref{CaseI1} 
satisfy 
\begin{equation}
\|\mu_1\|^2+\|\mu_2\|^2=\frac{1}{3},
\end{equation}
and either of the following two:
\begin{itemize}
\item[(1)]
$\eta_1(0)=-\eta_2(0)\in \R$ with $ |\eta_1(0)|\leq 1/(2\sqrt{n})$, and there exist 
$\kappa_1,\kappa_2\in \{1,-1\}$ satisfying  
$$\xi_1(0)=\xi_2(0)=-\frac{1}{2d}-\frac{\kappa_1\sqrt{1-4n\eta_1(0)^2}}{2\sqrt{n}},$$
$$\mu_0(0)=-\frac{1}{6d}+\frac{\kappa_1\sqrt{1-4n\eta_1(0)^2}}{2\sqrt{n}},$$
$$\mu_1(0)=-\frac{1}{6d}+\frac{\kappa_2}{2\sqrt{3n}},$$
$$\mu_2(0)=-\frac{1}{6d}-\frac{\kappa_2}{2\sqrt{3n}}.$$
\item[(2)] There exist $\kappa_1,\kappa_2 \in \{1,-1\}$ satisfying 
$$\xi_1(0)=\xi_2(0)=-\frac{1}{2d}-\frac{\kappa_1}{2\sqrt{n}},$$
$$\eta_1(0)=\eta_2(0)=0,$$
$$\mu_0(0)=-\frac{1}{6d}-\frac{\kappa_1}{6\sqrt{n}},$$
$$\mu_1(0)=-\frac{1}{6d}+\frac{\kappa_1}{3\sqrt{n}}+\frac{\kappa_2}{2\sqrt{3n}},$$
$$\mu_2(0)=-\frac{1}{6d}+\frac{\kappa_1}{3\sqrt{n}}-\frac{\kappa_2}{2\sqrt{3n}}.$$
\end{itemize}
\end{lemma}

We can show the following lemma in the same way except for Eq.(\ref{I52}), 
which follows from Eq.(\ref{Ir2}), (\ref{Ir3}), and (\ref{I51}). 

\begin{lemma}\label{CaseI5} 
If $\omega_1=1$ and $\omega_2=\zeta_3^{\pm 1}$, the solutions of the equations in Lemma \ref{CaseI1} 
satisfy 
\begin{equation}\label{I51}\|\mu_0\|^2+\|\mu_{\pm 1}\|^2=\frac{1}{3},\end{equation}
\begin{equation}\label{I52} \|\xi_2\|^2-2\|\eta_2\|^2-3\|\mu_{\mp 1}\|^2=-\frac{1}{d}, 
\end{equation}
and either of the following two:
\begin{itemize}
\item[(1)]
$\eta_1(0)=-\eta_2(0)\in \R$ with $ |\eta_1(0)|\leq 1/(2\sqrt{n})$, and there exist 
$\kappa_1,\kappa_2\in \{1,-1\}$ satisfying  
$$\xi_1(0)=\xi_2(0)=-\frac{1}{2d}-\frac{\kappa_1\sqrt{1-4n\eta_1(0)^2}}{2\sqrt{n}},$$
$$\mu_0(0)=-\frac{1}{6d}+\frac{\kappa_2}{2\sqrt{3n}},$$
$$\mu_{\pm 1}(0)=-\frac{1}{6d}-\frac{\kappa_2}{2\sqrt{3n}},$$
$$\mu_{\mp 1}(0)=-\frac{1}{6d}+\frac{\kappa_1\sqrt{1-4n\eta_1(0)^2}}{2\sqrt{n}}.$$
\item[(2)] There exist $\kappa_1,\kappa_2 \in \{1,-1\}$ satisfying 
$$\xi_1(0)=\xi_2(0)=-\frac{1}{2d}-\frac{\kappa_1}{2\sqrt{n}},$$
$$\eta_1(0)=\eta_2(0)=0,$$
$$\mu_0(0)=-\frac{1}{6d}+\frac{\kappa_1}{3\sqrt{n}}+\frac{\kappa_2}{2\sqrt{3n}},$$
$$\mu_{\pm 1}(0)=-\frac{1}{6d}+\frac{\kappa_1}{3\sqrt{n}}-\frac{\kappa_2}{2\sqrt{3n}},$$
$$\mu_{\mp 1}(0)=-\frac{1}{6d}-\frac{\kappa_1}{6\sqrt{n}}.$$
\end{itemize}
\end{lemma}

\subsection{Case II}
We assume $\chi_1=\chi_2=1$, $Je_1=e_2$, $Je_2=-e_1$. 
Since $c_1=c_2$, we denote $\omega=c_r/c$. 
In this case the gauge group $\cG(A,C,J)$ is $SU(2)$. 

\begin{lemma}\label{CaseII1} Let the notation be as above. 
Let $\xi(g)=b^{1,1}_{2,2}(g)$,  $\eta(g)=b^{1,1}_{2,1}(g)$, $\mu (g)=b^{1,2}_{2,1}(g)$. 
Then Eq.(\ref{p1})-(\ref{p9}) are equivalent to the following:

\begin{itemize}
\item[(1)] The matrix $\cB(g)=(b^{r,s}_{t,u}(g))_{(r,t),(s,u)}$ is expressed as 
$$\cB(g)=\left(
\begin{array}{cccc}
-\omega\cR^2\mu(g) &\eta(g) &-\cJ\eta(g) &\overline{\omega}\cR \mu(g)  \\
\eta(g) &\xi(g) &\mu(g) &-\eta(g) \\
-\cJ\eta(g)&\mu(g)&-\cJ \xi(g)&\cJ\eta(g)  \\
\overline{\omega}\cR \mu(g)  &-\eta(g)&\cJ\eta(g)&-\omega\cR^2\mu(g) 
\end{array}
\right).
$$

\item[(2)] For any $g\in G$, 
\begin{equation}
\cR\xi=\omega \xi, 
\end{equation}
\begin{equation}
\cR\eta=\omega \eta,
\end{equation}
\begin{equation}
\cJ \mu=-\mu,\quad \cJ \cR\mu=-\cR^2\mu,
\end{equation}
(where the second one follows from the first one in the last equation though). 

\item[(3)] $\overline{\omega}\cR \mu(0)+\omega\overline{\cR\mu(0)}=-1/d$. 

\item[(4)] For any $g \in G$, 
\begin{equation}
|\cR\mu(g)|^2+|\cR\mu(-g)|^2+|\eta(g)|^2+|\eta(-g)|^2=\frac{1}{n}-\frac{\delta_{g,0}}{d},
\end{equation}
\begin{equation}
|\xi(g)|^2+|\mu(g)|^2+2|\eta(g)|^2=\frac{1}{n},
\end{equation}
\begin{equation}
\omega\cR\mu(g)\overline{\cR^2\mu(g)}+\overline{\omega}\cR^2\mu(g)\overline{\cR\mu(g)}+|\eta(g)|^2+|\eta(-g)|^2=\frac{\delta_{g,0}}{d},
\end{equation}
\begin{equation}
\eta(g)\overline{\xi(g)}-\cJ\eta(g)\overline{\mu(g)}-(\overline{\omega}\cR\mu(g)+\omega\cR^2\mu(g))\overline{\eta(g)}=0,
\end{equation}
\begin{equation}
\eta(g)\overline{\mu(g)}+\cJ\eta(g)\overline{\cJ\xi(g)}+(\overline{\omega}\cR\mu(g)+\omega\cR^2\mu(g))\overline{\cJ\eta(g)}=0,
\end{equation}
\begin{equation}
2\eta(g)\overline{\cJ\eta(g)}+\mu(g)\overline{\cJ \xi(g)}-\xi(g)\overline{\mu(g)}=0.
\end{equation}
\end{itemize}
\end{lemma}

\begin{remark}\label{rCaseII} The condition (4) above imply
\begin{equation}\label{IIr1}
2\|\mu\|^2+2\|\eta\|^2=1-\frac{1}{d},
\end{equation}
\begin{equation}\label{IIr2}
\|\xi\|^2+\|\mu\|^2+2\|\eta\|^2=1,
\end{equation}
\begin{equation}\label{IIr3}
\omega\inpr{\mu}{\cR\mu}_{\ell^2(G)}+\overline{\omega}\inpr{\cR\mu}{\mu}_{\ell^2(G)}+2\|\eta\|^2=\frac{1}{d},
\end{equation}
\begin{equation}\label{IIr4}
\inpr{\eta}{\xi}_{\ell^2(G)}=\inpr{\mu}{\eta}_{\ell^2(G)},
\end{equation}
\begin{equation}\label{IIr5}
\inpr{\eta}{\cJ\eta}_{\ell^2(G)}=\inpr{\xi}{\mu}_{\ell^2(G)}.
\end{equation}
Eq(\ref{IIr1}) and Eq(\ref{IIr3}) imply
\begin{equation}\label{IIr6}
2\|\mu\|^2-\overline{\omega}\inpr{\cR\mu}{\mu}_{\ell^2(G)}-\omega\inpr{\cR\mu}{\mu}_{\ell^2(G)}=1-\frac{2}{d}. 
\end{equation}
\end{remark}

\begin{lemma}\label{CaseII2}  
The solutions of the equations in Lemma \ref{CaseII1} satisfy either of the following two: 
\begin{itemize}
\item[(1)] $|\eta(0)|\leq 1/(2\sqrt{n})$, and there exist $\kappa_1,\kappa_2\in \{1,-1\}$ satisfying 
$$\xi(0)=\frac{\eta(0)}{\overline{\eta(0)}}\frac{\kappa_1+\kappa_2\sqrt{1-4n|\eta(0)|^2}}{2\sqrt{n}}\ii ,$$
$$\mu(0)=\frac{\kappa_1-\kappa_2\sqrt{1-4n|\eta(0)|^2}}{2\sqrt{n}}\ii ,$$
$$\cR\mu(0)=\omega\big(-\frac{1}{2d}-\frac{\kappa_2\sqrt{1-4n|\eta(0)|^2}}{2\sqrt{n}}\ii \big),$$
where $\eta(0)/\overline{\eta(0)}$ is interpreted as an arbitrary phase if $\eta(0)=1$. 
\item[(2)] $\eta(0)=\xi(0)=0$, and there exist $\kappa\in \{1,-1\}$ satisfying
$$\mu(0)=\frac{\kappa\ii }{\sqrt{n}},$$ 
$$\cR\mu(0)=\omega\big(-\frac{1}{2d}-\frac{\kappa\ii }{2\sqrt{n}}\big).$$
\end{itemize}
\end{lemma}

\begin{proof} Lemma \ref{CaseII1},(2) shows $\mu(0)\in \ii \R$ and $\cR^2\mu(0)=-\overline{\cR\mu(0)}$, and 
Lemma \ref{CaseII1},(4) with $g=0$ shows 
$$|\cR\mu(0)|^2+|\eta(0)|^2=\frac{1}{2n}-\frac{1}{2d},$$
$$-(\overline{\omega}\cR\mu(0))^2-(\omega\overline{\cR\mu(0)})^2+2|\eta(0)|^2=\frac{1}{d},$$
$$\eta(0)\overline{\xi(0)}+\overline{\eta(0)}\mu(0)-(\overline{\omega}\cR\mu(0)-\omega\overline{\cR\mu(0)})\overline{\eta(0)}=0,$$
$$-\eta(0)\mu(0)+\overline{\eta(0)}\xi(0)+(\overline{\omega}\cR\mu(0)-\omega\overline{\cR\mu(0)})\eta(0)=0,$$
$$\eta(0)^2+\mu(0) \xi(0)=0.$$
$$|\xi(0)|^2+|\mu(0)|^2+2|\eta(0)|^2=\frac{1}{n}.$$
The first two with Lemma \ref{CaseII1},(3) is equivalent to 
$$\cR\mu(0)=\omega\big(-\frac{1}{2d}\pm\frac{\sqrt{1-4n|\eta(0)|^2}}{2\sqrt{n}}\ii \big).$$
If $\eta(0)=0$, we get (2) or the $\eta(0)=0$ case of (1). 
Assume $\eta(0)\neq 0$ now. 
Then the third and fourth equalities implies 
$$\overline{\omega}\cR\mu(0)-\omega\overline{\cR\mu(0)}=\frac{\eta(0)}{\overline{\eta(0)}}\overline{\xi(0)}+\mu(0)
=-\frac{\overline{\eta(0)}}{\eta(0)}\xi(0)+\mu(0).$$
Thus we can introduce a real parameter $l$ satisfying 
$$\xi(0)=l\frac{\eta(0)}{\overline{\eta(0)}}\ii ,$$
$$\mu(0)=\frac{|\eta(0)|^2}{l}\ii .$$
Iterating these into the last equality, we get 
$$l^2+\frac{|\eta(0)|^4}{l^2}+2|\eta(0)|^2=\frac{1}{n},$$
and 
$$l+\frac{|\eta(0)|^2}{l}=\frac{\kappa_1}{\sqrt{n}},$$
with $\kappa_1^2=1$. 
Solving this, we get the statement. 
\end{proof}

We can expand $\mu\in \ell^2(G)$ satisfying $\cJ\mu=-\mu$ as $\mu=\mu_0+\mu_1+\mu_2$ with 
$\cR\mu_i=\zeta_3^i\mu_i$ and $\cJ\mu_i=-\mu_i$. 

\begin{lemma}\label{CaseII3} There exists a solution of the equations in Lemma \ref{CaseII1} only if 
$$\dim\ker(\cR-\omega)\geq 2.$$ 
The solutions with $\omega=1$ satisfy 
$$\|\mu_1\|^2+\|\mu_2\|^2=\frac{1}{3}-\frac{2}{3d},$$
and either of the following two: 
\begin{itemize}
\item[(1)] $|\eta(0)|\leq 1/(2\sqrt{n})$, and there exist $\kappa_1,\kappa_2\in \{1,-1\}$ satisfying 
$$\xi(0)=\frac{\eta(0)}{\overline{\eta(0)}}\frac{\kappa_1+\kappa_2\sqrt{1-4n|\eta(0)|^2}}{2\sqrt{n}}\ii ,$$
$$\mu_0(0)=(\frac{\kappa_1}{6\sqrt{n}}-\frac{\kappa_2\sqrt{1-4n|\eta(0)|^2}}{2\sqrt{n}})\ii ,$$
$$\mu_1(0)=(\frac{\kappa_1}{6\sqrt{n}}+\frac{1}{2\sqrt{3}d})\ii ,$$
$$\mu_2(0)=(\frac{\kappa_1}{6\sqrt{n}}-\frac{1}{2\sqrt{3}d})\ii .$$
where $\eta(0)/\overline{\eta(0)}$ is interpreted as an arbitrary phase if $\eta(0)=1$. 
\item[(2)] $\eta(0)=\xi(0)=0$, and there exist $\kappa\in \{1,-1\}$ satisfying
$$\mu_0(0)=0,$$ 
$$\mu_1(0)=(\frac{\kappa}{2\sqrt{n}}+\frac{1}{2\sqrt{3}d})\ii ,$$
$$\mu_2(0)=(\frac{\kappa}{2\sqrt{n}}-\frac{1}{2\sqrt{3}d})\ii .$$
\end{itemize}
The second case could occur only if $\dim \{f\in \ker(\cR-1);\; f(0)=0\}\geq 2$. 
\end{lemma}

\begin{proof} We assume $\omega=1$. 
The proof for the general case can be obtained by applying the same argument to $\overline{\omega}\cR$, and 
replacing $(\mu_0,\mu_1,\mu_2)$ with $(\mu_1,\mu_2,\mu_0)$ or $(\mu_2,\mu_0,\mu_1)$. 
The first equation follows from Eq.(\ref{IIr6}). 

Assume that (1) in Lemma \ref{CaseII2} occurs. 
Then we have 
$$\mu_0(0)+\mu_1(0)+\mu_2(0)=\frac{\kappa_1-\kappa_2\sqrt{1-4n|\eta(0)|^2}}{2\sqrt{n}}\ii ,$$
$$\mu_0(0)+\zeta_3\mu_1(0)+\overline{\zeta_3}\mu_2(0)=-\frac{1}{2d}-\frac{\kappa_2\sqrt{1-4n|\eta(0)|^2}}{2\sqrt{n}}\ii ,$$
$$\mu_0(0)+\overline{\zeta_3}\mu_1(0)+\zeta_3\mu_2(0)=\frac{1}{2d}-\frac{\kappa_2\sqrt{1-4n|\eta(0)|^2}}{2\sqrt{n}}\ii ,$$
which implies (1).

We further assume $\dim\ker(\cR-1)=1$ and get contradiction. 
Since $(\xi(0),\mu_0(0))\neq (0,0)$, we can find $f\in \ker(\cR-1)$ with $\cJ f=f$ and $f(0)=1$. 
Then we have $\xi=\xi(0)f$, $\eta=\eta(0)f$, $\mu_0=\mu_0(0)f$, and Eq.(\ref{IIr4}),(\ref{IIr5}) imply 
$$\eta(0)\overline{\xi(0)}=\mu_0(0)\overline{\eta(0)},$$
$$\eta(0)^2=\xi(0)\overline{\mu_0(0)}.$$
Since $\mu_0(0)\in \R\ii $,  this implies either $\xi=\eta=0$ or $\mu_0=\eta=0$. 
The first case contradicts Eq.(\ref{IIr1}),(\ref{IIr2}). 
The second case with Eq.(\ref{IIr1}) gives $\|\mu\|^2=(1-1/d)/2$ on one hand, and on the other hand we have 
$$\|\mu^2\|=\|\mu_1\|^2+\|\mu_2\|^2=\frac{1}{3}-\frac{2}{3d},$$
which is contradiction too. 

Now we assume (2) in Lemma \ref{CaseII2} occurs. 
Then we have 
$$\mu_0(0)+\mu_1(0)+\mu_2(0)=\frac{\kappa}{\sqrt{n}}\ii ,$$ 
$$\mu_0(0)+\zeta_3\mu_1(0)+\overline{\zeta_3}\mu_2(0)=-\frac{1}{2d}-\frac{\kappa}{2\sqrt{n}}\ii ,$$
$$\mu_0(0)+\overline{\zeta_3}\mu_1(0)+\zeta_3\mu_2(0)=\frac{1}{2d}-\frac{\kappa}{2\sqrt{n}}\ii ,$$
which implies (2). 

In a similar way as above, we can show that the condition 
$$\dim\{f\in \ker(\cR-1);\; f(0)=0\}\leq 1$$ 
does not allow any solution. 
\end{proof}

\begin{remark} If $\omega=\zeta_3^i$, we can apply the same argument to $\zeta_3^{-i}\cR$ instead of $\cR$, 
and the same statement replacing $(\mu_0,\mu_1,\mu_2)$ with $(\mu_i,\mu_{1+i},\mu_{2+i})$ holds. 
\end{remark}
\subsection{Case III}
We assume $\chi_1=1$, and $\chi_2\neq 1$, $\chi_2^2=1$, $Je_1=e_1$, $Je_2=e_2$. 
For simplicity we denote $\chi=\chi_2$. 
There exists unique $g_\chi\in G$ satisfying $\chi(g)=\inpr{g}{g_\chi}$ for all $g\in G$. 
Note that we have the following for $a_\chi(g)=\chi(g) a(g)$:
$$\widehat{a_\chi}(0)=\frac{1}{\sqrt{n}}\sum_{h\in G}\inpr{g_\chi}{h}a(h)=\hat{a}(-g_\chi)=\hat{a}(0)\overline{a(g_\chi)}.$$
Thus
$$c_2^3=\overline{\widehat{a_\chi}(0)}=\overline{\hat{a}(0)}a(g_\chi)=c_1^3a(g_\chi).$$
Since $g_\chi$ has order 2, we have $a(g_\chi)^2=\inpr{g_\chi}{g_\chi}=\chi(g_\chi)$, and $a(g_\chi)^4=1$, 
which shows $c_2^3=(c_1\overline{a(g_\chi)})^3$.  

\begin{lemma}\label{CaseIII1} For $r=1,2$, we introduce unitaries $\cR_r$, $r=1,2$, of period 3 on $\ell^2(G)$ by 
$$\cR_rf(g)=\frac{\overline{c_r\chi_r(g)a(g)}}{\sqrt{n}}\sum_{h\in G}\inpr{g}{h}f(h),$$
and anti-unitaries $\cJ_r$ of period 2 on $\ell^2(G)$ by 
$$\cJ_rf(g)=\overline{\chi_r(g)a(g)f(-g)}.$$
Then Eq.(\ref{p1}), (\ref{p6}), (\ref{p9}), (\ref{p11}) are equivalent to the following conditions: 
\begin{itemize}
\item[(1)] Let $\xi_1(g)=b^{1,1}_{1,1}(g)$, $\xi_2(g)=b^{2,2}_{2,2}(g)$, $\mu (g)=b^{1,2}_{1,2}(g)$. 
The matrix $\cB(g)=(b^{r,s}_{t,u}(g))_{(r,t),(s,u)}$ is expressed as 
$$\cB(g)=\left(
\begin{array}{cccc}
\xi_1(g) &0 &0 &\mu (g)  \\
0 &\cR_2^2\mu (g) &\cR_1 \mu (g) &0  \\
0&\cR_2 \mu (g) &\cR_1^2\mu (g) &0  \\
\mu (g) &0 &0 &\xi_2(g) 
\end{array}
\right),
$$
and the following hold: 
$$\cR_1 \xi_1(g)=\xi_1(g),\quad \cJ_1\xi_1(g)=\xi_1(g), $$
$$\cR_2 \xi_2(g)=\xi_2(g),\quad \cJ_2\xi_2(g)=\xi_2(g). $$
\item[(2)] $\chi(g)\mu(g)=\mu(g)$. 
\end{itemize}
\end{lemma}

\begin{proof} 
Eq.(\ref{p11}) implies $b^{12}_{12}(g)=b^{21}_{21}(g)$. 
Then Eq.(\ref{p1}), (\ref{p6}) are equivalent to (1). 

Eq.(\ref{p9}), (\ref{p11}) are equivalent to  
$$\mu(g)=\chi(g)\mu(g),$$
$$\cR_1\mu(g)=\overline{c_1}c_2\chi(g)\cR_2\mu(g)=c_1^2\overline{c_2^2}\cR_2\mu(g+g_\chi),$$
$$\cR_2\mu(g)=c_1\overline{c_2}\chi(g)\cR_1\mu(g)=\overline{c_1^2}c_2^2\cR_1\mu(g+g_\chi),$$
$$\cR_1^2\mu(g)=c_1\overline{c_2}\cR_2^2\mu(g)=\cR_1^2\mu(g+g_\chi),$$
$$\cR_2^2\mu(g)=\overline{c_1}c_2\cR_1^2\mu(g)=\cR^2_2\mu(g+g_\chi).$$
We show that all the conditions follow from the first one. 
Note that we have the relation $\cR_2f(g)=c_1\overline{c_2}\chi(g)\cR_1f(g)$. 
One the other hand, the unitary $\cR_r^2=\cR_r^{-1}$ is given by 
$$\cR_r^2f(g)=\frac{c_r}{\sqrt{n}}\sum_{h\in G}\overline{\inpr{g}{h}}a(h)\chi_r(h)f(h),$$
and in particular it implies that we have $c_1\cR_2^2f(g)=c_2\cR_1^2f(g)$ for any function $f$ satisfying $\chi(g)f(g)=f(g)$. 
Note that the first condition implies 
$$\cR_r\mu(g)=a_r(g+g_\chi)\overline{a_r(g)}\cR_r\mu(g+g_\chi)=a_r(g_\chi)\chi(g)\cR_r\mu(g+g_r),$$
$$\cR_r^2\mu(g)=\cR_r^2\mu(g+g_\chi).$$
Thus the remaining conditions are equivalent to $c_2^3=c_1^3a(g_\chi)=c_1^3\overline{a(g_\chi)}\inpr{g_\chi}{g_\chi}$, 
which holds in general. 
\end{proof}

We arbitrarily fix $c\in \T$ with $c^3\hat{a}(0)=1$, and define $\cR$ and $\cJ$ as in the previous subsections. 
We set $c_\chi=c\overline{a(g_\chi)}$, which satisfies  $c_\chi^3\widehat{a_\chi}(0)=1$. 
We set $\omega_1=c_1/c$ and $\omega_2=c_2/c_\chi$. 
Then $\omega_1^3=\omega_2^3=1$. 
Note that using notation in the proof of the above lemma, we have $\cR_1=\overline{\omega_1}\cR$.

\begin{lemma}\label{CaseIII2} Let the notation be as above. 
Let $\xi_1(g)=b^{1,1}_{1,1}(g)$, $\xi_2(g)=b^{2,2}_{2,2}(g)$,  $\mu (g)=b^{1,2}_{1,2}(g)$. 
Then Eq.(\ref{p1})-(\ref{p9}) are equivalent to the following:
\begin{itemize}
\item[(1)] The matrix $\cB(g)=(b^{r,s}_{t,u}(g))_{(r,t),(s,u)}$ is expressed as 
$$\cB(g)=\left(
\begin{array}{cccc}
\xi_1(g) &0 &0 &\mu (g)  \\
0 &\omega_2\overline{a(g_\chi)}\cR^2\mu(g) &\overline{\omega_1}\cR \mu (g) &0  \\
0&\overline{\omega_2}a(g_\chi)\chi(g)\cR\mu(g) &\omega_1\cR^2\mu (g) &0  \\
\mu (g) &0 &0 &\xi_2(g) 
\end{array}
\right).
$$

\item[(2)] For any $g\in G$, 
\begin{equation}
\cR \xi_1(g)=\omega_1\xi_1(g),\quad \cJ\xi_1(g)=\xi_1(g), 
\end{equation}
\begin{equation}
\cR \xi_2(g)=\omega_2\overline{a(g_\chi)}\chi(g)\xi_2(g),\quad \cJ\xi_2(g)=\chi(g)\xi_2(g), 
\end{equation}
\begin{equation}
\chi(g)\mu(g)=\mu(g),
\end{equation}
\begin{equation}
\cJ\mu (g)=\mu (g).
\end{equation}

\item[(3)] $\xi_1(0)+\mu(0)=\xi_2(0)+\mu(0)=-1/d$. 

\item[(4)] For any $g \in G$.
\begin{equation}
|\xi_1(g)|^2+|\mu(g)|^2=\frac{1}{n}-\frac{\delta_{g,0}}{d},
\end{equation}
\begin{equation}
|\xi_2(g)|^2+|\mu(g)|^2=\frac{1}{n}-\frac{\delta_{g,0}}{d},
\end{equation}
\begin{equation}
|\cR \mu(g)|^2+|\cR^2\mu(g)|^2=\frac{1}{n},
\end{equation}
\begin{equation}
\mu(g)\overline{\xi_1(g)}+\overline{\mu(g)}\xi_2(g)=-\frac{\delta_{g,0}}{d}, 
\end{equation}
\begin{equation}
a(g_\chi)\overline{\omega_1\omega_2}\cR \mu(g)\overline{\cR^2\mu(g)}
+\overline{a(g_\chi)}\omega_1\omega_2\chi(g)\cR^2\mu(g)\overline{\cR \mu(g)}=0. 
\end{equation}
\end{itemize}
\end{lemma}

\begin{proof} The proof of the previous lemma shows that any solution satisfies 
$$\cR_2\mu(g)=\overline{\omega_2}a(g_\chi)\chi(g)\cR\mu(g),$$
$$\cR_2^2\mu(g)=\omega_2\overline{a(g_\chi)}\cR^2\mu(g).$$
Now the statement follows from the previous lemma and a straightforward argument. 
\end{proof}

\begin{remark}\label{rCaseIII} The condition (4) above imply
\begin{equation}\label{IIIr1}
\|\xi_1\|^2=\|\xi_2\|^2=\frac{1}{2}-\frac{1}{d},
\end{equation}
\begin{equation}\label{IIIr2}
\|\mu\|^2=\frac{1}{2},
\end{equation}
\end{remark}

\begin{lemma}\label{CaseIII3}  
Any solutions of the equations in Lemma \ref{CaseIII2} satisfy the following two conditions: 
\begin{itemize}
\item[(1)] There exists $\kappa\in \{1,-1\}$ satisfying  
$$\xi_1(0)=\xi_2(0)=-\frac{1}{2d}-\frac{\kappa}{2\sqrt{n}},$$
$$\mu(0)=-\frac{1}{2d}+\frac{\kappa}{2\sqrt{n}}.$$
\item[(2)] For every order 2 element $g\in G$, there exists $\kappa_g\in \{1,-1\}$ satisfying
$$(\cR\mu(g))^2=\frac{\omega_1\omega_2\overline{a(g_\chi)a(g)}\kappa_g\ii ^{(1+\chi(g))/2}}{2n},$$
and in particular 
$$(\sqrt{2n}\cR\mu(g))^{12}=\inpr{g_\chi}{g_\chi}\inpr{g}{g}(-1)^{\frac{1+\chi(g)}{2}}.$$ 
Consequently, when $a(g_\chi)=\pm 1$, we have $(\sqrt{2n}\cR\mu(0))^{12}=-1$, and 
when $a(g_\chi)=\pm \ii $ we have $(\sqrt{2n}\cR\mu(0))^{12}=1$. 
\end{itemize}
\end{lemma}

\begin{proof} Since $J\cR\mu(g)=\cR^2\mu(g)$, we have $\cR^2\mu(g)=\overline{a(g)\cR\mu(-g)}$. 
(1) follows from the $g=0$ case of Lemma \ref{CaseIII2}. 

If $g\in G$ has order 2, we have $\cR^2\mu(g)=\overline{a(g)\cR\mu(g)}$. 
Thus Lemma \ref{CaseIII2},(4) implies $|\cR\mu(g)|^2=1/(2n)$ and 
$$\overline{\omega_1\omega_2}a(g_\chi)a(g)(\cR\mu(g))^2+\chi(g)\omega_1\omega_2\overline{a(g_\chi)a(g)}\overline{\cR\mu(g)}^2=0,$$
which implies the statement. 
\end{proof}

Note that Case III could occur only if $G$ is an even group because there is no character of order 2 for odd groups. 
For small even groups we have 

\begin{lemma} Case III could occur only if $\# G\geq 8$. 
\end{lemma}

\begin{proof} Assume that we have a solution $(\xi_1,\xi_2,\mu)$ for the equations in Lemma \ref{CaseIII2}.  
Since $\chi(g)\mu(g)=\mu(g)$, the function $\mu$ is supported by 
$$\{\chi\}^\perp=\{g\in G;\; \chi(g)=1\}.$$

For $G=\Z_2$, we have 
$$\mu(0)=-\frac{1}{2d}+\frac{\kappa}{2\sqrt{n}}=-\frac{1}{2(2+\sqrt{6})}+\frac{\kappa_1}{2\sqrt{2}}.$$
On the other hand since $\{\chi\}^\perp=\{0\}$, we have 
$$\frac{1}{2}=\|\mu\|^2=\mu(0)^2,$$
which is impossible.

For $G$ with $\#G=4$, we have $\#\{\chi\}^\perp=2$, and we set $\{\chi\}^\perp=\{0,g_\perp\}$.  
Since $g_\perp$ has order 2 and $\cJ\mu(g_\perp)=\mu(g_\perp)$, we have $\overline{\mu(g_\perp)}=a(g_\perp)\mu(g_\perp)$. 
Thus  
$$\frac{1}{8}=|\cR\mu(0)|^2=|\frac{\mu(0)+\mu(g_\perp)}{2}|^2=\frac{\mu(0)^2+|\mu(g_\perp)|^2+(1+a(g_\perp))\mu(g_\perp)}{4}.$$
Since 
$$\mu(0)^2+|\mu(g_\perp)|^2=\|\mu\|^2=\frac{1}{2},$$
and $\mu(g_\perp)\neq 0$, we get $a(g_\perp)=-1$. 
This implies 
$$\cR\mu(0)=\overline{c}\frac{\mu(0)\pm\sqrt{\frac{1}{2}-\mu(0)^2}\ii }{2},$$ 
and on the other hand $(2\sqrt{2}\cR\mu(0))^{24}=1$. 
This is contradiction as we have $\overline{c}^{24}={\hat{a}(0)}^8=1$, and  
$$\mu(0)=\frac{2-\sqrt{5}+\kappa_1}{4}.$$

For $G=\Z_6$, we have $\chi(g)=(-1)^g$ and $\{\chi\}^\perp=\{0,2,4\}$. 
Up to complex conjugate, we may assume $\inpr{g}{h}=\zeta_6^{gh}$, where 
$\zeta_6=e^{\frac{2\pi\ii }{6}}$. 
There are only two possibilities of $a(g)$, and we can choose $a(g)=e^{-\frac{\pi\ii g^2}{6}}$ as the other one is $a(g)\chi(g)$. 
Since $\cJ\mu(g)=\mu(g)$, we have $\mu(4)=\overline{a(2)\mu(2)}=\zeta_3\overline{\mu(2)}$. 
Now we have 
$$\cR\mu(0)=\overline{c}\frac{\mu(0)+\zeta_6(\zeta_6^{-1}\mu(2)+\zeta_6\overline{\mu(2)})}{\sqrt{6}}.$$
Let $x=\zeta_6^{-1}\mu(2)+\zeta_6\overline{\mu(2)}$. 
Then 
$$\frac{1}{12}=|\cR\mu(0)|^2=\frac{\mu(0)^2+x^2+\mu(0)x}{6},$$
and 
$$x=\frac{-\mu(0)\pm \sqrt{2-3\mu(0)^2}}{2}.$$
This implies 
$$\cR\mu(0)=\overline{c}\frac{3\mu(0)\pm\sqrt{2-3\mu(0)^2}+(-\mu(0)\pm\sqrt{2-3\mu(0)^3})\sqrt{3}\ii}{4\sqrt{6}}.$$
On the other hand we have $(\sqrt{12}\cR\mu(0))^{24}=1$, which is a contradiction as we have $c^{24}=1$ and 
$$\mu(0)=\frac{\sqrt{6}-\sqrt{7}+\kappa_1}{2\sqrt{6}}.$$
\end{proof}

To see if Case III really occurs, the first test case is order 8 abelian groups $\Z_8$, $\Z_4\times \Z_2$, and 
$\Z_2\times \Z_2\times \Z_2$. 
\subsection{The case of $G=\Z_2$}
Ostrik \cite{O13} showed that there is no near-group category for $G=\Z_2$ with $m\geq 3$. 
We give a proof of this fact in the case of $m=4$.  

There is a unique non-degenerate symmetric bicharacter $\inpr{g}{h}=(-1)^{gh}$ for $\Z_2$, 
and there are exactly two $a(g)$.  
We may assume $a(0)=1$, $a(1)=\ii $ as the other one is its complex conjugate. 
Then $\hat{a}(0)=e^{\frac{\pi\ii }{4}}$, and we choose $c=e^{-\frac{\pi\ii }{12}}$. 
We identify $\ell^2(\Z_2)$ with $\C^2$, the set of column vectors with two components, and express $\cR$ as a matrix 
$$\cR=e^{\frac{\pi\ii }{12}}
\left(
\begin{array}{cc}
1 &0  \\
0 &-\ii  
\end{array}
\right)
\frac{1}{\sqrt{2}}
\left(
\begin{array}{cc}
1 &1  \\
1 &-1 
\end{array}
\right)
=\frac{\sqrt{3}+1+(\sqrt{3}-1)\ii }{4}
\left(
\begin{array}{cc}
1 &1  \\
-\ii  &\ii  
\end{array}
\right).
$$
$\cJ$ is given by 
$$\cJ\left(
\begin{array}{c}
f(0)  \\
f(1) 
\end{array}
\right)
=\left(
\begin{array}{c}
\overline{f(0)}  \\
-\ii\overline{f(1)} 
\end{array}
\right).
$$
Let 
$$f_0=\left(
\begin{array}{c}
1  \\
\frac{\sqrt{3}-1}{\sqrt{2}}e^{-\frac{\pi\ii }{4}} 
\end{array}
\right), 
\quad 
f_1=\left(
\begin{array}{c}
1  \\
-\frac{\sqrt{3}+1}{\sqrt{2}}e^{-\frac{\pi\ii }{4}}
\end{array}
\right).
$$
Then we have $\cR f_0=f_0$, $\cR f_1=\zeta_3 f_1$, $Jf_0=f_0$, $Jf_1=f_1$, 
$\|f_0\|^2=3-\sqrt{3}$, $\|f_1\|^2=3+\sqrt{3}$. 

Lemma \ref{CaseII3} shows that Case II never occurs, and Lemma \ref{CaseI3} and Lemma \ref{CaseI4} show 
that the only possible case is Case I with $(\omega_1,\omega_2)=(1,\zeta_3)$ as $\zeta_3^2$ is not 
an eigenvalue of $\cR$. 
Since $\mu_2=0$, the case (1) of Lemma \ref{CaseI4} is the only possibility. 
Note that we have $\mu_0=\mu_0(0)f_0$ and $\mu_1=\mu_1(0)f_1$. 
Eq.(\ref{I51}) implies
\begin{align*}
\lefteqn{\frac{1}{3}=\mu_0(0)^2\|f_0\|^2+\mu_1(0)^2\|f_1\|^2} \\
 &=(3+\sqrt{3})(\frac{1}{6d}-\frac{\kappa_2}{2\sqrt{6}})^2+(3-\sqrt{3})(\frac{1}{6d}+\frac{\kappa_2}{2\sqrt{6}})^2,
\end{align*}
which is contradiction as $d=2+\sqrt{6}$.

\subsection{The case of $G=\Z_3$}\label{Z3m=6} 
Larson \cite{L14} showed that there exists no near-group category for $G=\Z_3$ with $m\geq 7$. 
On the other hand, it is known that there exists at least one C$^*$ near-group category for $G=\Z_3$ with $m=6$, 
which was first observed by Zhengwei Liu and Noah Snyder (\cite[page 59]{L15}, see also \cite[page 14]{EP12}). 

\begin{theorem}\label{Z36} There exist exactly two C$^*$ near-group categories for $\Z_3$ with $m=6$, and they are complex conjugate to 
each other. 
\end{theorem}

There are exactly two symmetric bicharacters of $\Z_3$ and they are complex conjugate to each other. 
We choose $\inpr{g}{h}=\zeta_3^{gh}$. 
For this, there is a unique $a(g)$, given by $a(1)=a(2)=\zeta_3$. 
Since $\hat{a}(0)=\ii $, we choose $c=e^{-\frac{\pi\ii }{6}}$. 
Then 
$$\cR =e^{\frac{\pi\ii }{6}}
\left(
\begin{array}{ccc}
1 &0 &0  \\
0 &\zeta_3^2 & 0 \\
0 &0         &\zeta_3^2 
\end{array}
\right)
\frac{1}{\sqrt{3}}
\left(
\begin{array}{ccc}
1 &1 &1  \\
1 &\zeta_3 &\zeta_3^2  \\
1 &\zeta_3^2 &\zeta_3 
\end{array}
\right)
=\frac{e^{\frac{\pi\ii }{6}}}{\sqrt{3}}
\left(
\begin{array}{ccc}
1 &1 &1  \\
\zeta_3^2 &1 &\zeta_3  \\
\zeta_3^2 &\zeta_3 &1 
\end{array}
\right),
$$
$$\cJ\left(
\begin{array}{c}
f(0)  \\
f(1)  \\
f(2) 
\end{array}
\right)
=\left(
\begin{array}{c}
\overline{f(0)}  \\
\zeta_3^{-1}\overline{f(2)}  \\
\zeta_3^{-1}\overline{f(1)}
\end{array}
\right).
$$
Let 
$$f_0=\left(
\begin{array}{c}
1  \\
-\frac{\zeta_3}{2} \\
-\frac{\zeta_3}{2}
\end{array}
\right),
\quad 
f_0'=
\left(
\begin{array}{c}
0  \\
\zeta_3\ii \\
-\zeta_3\ii
\end{array}
\right),
\quad 
f_1=
\left(
\begin{array}{c}
1  \\
\zeta_3  \\
\zeta_3 
\end{array}
\right). 
$$
Then $\cR f_0=f_0$, $\cR f_0'=f_0'$, $\cR f_1=\zeta_3f_1$, and  
$\cJ f_0=f_0$, $\cJ f_0'=f_0'$, $\cJ f_1=f_1$. 
They form a (non-normalized) 
orthogonal basis of $\ell^2(\Z_3)$ as well as the real vector space 
$$\{f\in \ell^2(\Z_3)\;\; \cJ f=f\}.$$

We prove the above theorem by showing the following lemma. 

\begin{lemma} For $\cR$ and $\cJ$ as above, there is no solution of the equations in Lemma \ref{CaseII1}, 
and there is a unique solutions of the equations in Lemma \ref{CaseI1} up to equivalence given as follows: 
$\omega_1=\omega_2=1$, and 
$$\xi_1(0)=-\frac{\sqrt{3-1}}{2},\quad \xi_1(1)=\zeta_3(\frac{\sqrt{3}-1}{4}+x\ii),\quad \xi_1(2)=\zeta_3(\frac{\sqrt{3}-1}{4}-x\ii),$$
$$\xi_2(0)=-\frac{\sqrt{3-1}}{2},\quad \xi_2(1)=\zeta_3(\frac{\sqrt{3}-1}{4}+x\ii),\quad \xi_2(2)=\zeta_3(\frac{\sqrt{3}-1}{4}-x\ii),$$
$$\eta_1(0)=0,\quad \eta_1(1)=y\zeta_3\ii,\quad \eta_1(2)=-y\zeta_3 \ii,$$
$$\eta_1(0)=0,\quad \eta_2(1)=-y\zeta_3\ii,\quad \eta_2(2)=y\zeta_3 \ii,$$
$$\mu(0)=\frac{\sqrt{3}-1}{2\sqrt{3}},\quad \mu(1)=\zeta_3(\frac{\sqrt{3}+1}{4\sqrt{3}}-x\ii),
\quad \mu(2)=\zeta_3(\frac{\sqrt{3}+1}{4\sqrt{3}}+x\ii),$$
$$\cR\mu(0)=\frac{-1+\ii}{2\sqrt{3}},\; \cR\mu(1)=\zeta_3(\frac{1-\sqrt{3}+2\ii}{4\sqrt{3}}-x\ii),
\; \cR\mu(2)=\zeta_3(\frac{1-\sqrt{3}+2\ii}{4\sqrt{3}}+x\ii),$$
$$\cR^2\mu(0)=-\frac{1+\ii}{2\sqrt{3}},\; \cR^2\mu(1)=\zeta_3(\frac{1-\sqrt{3}-2\ii}{4\sqrt{3}}-x\ii),
\; \cR^2\mu(2)=\zeta_3(\frac{1-\sqrt{3}-2\ii}{4\sqrt{3}}+x\ii),$$
where $x,y\in \R$ with 
$$x^2+y^2=\frac{\sqrt{3}}{24}.$$
\end{lemma}

\begin{proof} Since $\zeta_3^2$ is not an eigenvalue of $\cR$, we have $\mu_2(0)=0$, and 
Lemma \ref{CaseII3} shows that Case II never occur. 
Lemma \ref{CaseI3} shows that the only possibilities in Case I are $(\omega_1,\omega_2)=(1,\zeta_3)$ and 
$(\omega_1,\omega_2)=(1,1)$. 

First we assume that $(\omega_1,\omega_2)=(1,\zeta_3)$ holds. 
Since $\mu_2(0)=0$, Lemma \ref{CaseI4} shows that only the case (1) of Lemma \ref{CaseI4} is possible,  
and we have 
$$\kappa_1\sqrt{1-12\eta_1(0)^2}=\frac{2-\sqrt{3}}{3},$$
$$\eta_1(0)^2=\frac{1+2\sqrt{3}}{54},$$
$$\xi_1(0)=\xi_2(0)=-\frac{2(2-\sqrt{3})}{3\sqrt{3}}.$$
Since $\xi_2=\xi_1(0)f_1$, $\eta_2=-\eta_1(0)f_1$, $\mu_2=0$, 
Eq.(\ref{I52}) implies 
$$-\frac{1}{d}=\xi_1(0)^2\|f_1\|^2-2\eta_1(0)^2\|f_1\|^2=3(\frac{2(2-\sqrt{3})}{3\sqrt{3}})^2-6\frac{1+2\sqrt{3}}{54}=3-2\sqrt{3},$$
which is contradiction as we have $d=3+2\sqrt{3}$. 

Now we solve the equations for $(\omega_1,\omega_2)=(1,1)$. 
Since $\mu_2(0)=0$, only the case (2) in Lemma \ref{CaseI3} with $\kappa_1=\kappa_2=1$ is possible, and 
$$\xi_1(0)=\xi_2(0)=3\mu_0(0)=-\frac{\sqrt{3}-1}{2},$$
$\eta_1(0)=\eta_2(0)=0$, $\mu_1(0)=1/3$. 
Since $\xi_1$, $\xi_2$, $\eta_1$, $\eta_2$, $\mu_0$ belong to the real linear space 
$$\{f\in \ker(\cR-1);\;\cJ f=f \},$$ 
there exist real numbers $x_1,x_2,y_1,y_2,p$ satisfying $\xi_r=\xi_r(0)f_0+x_rf_0'$, $\eta_r=y_rf_0'$ for $r=1,2$, 
and $\mu_0=\mu_0(0)f_0+pf_0'$. 
Since $\dim\ker(\cR-\zeta_3)=1$, we have $\mu_1=\frac{1}{3}f_1$. 

Eq.(\ref{Ir1}), (\ref{Ir2}), (\ref{Ir3}) imply 
$$\|\xi_1\|^2+\|\xi_2\|^2-2\|\mu\|^2=-\frac{2}{d}.$$
Thus 
$$2\xi_1(0)^2\|f_0\|^2+(x_1^2+x_2^2)\|f_0'\|^2-2(\mu_0(0)^2\|f_0\|^2+p^2\|f_0'\|^2+\frac{1}{3})=-\frac{2}{d},$$
and we get $x_1^2+x_2^2=2p^2$. 
Eq.(\ref{Ir5}), (\ref{Ir6}) imply 
$$py_1+(2p+x_1)p=0,$$
$$(2p+x_2)y_1+py_2=0,$$
and so either $y_1=y_2=0$ or $p^2=(2p+x_1)(2p+x_2)$. 
We claim that $x_1=x_2=-p$, $y_1=-y_2$, and 
$$x_1^2+y_1^2=\frac{1}{8\sqrt{3}}.$$

Indeed, assume first that $(y_1,y_2)\neq (0,0)$ holds. 
Then 
\begin{align*}
\lefteqn{0=3p^2+2(x_1+x_2)p+x_1x_2} \\
 &=3p^2+2(x_1+x_2)p+\frac{(x_1+x_2)^2-(x_1^2+x_2^2)}{2}\\
 &=\frac{6p^2+4(x_1+x_2)p+(x_1+x_2)^2-2p^2}{2}\\
 &=\frac{(2p+x_1+x_2)^2}{2},
\end{align*}
and we get $x_1+x_2=2p$. 
Therefore we have $x_1=x_2=-p$ and $p(y_1+y_2)=0$. 
Eq.(\ref{Ir1}), (\ref{Ir2}) imply 
$$\|\xi_1\|^2-\|\xi_2\|^2+2(\|\eta_2\|^2-\|\eta_1\|^2)=0,$$
which shows $y_1^2=y_2^2$. 
Eq.(\ref{Ir1}),(\ref{Ir4}) imply
$$4y_2^2+4p^2=-4y_1y_2+4p^2=\frac{1}{2\sqrt{3}},$$ 
which shows the claim. 

Assume now that $y_1=y_2=0$ holds. 
Then $\eta_1=\eta_2=0$, and the equations in Remark \ref{rCaseI} are equivalent to 
$$\|\mu\|^2=\frac{1}{2},$$
$$\|\xi_1\|^2=\|\xi_2\|^2=\frac{3}{2}-\frac{2}{\sqrt{3}},$$
$$\inpr{\mu}{\xi_1}_{\ell^2(\Z_3)}+\inpr{\xi_2}{\mu}_{\ell^2(\Z_3)}=1-\frac{2}{\sqrt{3}}.$$
In terms of $p$, $x_1$, $x_2$, these are equivalent to 
$$p^2=x_1^2=x_2^2=\frac{1}{8\sqrt{3}},$$
$$p(x_1+x_2)=-\frac{1}{4\sqrt{3}},$$
which shows the claim again. 

Let $x=x_1$ and $y=y_1$. 
Then $\xi_r$, $\eta_r$, $\mu$ are as in the statement. 
It is straightforward to show that they satisfy the conditions in Lemma \ref{CaseI1}. 

Since $(\omega_1,\omega_2)=(1,1)$, the gauge group is $O(2)$, and we show that they act on the solutions transitively.  
Indeed, it is easy to show that 
$$\left(
\begin{array}{cc}
1 &0  \\
0 &-1 
\end{array}
\right)
$$
acts on the solutions as $(x,y)\mapsto (x,-y)$. 
Let 
$$R(\theta)=\left(
\begin{array}{cc}
\cos \theta &-\sin \theta  \\
\sin\theta &\cos\theta 
\end{array}
\right).$$
Note that we have 
\begin{align*}
\lefteqn{\cB(g)} \\
 &=-\frac{\sqrt{3}-1}{6}f_0(g)\left(
\begin{array}{cccc}
3 &0 &0 &1  \\
0 &1 &1 &0  \\
0 &1 &1 &0  \\
1 &0 &0 &3 
\end{array}
\right)
+\frac{1}{3}f_1(g)\left(
\begin{array}{cccc}
0 &0 &0 &1  \\
0 &\zeta_3^2 &\zeta_3 &0  \\
0 &\zeta_3 &\zeta_3^2 &0  \\
1 &0 &0 &0 
\end{array}
\right)
 \\
 &+f_0'(g)[x(X\otimes X-Y\otimes Y)-y(X\otimes Y+Y\otimes X)],
\end{align*}
where 
$$X=\left(
\begin{array}{cc}
1 &0  \\
0 &-1 
\end{array}
\right), 
\quad Y=\left(
\begin{array}{cc}
0 &1  \\
1 &0 
\end{array}
\right). 
$$
The first two terms commute with $R(\theta)\otimes R(\theta)$. 
Since 
$$R(\theta)XR(\theta)^{-1}=\cos2\theta X+\sin2\theta Y,$$
$$R(\theta)YR(\theta)^{-1}=-\sin2\theta X+\cos2\theta Y,$$
we have 
\begin{align*}
\lefteqn{(R(\theta)\otimes R(\theta))(X\otimes X-Y\otimes Y)(R(\theta)^{-1}\otimes R(\theta)^{-1})} \\
 &=\cos 4\theta (X\otimes X-Y\otimes Y)+\sin4\theta(X\otimes Y+Y\otimes X),
\end{align*}
\begin{align*}
\lefteqn{(R(\theta)\otimes R(\theta))(X\otimes Y+Y\otimes X)(R(\theta)^{-1}\otimes R(\theta)^{-1})} \\
 &=-\sin 4\theta (X\otimes X-Y\otimes Y)+\cos4\theta(X\otimes Y+Y\otimes X).
\end{align*}
Therefore $R(\theta)$ acts on the solutions as 
$$(x,y)\mapsto (x\cos4\theta +y\sin4\theta ,-x\sin4\theta +y\cos4\theta).$$
\end{proof}

\begin{remark}
We already know that there is a C$^*$ near-group category for $\Z_3$ with $m=6$, and we do not need to verify 
Eq.(\ref{p10}). 
\end{remark}

Among the equivalence class of the solutions above, we can choose a representative with $y=0$. 
Such a solution is invariant under the subgroup of $O(2)$ generated by 
$$\left(
\begin{array}{cc}
1 &0  \\
0 &-1 
\end{array}
\right),\quad 
\left(
\begin{array}{cc}
0 &-1  \\
1 &0 
\end{array}
\right),
$$
which is the dihedral group $D_8$ of order 8. 
Thus there is a $D_8$-action $\gamma$ on $\cO_9$ commuting with $\alpha_g$ and $\rho$ corresponding 
such a solution. 
We can perform equivariantization with this $D_8$-action, and the resulting category includes a 
near-group category for $G=\Z_2\times \Z_2\times \Z_3$ with $m=12$ as a subcategory (see Example \ref{36equi}). 
\subsection{The case of $G=\Z_4$} 
\begin{theorem}\label{Z41} There is no C$^*$ near-group category for $\Z_4$ with $m=8$. 
\end{theorem}

There are exactly two non-degenerate symmetric bicharacters for $\Z_4$, which are complex conjugate to each other, 
and we may assume $\inpr{g}{h}=\ii^{gh}$ to show the statement. 
For this, there are exactly two $a(g)$, and one of them is $a(g)=e^{-\frac{g^2\pi \ii}{4}}$. 
In this case we can choose $c=e^{\frac{3\pi \ii}{4}}$, and 
$$\cR=\frac{e^{-\frac{3\pi\ii}{4}}}{2}\left(
\begin{array}{cccc}
1 &0 &0 &0  \\
0 &e^{\frac{\pi\ii}{4}}&0 &0  \\
0 &0 &-1 &0  \\
0 &0 &0 &e^{\frac{\pi\ii}{4}} 
\end{array}
\right)
\left(
\begin{array}{cccc}
1 &1 &1 &1  \\
1 &\ii &-1 &-\ii  \\
1 &-1 &1 &-1  \\
1 &-\ii &-1 &\ii 
\end{array}
\right),
$$
$$\cJ \left(
\begin{array}{c}
f(0)  \\
f(1)  \\
f(2)  \\
f(3) 
\end{array}
\right)
=\left(
\begin{array}{c}
\overline{f(0)}  \\
e^{\frac{\pi\ii}{4}}\overline{f(3)}  \\
- \overline{f(2)}  \\
e^{\frac{\pi\ii}{4}}\overline{f(1)} 
\end{array}
\right).
$$
Let 
$$f_0=\left(
\begin{array}{c}
1  \\
-\frac{\zeta_{16}}{2\cos\frac{3\pi}{8}}  \\
\tan\frac{3\pi}{8}\ii  \\
 -\frac{\zeta_{16}}{2\cos\frac{3\pi}{8}}
\end{array}
\right),\quad 
f_0'=\left(
\begin{array}{c}
0  \\
\zeta_{16}\ii  \\
0  \\
-\zeta_{16}\ii 
\end{array}
\right),
\quad f_1=
\left(
\begin{array}{c}
1  \\
\frac{\zeta_{16}}{2\cos\frac{17\pi}{24}}  \\
\tan\frac{17\pi}{24}\ii  \\
\frac{\zeta_{16}}{2\cos\frac{17\pi}{24}}
\end{array}
\right), \quad f_2
=\left(
\begin{array}{c}
1  \\
\frac{\zeta_{16}}{2\cos\frac{\pi}{24}}  \\
\tan\frac{\pi}{24}\ii  \\
\frac{\zeta_{16}}{2\cos\frac{\pi}{24}}
\end{array}
\right).
$$
Then $\cR f_0=f_0$, $\cR f_0'=f_0'$, $\cR f_1=\zeta_3f_1$, $\cR f_2=\zeta_3^2f_2$, and 
they are invariant under $\cJ$. 
They form a (non-normalized) orthogonal basis of $\ell^2(\Z_4)$ as well as 
the real subspace of $\cJ$-invariant vectors. 
We have 
$$\|f_1\|^2=\frac{3}{2\cos^2\frac{17\pi}{24}}=\frac{3}{1+\cos\frac{7\pi}{12}}=\frac{6\sqrt{2}}{1+2\sqrt{2}-\sqrt{3}},$$
$$\|f_2\|^2=\frac{3}{2\cos^2\frac{\pi}{24}}=\frac{3}{1+\cos\frac{\pi}{12}}=\frac{6\sqrt{2}}{1+2\sqrt{2}+\sqrt{3}}.$$

The other choice of $a(g)$ is $a(g)=(-1)^ge^{-\frac{g^2\pi \ii}{4}}$. 
In this case we can choose $c=e^{-\frac{\pi \ii}{4}}$, and 
$$\cR=\frac{e^{\frac{\pi\ii}{4}}}{2}\left(
\begin{array}{cccc}
1 &0 &0 &0  \\
0 &-e^{\frac{\pi\ii}{4}}&0 &0  \\
0 &0 &-1 &0  \\
0 &0 &0 &-e^{\frac{\pi\ii}{4}} 
\end{array}
\right)
\left(
\begin{array}{cccc}
1 &1 &1 &1  \\
1 &\ii &-1 &-\ii  \\
1 &-1 &1 &-1  \\
1 &-\ii &-1 &\ii 
\end{array}
\right),
$$
$$\cJ \left(
\begin{array}{c}
f(0)  \\
f(1)  \\
f(2)  \\
f(3) 
\end{array}
\right)
=\left(
\begin{array}{c}
\overline{f(0)}  \\
-e^{\frac{\pi\ii}{4}}\overline{f(3)}  \\
- \overline{f(2)}  \\
-e^{\frac{\pi\ii}{4}}\overline{f(1)} 
\end{array}
\right).
$$
Let 
$$f_0=\left(
\begin{array}{c}
1  \\
\frac{\zeta_{16}^5}{2\cos\frac{7\pi}{8}}  \\
\tan\frac{7\pi}{8}\ii  \\
\frac{\zeta_{16}^5}{2\cos\frac{7\pi}{8}}
\end{array}
\right),\;
f_0'=\left(
\begin{array}{c}
0  \\
\zeta_{16}^5\ii  \\
0  \\
-\zeta_{16}^5\ii 
\end{array}
\right), \;
f_1=
\left(
\begin{array}{c}
1  \\
\frac{\zeta_{16}^5}{2\cos\frac{5\pi}{24}}  \\
\tan\frac{5\pi}{24}\ii  \\
\frac{\zeta_{16}^5}{2\cos\frac{5\pi}{24}}
\end{array}
\right), \;
f_2
=\left(
\begin{array}{c}
1  \\
\frac{\zeta_{16}^5}{2\cos\frac{11\pi}{24}}  \\
-\tan\frac{11\pi}{24}\ii  \\
\frac{\zeta_{16}^5}{2\cos\frac{11\pi}{24}}
\end{array}
\right).
$$
Then $\cR f_0=f_0$, $\cR f_0'=f_0'$, $\cR f_1=\zeta_3f_1$, $\cR f_2=\zeta_3^2f_2$, and 
they are invariant under $\cJ$. 
They form a (non-normalized) orthogonal basis of $\ell^2(\Z_4)$ as well as 
the real subspace of $\cJ$-invariant vectors. 
$$\|f_1\|^2=\frac{3}{2\cos^2\frac{5\pi}{24}}=\frac{3}{1+\cos\frac{5\pi}{12}}=\frac{6\sqrt{2}}{-1+2\sqrt{2}+\sqrt{3}},$$
$$\|f_2\|^2=\frac{3}{2\cos^2\frac{11\pi}{24}}=\frac{3}{1+\cos\frac{11\pi}{12}}=\frac{6\sqrt{2}}{-1+2\sqrt{2}-\sqrt{3}}.$$

We give a unified argument for $\Z_4$ and $\Z_2\times \Z_2$ to show that certain cases never occur. 
Note that we have $d=4+2\sqrt{5}$ in the both cases. 

\begin{lemma}\label{Z42} Assume that $\# G=4$, $\dim\ker(\cR-\zeta_3^r)=1$ for $r=1,2$, and there exist 
$f_r\in \ker(\cR-\zeta_3^r)$ satisfying $Jf_r=f_r$ and $f_r(0)=1$ for $r=1,2$. 
We assume $\sqrt{5}\notin \Q(\sqrt{3},\|f_1\|^2,\|f\|_2^2)$. 
\begin{itemize}
\item[(1)] Neither the case (1) of Lemma \ref{CaseI3} nor the case (1) of Lemma \ref{CaseI4} occurs. 
\item[(2)] The case (2) of Lemma \ref{CaseI3} or the case (2) of Lemma \ref{CaseI4} could occur only if 
$\kappa_1=-1$ and $\|f_1\|^2=\|f_2\|^2=3$. 
\item[(3)] The case (1) of Lemma \ref{CaseII3} never occurs. 
\item[(4)] The case (2) of Lemma \ref{CaseII3} could occur only if $\|f_1\|^2=\|f_2\|^2=2$. 
\end{itemize}
\end{lemma}

\begin{proof} Note that we have $\mu_1=\mu_1(0)f_1$ and $\mu_2=\mu_2(0)f_2$. 
The statements follow from 
$$\mu_1(0)^2\|f_1\|^2+\mu_2(0)^2\|f_2\|^2=\frac{1}{3},$$
for the cases of Lemma \ref{CaseI3} and Lemma \ref{CaseI4}, and 
$$\mu_1(0)^2\|f_1\|^2+\mu_2(0)^2\|f_2\|^2=\frac{1}{3}(1-\frac{2}{d}),$$
for the case of Lemma \ref{CaseII3}. 
\end{proof}

\begin{proof}[Proof of Theorem \ref{Z41}] Thanks to the previous lemma, the only possibilities are 
Case I with $(\omega_1,\omega_2)=(1,\zeta_3^{\pm 1})$.  
We assume $(\omega_1,\omega_2)=(1,\zeta_3)$ as the other case can be treated in the same way. 
Lemma \ref{CaseI1},(4) with $g=2$ implies 
$$|\xi_1(2)|^2+2|\eta_2(2)|^2=|\xi_2(2)|^2+2|\eta_1(2)|^2.$$
Since $\dim\ker(\cR-\zeta_3)=1$, we have $\xi_2=\xi_2(0)f_1=\xi_1(0)f_1$ and $\eta_2=\eta_2(0)f_1=-\eta_1(0)f_1$.  
Since $f_0'(2)=0$, we have $\xi_1(2)=\xi_1(0)f_0(2)$ and $\eta_1(2)=\eta_1(0)f_0(2)$. 
Thus 
$$(\xi_0(0)^2-2\eta_1(0)^2)(|f_0(2)|-|f_1(2)|)=0,$$
and $\xi_0(0)=\pm \sqrt{2}\eta_1(0)$. 
This shows that the case (2) of Lemma \ref{CaseI5} never occurs. 
Since $\xi_2=\xi_1(0)f_1$ and $\eta_2=-\eta_1(0)f_1$, we have $\xi_2=\mp\sqrt{2}\eta_2$, 
and Eq.(\ref{I52}) implies 
$$\mu_2(0)^2\|f_1\|^2=\|\mu_2\|^2=\frac{1}{3d}=\frac{\sqrt{5}-2}{6}.$$
Now the equations in (1) of Lemma \ref{CaseI5} imply  
$$-\frac{\sqrt{5}-2}{4}-\frac{\kappa_1\sqrt{1-16\eta_1(0)^2}}{4}=\pm \sqrt{2}\eta_1(0),$$
$$(-\frac{\sqrt{5}-2}{12}+\frac{\kappa_1\sqrt{1-16\eta_1(0)^2}}{4})^2\|f_1\|^2=\frac{\sqrt{5}-2}{6}.$$
Direct computation shows that for $\|f_1\|^2$ as above, these is no $\eta_1(0)$ satisfying these two equations. 
\end{proof}

\subsection{The case of $G=\Z_2\times \Z_2$}

\begin{theorem} There is no C$^*$ near-group category for $\Z_2\times \Z_2$ with $m=8$. 
\end{theorem}

We use the notation in Example \ref{Z2Z2}. 
Up to complex conjugate, there are exactly two $a(g)$ for each of $\inpr{\cdot}{\cdot}_1$ and $\inpr{\cdot}{\cdot}_2$, 
and we show the statement for these 4 cases separately. 

\begin{lemma} There is no C$^*$ near-group category for $\inpr{\cdot}{\cdot}_1$ and $a(g_0)=a(g_1)=\ii$, $a(g_2)=-1$. 
\end{lemma}

\begin{proof}
In this case we can choose $c=\ii$, and 
$$\cR=\frac{-\ii}{2}\left(
\begin{array}{cccc}
1 &0 &0 &0  \\
0 &-\ii &0 &0  \\
0 &0 &-\ii &0  \\
0 &0 &0 &-1 
\end{array}
\right)
\left(
\begin{array}{cccc}
1 &1 &1 &1  \\
1 &-1&1 &-1  \\
1 &1 &-1&-1  \\
1 &-1 &-1 &1 
\end{array}
\right)
=\frac{1}{2}
\left(
\begin{array}{cccc}
-\ii &-\ii &-\ii &-\ii  \\
-1 &1 &-1 &1  \\
-1 &-1 &1 &1  \\
\ii &-\ii &-\ii &\ii 
\end{array}
\right),
$$
$$J\left(
\begin{array}{c}
f(0)  \\
f(g_0)  \\
f(g_1)  \\
f(g_2) 
\end{array}
\right)
=\left(
\begin{array}{c}
\overline{f(0)}  \\
-\ii\overline{f(g_0)}  \\
-\ii\overline{f(g_1)} \\
-\overline{f(g_2)} 
\end{array}
\right).
$$
Let 
$$f_0=\left(
\begin{array}{c}
1  \\
-\frac{\zeta_8^{-1}}{\sqrt{2}}  \\
-\frac{\zeta_8^{-1}}{\sqrt{2}}  \\
\ii 
\end{array}
\right),\quad
f_0'=
\left(
\begin{array}{c}
0  \\
\zeta_8^{-1}\ii  \\
-\zeta_8^{-1}\ii  \\
0 
\end{array}
\right),\quad 
f_1=
\left(
\begin{array}{c}
1  \\
\frac{\zeta_8^{-1}}{2\cos\frac{7\pi}{12}}  \\
\frac{\zeta_8^{-1}}{2\cos\frac{7\pi}{12}}  \\
\tan\frac{7\pi}{12}\ii
\end{array}
\right),\quad 
f_2=
\left(
\begin{array}{c}
1  \\
\frac{\zeta_8^{-1}}{2\cos\frac{\pi}{12}}  \\
\frac{\zeta_8^{-1}}{2\cos\frac{\pi}{12}}  \\
-\tan\frac{\pi}{12}\ii
\end{array}
\right).
$$
Then $\cR f_0=f_0$, $\cR f_0'=f_0'$, $\cR f_1=\zeta_3f_1$, $\cR f_2=\zeta_3^2f_2$, and 
they are invariant under $\cJ$. 
They form a (non-normalized) orthogonal basis of $\ell^2(\Z_2\times \Z_2)$ as well as 
the real subspace of $\cJ$-invariant vectors. 
We have 
$$\|f_1\|^2=\frac{3}{2\cos^2\frac{7\pi}{12}}=6(2+\sqrt{3}),$$
$$\|f_2\|^2=\frac{3}{2\cos^2\frac{\pi}{12}}=6(2-\sqrt{3}).$$

Since $f_0'(g_2)=0$, we can show the statement in exactly the same way as in the case of $\Z_4$. 
\end{proof}

\begin{lemma} There is no C$^*$ near-group category for $\inpr{\cdot}{\cdot}_1$ and $a(g_0)=-a(g_1)=\ii$, $a(g_2)=1$. 
\end{lemma}

\begin{proof}
In this case we can choose $c=1$, and 
$$\cR=\frac{1}{2}\left(
\begin{array}{cccc}
1 &0 &0 &0  \\
0 &-\ii &0 &0  \\
0 &0 &\ii &0  \\
0 &0 &0 &1 
\end{array}
\right)
\left(
\begin{array}{cccc}
1 &1 &1 &1  \\
1 &-1&1 &-1  \\
1 &1 &-1&-1  \\
1 &-1 &-1 &1 
\end{array}
\right)
=\frac{1}{2}
\left(
\begin{array}{cccc}
1 &1 &1 &1  \\
-\ii &\ii &-\ii &\ii  \\
\ii &\ii &-\ii &-\ii  \\
1 &-1 &-1 &1 
\end{array}
\right),
$$
$$J\left(
\begin{array}{c}
f(0)  \\
f(g_0)  \\
f(g_1)  \\
f(g_2) 
\end{array}
\right)
=\left(
\begin{array}{c}
\overline{f(0)}  \\
-\ii\overline{f(g_0)}  \\
\ii\overline{f(g_1)} \\
\overline{f(g_2)} 
\end{array}
\right).
$$
Let 
$$f_0=\left(
\begin{array}{c}
1  \\ 
\frac{\zeta_8^{-1}}{2\sqrt{2}}  \\
\frac{\zeta_8}{2\sqrt{2}}  \\
\frac{1}{2}
\end{array}
\right),\;
f_0'=
\left(
\begin{array}{c}
0  \\
\frac{\zeta_8^{-1}}{\sqrt{2}}  \\
\frac{\zeta_8}{\sqrt{2}}  \\
-1
\end{array}
\right),\;
f_1=
\left(
\begin{array}{c}
1  \\
-2\cos\frac{\pi}{12}\zeta_8^{-1}  \\
2\sin\frac{\pi}{12}\zeta_8  \\
-1
\end{array}
\right),\; 
f_2=
\left(
\begin{array}{c}
1  \\
2\sin\frac{\pi}{12}\zeta_8^{-1}  \\  
-2\cos\frac{\pi}{12}\zeta_8 \\
-1
\end{array}
\right).
$$
Then $\cR f_0=f_0$, $\cR f_0'=f_0'$, $\cR f_1=\zeta_3f_1$, $\cR f_2=\zeta_3^2f_2$, and 
they are invariant under $\cJ$. 
They form a (non-normalized) orthogonal basis of $\ell^2(\Z_2\times \Z_2)$ as well as 
the real subspace of $\cJ$-invariant vectors. 
We have $\|f_1\|^2=\|f_2\|^2=6$.

Thanks to Lemma \ref{Z42}, the only remaining case is Case I with $(\omega_1,\omega_2)=(1,\zeta_3^{\pm1})$. 
We assume $(\omega_1,\omega_2)=(1,\zeta_3)$ because the other case can be treated in the same way. 
Eq.(\ref{I52}) implies 
$$(\xi_1(0)^2-2\eta_1(0)^2)\|f_1\|^2-3\mu_2(0)^2\|f_2\|^2=-\frac{1}{d},$$
and so 
$$3\mu_2(0)^2-\xi_1(0)^2+2\eta_1(0)^2=\frac{\sqrt{5}-2}{12}.$$
This shows that case (2) in Lemma \ref{CaseI5} never occurs because the left-hand side would be negative 
in that case. 
We assume that the case (1) in Lemma \ref{CaseI5} holds. 
Then we would get\begin{align*}
\lefteqn{\frac{\sqrt{5}-2}{12}=3\mu_2(0)^2-\xi_1(0)^2+2\eta_1(0)^2} \\
 &=3(-\frac{\sqrt{5}-2}{12}+\frac{\kappa_1\sqrt{1-16\eta_1(0)^2}}{4})^2-
 (\frac{\sqrt{5}-2}{4}+\frac{\kappa_1\sqrt{1-16\eta_1(0)^2}}{4})^2+2\eta_1(0)^2 \\
 &=\frac{-3+2\sqrt{5}}{12}-\frac{(\sqrt{5}-2)\kappa_1\sqrt{1-16\eta_1(0)^2}}{4},
\end{align*}
and 
$$\kappa_1\sqrt{1-16\eta_1(0)^2}=\frac{3+\sqrt{5}}{3}>1.$$
This is contradiction. 
\end{proof}

Up to group automorphism there are exactly two $a(g)$ for $\inpr{\cdot}{\cdot}_2$, the one given by 
$a(g_0)=a(g_1)=1,a(g_2)=-1$, and the one given by $a(g_0)=a(g_1)=a(g_2)=-1$. 

\begin{lemma} There is no C$^*$ near-group category for $\inpr{\cdot}{\cdot}_2$ and $a(g_0)=a(g_1)=1$, $a(a_2)=-1$. 
\end{lemma}

\begin{proof}
In this case, we can choose $c=1$, and 
$$\cR=\frac{1}{2}\left(
\begin{array}{cccc}
1 &0 &0 &0  \\
0 &1 &0 &0  \\
0 &0 &1 &0  \\
0 &0 &0 &-1 
\end{array}
\right)
\left(
\begin{array}{cccc}
1 &1 &1 &1  \\
1 &1 &-1 &-1  \\
1 &-1 &1 &-1  \\
1 &-1 &-1 &1 
\end{array}
\right)
=\frac{1}{2}\left(
\begin{array}{cccc}
1 &1 &1 &1  \\
1 &1 &-1 &-1  \\
1 &-1 &1 &-1  \\
-1 &1 &1 & -1
\end{array}
\right),
$$
$$\cJ \left(
\begin{array}{c}
f(0)  \\
f(g_0)  \\
f(g_1)  \\
f(g_2) 
\end{array}
\right)
=\left(
\begin{array}{c}
\overline{f(0)}  \\
\overline{f(g_0)}  \\
\overline{f(g_1)}  \\
-\overline{f(g_2)} 
\end{array}
\right). 
$$
Let 
$$f_0=\left(
\begin{array}{c}
1  \\ 
\frac{1}{2}  \\  
\frac{1}{2}  \\
0
\end{array}
\right),\;
f_0'=
\left(
\begin{array}{c}
0  \\
1  \\
-1 \\
0
\end{array}
\right),\;
f_1=
\left(
\begin{array}{c}
1  \\
-1  \\
-1  \\
\sqrt{3}\ii
\end{array}
\right),\; 
f_2=
\left(
\begin{array}{c}
1  \\
-1  \\  
-1\\
-\sqrt{3}\ii
\end{array}
\right).
$$
Then $\cR f_0=f_0$, $\cR f_0'=f_0'$, $\cR f_1=\zeta_3f_1$, $\cR f_2=\zeta_3^2f_2$, and 
they are invariant under $\cJ$. 
They form a (non-normalized) orthogonal basis of $\ell^2(\Z_2\times \Z_2)$ as well as 
the real subspace of $\cJ$-invariant vectors. 
We have $\|f_1\|^2=\|f_2\|^2=6$. 

Since $f_0'(g_3)=0$, we can show the statement in exactly the same way as in the case of $\Z_4$. 
\end{proof}

\begin{lemma} There is no C$^*$ near-group category for $\inpr{\cdot}{\cdot}_2$ and $a(g_0)=a(g_1)=a(g_2)=-1$. 
\end{lemma}

\begin{proof}
In this case, we can choose $c=-1$, and 
$$\cR=\frac{-1}{2}\left(
\begin{array}{cccc}
1 &0 &0 &0  \\
0 &-1 &0 &0  \\
0 &0 &-1 &0  \\
0 &0 &0 &-1 
\end{array}
\right)
\left(
\begin{array}{cccc}
1 &1 &1 &1  \\
1 &1 &-1 &-1  \\
1 &-1 &1 &-1  \\
1 &-1 &-1 &1 
\end{array}
\right)
=\frac{1}{2}\left(
\begin{array}{cccc}
-1 &-1 &-1 &-1  \\
1 &1 &-1 &-1  \\
1 &-1 &1 &-1  \\
1 &-1 &-1 & 1
\end{array}
\right),
$$
$$\cJ \left(
\begin{array}{c}
f(0)  \\
f(g_0)  \\
f(g_1)  \\
f(g_2) 
\end{array}
\right)
=\left(
\begin{array}{c}
\overline{f(0)}  \\
-\overline{f(g_0)}  \\
-\overline{f(g_1)}  \\
-\overline{f(g_2)} 
\end{array}
\right). 
$$
Let 
$$f_0=\left(
\begin{array}{c}
0  \\ 
1  \\  
\zeta_3  \\
\zeta_3^2
\end{array}
\right),\;
f_0'=
\left(
\begin{array}{c}
0  \\
1  \\
\zeta_3^2 \\
\zeta_3
\end{array}
\right),\;
f_1=
\left(
\begin{array}{c}
1  \\
-\frac{\ii}{\sqrt{3}}  \\
-\frac{\ii}{\sqrt{3}}  \\
-\frac{\ii}{\sqrt{3}}
\end{array}
\right),\; 
f_2=
\left(
\begin{array}{c}
1  \\
\frac{\ii}{\sqrt{3}}  \\  
\frac{\ii}{\sqrt{3}}\\
\frac{\ii}{\sqrt{3}}
\end{array}
\right).
$$
Then $\cR f_0=f_0$, $\cR f_0'=f_0'$, $\cR f_1=\zeta_3f_1$, $\cR f_2=\zeta_3^2f_2$, and 
$\cJ f_0=-f_0'$, $\cJ f_0'=-f_0$, $\cJ f_1=f_1$, $\cJ f_2=f_2$. 
They form a (non-normalized) orthogonal basis of $\ell^2(\Z_2\times \Z_2)$ as well as 
the real subspace of $\cJ$-invariant vectors. 
We have $\|f_1\|^2=\|f_2\|^2=2$. 

Lemma \ref{Z42} and $f_0(0)=f_0'(0)=0$ imply that the only possible case is the case (2) of Lemma \ref{CaseII3}. 
Thus there exist $p\in \C$ and $\kappa\in \{1,-1\}$ such that $\mu_0=pf_0+\overline{p}f_0'$, and 
$$\mu=pf_0+\overline{p}f_0'+ (\frac{\kappa}{4}+\frac{\sqrt{5}-2}{4\sqrt{3}})\ii f_1+(\frac{\kappa}{4}-\frac{\sqrt{5}-2}{4\sqrt{3}})\ii f_2.$$
Let $p_i=p\zeta_3^i+\overline{p}\zeta_3^{-i}\in \R$. 
Then
$$\mu(0)=\frac{\kappa}{2}\ii,\quad \mu(g_i)=p_i+\frac{\sqrt{5}-2}{6},$$
$$\cR\mu(0)=-\frac{\sqrt{5}-2}{4}-\frac{\kappa}{4}\ii,\quad 
\cR\mu(g_i)=p_i-\frac{\sqrt{5}-2}{12}+\frac{\kappa}{4}\ii,$$
$$\cR^2\mu(0)=\frac{\sqrt{5}-2}{4}-\frac{\kappa}{4}\ii,\quad 
\cR^2\mu(g_i)=p_i-\frac{\sqrt{5}-2}{12}-\frac{\kappa}{4}\ii.$$

Now Lemma \ref{CaseI1},(4) with $g=g_i$ is equivalent to 
\begin{equation}\label{Z2Z21}
p_i^2-\frac{\sqrt{5}-2}{6}p_i+|\eta(g_i)|^2=\frac{\sqrt{5}}{36},
\end{equation}
\begin{equation}\label{Z2Z22}
|\xi(g_i)|^2+2|\eta(g_i)|^2+p_i^2+\frac{\sqrt{5}-2}{3}p_i=\frac{\sqrt{5}}{9},
\end{equation}
\begin{equation}\label{Z2Z23}
\xi(g_i)\overline{\eta(g_i)}=(p_i-\frac{\sqrt{5}-2}{3})\eta(g_i),
\end{equation}
\begin{equation}\label{Z2Z24}
\eta(g_i)^2+\xi(g_i)(p_i+\frac{\sqrt{5}-2}{6})=0. 
\end{equation}
From Eq.(\ref{Z2Z23}) and (\ref{Z2Z24}), we have 
\begin{align*}
\lefteqn{0=\eta(g_i)^2\overline{\eta(g_i)}+\xi(g_i)\overline{\eta(g_i)}(p_i+\frac{\sqrt{5}-2}{6})} \\
 &=\eta(g_i)^2\overline{\eta(g_i)}+\eta(g_i)(p_i-\frac{\sqrt{5}-2}{3})(p_i+\frac{\sqrt{5}-2}{6}) \\
 &=\eta(g_i)(|\eta(g_i)|^2+p_i^2-\frac{\sqrt{5}-2}{6}p_i-\frac{9-4\sqrt{5}}{18}). 
\end{align*}
and $\eta(g_i)=0$. 
Eq.(\ref{Z2Z24}) implies either $\xi(g_i)=0$ or $p_i=-(\sqrt{5}-2)/6$, 
which in either case is not compatible with Eq.(\ref{Z2Z21}), (\ref{Z2Z22}). 
\end{proof}
\section{$2^G_l$1 subfactors}
Let $n$ and $l$ be natural numbers. 
A bipartite graph is said to be $2^n_l1$ if the following conditions hold: 
\begin{itemize}
\item[(1)] the set of even vertices is $\{v_i\}_{i=0}^{n-1}\cup \{v_\rho\}$,
\item[(2)] the set of odd vertices is $\{w\}_{i=0}^{n-1}\cup\{w_\pi\}$,  
\item[(3)] the only non-zero entries of the incidence matrix $\Gamma$ are 
$$\Gamma_{v_i,w_i}=\Gamma_{w_i,v_i}=1,$$
$$\Gamma_{w_i,v_\rho}=\Gamma_{v_\rho,w_i}=l,$$
$$\Gamma_{v_\rho,w_\pi}=\Gamma_{w_\pi,v}=1.$$
\end{itemize}

Let $N\subset M$ be a subfactor whose principal graph is $2^n_l1$. 
More precisely, by the principal graph, we mean the induction reduction graph for $M-M$ and $M-N$ sectors. 
Then the vertices $\{v_i\}_{i=0}^{n-1}$ correspond to automorphisms, forming a group $G$ in $\Out(M)$, and 
we denote them by $\{\alpha_g\}_{g\in G}$. 
We call such a subfactor $2^G_l1$.  
When $l=1$, we just call it a $2^G1$ subfactor. 
For example, the $E_6$ subfactor is $2^{\Z_2}1$. 
Let $\rho$ be the endomorphism of $M$ corresponding to $v_\rho$. 
Then $\{\alpha_g\}_{g\in G}\cup\{\rho\}$ generate a C$^*$ near-group category for $G$ with $m=ln$. 
In particular $G$ is always abelian. 
In this section, we determine the structure of $2^G_l1$ subfactors in terms of the corresponding 
C$^*$ near group categories.

\begin{figure}[htbp]
\begin{center}
\hspace{10pt}
\begin{xy}\ar@{-}(0,0)*{\mathrm{id}};(10,0)*{\iota}="A",
\ar@{=}"A";(20,00)*{\rho}="B",
\ar@{=}"B";(30,00)*{\alpha^2\iota}="C",
\ar@{-}"C";(40,00)*{\alpha^2},
\ar@{-}"B";(20,10)*{\pi},
\ar@{=}"B";(20,-10)*{\alpha\iota}="D",
\ar@{-}"D";(20,-20)*{\alpha}
\end{xy}
\end{center}
\caption{$2^{\Z_3}_21$ subfactor}
\end{figure}

Let $N\subset M$ be a $2^G_l1$ subfactor and let $\{\alpha_g\}_{g\in G}$ and $\rho$ be as above. 
We may and do assume $\alpha_g\circ \rho=\rho$, and we use the same notation as in the previous sections 
for the restriction of $\alpha_g$ and $\rho$ to the intertwiner spaces $(\alpha_g,\rho^2)$, 
and $(\rho,\rho^2)$. 
Let $\iota :N\hookrightarrow M$ be the inclusion map. 
Then odd vertices $\{w_i\}_{i=0}^{n-1}$ correspond to $\{\alpha_g\circ\iota\}_{g\in G}$. 
Let $\pi:N\rightarrow M$ be the homomorphism corresponding to $w_\pi$. 
Direct computation shows $d(\iota)=d/\sqrt{n}$, $d(\pi)=\sqrt{n}$, and $[\rho]=[\pi\biota]$. 
Since $\rho$ is self-conjugate, we have $[\rho]=[\iota\bpi]$ too. 
Since $[\alpha_g\circ\pi]=[\pi]$ and $d(\pi)=\sqrt{n}$, we have 
$$[\pi\bpi]=\bigoplus_{g\in G}[\alpha_g].$$

From $[\rho]=[\iota\bpi]$, we may assume $\rho=\iota\bpi$ by replacing $\rho$ with $\Ad v\circ \rho$ 
and $\alpha_g$ by $\Ad v\circ \alpha_g\circ \Ad v^*$ for a unitary $v\in M$ if necessary.  
This means that we have $M\supset N\supset \rho(M)$. 

\begin{lemma} Let the notation be as above. 
\begin{itemize}
\item[(1)] $N=\rho(M)\vee\{U(g)\}_{g\in G}''$. 
\item[(2)] The dual inclusion of $M\supset N$ is isomorphic to $M^\alpha\supset \rho(M)$, 
where $M^\alpha$ is the fixed point subalgebra 
$$M^\alpha=\{x\in M;\;\alpha_g(x)=x,\;\forall g\in G\}.$$ 
\end{itemize}
\end{lemma}

\begin{proof} (1) Since $[\bpi\circ \alpha_g]=[\bpi]$, there exist unitaries $u_g\in N$ satisfying 
$\bpi\circ \alpha_g=\Ad u_g\circ \bpi$. 
On the other hand, we have 
$$\Ad U(g)\circ \rho=\rho\circ\alpha_g=\iota\circ \bpi\circ \alpha_g=\iota\circ \Ad u_g\circ \bpi,$$
and $u_g$ is a multiple of $U(g)$. 
This means that $U(g)\in N$, and we have the inclusion relation 
$N\supset \rho(M)\vee \{U(g)\}_{g\in G}$. 
Let $\rho_0$ be $\rho$ regarded as an isomorphism from $M$ onto $\rho(M)$, and let $\beta_g=\rho_0\circ \alpha_g\circ \rho_0^{-1}$, 
which is an outer action of $G$ on $\rho(M)$. 
Since $\rho(M)\vee \{U(g)\}_{g\in G}$ is identified with the crossed product $\rho(M)\rtimes_\beta G$, 
its index in $M$ is $d^2/n$, which coincides with $[M:N]$. 
Thus we have  the equality $N=\rho(M)\vee \{U(g)\}_{g\in G}$. 

(2) Since $[\rho]=[\pi\biota]$, we may assume $\rho=\pi\circ \biota$ by replacing $\pi$ with an 
equivalent sector if necessary. 
Since $[\alpha_g\circ \pi]=[\pi]$, there exist unitaries $v_g\in M$ such $\alpha_g\circ \pi=\Ad v_g\circ \pi$. 
On the other hand, we have $\alpha_g\circ \rho=\rho$, and so $v_g$ is a scalar and $\alpha_g\circ \pi=\pi$. 
Since $d(\pi)=\sqrt{n}$, this implies that the image of $\pi$ coincides with $M^\alpha$. 
Thus the dual inclusion $N\supset \biota(N)$ is isomorphic to $M^\alpha\supset \rho(M)$. 
\end{proof}

\begin{theorem}\label{classify2G1} For an abelian group $G$, there is a one-to-one correspondence between 
the isomorphism classes of $2^G_l1$ subfactors and the set of equivalence classes of 
a C$^*$ near-group category for $G$ with $m=ln$. 
\end{theorem}

\begin{proof} It is obvious that isomorphic $2^G_l1$ subfactors give rise to equivalent C$^*$ near-group categories. 
Moreover the previous lemma shows that if the two resulting C$^*$ near-group categories are equivalent, the two 
$2^G_l1$ subfactors are isomorphic. 

It remains to show that every C$^*$ near-group category for $G$ with $m=ln$ gives rise to a $2^G_l1$ subfactor. 
We may assume that such a category is realized by $\{\alpha_g\}_{g\in G}$ and $\rho$ acting on a type III factor $M$, 
and we use the same notation as before. 
We set $N=\rho(M)\vee\{U(g)\}_{g\in G}''$. 
Then it is identified with the crossed product $\rho(M)\rtimes G$, and its index in $M$ is $d^2/n=1+ld$. 
Since $\rho(M)$ is irreducible in $M$, so is $N$ too. 
We denote by $\iota$ the inclusion map $\iota: N\hookrightarrow M$. 
All we have to show is $\dim (\rho,\iota\biota)=l$, or equivalently $\dim (\iota,\rho\iota)=l$, 
which will show 
$$[\iota\biota]=[\id]\oplus l[\rho].$$
Since $(\iota,\rho\iota)\subset (\rho,\rho^2)$, we have 
$$(\iota,\rho\iota)=\{T\in \cK;\;\forall g\in G,\; TU(g)=\rho(U(g))T\},$$
and in view of Eq.(\ref{rhoU}), we have 
$$(\iota,\rho\iota)=\{T\in \cK;\;\forall g\in G,\; (j_1\circ j_2)U(g)(j_1\circ j_2)^*T=T\},$$
which shows $\dim (\iota,\rho\iota)=l$. 
\end{proof}

Grossman-Jordan-Snyder \cite[Section 4]{GJS15} discussed $\Z_2$-graded extensions of near-group categories $\cC$ in the case of prime $m=n$.  
They constructed a non-trivial $\Z_2$-graded extension of $\cC$ under the assumption that the outer automorphism group $\Out(\cC)$ is trivial and 
the corresponding $2^G1$ subfactor is self-dual. 
We will determine the structure of $\Out(\cC)$ in Section \ref{Equi}. 
The second condition turns out to always hold. 

\begin{theorem}\label{self-dual} Let $N\subset M$ be a $2^G_l1$ subfactor whose associated C$^*$ near-group category satisfies the following condition: 
the operator $A(g)$ in Lemma \ref{ACJ} is a scalar for any $g\in G$. 
Then $N\subset M$ is self-dual. 
In particular, every $2^G1$ subfactor is self-dual. 
\end{theorem}

\begin{proof} We use the same notation for a $2^G_l1$ subfactor $N\subset M$ as before. 
Let $\iota':\rho(M)\hookrightarrow M^\alpha$ and $\kappa:M^\alpha \hookrightarrow M$ be the inclusion maps. 
Since $\rho(M)\subset M^\alpha$, the restriction of $\rho$ to $M^\alpha$ makes sense as an endomorphism of 
$M^\alpha$, which will be denoted by $\rho'$. 
Since $\alpha_g(U(h))=\overline{\inpr{g}{h}}U(h)$, the restriction of $\Ad U(-g)$ to $M^\alpha$ makes sense too 
as a $G$-action on $M^\alpha$, which will be denoted by $\alpha'$. 
By definition, we have $\rho\circ\kappa=\kappa\circ \rho'$ and $\kappa\circ \alpha'_g=\Ad U(-g)\circ \kappa$. 
We also have $\alpha'_g\circ\rho'=\rho'$ as $\Ad U(-g)\circ \rho=\rho\circ \alpha_g$. 

We first claim that that $\rho'$ is irreducible. 
Indeed, we have 
$$(\rho',\rho')=(\kappa\circ \rho',\kappa\circ\rho')\cap M^\alpha=(\rho\circ\kappa,\rho\circ\kappa)\cap M^\alpha,$$
and we first determine $(\rho\circ\kappa,\rho\circ\kappa)$. 
Since $\Ad U(g)\circ \rho=\rho\circ \alpha_g$, we have $U(g)\in (\rho\circ\kappa,\rho\circ\kappa)$. 
On the other hand, 
$$\dim(\rho\circ\kappa,\rho\circ\kappa)=\dim (\kappa\overline{\kappa},\rho^2)=\dim(\bigoplus_{g\in G}\alpha_g,\rho^2)=n,$$
and so $(\rho\circ\kappa,\rho\circ\kappa)=\Span\{U(g)\}_{g\in G}$. 
This shows that $\rho'$ is irreducible. 

Next we observe 
$$[\iota'\overline{\iota'}]=[\id]\oplus l[\rho'],$$
or equivalently $\dim(\iota',\rho'\circ\iota')=l$. 
This follows from 
$$(\iota',\rho'\circ\iota')=(\rho,\rho^2)\cap M^\alpha=\{T\in \cK;\;\forall g\in G,\; V(g)T=T\}.$$

Now it suffices to show that the two categories generated by $\{\alpha_g\}_{g\in G}\cup \{\rho\}$ and by 
$\{\alpha_g'\}_{g\in G}\cup\{\rho'\}$ are mutually equivalent, and for this 
we show that the their Cuntz algebra models are the same. 

Let 
$$S'_0=\frac{1}{\sqrt{n}}\sum_{g\in G}S_g\in M^\alpha.$$
Then for any $x\in M^\alpha$, we have 
$$\rho'^2(x)S'_0=\frac{1}{\sqrt{n}}\sum_{g\in G}S_g\alpha_g(x)=\frac{1}{\sqrt{n}}\sum_{g\in G}S_gx=S'_0x,$$
and $S'_0\in (\id,\rho'^2)$. 
We set 
$$S'_g=\alpha'_g(S'_0)=\frac{1}{\sqrt{n}}\sum_{h\in G}U(-g)S_hU(-g)^*=\frac{1}{\sqrt{n}}\sum_{h\in G}\overline{\inpr{g}{h}}S_hU(g),$$
which is in $(\alpha'_g,\rho'^2)$. 
Note that we have 
$$\sum_{g\in G}S'_g\alpha'_g(x)S_g'^*=\frac{1}{n}\sum_{g,h,k\in G}\inpr{g}{k-h}S_hxS_k^*=\sum_{h\in G}S_hxS_h^*.$$
On the other hand, 
\begin{align*}
\lefteqn{\rho'^2(x)=\sum_g S_g\alpha_g(x)S_g^*+\sum_{g,r}{T_g(e_r)}\rho(x)T_g(e_r)^*,} \\
 &=\sum_gS_gxS_g^*+\sum_{g,r} T_g(e_r)U(g)\rho(x)U(g)^*T_g(e_r)^*\\
 &=\sum_gS'_g\alpha'_g(x){S'_g}^*+\sum_{g,r} T_g(e_r)U(g)\rho'(x)U(g)^*T_g(e_r)^*.
\end{align*}
Since $T_g(\xi)U(g)\in M^\alpha$, we see that $\{T_g(e_r)U(g)\}_{g,r}$ form an orthonormal basis of $(\rho',\rho'^2)$. 
Since $\rho(U(-g))\in M^\alpha$, and $\Ad \rho(U(-g))\circ \rho'=\rho'\circ \alpha'_g$, we set $U'(-g)=\rho(U(g))$. 
Then thanks to Eq.(\ref{rhoU}), we have 
$$U'(g)S'_h=\frac{1}{\sqrt{n}}\sum_k\overline{\inpr{h}{k}}\rho(U(-g))S_kU(h)
=\frac{1}{\sqrt{n}}\sum_k\overline{\inpr{h}{k}}S_{k+g}U(h)=\inpr{h}{g}S'_h.$$

By assumption, the operator $A(g)$ in Lemma \ref{ACJ} is a scalar $a(g)$. 
We choose a square root of $\hat{a}(0)\in \T$ and fix it, and set 
$$T'_0(\xi)=\frac{\hat{a}(0)^{1/2}}{\sqrt{n}}\sum_{h\in G}T_h(\xi)U(h)\in (\rho',\rho'^2)\cap (M^\alpha)^{\alpha'},$$
\begin{align*}
\lefteqn{T'_g(\xi)=U'(-g)T'_{0}(\xi)} \\
 &=\frac{\hat{a}(0)^{1/2}}{\sqrt{n}}\sum_{h\in G}\rho(U(g))T_h(\xi)U(h) \\
 &=\frac{\hat{a}(0)^{1/2}}{\sqrt{n}}\sum_{h\in G}(j_1\circ j_2)^*U(g)(j_1\circ j_2)T_h(\xi)U(h+g) \\
 &=\frac{\hat{a}(0)^{1/2}}{\sqrt{n}}\sum_{h\in G}T_{h+g}(A(h+g)A(h)^*\xi)U(h+g)\\
 &=\frac{\hat{a}(0)^{1/2}\overline{a(g)}}{\sqrt{n}}\sum_{h\in G}\overline{\inpr{g}{h}}T_h(\xi)U(h).
\end{align*}
Then $\{T'_g(e_r)\}_{g,r}$ is an orthonormal basis of $(\rho',\rho'^2)$, and we have $\alpha'_g(T'_h(e_r))=\inpr{g}{h}T'_g(e_r)$ 
and $U'(g)T'_h(e_r)=T'_{h-g}(e_r)$. 

For $T'\in (\rho',\rho'^2)$, we defined $j'_1(T')$ and $j'_2(T')$ as in the definition of $j_1$ and $j_2$ 
by replacing $\rho$ and $S_0$ with $\rho'$ and $S'_0$. 
Then \begin{align*}
\lefteqn{j'_1(T'_g(\xi))=\sqrt{d}T'_g(\xi)^*\rho'(S_0)} \\
 &=\frac{\hat{a}(0)^{-1/2}a(g)\sqrt{d}}{n}\sum_{h,k\in G}\inpr{g}{h}U(h)^*T_h(\xi)^*\rho(S_k)\\
 &=\frac{\hat{a}(0)^{-1/2}a(g)\sqrt{d}}{n}\sum_{h,k\in G}\inpr{g}{h}U(h)^*T_h(\xi)^*U(k)\rho(S_0)U(k)^*\\
 &=\frac{\hat{a}(0)^{-1/2}a(g)}{n}\sum_{h,k\in G}\inpr{g}{h}U(h)^*j_1(T_{h+k}(\xi))U(k)^*\\
 &=\frac{\hat{a}(0)^{-1/2}a(g)}{n}\sum_{h,k\in G}a(k+h)\inpr{g}{h}T_{-k}(J\xi)U(k)^*\\
 \end{align*}
Since
$$\frac{1}{\sqrt{n}}\sum_{h}\inpr{g}{h}a(k+h)=\overline{\inpr{g}{k}}\hat{a}(-g)=\overline{\inpr{g}{k}}\hat{a}(0)\overline{a(-g)},$$
we get  
$$j'_1(T'_g(\xi))=\frac{\hat{a}(0)^{1/2}}{\sqrt{n}}\sum_{k\in G}\overline{\inpr{g}{k}}
T_{-k}(J\xi)U(k)^*=T'_{-g}(A(-g)J\xi),$$
which shows that $j_1'$ has the same form as $j_1$. 

For $j_2'$, we have \begin{align*}
\lefteqn{j_2'(T'_g(\xi))=\sqrt{d}\rho(T'_g(\xi)S'_0)} \\
&=\frac{\hat{a}(0)^{-1/2}a(g)\sqrt{d}}{n}\sum_{h,k}\inpr{g}{h}\rho(U(h)^*T_h(\xi)^*)S_k\\
 &=\frac{\hat{a}(0)^{-1/2}a(g)}{n}\sum_{h,k}\inpr{g}{h}\rho(U(h)^*)\alpha_k(j_2(T_h(\xi))) \\
 &=\frac{\epsilon\hat{a}(0)^{-1/2}a(g)}{n\sqrt{n}}\sum_{h,k,p}\inpr{g-p}{h}\inpr{k}{p}\rho(U(h)^*)T_p(C^*J\xi)\\
 &=\frac{\epsilon\hat{a}(0)^{-1/2}a(g)}{\sqrt{n}}\sum_h\inpr{g}{h}\rho(U(h)^*)T_0(C^*J\xi)\\
 &=\frac{\epsilon\hat{a}(0)^{-1/2}a(g)}{\sqrt{n}}\sum_h\inpr{g}{h}a(h)T_{-h}(C^*J\xi)U(h)^*\\
 &=\frac{\epsilon\hat{a}(0)^{-1/2}}{\sqrt{n}}\sum_ha(g-h)T_{-h}(C^*J\xi)U(h)^*.
\end{align*}
On the other hand, 
\begin{align*}
\lefteqn{\frac{\epsilon}{\sqrt{n}}\sum_{k}\overline{\inpr{g}{k}}T'_k(C^*J\xi)} \\
 &=\frac{\epsilon\hat{a}(0)^{1/2}}{n}\sum_{h,k}\overline{\inpr{g+h}{k}}\overline{a(k)}T_h(C^*J\xi)U(h) \\
 &= \frac{\epsilon\hat{a}(0)^{1/2}}{\sqrt{n}}\sum_{h}\overline{\hat{a}(-g-h)}T_h(C^*J\xi)U(h)\\
 &= \frac{\epsilon\hat{a}(0)^{-1/2}}{\sqrt{n}}\sum_{h}a(g-h)T_h(C^*J\xi)U(h).
\end{align*}
Thus $j'_2$ has the same form as $j_2$. 

Now it suffices to show that $B_h(\xi)$ for $\rho'$ takes the same form as the original one. 
Note that we have 
$$T_g^*(e_r)\rho(T_0(e_u))T_0(e_t)=a(g)\sum_{s}b^{r,s}_{t,u}(g)T_{-g}(e_s).$$
On the other hand, 
\begin{align*}
\lefteqn{T'_g(e_r)^*\rho(T'_0(e_u))T'_0(e_t)
=\frac{\hat{a}(0)^{1/2}}{\sqrt{n}}\sum_hT'_g(e_r)^*\rho(T_h(e_u)U(h))T'_0(e_t)} \\
 &=\frac{\hat{a}(0)^{1/2}}{\sqrt{n}}\sum_hT'_g(e_r)^*U'(h)\rho(T_0(e_u))U'(-h)T'_0(e_t)\\
 &=\frac{\hat{a}(0)^{1/2}}{\sqrt{n}}\sum_hT'_{g+h}(e_r)^*\rho(T_0(e_u))T'_h(e_t)\\
  &=\frac{\hat{a}(0)^{1/2}}{n\sqrt{n}}\sum_{h,p,q}a(g+h)\overline{a(h)}\inpr{q}{g+h}\overline{\inpr{p}{h}}U(q)^*T_q(e_r)^*
  \rho(T_0(e_u))T_p(e_t)U(p)\\
  &=\frac{\hat{a}(0)^{1/2}a(g)}{n\sqrt{n}}
  \sum_{h,p,q}\inpr{q-p-g}{h}\inpr{q}{g} U(q)^*T_{q-p}(e_r)^*\rho(T_0(e_u))T_0(e_t)U(p),
\end{align*}
where we used \begin{align*}
\lefteqn{T_q(\zeta)^*\rho(T_0(\xi))T_p(\eta)=T_q(\zeta)^*\rho(T_0(\xi))U(-p)T_0(\eta)} \\
 &=T_q(\zeta)^*U(-p)\rho(\alpha_{p}(T_0(\xi)))T_0(\eta)=T_{q-p}(\zeta)^*\rho(T_0(\xi))T_0(\eta).
\end{align*}
Thus we get 
\begin{align*}
\lefteqn{T'_g(e_r)^*\rho(T'_0(e_u))T'_0(e_t)} \\
  &=\frac{\hat{a}(0)^{1/2}a(g)}{\sqrt{n}}
  \sum_{p}\inpr{p+g}{g} U(p+g)^*T_g(e_r)^*\rho(T_0(e_u))T_0(e_t)U(p)\\
  &=\frac{\hat{a}(0)^{1/2}a(g)^2}{\sqrt{n}}
  \sum_{p,s}b^{r,s}_{t,u}(g)\inpr{p+g}{g} U(p+g)^*T_{-g}(e_s)U(p)\\
  &=\frac{\hat{a}(0)^{1/2}}{\sqrt{n}}
  \sum_{p,s}b^{r,s}_{t,u}(g)\inpr{p}{g} T_p(e_s)U(p)\\
  &=a(g)\sum_sb^{r,s}_{t,u}(g)T'_{-g}(e_s),
\end{align*}
which finishes the proof. 
\end{proof}

\section{Orbifold construction I (de-equivariantization)}\label{De-equi}
Orbifold construction for subfactors was first introduced in \cite{EK94} 
to construct new subfactors from given ones with group actions (see \cite{K94} too). 
Purely categorical versions (see \cite{B00}, \cite{EGNO15}, \cite{Y02} for example) of it are now called equivariantization and 
de-equivariantization, which are dual operations to each other via Takesaki duality. 
In the final two sections, we systematically investigate these operations for $2^G_l1$ subfactors and 
near-group categories in the irrational case. 

Victor Ostrik is the first to observe that a near-group category for 
$\Z_3\times \Z_3$ with $m=9$ produces the Haagerup category via de-equivariantization (see Example \ref{Z3Z3} below). 
In this section, we systematically pursue this kind of phenomena. 
We concentrate on the case with $m=|G|$, though our argument makes sense for the case of 
$m>|G|$ under a mild assumption. 
A C$^*$ near group category in this class is completely described by the data 
$(\inpr{\cdot}{\cdot},a,b,c)$ as in Section \ref{m=n}.  
The pair $(\inpr{\cdot}{\cdot},\overline{a})$ is often called a quadratic form on $G$ in the literature, 
where $\overline{a}$ is the complex conjugate of $a$. 

Let $N\subset M$ be a $2^G1$ subfactor whose even part is a C$^*$ near-group category given by $\alpha$ and $\rho$ with 
$\alpha_g\circ \rho=\rho$. 
We assume that it has an invariant $(\inpr{\cdot}{\cdot},a,b,c)$, and we use the same notation as before for the Cuntz algebra model.  
Now Eq.(\ref{rhoU}) takes the form 
$$\rho(U(g))=\sum_{h\in G}S_{h-g}S_{h}^*+\sum_{h\in G}a(h)\overline{a(h-g)}T_{h}U(g)T_{h-g}^*.$$

\subsection{Untwisted case} 
We fix a subgroup $H\subset G$, and consider the crossed product $M\rtimes_\alpha H$, which 
is a factor generated by $M$ and a unitary representation $\{\lambda_h\}$ of $H$ with the relation 
$\lambda_hx=\alpha_h(x)\lambda_h$ for $x\in M$. 
Then we can extend $\rho$ and $\alpha$ to $M\rtimes H$ by setting $\trho(\lambda_h)=a(h)U(h)\lambda_h$, 
and $\talpha_g(\lambda_h)=\inpr{g}{h}\lambda_h$. 
Note that we have $\Ad (a(h)U(h)\lambda_h)(\rho(x))=\rho(\alpha_h(x))$ for $x\in M$, and 
\begin{align*}
\lefteqn{(a(h)U(h)\lambda_h)(a(k)U(k)\lambda_k)=a(h)a(k)U(h)\alpha_h(U(k))\lambda_{h+k}} \\
 &=a(h)a(k)\overline{\inpr{h}{k}}U(h+k)\lambda_{h+k}=a(h+k)U(h+k)\lambda_{h+k}.
\end{align*}

\begin{lemma}\label{deL1} With the notation as above, we have the following for all $x\in M\rtimes_\alpha H$. 
$$\trho^2(x)=\sum_{g\in G}S_g\talpha_g(x)S_g^*+\sum_{g\in G}T_g\trho(x)T_g^*,$$
$$\trho\circ\talpha_g=\Ad U(g)\circ\trho.$$
\end{lemma}

\begin{proof} It suffices to show the statement for $x=\lambda_h$. 
For $x=\lambda_h$ the left-hand side is $a(h)^2\rho(U(h))U(h)\lambda_h$, and 
\begin{align*}
\lefteqn{S_g^*\trho^2(\lambda_h)S_g=a(h)^2S_g^*\rho(U(h))U(h)S_{g+h}\lambda_h} \\
 &=a(h)^2\inpr{h}{g+h}S_g^*\rho(U(h))S_{g+h}\lambda_h=\inpr{h}{g}\lambda_g \\
 &=\talpha_g(\lambda_h),
\end{align*}
\begin{align*}
\lefteqn{T_g^*\trho^2(\lambda_h)T_g=a(h)^2T_g^*\rho(U(h))U(h)\lambda_hT_g} \\
 &=a(h)^2\inpr{h}{g}T_g^*\rho(U(h))U(h)T_g \lambda_h=a(h)^2\inpr{h}{g}T_g^*\rho(U(h))T_{g-h} \lambda_h\\
 &=a(h)^2a(g)\overline{a(g-h)}\inpr{h}{g}U(h) \lambda_h=a(h)U(h)\lambda_h\\
 &=\trho(\lambda_h),
\end{align*}
which show the first statement. 

For the second statement, we have 
$$\trho\talpha_g(\lambda_h)=\inpr{g}{h}a(h)U(h)\lambda_h,$$
and on the other hand, 
\begin{align*}
\lefteqn{\Ad U(g)\rho(\lambda_h)=a(h)U(g)U(h)\lambda_hU(g)^*} \\
 &=a(h)U(g+h)\alpha_h(U(g)^*)\lambda_h=a(h)\inpr{h}{g}U(h)\lambda_h.
\end{align*}
\end{proof}

For a subgroup $H\subset G$, we set $H^\perp=\{g\in G;\; \inpr{g}{h}=1\}$. 
Since $\Ad \lambda_h^*\circ\talpha_h$ is trivial on $M$, it comes from the dual action of $\alpha$. 
If moreover $H\subset H^\perp$, it is trivial, that is $\talpha_h=\Ad \lambda_h$.

\begin{lemma}\label{deL2} If $H\subset H^\perp$, then $(\trho,\trho)=\{\lambda_h\}_{h\in H}''$. 
\end{lemma}

\begin{proof} Note that we have $(\trho,\trho)\subset M\rtimes_\alpha H\cap \rho(M)'$. 
Since $\alpha_h\circ\rho=\rho$, and $\rho$ is irreducible, we get $M\rtimes_\alpha H\cap \rho(M)'=\{\lambda_h\}_{h\in H}''$. 
Now we have 
$$(\trho,\trho)=\{\lambda_h\}_{h\in H}''\cap \{\trho(\lambda_k)\}_{k\in H}'=\{\lambda_h\}_{h\in H}''\cap \{a(k)U(k)\lambda_k\}_{k\in H}',$$ 
which is again $\{\lambda_h\}_{h\in H}''$ because $\lambda_h$ commutes with $U(k)$ thanks to
$$\alpha_h(U(k))=\overline{\inpr{h}{k}}U(k)=U(k).$$  
\end{proof} 

For the character group $\hH$ of $H$, we use the additive notation, that is, 
for $\chi_1,\chi_2\in \hH$, we denote $(\chi_1+\chi_2)(h)=\chi_1(h)\chi_2(h)$.  
For $\chi\in \hH$, we  denote by $e_\chi$ the corresponding minimal projection in $\{\lambda_h\}_{h\in H}''$, 
that is 
$$e_\chi=\frac{1}{\sqrt{|H|}}\sum_{h\in H}\overline{\chi(h)}\lambda_h.$$ 
Assume $H\subset H^\perp$. 
Then since 
$$(\trho,\trho)=\bigoplus_{\chi\in \hH}\C e_\chi,$$
the endomorphism $\trho$ is decomposed into irreducible components $\sigma_\chi$ parametrized by $\chi\in \hH$. 
More precisely, we choose an isometry $V_\chi\in M\rtimes_\alpha H$ for each $\chi\in \hH$ satisfying $V_\chi V_\chi^*=e_\chi$, 
and set $\sigma_\chi(x)=V_\chi^*\trho(x)V_\chi$. 
For simplicity, we denote $\sigma=\sigma_0$. 
Then we have 
$$[\trho]=\bigoplus_{\chi\in \hH}[\sigma_\chi].$$
Note that $\{\sigma_\chi\}_{\chi\in \hH}$ are all inequivalent. 

Recall that we may assume $N=\rho(M)\vee\{U(g)\}_{g\in G}''$, and 
$N$ is regarded as the crossed product of $\rho(M)$ by the $G$-action $\rho_0\circ \alpha_g\circ \rho_0^{-1}$, 
where $\rho_0$ is $\rho$ regarded as an isomorphism from $M$ onto $\rho(M)$. 
Since $\alpha_h(U(h))=\overline{\inpr{g}{h}}U(g)$, the restriction of $\alpha$ to $N$ is identified with the dual action 
for this crossed product, and in consequence $\alpha$ on $N$ is outer too. 
Thus $N\rtimes_\alpha H$ is a subfactor of $M\rtimes_\alpha H$. 

\begin{theorem}\label{deT1} Let $\kappa:N\rtimes_\alpha H\hookrightarrow M\rtimes_\alpha H$ be the inclusion map. 
Then 
$$[\kappa\overline{\kappa}]=[\id]\oplus \bigoplus_{\chi\in \hH}[\sigma_\chi].$$
\end{theorem}

\begin{proof} Since $[M\rtimes_\alpha H:N\rtimes_\alpha H]=[M:N]$, it suffices to show that $\sigma_\chi$ is contained in $\kappa\overline{\kappa}$, 
or equivalently $(\kappa,\sigma_\chi\kappa)\neq \{0\}$. 

Let $\iota:N \hookrightarrow M$ be the inclusion map. 
In the proof of Theorem \ref{classify2G1}, we showed 
$$(\iota,\rho\iota)=\{T\in \cK;\;\forall g\in G,\; (j_1\circ j_2)U(g)(j_1\circ j_2)^*T=T\}=\C(j_1\circ j_2)^*T_0.$$
Note that we have 
$$(j_1\circ j_2)U(g)(j_1\circ j_2)^*=a(g)U(-g)V(-g). $$
We claim that $(j_1\circ j_2)^*T_0\lambda_h=\trho(\lambda_h)(j_1\circ j_2)^*T_0$ holds. 
Indeed, the right-hand side is 
$$a(h)U(h)\lambda_h(j_1\circ j_2)^*T_0=a(h)U(h)\alpha_h((j_1\circ j_2)^*T_0)\lambda_h=a(h)U(h)V(h)(j_1\circ j_2)^*T_0\lambda_h,$$
and the claim follows. 
Now we have $V_\chi^* (j_1\circ j_2)^*T_0\in (\kappa,\sigma_\chi\kappa)$. 
It remains to show $V_\chi^* (j_1\circ j_2)^*T_0\neq 0$, which follows from 
\begin{align*}
(V_\chi^* (j_1\circ j_2)^*T_0)^*V_\chi^* (j_1\circ j_2)^*T_0
 &=\frac{1}{|H|}\sum_{h\in H}\overline{\chi(h)}T_0^*j_1\circ j_2\lambda_h(j_1\circ j_2)^*T_0 \\
 &=\frac{1}{|H|}\sum_{h\in H}\overline{\chi(h)}T_0^*j_1\circ j_2\alpha_h((j_1\circ j_2)^*T_0)\lambda_h\\
 &=\frac{1}{|H|}\sum_{h\in H}\overline{\chi(h)}T_0^*j_1\circ j_2V(h)(j_1\circ j_2)^*T_0\lambda_h\\ 
 &=\frac{1}{|H|}. 
\end{align*}
\end{proof}

\begin{remark} From Theorem \ref{self-dual} and its proof, we can see that the subfactor $ N\rtimes_\alpha H\subset M\rtimes_\alpha H$ 
is self-dual too. 
\end{remark}

To determine the principal graph of $N\rtimes_\alpha H\subset M\rtimes_\alpha H$, it suffices to determine the fusion rules for 
the categories generated by $\{\sigma_\chi\}_{\chi\in \hH}$. 

For $g\in G$, we denote by $\chi_g\in \hH$ the character of $H$ determined by $\chi_g(h)=\inpr{h}{g}$. 
Since every character of $H$ extends to a character of $G$ and the bicharacter $\inpr{\cdot}{\cdot}$ 
is non-degenerate, the map $G\ni g\mapsto \chi_g\in \hat{H}$ is a surjection, giving an isomorphism 
from $G/H^\perp$ onto $\hat{H}$. 

Direct computation shows that $\talpha_g(e_\chi)=e_{\chi-\chi_g}$, and $e_\chi U(g)=U(g)e_{\chi+\chi_g}$. 
This implies $[\talpha_g\sigma_\chi]=[\sigma_{\chi-\chi_g}]$ and $[\sigma_{\chi}\talpha_g]=[\sigma_{\chi+\chi_g}].$ 
In particular, when $k\in H^\perp$, we have 
$$[\talpha_k\sigma_\chi]=[\sigma_\chi\talpha_k]=[\sigma_\chi].$$

\begin{theorem}\label{deT2} Assume that $H\subset H^\perp$. 
Then there exists $g_a\in G$ with $a(h)=\inpr{h}{g_a}$ for all $h\in H$, and 
$$[\sigma][\sigma]=\bigoplus_{k\in H^\perp/H}[\talpha_{k-g_a}]\oplus |H^\perp/H|\bigoplus_{g\in G/H^\perp} [\talpha_{g}\sigma],$$
$$[\talpha_g][\sigma]=[\sigma][\talpha_{-g}].$$
Note that since $[\talpha_{g+h}]=[\talpha_{g}]$ for any $h\in H$ and $[\talpha_k\sigma]=[\sigma]$ for any $k\in H^\perp$, 
the above expression makes sense. 
\end{theorem}

\begin{proof} For $h_1,h_2\in H$, we have 
$$a(h_1)a(h_2)=\inpr{h_1}{h_2}a(h_1+h_2)=a(h_1+h_2),$$
and the restriction of $a$ to $H$ is a character. 
Thus there exists $g_a\in G$ satisfying $a(h)=\inpr{h}{g_a}$ for any $h\in H$. 

We choose a transversal $\{k_i\}_{i\in H^\perp/H}\subset H^\perp$ for $H$. 
Note that we have 
\begin{align*}
\lefteqn{\sigma^2(x)=V_0^*\trho(V_0^*)\trho^2(x)\trho(V_0)V_0} \\
 &=\sum_{g\in G}V_0^*\trho(V_0^*)S_g\talpha_g(x)S_g^*\trho(V_0)V_0
+\sum_{g\in G}V_0^*\trho(V_0^*)T_g\trho(x)T_g^*\trho(V_0)V_0,
\end{align*}
and the range projection of $\trho(V_0)V_0$ is 
\begin{align*}
\lefteqn{\trho(V_0)e_0\trho(V_0)^*=\trho(V_0V_0^*)e_0=\trho(e_0)e_0=\frac{1}{|H|}\sum_{h\in H}a(h)U(h)\lambda_he_0} \\
 &=\frac{1}{|H|}\sum_{h\in H}a(h)U(h)e_0=\frac{1}{|H|}\sum_{h\in H}a(h)e_{-\chi_h}U(h)
 =\frac{1}{|H|}e_{0}\sum_{h\in H}a(h)U(h). 
\end{align*}
We have 
$$\frac{1}{|H|}e_{0}\sum_{h\in H}a(h)U(h)S_g=\frac{1}{|H|}e_{0}\sum_{h\in H}\inpr{h}{g_a+g}S_g
=1_{H^\perp}(g_a+g)e_0S_g,$$
where $1_{H^\perp}$ is the indicator function of $H^\perp$. 
Now \begin{align*}
\lefteqn{\sum_{g\in G}V_0^*\trho(V_0^*)S_g\talpha_g(x)S_g^*\trho(V_0)V_0} \\
 &=\sum_{k\in H^\perp}V_0^*\trho(V_0^*)e_0S_{k-g_a}\talpha_{k-g_a}(x)S_{k-g_a}^*e_0\trho(V_0)V_0\\
 &=\sum_{i\in H^\perp/H}\sum_{h\in H}V_0^*\trho(V_0^*)e_0S_{h+k_i-g_a}\talpha_{h+k_i-g_a}(x)S_{h+k_i-g_a}^*e_0\trho(V_0)V_0 \\
 &=\sum_{i\in H^\perp/H}\sum_{h\in H}V_0^*\trho(V_0^*)e_0\lambda_hS_{k_i-g_a}\talpha_{k_i-g_a}(x)
 S_{k_i-g_a}^*\lambda_h^*e_0\trho(V_0)V_0 \\
 &=|H|\sum_{i\in H^\perp/H}V_0^*\trho(V_0^*)S_{k_i-g_a}\talpha_{k_i-g_a}(x)
 S_{k_i-g_a}^*\trho(V_0)V_0.
\end{align*}
It is straightforward to see that $\{\sqrt{|H|}V_0^*\trho(V_0^*)S_{k_i-g_a}\}_{i\in H^\perp/H}$ are isometries 
with mutually orthogonal ranges. 

For the second term we have 
\begin{align*}
\lefteqn{\sum_{g\in G}V_0^*\trho(V_0^*)T_g\trho(x)T_g^*\trho(V_0)V_0} \\
 &=\frac{1}{|H|^2}\sum_{g\in G}\sum_{h_1,h_2\in H}a(h_1)\overline{a(h_2)}
 V_0^*\trho(V_0^*)e_0U(h_1)T_g\trho(x)T_g^*U(h_2)^*e_0\trho(V_0)V_0 \\
 &=\frac{1}{|H|^4}\sum_{g\in G}\sum_{h_1,h_2,h_3,h_4\in H}a(h_1)\overline{a(h_2)}
 V_0^*\trho(V_0^*)\lambda_{h_3}T_{g-h_1}\trho(x)T_{g-h_2}^*\lambda_{h_4}\trho(V_0)V_0 \\
 &=\frac{1}{|H|^4}\sum_{g\in G}\sum_{h_1,h_2,h_3,h_4\in H}a(h_1)\overline{a(h_2)}\inpr{h_3+h_4}{g}
 V_0^*\trho(V_0^*)T_{g-h_1}\lambda_{h_3+h_4}\trho(x)T_{g-h_2}^*\trho(V_0)V_0 \\
 &=\frac{1}{|H|^2}\sum_{g\in G}\sum_{h_1,h_2\in H}a(h_1)\overline{a(h_2)}
 V_0^*\trho(V_0^*)T_{g-h_1}e_{-\chi_g}\trho(x)T_{g-h_2}^*\trho(V_0)V_0 \\
 &=\frac{1}{|H|^2}\sum_{g\in G}\sum_{h_1,h_2\in H}a(h_1)\overline{a(h_2)}
 V_0^*\trho(V_0^*)T_{g-h_1}V_{-\chi_g}\sigma_{-\chi_g}(x)V_{-\chi_g}^*T_{g-h_2}^*\trho(V_0)V_0.
 \end{align*}
We choose a transversal $\{g_j\}_{j\in G/H^\perp}\subset G$ for $H^\perp$. 
Then we get 
\begin{align*}
\lefteqn{\sum_{g\in G}V_0^*\trho(V_0^*)T_g\trho(x)T_g^*\trho(V_0)V_0} \\
 &=\frac{1}{|H|^2}\sum_{j\in G/H^\perp}\sum_{i\in H^\perp/H}\sum_{h_0, h_1,h_2\in H}a(h_1)\overline{a(h_2)}\\
 &\times V_0^*\trho(V_0^*)T_{g_j+k_i+h_0-h_1}V_{-\chi_{g_j}}\sigma_{-\chi_{g_j}}(x)V_{-\chi_{g_j}}^*T_{g_j+k_i+h_0-h_2}^*\trho(V_0)V_0.\\
 &=\frac{1}{|H|}\sum_{j\in G/H^\perp}\sum_{i\in H^\perp/H}\sum_{h_1,h_2\in H}a(h_1)\overline{a(h_2)}\\
 &\times V_0^*\trho(V_0^*)T_{g_j+k_i-h_1}V_{-\chi_{g_j}}\sigma_{-\chi_{g_j}}(x)V_{-\chi_{g_j}}^*T_{g_j+k_i-h_2}^*\trho(V_0)V_0.
\end{align*}
Direct computation shows that 
$$\{\frac{1}{\sqrt{|H|}}\sum_{h\in H}a(h)V_0^*\trho(V_0^*)T_{g_j+k_i+h}V_{-\chi_{g_j}}\}_{j\in G/H^\perp,\;i\in H^\perp/H},$$
is a family of isometries with mutually orthogonal ranges. 
Therefore the statement is proved. 
\end{proof}

\begin{remark} 
(1) The above theorem in particular shows that $\sigma$ is self-conjugate if and only if $a(h)=1$ for any $h\in H$. 
Since $a(h)a(h)=\overline{\inpr{h}{h}}=1$ for $h\in H$, we have $2g_a=0$, and non-trivial $g_a$ could occur only if 
$H$ is an even group. \\
(2) The original near-group category is Morita equivalent to the fusion category generated by $\sigma$ and the dual 
action $\hat{\alpha}$ of $\alpha$. 
\end{remark}

A subgroup $H\subset G$ is called Lagrangian if $H=H^\perp$ and the restriction of $a$ to $H$ is 1 (see \cite{ENO10}). 

\begin{cor}\label{deCor} If $H$ is a Lagrangian, we have 
$$[\sigma][\sigma]=[\id]\oplus \bigoplus_{g\in G/H} [\talpha_{g}\sigma],$$
$$[\talpha_g][\sigma]=[\sigma][\talpha_{-g}].$$
\end{cor}

In \cite{I01}, we gave a Cuntz algebra construction of the fusion categories with the above type of fusion rules, 
other than near-group categories, which was further pursued in \cite{EG11}. 
In a forthcoming paper, we systematically investigate such fusion categories for arbitrary finitely abelian groups. 
In this note, we just give a converse construction of the above corollary in the case of odd groups. 

Let $R$ be a type III factor and $K$ be an odd abelian group, which will play the role of $\hH$ as well as $G/H$. 
Assume that we are given a map $\beta:K\rightarrow \Aut(R)$ inducing an injective homomorphism from $K$ 
into $\Out(R)$, and irreducible $\sigma\in \End(R)$ satisfying 
$$[\beta_k][\sigma]=[\sigma][\beta_{-k}],$$
$$[\sigma][\sigma]=[\id]\oplus l\bigoplus_{k\in K}[\beta_k][\sigma].$$
Note that the Frobenius reciprocity implies that $\{\beta_k\sigma\}_{k\in K}$ are all inequivalent. 

Since $K\ni k\mapsto [\beta_k]\in \Out(R)$ is a homomorphism, there exist $u(k_1,k_2)\in U(R)$ satisfying 
$\beta_{k_1}\circ \beta_{k_2}=\Ad u(k_1,k_2)\circ \beta_{k_1+k_2}$. 
The associativity 
$$\beta_{k_1}\circ(\beta_{k_2}\circ \beta_{k_3})=(\beta_{k_1}\circ \beta_{k_2})\circ \beta_{k_3}$$ 
implies that there exists a 
3-cocycle $\omega\in Z^3(K,\T)$ satisfying 
$$\beta_{k_1}(u(k_2,k_3))u(k_1,k_2+k_3)=\omega(k_1,k_2,k_3)u(k_1,k_2)u(k_1+k_2,k_3).$$

Recall that the opposite algebra $R^{\mathrm{opp}}$ of $R$ is a linear space $R$ equipped with a new product $x\cdot y=yx$. 
To realize $R^{\mathrm{opp}}$ as a von Neumann algebra, we identify it with the commutant $R'$ of $R$ under 
identification of $x\in R$ and $J_Rx^*J_R\in R'$ where $J_R$ is the modular conjugation of $R$. 
We denote by $j$ the anti-linear multiplicative map $:R\ni x\mapsto J_RxJ_R\in R^{\mathrm{opp}}$, 
and define $\beta^{\mathrm{opp}}_k(x)=j\circ \beta_k\circ j^{-1}\in \Aut(R^{\mathrm{opp}})$. 
Let $\cR=R\otimes R^{\mathrm{opp}}$, let $\gamma_k=\beta_k\otimes \beta^{\mathrm{opp}}_k\in \cR$, and let 
$v(k_1,k_2)=u(k_1,k_2)\otimes j(u(k_1,k_2))$. 
Then $\gamma$ and $v$ satisfy the 2-cocycle relation  
$$\gamma_{k_1}(v(k_2,k_3))v(k_1,k_2+k_3)=v(k_1,k_2)v(k_1+k_2,k_3),$$
and we can construct the twisted crossed product $M=\cR\rtimes_{\gamma,v} K$, 
that is a factor generated by $\cR$ and unitaries $\{\lambda^v_k\}_{k\in K}$ satisfying 
$\lambda^v_k x=\gamma_k(x)\lambda^v_k$ for $x\in \cR$ and $\lambda^v_{k_1}\lambda^v_{k_2}=v(k_1,k_2)\lambda^v_{k_1+k_2}$. 
Let $\kappa:\cR\hookrightarrow M$ be the inclusion map. 
Then the fusion categories generated by $\{\kappa(\beta_k\otimes \id)\overline{\kappa}\}_{k\in K}$ is nothing but the 
quantum double $\cD^\omega(K)$ (see \cite{I00}). 

Let 
$$\theta_{k_1}(k_2,k_3)=\frac{\omega(k_2,k_1,k_3)}{\omega(k_1,k_2,k_3)\omega(k_2,k_3,k_1)}.$$
Then it is known that $\theta_k$ is a 2-cocycle for any fixed $k\in K$ (see \cite{DPR90}, \cite{DW90}). 
We assume that $\theta_k$ is a coboundary for any $k\in K$, which automatically holds for cyclic groups. 
This is equivalent to the condition that $\cD^\omega(D)$ is pointed, that is, every simple object of 
$\cD^\omega(K)$ is invertible, which in our situation means that $\beta_k\otimes \id$ extends to $M$ as we see now. 
Thanks to this assumption, there exists $\nu_{k_1}(k_2)$ satisfying 
$$\theta_k(k',k'')=\nu_k(k'+k'')\overline{\nu_k(k')\nu_k(k'')}.$$

\begin{lemma} Let the notation be as above. 
The automorphism $\beta_k\otimes \id$ on $\cR$ extends to $M$ by 
$$\tbeta_k(\lambda^v_{k'})=\nu_k(k')(u(k,k')u(k',k)^*\otimes 1)\lambda^v_{k'}.$$
\end{lemma}

\begin{proof} It suffices to verify the following two relations: 
\begin{align*}
\lefteqn{\nu_k(k')(u(k,k')u(k',k)^*\otimes 1)\lambda^v_{k'}(\beta_k\otimes \id)(x)} \\
 &=(\beta_k\otimes \id)\circ\gamma_{k'}(x) \nu_k(k')(u(k,k')u(k',k)^*\otimes 1)\lambda^v_{k'},\quad \forall x\in \cR,
\end{align*}
\begin{align*}
\lefteqn{\nu_k(k')(u(k,k')u(k',k)^*\otimes 1)\lambda^v_{k'}\nu_k(k'')(u(k,k'')u(k'',k)^*\otimes 1)\lambda^v_{k''}} \\
 &=\nu_k(k'+k'') (\beta_{k}\otimes 1)(w(k',k''))(u(k,k'+k'')u(k'+k'',k)^*\otimes 1)\lambda^v_{k'+k''}. 
\end{align*}
The first one is easy to verify and the second one is equivalent to 
\begin{align*}
\lefteqn{ \nu_k(k')\nu_k(k'')u(k,k')u(k',k)^*\beta_{k'}(u(k,k'')u(k'',k)^*)u(k',k'')} \\
 &=\nu_k(k'+k'')\beta_{k}(u(k'+k''))u(k,k'+k'')u(k'+k'',k)^*. 
\end{align*}
Using the defining relation of $\omega$, we get 
\begin{align*}
\lefteqn{\beta_k(u(k'+k''))u(k,k'+k'')u(k'+k'',k)^*} \\
 &=\omega(k,k',k'')u(k,k')u(k+k',k'')u(k'+k'',k)^*\\
 &=\omega(k,k',k'')u(k,k')u(k',k)^*(u(k',k)u(k+k',k''))u(k'+k'',k)^* \\
 &=\omega(k,k',k'')\overline{\omega(k',k,k'')}u(k,k')u(k',k)^*
 \beta_{k'}(u(k,k''))u(k',k+k'')u(k'+k'',k)^*  \\
 &=\omega(k,k',k'')\overline{\omega(k',k,k'')}\omega(k',k'',k)
 u(k,k')u(k',k)^*\beta_{k'}(u(k,k'')u(k'',k)^*)u(k'k'')\\
 &=\overline{\theta_k(k',k'')} u(k,k')u(k',k)^*\beta_{k'}(u(k,k'')u(k'',k)^*)u(k'k''),\\
\end{align*}
which shows the statement. 
\end{proof}

Note that the group of the equivalence classes of the simple objects of $\cD^\omega(K)$ is 
$\{[\tbeta_k\circ\hat{\gamma}_\chi]\}_{k\in K,\;\chi\in \hat{K}}$, 
where $\hat{\gamma}$ is the dual action of $\gamma$.  
We denote this group by $G$. 
Thus $G$ is an extension:
$$0\to\hat{K}\to G\to K\to 0.$$
 
\begin{theorem}\label{KtoG} Let the notation be as above. 
The fusion category generated by $\kappa(\sigma\otimes \id)\overline{\kappa}$ is a C$^*$ near-group category 
for $G$ with $m=l|K|^2$. 
\end{theorem}

\begin{proof} Let $\rho=\kappa(\sigma\otimes \id)\overline{\kappa}$. 
Then $\rho$ is irreducible because
$$\dim(\rho,\rho)=\dim(\overline{\kappa}\kappa(\sigma\otimes \id),(\sigma\otimes \id)\overline{\kappa}\kappa)
=\sum_{k,k'\in K}\dim(\beta_k\sigma\otimes \beta^{\mathrm{opp}}_k,\sigma\beta_{k'}\otimes \beta^{\mathrm{opp}}_{k'})=1.$$
For $\rho^2$, we have 
\begin{align*}
[\rho^2]&=[\kappa(\sigma\otimes \id)\overline{\kappa}\kappa(\sigma\otimes \id) \overline{\kappa}]
=\bigoplus_{k\in K}[\kappa(\sigma\beta_k\sigma\otimes \beta^{\mathrm{opp}}_k)\overline{\kappa}]
=\bigoplus_{k\in K} [\kappa(\beta_{-k}\sigma^2\otimes\beta^{\mathrm{opp}}_k) \overline{\kappa}] \\
 &=\bigoplus_{k\in K} [\kappa\gamma_k(\beta_{-2k}\sigma^2\otimes\id) \overline{\kappa}]
 =\bigoplus_{k\in K} [\kappa(\beta_{-2k}\otimes \id)\overline{\kappa}]
 \oplus l\bigoplus_{k,k'\in K}[\kappa(\beta_{-2k}\beta_{k'}\sigma\otimes \id)\overline{\kappa}]\\
 &=\bigoplus_{k\in K} [\tbeta_{-2k}\kappa\overline{\kappa}]
 \oplus l|K|\bigoplus_{k'\in K}[\kappa(\beta_{k'}\sigma\otimes \id)\overline{\kappa}].
\end{align*}
where we used $\kappa\circ\beta_k=\tbeta_k\circ\kappa$. 
Since $K$ is an odd group, the first term is 
$$\bigoplus_{k\in K,\;\chi\in \hat{K}}[\tbeta_k\hat{\gamma}_\chi].$$

For the second term,
\begin{align*}
\lefteqn{\dim(\kappa(\beta_k\sigma\otimes \id)\overline{\kappa},\kappa(\sigma\otimes \id)\overline{\kappa})
=\dim(\overline{\kappa}\kappa(\beta_k\sigma\otimes \id),(\sigma\otimes \id)\overline{\kappa}\kappa)} \\
&=\sum_{k',k''\in K}\dim(\beta_{k+k'}\sigma\otimes \beta^{\mathrm{opp}}_{k'}), ((\sigma\beta_{k''}\otimes \beta^{\mathrm{opp}}_{k''}))
=\sum_{k'\in K}\dim(\beta_{k+k'}\sigma, \beta_{-k'}\sigma)\\
&=\sum_{k'\in K}\dim(\beta_{k+2k'}\sigma, \sigma)=1.
\end{align*}
This shows 
$$[\rho^2]=\bigoplus_{k\in K,\;\chi\in \hat{K}}[\tbeta_k\hat{\gamma}_\chi]+l|K|^2[\rho].$$
\end{proof}

\begin{remark} 
Let $\cC_\sigma$ be the fusion category generated by $\sigma$, and let $\cC_\rho$ be the near-group category generated by $\rho$ 
in Theorem \ref{KtoG}. 
Then the above construction shows that $\cC_\rho$ is categorically Morita equivalent to 
$\cC_\sigma\boxtimes \mathrm{Vec}^{\overline{\omega}}_K$, 
where $\mathrm{Vec}^{\overline{\omega}}_K$ is the category of $K$-graded vector spaces with twist $\overline{\omega}$. 
In consequence, we have equivalence $\cZ(\cC_\rho)\cong \cZ(\cC_\sigma)\boxtimes \cZ(\mathrm{Vec}^{\overline{\omega}}_K)$, 
where $\cZ(\cC)$ denotes the Drinfeld center of $\cC$.   
This explains why the modular data of $\cZ(\cC_\rho)$ factors as the product of those of $\cZ(\cC_\sigma)$ 
and $\cD^{\overline{\omega}}(K)$ in Example \ref{Z3Z3} and Example \ref{Z9} below (see \cite[Section2 and Section 3]{EG11} 
and \cite[p.636]{EG14}). 
\end{remark}

\begin{example} For $G=\Z_2\times \Z_2$, there is a unique C$^*$ near-group category with $m=4$, 
and a non-trivial subgroup $H$ with $H\subset H^\perp$ is unique, which is $H=\{0,g_3\}$ in Example \ref{Z2Z2}. 
The subgroup $H$ is Lagrangian, and by de-equivariantization, we get the even part of the $A_7$-subfactor.  
\end{example}

\begin{example} There are two C$^*$ near-group categories for $G=\Z_4$, which are mutually complex conjugate, and 
$H=\Z_2$ satisfies $H=H^\perp$, though it is not Lagrangian. 
In this case, de-equivariantization by $H$ gives two 1-supertransitive subfactors constructed Liu-Morrison-Penneys \cite{LMP15}. 
In fact they already observed that the two subfactors they constructed give rise to the $2^{\Z_4}1$ subfactors by equivariantization 
(see \cite[Section 4.3]{LMP15}).
\end{example}

\begin{example}\label{Z3Z3} Let $G=\Z_3\times \Z_3$. Evans-Gannon \cite[Table 2]{EG14} showed that there is a unique C$^*$ near-group 
category for $G=\Z_3\times \Z_3$, and the quadratic form in the solution is given as $\inpr{(x_1,x_2)}{(y_1,y_2)}=\zeta_3^{x_1x_2-y_1y_2}$, 
$a(x_1,x_2)=\zeta_3^{x_1^2-y_1^2}.$ 
There are exactly two Lagrangians: $H_1=\{(0,0),(1,1),(2,2)\}$ and $H_2=\{(0,0), (1,2),(2,1)\}$, and 
they produce two different fusion categories by de-equivariantization. 
A priori, two different subgroups could produce equivalent categories, but this is not the case now because of the following reason. 
In the case of $H=H_1$, we have 
$$[\sigma^2]=[\id]\oplus \bigoplus_{h\in H_2} [\talpha_h\sigma],$$
and $\talpha$ restricted to $H_2$ is an action, and therefore there is no third cohomology obstruction for the $\Z_3$ part. 
When $H=H_2$, the situation is the same. 
In fact, Grossman-Snyder \cite{GS12} showed that there are exactly two C$^*$-fusion categories with the same fusion rules 
as the Haagerup category (the even part of the Haagerup subfactor) and with trivial third cohomology obstruction for the $\Z_3$ part. 
Thanks to Theorem \ref{KtoG}, which is the converse construction of Corollary \ref{deCor}, 
the two groups $H_1$ and $H_2$ indeed produce two different fusion categories. 
\end{example}

\begin{example}\label{Z9} Evans-Gannon \cite{EG11} showed that there are exactly two C$^*$ near-group categories for $\Z_9$, 
and they are complex conjugate to each other. 
In this case, the subgroup $H=<3>\cong \Z_3$ is a Lagrangian, which produces two more fusion categories with the same fusion rules 
as the Haagerup category. 
The $\Z_3$ part of the resulting fusion categories is generated by $\talpha_1$, and its third cohomology obstruction, known as 
Connes obstruction \cite{C77}, can be easily computed as follows. 
Indeed, we have $\talpha_1^3=\lambda_3$ and $\talpha_1(\lambda_3)=\inpr{1}{1}^3\lambda_3$, which shows that the Connes obstruction is 
either $\zeta_3$ or $\zeta_3^{-1}$. 
The existence of these two categories has been expected since Evans-Gannon \cite[Question 5]{EG11} obtained the corresponding hypothetical 
modular data of the Drinfeld center. 
\end{example}

\begin{example} It is known that there exist fusion categories as in Theorem \ref{KtoG} for $G=\Z_n$, $n=5,7,9$, 
with $l=1$ and trivial third cohomology obstruction on the group part (see \cite{I01}, \cite{EG11}).  
It follows that there exist C$^*$ near-group categories for $G=\Z_n\times \Z_n$ with $m=n^2$ for $n=5,7,9$. 
\end{example}

\subsection{Twisted case} 
In the tensor category setting (see \cite{EGNO15}), de-equivariantization with respect to a subgroup $H\subset G$ 
can be defined for an algebra structure of the object
$$\bigoplus_{h\in H}\alpha_h,$$
because it is known that $\alpha_h$ has a unique half-braiding (see \cite[Lemma 6.1]{I01}), and hence 
$\alpha_h$ uniquely lifts to the Drinfeld center. 
In operator algebra setting, an algebra structure of the above object corresponds to a lifting of 
$\{[\alpha_h]\}_{h\in H}\subset \Out(M)$ to a group action. 
In the case of near-group categories, we have a privileged lifting determined by $\alpha_h\circ \rho=\rho$, 
and we essentially discussed de-equivariantization with this algebra structure in the previous subsection. 
The other algebra structures are in one-to-one correspondence with $H^2(H,\T)\setminus \{0\}$, 
and the corresponding operation in terms of operator algebras is cocycle twisted crossed product. 

Let $G$, $M$, $\rho$, and $\alpha$ be as in the previous subsection. 
Now we do not assume that $H$ satisfies the condition $H\subset H^\perp$. 
Instead, we assume that $H\cong \Z_2^{2s}$ for a natural number $s$, and that the restriction of 
$\inpr{\cdot}{\cdot}$ to $H$ is non-degenerate. 
We further assume that there exists a 2-cocycle $\omega\in Z^2(H,\T)$ satisfying 
$\omega(h,k)\overline{\omega(k,h)}=\inpr{h}{k}\in \R$. 

Let $M\rtimes_{\alpha,\omega} H$ be the twisted crossed product, that is, the von Neumann algebra generated by $M$ and 
unitaries $\{\lambda^\omega_h\}_{h\in H}$ satisfying $\lambda^\omega_hx=\alpha_h(x)\lambda^\omega_h$ for any $x\in M$ and 
$\lambda^\omega_{h_1}\lambda^\omega_{h_2}=\omega(h_1,h_2)\lambda^\omega_{h_1+h_2}$. 
As before we can extend $\rho$ and $\alpha$ to $M\rtimes_{\alpha,\omega}H$ by setting 
$\trho(\lambda^\omega_h)=a(h)U(h)\lambda^\omega_h$, and $\talpha_g(\lambda^\omega_h)=\inpr{g}{h}\lambda^\omega_h$. 
As in Lemma \ref{deL1}, we have the following for all $x\in M\rtimes_\alpha H$. 
$$\trho^2(x)=\sum_{g\in G}S_g\talpha_g(x)S_g^*+\sum_{g\in G}T_g\trho(x)T_g^*,$$
$$\trho\circ\talpha_g=\Ad U(g)\circ\trho.$$

We claim that $\lambda^\omega_h$ commutes with $\trho(\lambda^\omega_k)$ for any $h,k\in H$.
Indeed, we have 
\begin{align*}
\lefteqn{\lambda^\omega_h\trho(\lambda^\omega_k)=\lambda^\omega_ha(k)U(k)\lambda^\omega_k
=a(k)\alpha_h(U(k))\lambda^\omega_h\lambda^\omega_k}\\
&=a(k)\overline{\inpr{h}{k}}\omega(h,k)U(k)\lambda^\omega_{h+k} 
=a(k)\overline{\inpr{h}{k}} \omega(h,k)\overline{\omega(k,h)}U(k)\lambda^\omega_k\lambda^\omega_h\\
 &=\trho(\lambda^\omega_k)\lambda^\omega_h=\trho(\lambda^\omega_k)\lambda^\omega_h. 
\end{align*}

\begin{lemma} Let the notation be as above. 
\begin{itemize}
\item[(1)] $\talpha_h=\Ad \lambda^\omega_h$ for any $h\in H$. 
\item[(2)] $(\trho,\trho)=\{\lambda^\omega_h\}_{h\in H}''$. 
\end{itemize}
\end{lemma}

\begin{proof} (1) 
Note that the restriction of $\inpr{\cdot}{\cdot}$ to $H$ is real-valued. 
Since $\Ad {\lambda^\omega_h}^*\circ\talpha_h$ is trivial on $M$, it suffices to show 
\begin{align*}
\lefteqn{\Ad {\lambda^\omega_h}^*\circ\talpha_h(\lambda^\omega_k)=\inpr{h}{k} \lambda^\omega_h\lambda^\omega_k{\lambda^\omega_h}^*
=\inpr{h}{k}\omega(h,k)\lambda^\omega_{h+k}{\lambda^\omega_h}^*}\\
 &=\inpr{h}{k}\omega(h,k)\overline{\omega(k,h)}
 \lambda^\omega_{k}\lambda^\omega_h{\lambda^\omega_h}^* =\inpr{h}{k}^2\lambda^\omega_k=\lambda^\omega_k. 
\end{align*}

(2) As in the proof of Lemma \ref{deL2}, we have $(M\rtimes_{\alpha,\omega} H)\cap \rho(M)'=\{\lambda^\omega_h\}_{h\in H}''$, and  
$$(\trho,\trho)=\{\lambda^\omega_h\}_{h\in H}''\cap \{\trho(\lambda^\omega_k)\}_{k\in H}'=\{\lambda^\omega_h\}_{h\in H}''.$$
\end{proof} 

Since $\inpr{\cdot}{\cdot}$ restricted to $H$ is non-degenerate, the twisted group algebra $\{\lambda^\omega_h\}_{h\in H}''$ is isomorphic 
to the full matrix algebra $M_{2^s}(\C)$. 
Thus (2) above shows that there exists an irreducible $\sigma\in \End(M\rtimes_{\alpha,\omega}H)$ satisfying 
$[\trho]=2^s[\sigma]$. 
On the other hand we have 
$$[\trho^2]=\bigoplus_{g}[\talpha_g]\oplus n[\trho]
=|H|\bigoplus_{g\in G/H}[\talpha_g]\oplus n2^s[\sigma].$$
This implies 

\begin{theorem}\label{twisted} Let the notation be as above. Then 
$$[\sigma^2]=\bigoplus_{g\in G/H}[\talpha_g]\oplus 2^s|G/H|[\sigma].$$
In particular $\sigma$ generates a C$^*$ near-group category for $G/H$ with $m=2^s|G/H|$. 
\end{theorem}

We have already seen in Subsection \ref{Z3m=6} that there are exactly two C$^*$ near-group categories for $\Z_3$ with $m=6$  
(resp. $2^{\Z_3}_21$ subfactors), and there is a $D_8$-actions on each of them, where $D_8$ is the dihedral group of order 8. 
We will see in Example \ref{36equi} that equivariantization by this action produces a fusion category containing a 
C$^*$ near-group category for $\Z_2\times \Z_2\times \Z_3$ with $m=12$ (resp. a $2^{\Z_2\times \Z_2\times \Z_3}1$ subfactor). 
Theorem \ref{twisted} should be considered as a converse construction of this equivariantization in spite of a superficial discrepancy: 
twisted crossed product is used instead of ordinary crossed product, and $\Z_2\times \Z_2$ is not the dual group of $D_8$.    
The first point can be easily fixed because it is well-known that every outer cocycle action of a finite group is known to be 
equivalent to an ordinary action. 
It means that there is a family of unitaries $\{w_h\}_{h\in H}$ in $M$ satisfying $w_h\alpha_h(w_k)=\omega(h,k)w_{h+k}$. 
Now the twisted crossed product in Theorem \ref{twisted} is identified with the ordinary crossed product by the new action 
$\Ad w_h\circ \alpha_h$.  
For the second point, let $N\subset M$ be one of the $2^{\Z_3}_21$ subfactors, and let $\gamma$ be the $D_8$-action, 
which we will see in Section \ref{Equi}. 
Let $Z(D_8)\cong \Z_2$ be the center of $D_8$. 
Then it is easy to show that the inclusion  $ N\rtimes_\gamma Z(D_8)\subset M\rtimes_\gamma Z(D_8)$ is still $2^{\Z_3}_21$, 
and it is the fixed point algebra pair of $N\rtimes_\gamma D_8\subset M\rtimes_\gamma D_8$ under the dual action 
$\hat{\gamma}$ restricted to the group part $\Hom(D_8,\T)\cong \Z_2\times \Z_2$ of the unitary dual $\hat{D_8}$. 

This reasoning indicates that there must be two missing solutions of the polynomial equations for 
$G=\Z_2\times \Z_2\times \Z_3$ with $m=12$ in Evans-Gannon's list \cite[Table 2]{EG14}, for which the restriction of 
$\inpr{\cdot}{\cdot}$ to $\Z_2\times \Z_2$ should be $\inpr{\cdot}{\cdot}_2$ in Example \ref{Z2Z2}. 
Since they are complex conjugate to each other, we give one of them now. 

\begin{example}\label{Z2Z2Z3} Let $G=H\times K$ with $H=\Z_2\times \Z_2$ and $K=\Z_3$, and let 
$$\inpr{(h,k)}{(h',k')}=\inpr{h}{h'}_2\zeta_3^{kk'},$$
$$a(h,k)=a_H(h)\zeta_3^{k^2},$$
with $a_H(0)=a_H(g_1)=a_H(g_2)=1$, $a_H(g_3)=-1$, where we use the parametrization $H=\{0,g_0,g_1,g_2\}$ as in Example \ref{Z2Z2}. 
Then there is a unique solution for the equations in Theorem \ref{m=nTh}, which is given by $c=e^{-\pi\ii/6}$, $d=6+4\sqrt{3}$ and 
\begin{align*}
b &=\frac{1-\sqrt{3}}{3}\left(
\begin{array}{c}
1  \\
\frac{1}{2}  \\
\frac{1}{2}  \\
0 
\end{array}
\right)\otimes 
\left(
\begin{array}{c}
1  \\
-\frac{\zeta_3}{2}  \\
-\frac{\zeta_3}{2} 
\end{array}
\right)
+\frac{1}{2\sqrt{2\sqrt{3}}}\left(
\begin{array}{c}
0  \\
1  \\
-1  \\
0 
\end{array}
\right)\otimes 
\left(
\begin{array}{c}
0  \\
\zeta_3\ii  \\
-\zeta_3\ii 
\end{array}
\right)\\
&+\frac{1}{6}\left(
\begin{array}{c}
1  \\
-1  \\
-1  \\
-\sqrt{3}\ii 
\end{array}
\right)\otimes 
\left(
\begin{array}{c}
1  \\
\zeta_3\\
\zeta_3 
\end{array}
\right),
\end{align*}
where we identify $\ell^2(H\times K)$ with $\ell^2(H)\otimes \ell^2(K)$, and $\ell^2(H)$ and $\ell^2(K)$ with 
$\C^4$ and $\C^3$ via orthonormal bases $<\delta_0,\delta_{g_0},\delta_{g_1},\delta_{g_2}>$ and $<\delta_0, \delta_1,\delta_2>$ 
respectively. 
We can apply Theorem \ref{twisted} to this solution with $\omega$ given by following table, 
\begin{table}[ht]
\begin{tabular}{|c|c|c|c|c|}\hline
&0&$g_0$&$g_1$&$g_2$\\ \hline
0    &1&1&1&1\\ \hline
$g_0$&1&$1$&$\ii$&-$\ii$\\ \hline
$g_1$&1&-$\ii$&$1$&$\ii$\\ \hline
$g_2$&1&$\ii$&-$\ii$&$1$\\ \hline
\end{tabular}
\end{table}
and obtain the C$^*$ near-group category for $\Z_3$ with $m=6$ discussed in Subsection \ref{Z3m=6}. 
\end{example}

Note that the above solution is invariant under $\theta\in \Aut(G)$ given by $\theta(0,x)=(0,x)$, 
$\theta(g_0,x)=(g_1,-x)$, $(g_1,x)=(g_0,-x)$, $(g_2,x)=(g_2,x)$. 

\section{Orbifold construction II (equivariantization)}\label{Equi}
\subsection{Automorphism groups} 
In this subsection, we investigate the structure of the outer automorphism group of a 
C$^*$ near-group category $\cC$ in the irrational case. 
To realize group actions on $\cC$ as those on a concrete operator algebra, 
we use the Cuntz algebra model discussed in Section \ref{RCS}.  

We fix a solution $(\epsilon, \inpr{\cdot}{\cdot},A,C,J,B_g)$ of the equations in Theorem \ref{irrpe}, 
and define a $G$-action $\alpha$ on $\cO_{n+m}$ and an endomorphism $\rho$ as before. 
We introduce a subgroup $\Aut_\rho(\cO_{n+m})$ of $\Aut(\cO_{n+m})$ by 
$$\Aut_\rho(\cO_{n+m})=\{\gamma\in \Aut(\cO_{n+m});\; \gamma\circ \rho=\rho\circ\gamma\}.$$

\begin{lemma} For any $\gamma\in \Aut_\rho(\cO_{n+m})$, there exist $\omega\in \T$, $\theta\in \Aut(G)$, and 
$u\in \cU(\cK_0)$ satisfying $\inpr{\theta(g)}{\theta(h)}=\inpr{g}{h}$, $A(\theta(g))=uA(g)u^*$, $uC=Cu$, $uJ=Ju$, and 
$B_{\theta(g)}(\xi)=(u\otimes u)B_g(u^*\xi)u^*,$ 
such that 
$$\gamma(S_g)=\omega^2S_{\theta(g)},\quad \gamma(T_g(\xi))=\omega T_{\theta(g)}(u\xi).$$

On the other hand, for any triple $(\omega,\theta,u)\in \T\times \Aut(G)\times \cU(\cK_0)$ satisfying 
the above condition, there exists $\gamma_{(\omega,\theta,u)}\in \Aut_\rho(\cO_{n+m})$ satisfying 
$$\gamma_{(\omega,\theta,u)}(S_g)=\omega^2 S_{\theta(g)},\quad \gamma_{(\omega,\theta,u)}(T_g(\xi))=\omega T_{\theta(g)}(u\xi).$$ 
Moreover two different $(\omega,\theta,u)$ and $(\omega',\theta',u')$ give the same automorphism 
if and only if $\omega'=- \omega$, $\theta'=\theta$, and $u'=-u$. 
\end{lemma}

\begin{proof} It is straightforward to show the second part, and we show only the first part here. 
Let $\gamma\in \Aut_\rho(\cO_{n+m})$. 
Since $(\id,\rho^2)=\C S_0$ and $\gamma$ commutes with $\rho$, we have $\gamma(S_0)\in (\id,\rho^2)$, and 
$\gamma(S_0)$ is a multiple of $S_0$. 
Replacing $\gamma$ with $\gamma_{(e^{\ii t},\id,I)}\circ \gamma$ for appropriate $t$, we may assume $\gamma(S_0)=S_0$. 

We claim that there exists $\theta\in \Aut(G)$ satisfying $\gamma\circ \alpha_g\circ \gamma^{-1}=\alpha_{\theta(g)}$ and 
$\gamma(S_g)=S_{\theta(g)}$. 
Since $\gamma(\rho^2,\rho^2)=(\rho^2,\rho^2)$, and 
$$(\rho^2,\rho^2)=\bigoplus_{g\in G}\C S_gS_g^*\oplus \cK\cK^*,$$
there exists a permutation $\theta$ of $G$ fixing $0$ such that $\gamma(S_gS_g^*)=S_{\theta(g)}S_{\theta(g)}^*$. 
Thus $S_{\theta(g)}^*\gamma(S_g)$ is a unitary belonging to $(\gamma\circ\alpha_g\circ\gamma^{-1},\alpha_{\theta(g)})$. 
Since 
$$\gamma\circ\alpha_g\circ\gamma^{-1}\circ \rho=\rho=\alpha_g\circ \rho,$$
we have $(\gamma\circ\alpha_g\circ\gamma^{-1},\alpha_{\theta(g)})\subset (\rho,\rho)=\C$, which implies 
that $\gamma\circ\alpha_g\circ\gamma^{-1}=\alpha_{\theta(g)}$ and $\theta$ is a group automorphism of $G$. 
Moreover, 
$$\gamma(S_g)=\gamma(\alpha_g(S_0))=\alpha_{\theta(g)}(\gamma(S_0))=\alpha_{\theta(g)}(S_0)=S_{\theta(g)}.$$
Since 
$$\Ad \gamma(U(g))\circ \rho=\gamma\circ \Ad U(g)\circ \rho\circ \gamma^{-1}=\gamma\circ\rho\circ \alpha_g\circ \gamma^{-1}
=\rho\circ\alpha_{\theta(g)},$$
$$\gamma(U(g))S_0=\gamma(U(g)S_0)=\gamma(S_0)=S_0,$$
We have $\gamma(U(g))=U(\theta(g))$, which together with $\alpha_g(U(h))=\overline{\inpr{g}{h}}U(h)$ implies 
$\inpr{\theta(g)}{\theta(h)}=\inpr{g}{h}$. 

Since $\gamma$ commutes with $\rho$, it induces a unitary transformation of $(\rho,\rho^2)=\cK$. 
Since 
$$\alpha_g(\gamma(T_0(\xi)))=\gamma(\alpha_{\theta^{-1}(g)}(T_0(\xi)))=\gamma(T_0(\xi)),$$ 
we have $\gamma(T_0(\xi))\in \cK_0$, and there exists a unitary 
$u\in \cU(\cK_0)$ satisfying $\gamma(T_0(\xi))=T_0(u\xi)$. 
For $T_g(\xi)$, we have
$$\gamma(T_g(\xi))=\gamma(U(-g)T_0(\xi))=U(-\theta(g))T_0(u\xi)=T_{\theta(g)}(u\xi).$$ 
Since $\gamma(S_0)=S_0$, we get $\gamma\circ j_1=j_1\circ \gamma$ and $\gamma\circ j_2\circ \gamma$, and in consequence 
Lemma \ref{ACJ} implies that $A(\theta(g))=uA(g)u^*$, $uC=Cu$, and $uJ=Ju$. 

It remains to show that the condition $\gamma(l(T))=l(\gamma(T))$ is equivalent to $B_{\theta(g)}(\xi)=(u\otimes u)B_g(u^*\xi)u^*$ 
under the other conditions we have verified so far. 
Thanks to Lemma \ref{lB}, the former is equivalent to $\gamma(B_g(\xi))=B_{\theta(g)}(u\xi)$. 
Since $B_g(\xi)\in \cK_0^2\cK_0^*$, we get $\gamma(B_g(\xi))=(u\otimes u)B_g(\xi)u^*$, and we are done. 
\end{proof}

Recall that the periodic 1-parameter automorphism group $\{\gamma_{(e^{\ii t},\id,I)}\}_{t\in \R}$, 
which is a central subgroup of $\Aut_\rho(\cO_{n+m})$, has a unique KMS state, 
whose GNS construction produces a type III factor $\cM$ extending $\cO_{n+m}$ (see Appendix and \cite{I93}). 
We use the same symbols $\alpha$, $\rho$, $\gamma_{(\omega,\theta,u)}$ for their extensions to $\cM$ as before. 
We denote by $\cC\subset \End(\cM)$ the fusion category generated by $\rho$. 
To avoid possible confusion, we make the only exception in this note here, and use the tensor symbol $\otimes$ for 
the monoidal product of $\cC$ instead of composition, and do not omit the symbol $\Hom_\cC$ for the spaces of intertwiners in the following arguments. 
To emphasize that $X\in \cM$ is regarded as an element in 
$$\Hom_\cC(\mu,\nu)=\{X\in \cM;\;X\mu(x)=\nu(x)X,\; \forall x\in \cM\},$$ 
we use the notation $(\nu|X|\mu)$ following \cite[Section IV]{DHR71}. 
(Caution: the order of $\mu$ and $\nu$ are switched in $(\nu|X|\mu)$ and $\Hom_\cC(\mu,\nu)$.) 
Note that we have 
$$I_\sigma\otimes (\nu|X|\mu)=(\sigma\otimes\nu|\sigma(X)|\sigma\otimes\mu)\in \Hom_\cC(\sigma\otimes \mu,\sigma\otimes \nu),$$
$$(\nu|X|\mu)\otimes I_\sigma=(\nu\otimes \sigma|X|\mu\otimes \sigma)\in \Hom_\cC(\mu\otimes \sigma,\nu\otimes \sigma).$$

An automorphism of $\cC$ is a tensor auto-equivalence of $\cC$ as a C$^*$-category,  
which is a pair $(F,L)$ consisting of an auto-equivalence $F$ and natural isomorphisms 
$L_{\sigma_1,\sigma_2}:F(\sigma_1)\otimes (\sigma_2)\to F(\sigma_1\otimes \sigma_2)$ satisfying 
$$L_{\sigma_1\otimes \sigma_2,\sigma_3}\circ (L_{\sigma_1,\sigma_2}\otimes I_{F(\sigma_3)})
=L_{\sigma_1,\sigma_2\otimes \sigma_3}\circ (I_{\sigma_1}\otimes L_{\sigma_2,\sigma_3}).$$
The automorphism group $\Aut(\cC)$ of $\cC$ is the collection of automorphisms of $\cC$ modulo tensor natural isomorphisms. 
Since we are working on C$^*$ categories, an isomorphism between two objects are always assumed to be unitary. 
An automorphism is inner if it is equivalent to the one given by $\beta\otimes \cdot \otimes \beta^{-1}$ 
for an invertible object $\beta$ in $\cC$, which is always equivalent to $\alpha_g\otimes \cdot \otimes \alpha_g^{-1}$ 
for some $g\in G$ in our case. 
The outer automorphism group $\Out(\cC)$ of $\cC$ is $\Aut(\cC)$ modulo the normal subgroup of inner automorphisms. 

Every $\gamma\in \Aut_\rho(\cO_{n+m})$ extends to $\cM$ and induces an automorphism of $\cC$, which is denoted by 
$F_\gamma$ (we omit the natural isomorphisms $L:F_\gamma(\sigma_1)\otimes F_\gamma(\sigma_2)\to 
F_\gamma(\sigma_1\otimes \sigma_2)$ as it is trivial now). 
More precisely, we define $F_\gamma(\sigma)=\gamma\circ\sigma\circ \gamma^{-1}$ for an object $\sigma\in \cC$, and 
define $F_\gamma$ on the $\mathrm{Hom}$-spaces by the restriction of $\gamma$. 
For $\gamma=\gamma_{(\omega,\theta,u)}$, we have $F_\gamma(\alpha_g)=\alpha_{\theta(g)}$, $F_\gamma(\rho)=\rho$, 

\begin{lemma} An automorphism $\gamma\in \Aut_\rho(\cO_{n+m})$ induces a trivial automorphism of $\cC$ if and only if 
$\gamma=\gamma_{(\omega,\id,I_{\cK_0})}$ for $\omega\in \T$. 
\end{lemma}

\begin{proof}  
Since $F_\gamma$ is a trivial automorphism only if it fixes the equivalence class of each simple object, we assume 
$\gamma=\gamma_{(\omega,\id,u)}$. 
Assume that $\eta$ is a natural tensor isomorphism between $F_\gamma$ and $\id$. 
We identify $\eta_{\alpha_g}\in \Hom_\cC(\alpha_g,\alpha_g)$ and $\eta_\rho\in \Hom_\cC(\rho,\rho)$ with complex numbers of modulus 1. 
Note that we have $\Hom_\cC(\alpha_{g+h},\alpha_g\otimes \alpha_h)=\C I_{\cM}$, 
$\Hom_\cC(\rho,\alpha_g\otimes \rho)=\C I_{\cM}$, $\Hom_\cC(\rho,\rho\otimes \alpha_g)=\C U(g)$, 
$\Hom_\cC(\alpha_g,\rho\otimes\rho)=\C S_g$, and 
$\Hom_\cC(\rho,\rho\otimes \rho)=\cK$. 
Since $\gamma$ acts on the first three trivially, we get $\eta_{\alpha_g}=I_{\alpha_g}$. 
For $S_g\in \Hom_\cC(\alpha_g,\rho\otimes \rho)$ and $T_g(\xi)\in \Hom_\cC(\rho,\rho\otimes \rho)$, we have 
$$(\eta_\rho\otimes \eta_\rho)\circ F_\gamma((\rho\otimes \rho|S_g|\alpha_g))=(\rho\otimes \rho|S_g|\alpha_g)\circ \eta_{\alpha_g},$$
$$(\eta_\rho\otimes \eta_\rho)\circ F_\gamma((\rho\otimes \rho|T_g(\xi)|\rho))
=(\rho\otimes \rho|T_g(\xi)|\rho)\circ \eta_\rho,$$
which are equivalent to $\eta_{\rho}^2\omega^2=1$ and $\eta_\rho\omega u=I_{\cK_0}$. 
Since $\gamma_{(\omega,\id,u)}=\gamma_{(-\omega,\id,-u)}$, we may assume $u=I_{\cK_0}$, and $\gamma=\gamma_{(\eta_\rho^{-1},\id,I_{\cK_0})}$. 

Tracing back the above argument, we can see that $\gamma_{(\omega,\id,I_{\cK_0})}$ induces a trivial automorphism of $\cC$. 
\end{proof}

We set
$$\Aut_{\rho,0}(\cO_{n+m})=\{\gamma\in \Aut_\rho(\cO_{n+m});\; \gamma(S_0)=S_0\} 
=\{\gamma_{(1,\theta,u)}\in \Aut_\rho(\cO_{n+m}) \}.$$
Then we have 
$$\Aut_\rho(\cO_{n+m})/\T\cong \Aut_{\rho,0}(\cO_{n+m})/\Z_2,$$
where $\T$ is identified with $\{\gamma_{(e^{\ii t},\id,I)}\}_{t\in \R/2\pi\Z}$ and $\Z_2$ is identified with 
$\{\gamma_{(1,\id,\pm I_{\cK_0})}\}$. 
The above lemma shows that there is an injective homomorphism from $\Aut_{\rho,0}(\cO_{n+m})/\Z_2$ into $\Aut(\cC)$. 

Recall that the gauge group $\cG(A,C,J)$ is the set of all unitaries in $\cU(\cK_0)$ commuting with 
$A(g)$, $C$, and $J$. 

\begin{theorem} With the above notation, we have 
\begin{align*}
\lefteqn{\Aut(\cC)\cong \Aut_\rho(\cO_{n+m})/\T\cong \Aut_{\rho,0}(\cO_{n+m})/\Z_2} \\
 &\cong\{(\theta,u)\in \Aut(G)\times \cU(\cK_0);\; \inpr{\theta(g)}{\theta(h)}=\inpr{g}{h},\; uC=Cu,\; uJ=Ju,\\
 & A(\theta(g))=uA(g)u^*,\; 
 B_{\theta(g)}(\xi)=(u\otimes u)B_g(u^*\xi)u^*\}/\{(\id,\pm I_{\cK_0})\}. 
\end{align*}
In particular, the kernel of the canonical homomorphism from $\Aut(\cC)$ to $\Aut(G)$ is isomorphic to 
$$\{u\in \cG(A,C,J);\; B_{g}(\xi)=(u\otimes u)B_g(u^*\xi)u^*\}/\{\pm I_{\cK_0}\}.$$  
\end{theorem}

\begin{proof} It suffices to show that the association 
$\Aut_{\rho,0}(\cO_{n+m})\ni \gamma \mapsto F_\gamma$ induces a map onto $\Aut(\cC)$. 
Let $(F,L)$ be an automorphism of $\cC$. 
Up to a tensor natural isomorphism, we may assume $F(\alpha_g)=\alpha_{\theta(g)}$ and $F(\rho)=\rho$ 
with a group automorphism $\theta\in \Aut(G)$. 
Then we have 
$$F(\alpha_g)\otimes F(\alpha_h)=\alpha_{\theta(g)}\otimes \alpha_{\theta(h)}=\alpha_{\theta(g+h)}
=F(\alpha_{g+h})=F(\alpha_g\otimes \alpha_h),$$  
$$F(\alpha_g)\otimes F(\rho)=\alpha_{\theta(g)}\otimes \rho=\rho=F(\rho)=F(\alpha_g\otimes \rho).$$
We claim that we may further assume  
$$F(\rho\otimes \alpha_g)=F(\rho)\otimes F(\alpha_g),$$
$$F(\rho\otimes \rho)=F(\rho)\otimes F(\rho),$$
and $L_{\alpha_g,\alpha_h}$, $L_{\alpha_g,\rho}$, $L_{\rho,\alpha_g}$, and $L_{\rho,\rho}$ are trivial.  
Indeed, there exist $x(g,h)\in \T$ and $y(g)\in \T$ satisfying 
$L_{\alpha_g,\alpha_h}=(\alpha_{\theta(g+h)}|x(g,h)|\alpha_{\theta(g+h)})$ and 
$L_{\alpha_g,\rho}=(\rho|y(g)|\rho)$, 
and the relation 
$$L_{\alpha_g\otimes\alpha_h,\rho}\circ(L_{\alpha_g,\alpha_h}\otimes I_{F(\rho)})
=L_{\alpha_g,\alpha_h\otimes \rho}\circ(I_{F(\alpha_g)}\otimes L_{\alpha_h,\rho})$$ 
implies $x(g,h)y(g+h)=y(g)y(h)$. 
Now passing from $(F,L)$ to $(F',L')$ by a tensor natural isomorphism $\eta:F\to F'$ with 
$\eta_{\alpha_g}=(\alpha_{\theta(g)}|y(g)|\alpha_{\theta(g)})$, 
$\eta_\rho=(\rho|I_\cM|\rho)$, $\eta_{\rho\otimes\alpha_g}=y(g)L_{\rho,\alpha_g}^{-1}$, and 
$\eta_{\rho\otimes \rho}=L_{\rho,\rho}^{-1}$, 
we get the claim. 

Note that we have \begin{align*}
\lefteqn{F((\alpha_g\otimes \alpha_h |I_{\cM}|\alpha_{g+h}))=F(\alpha_{g+h}|I_\cM|\alpha_{g+h})} \\
 &=(\alpha_{\theta(g+h)}|I_\cM|\alpha_{\theta(g+h)})
=(\alpha_{\theta(g)}\otimes \alpha_{\theta(h)}|I_{\cM}|\alpha_{\theta(g+h)}), 
\end{align*}
$$F((\alpha_g\otimes \rho |I_{\cM}|\rho))=F(\rho|I_\cM|\rho)=(\rho|I_\cM|\rho)
=(\alpha_{\theta(g)}\otimes \rho|I_{\cM}|\rho).$$
We will freely use these relations and their adjoints in what follows. 

Since $F((\rho\otimes\rho|S_0|\id))$ is an isometry belonging to $\Hom_\cC(\id,\rho\otimes\rho)$, 
it is of the form $\omega(\rho\otimes\rho|S_0|\id)$ with $\omega\in \T$. 
Thus replacing $F$ with $F_{\gamma_{(\omega^{-1},\id,I_\cK)}}\circ F$, we may further assume 
$F((\rho\otimes\rho|S_0|\id))=(\rho\otimes\rho|S_0|\id)$. 
Since $S_g=\alpha_g(S_0)$, we have 
$$(\rho\otimes \rho|S_g|\alpha_g)=((\rho|1|\alpha_g\otimes \rho)\otimes I_\rho)\circ (I_{\alpha_g}\otimes (\rho\otimes \rho|S_0|\id)).$$
Applying $F$ to this, we get $F((\rho\otimes\rho|S_g|\alpha_g))=(\rho\otimes \rho|S_{\theta(g)}|\alpha_{\theta(g)})$ and its adjoint. 

Since $F$ restricted to $\Hom_\cC(\rho,\rho\otimes \rho)$ is a unitary transformation of $\Hom_\cC(\rho,\rho\otimes \rho)$, 
we introduce an automorphism $\gamma\in \Aut(\cO_{n+m})$ by 
$$F(((\rho\otimes\rho|S_g|\alpha_g)))=(\rho\otimes \rho|\gamma(S_g)|\alpha_{\theta(g)}),$$ 
$$F((\rho\otimes \rho|T|\rho))=(\rho\otimes \rho|\gamma(T)|\rho).$$
The above computation shows $\gamma(S_g)=S_{\theta(g)}$. 

Recall that we have $\Hom_\cC(\rho,\rho\otimes \alpha_g)=\C U(g)$. 
Thus $F((\rho\otimes \alpha_g|U(g)|\rho))$ is a multiple of $(\rho\otimes \alpha_{\theta(g)}|U(\theta(g))|\rho)$. 
Since $U(g)S_0=S_0$, we have 
\begin{align*}
\lefteqn{(\rho\otimes \rho|S_0|\id)} \\
 &=(I_\rho\otimes (\rho|I_{\cM}|\alpha_g\otimes \rho))\circ((\rho\otimes \alpha_g|U(g)|\rho)\otimes I_\rho)
 \circ (\rho\otimes \rho|S_0|\id).
\end{align*}
Applying $F$ to this, we get $F((\rho\otimes \alpha_g|U(g)|\rho))=(\rho\otimes \alpha_{\theta(g)}|U(\theta(g))|\rho)$ and 
its adjoint. 
We claim $F((\rho\otimes \alpha_g|U(g)|\rho))=(\rho\otimes \alpha_{\theta(g)}|\gamma(U(g))|\rho)$, or equivalently 
$\gamma(U(g))=U(\theta(g))$. 
Indeed, applying $F$ to the two equalities 
\begin{align*}
\lefteqn{(\rho\otimes \rho|U(g)S_h|\alpha_h)} \\
 &=(I_\rho\otimes (\rho|I_{\cM}|\alpha_g\otimes \rho))\circ ((\rho\otimes \alpha_g|U(g)|\rho)\otimes I_\rho)\circ(\rho\otimes \rho|S_h|\alpha_h),
\end{align*}
\begin{align*}
\lefteqn{(\rho\otimes \rho|U(g)T|\rho)} \\
 &=(I_\rho\otimes (\rho|I_{\cM}|\alpha_g\otimes \rho))\circ ((\rho\otimes \alpha_g|U(g)|\rho)\otimes I_\rho)
 \circ(\rho\otimes \rho|T|\rho),
\end{align*}
we get $\gamma(U(g)S_h)=U(\theta(g))\gamma(S_h)$ and $\gamma(U(g)T)=U(\theta(g))\gamma(T)$. 
Choosing an orthonormal basis $\{T_i\}_{i=1}^m$ of $\Hom_\cC(\rho,\rho\otimes \rho)$,  we get
\begin{align*}
\lefteqn{\gamma(U(g))=\sum_{h\in G}\gamma(U(g)S_h)\gamma(S_h)^*+\sum_{i=1}^m\gamma(U(g)T_i)\gamma(T_i)^*} \\
 &=\sum_{h\in G}U(\theta(g))\gamma(S_hS_h^*)+\sum_{i=1}^mU(\theta(g))\gamma(T_iT_i^*)=U(\theta(g)),
\end{align*}
which shows the claim. 

We have already seen that $F$ coincides with $\gamma$ on the following morphism spaces:  
$$\Hom_\cC(\alpha_{g+h},\alpha_g\otimes \alpha_h),\quad \Hom_\cC(\rho,\alpha_g\otimes \rho),\quad 
\Hom_\cC(\rho,\rho\otimes \alpha_g),$$ 
$$\Hom_\cC(\alpha_g,\rho\otimes\rho),\quad \Hom_\cC(\rho,\rho\otimes \rho),$$ 
and their adjoints.  
To finish the proof, it suffices to show $\gamma\in \Aut_{\rho,0}(\cO_{n+m})$.  
Before showing it, we claim $\gamma\circ \alpha_g=\alpha_{\theta(g)}\circ \gamma$. 
This is easily shown for $S_h$. 
Applying $F$ to 
$$(\rho\otimes \rho|\alpha_g(T)|\rho)=((\rho|I_{\cM}|\alpha_g\otimes \rho)\otimes I_\rho)\circ ((I_{\alpha_g}\otimes (\rho\otimes\rho|T|\rho))
\circ (\alpha_g\otimes \rho|I_{\cM}|\rho),$$
we get $\gamma(\alpha_g(T))=\alpha_{\theta(g)}(\gamma(T))$. 

Applying $F$ to the two equations 
$$(\rho\otimes \rho|j_1(T)|\rho)=\sqrt{d}((\rho\otimes \rho|T|\rho)^*\otimes I_\rho)\circ (I_\rho\otimes (\rho\otimes \rho|S_0|\id)),$$
$$(\rho\otimes \rho|j_2(T)|\rho)=\sqrt{d}(I_\rho\otimes (\rho\otimes \rho|T|\rho)^*)\circ ((\rho\otimes \rho|S_0|\id)\otimes I_\rho),$$
we see that $\gamma$ commutes with $j_1$ and $j_2$. 
Thus Eq.(\ref{rhoS}) implies 
$$\gamma(\rho(S_0))=\rho(S_0)=\rho(\gamma(S_0)),$$ 
and 
$$\gamma(\rho(S_g))=\gamma(U(g)\rho(S_0)U(g)^*)=U(\theta(g))\rho(S_0)U(\theta(g))^*=\rho(S_{\theta(g)})=\rho(\gamma(S_g)).$$

It remains to show $\gamma(\rho(T))=\rho(\gamma(T))$ for $T\in \Hom_\cC(\rho,\rho\otimes \rho)$. 
We recall the two projections $P=\sum_{g\in G}S_gS_g^*$ and $Q=\sum_{i=1}^mT_iT_i^*$, which satisfies  
$\gamma(P)=P$, $\gamma(Q)=Q$, and $P+Q=I_{\cM}$. 
Then Eq.(\ref{rhoT}) implies 
$\gamma(\rho(T))P=\rho(\gamma(T))P$ and $P\gamma(\rho(T))=P\rho(\gamma(T))$. 
Note that $Q\rho(T)Q=l(T)\in \cK^2\cK^*$. 
Let $T',T'',T'''\in \Hom_\cC(\rho,\rho\otimes\rho)$. 
Then ${T''}^*{T'}^*\rho(T)T'''$ is a scalar. 
Denoting it by $\xi$, we get 
$$\xi I_\rho
=(\rho|T''|\rho\otimes \rho)^*\circ ((\rho|T'|\rho\otimes \rho)^*\otimes I_\rho)\circ(I_\rho\otimes (\rho|T|\rho\otimes \rho))
 \circ (\rho|T'''|\rho\otimes \rho).$$
Applying $F$ to it, we get ${\gamma(T'')}^*{\gamma(T')}^*\rho(\gamma(T))\gamma(T''')=\xi$, and so 
$${T''}^*{T'}^*\gamma^{-1}(\rho(\gamma(T)))T'''=\xi={T''}^*{T'}^*\rho(T)T'''.$$
This shows  $Q\gamma^{-1}(\rho(\gamma(T)))Q=Q\rho(T)Q$ and equivalently $Q\rho(\gamma(T)))Q=Q\gamma(\rho(T))Q$, 
which finishes the proof. 
\end{proof}

\begin{theorem} When $A(g)$ is a scalar $a(g)I$ for any $g\in G$, the inner automorphism group of $\cC$ is trivial. 
In consequence, 
\begin{align*}
\lefteqn{\Out(\cC)=\Aut(\cC)\cong \Aut_\rho(\cO_{n+m})/\T\cong \Aut_{\rho,0}(\cO_{n+m})/\Z_2} \\
 &\cong\{(\theta,u)\in \Aut(G)\times \cG(A,C,J);\; \\
 &a(\theta(g))=a(g),\; B_{\theta(g)}(\xi)=(u\otimes u)B_g(u^*\xi)u^*\}/\{(\id,\pm I)\}. 
\end{align*}
\end{theorem}

\begin{proof} Let $F=\alpha_h\otimes \cdot \otimes \alpha_h^{-1}$. 
Then $F(\alpha_g)=\alpha_g$ and $F(\rho)=\Ad U(h)^{-1}\circ \rho$. 
We show that $\eta$ determined by 
$$\eta_{\alpha_g}=\inpr{g}{h}I_{\alpha_g}\in \Hom_\cC(\alpha_h\otimes \alpha_g\otimes \alpha_h^{-1},\alpha_g),$$
$$\eta_\rho=(\rho|a(h)U(h)|\Ad U(h)^{-1}\circ \rho)\in \Hom_\cC(\alpha_h\otimes \rho\otimes \alpha_h^{-1},\rho),$$ 
is a tensor natural isomorphism from $F$ to $\id$. 
Indeed, it suffices to verify
$$(\eta_{\alpha_{g_1}}\otimes \eta_{\alpha_{g_2}})\circ F((\alpha_{g_1}\otimes\alpha_{g_2}|I_{\cM}|\alpha_{g_1+g_2}))
=(\alpha_{g_1}\otimes\alpha_{g_2}|I_{\cM}|\alpha_{g_1+g_2})\circ \eta_{\alpha_{g_1+g_2}},$$
$$(\eta_{\alpha_g}\otimes \eta_\rho)\circ F((\alpha_g\otimes\rho|I_{\cM}|\rho))
=(\alpha_g\otimes\rho|I_{\cM}|\rho)\circ \eta_\rho,$$
$$(\eta_\rho\otimes \eta_{\alpha_g})\circ F((\rho\otimes\alpha_g|U(g)|\rho))
=(\rho\otimes\alpha_g|U(g)|\rho)\circ \eta_\rho,$$
$$(\eta_\rho\otimes \eta_\rho)\circ F((\rho\otimes \rho|S_g|\alpha_g))=(\rho\otimes \rho|S_g|\alpha_g)\circ \eta_{\alpha_g},$$
$$(\eta_\rho\otimes \eta_\rho)\circ F((\rho\otimes \rho|T|\rho))=(\rho\otimes \rho|T|\rho)\circ \eta_\rho,$$
which are equivalent to 
$$\inpr{g_1}{h}\inpr{g_2}{h}=\inpr{g_1+g_2}{h},$$
$$\alpha_g(a(h)U(h))\inpr{g}{h}=a(h)U(h),$$
$$\inpr{h}{g}a(h)U(h)\alpha_h(U(g))=U(g)a(h)U(h),$$
$$\rho(a(h)U(h))a(h)U(h)\alpha_h(S_g)=\inpr{g}{h}S_g,$$
$$\rho(a(h)U(h))a(h)U(h)\alpha_h(T)=Ta(h)U(h).$$
It is straightforward to show these equalities. 
\end{proof}

\begin{cor} When $m=n$, 
\begin{align*}
\lefteqn{\Out(\cC)=\Aut(\cC)} \\
 &\cong \{\theta\in \Aut(G);\; \inpr{\theta(g)}{\theta(h)}=\inpr{g}{h},\;a(\theta(g))=a(g),\; b(\theta(g))=b(g)\}.
\end{align*}
\end{cor}

Let 
$$\Gamma_\cC=\{(\theta,u)\in \Aut(G)\times \cG(A,C,J);\; B_{\theta(g)}(\xi)=(u\otimes u)B_g(u^*\xi)u^*\},$$
which is isomorphic to $\Aut_{\rho,0}(\cO_{n+m})$.  
For $x\in \Gamma_\cC$, we denote $\gamma_x$ instead of $\gamma_{(1,x)}$ for simplicity. 

\begin{remark} Assume that $A(g)$ is a scalar for any $g\in G$. 
Then there exists a $\Out(\cC)$-graded extension $\cC'$ of $\cC$ such that $\Out(\cC)$ 
on $\cC\subset \cC'$ is implemented by inner automorphisms of $\cC'$.  
Indeed, let $\mu=\gamma_{(\id,-I)}$, and consider the crossed product $\cM\rtimes_\mu \Z_2$, 
which is the von Neumann algebra generated by $\cM$ and a period two unitary $\lambda$ satisfying 
$\lambda x=\mu(x)\lambda $ for any $x\in \cM$. 
We can extend $\rho$, $\alpha$, and $\gamma$ to $\cM\rtimes_\mu \Z_2$ by $\trho(\lambda)=\lambda$, $\talpha_g(\lambda)=\lambda$, and 
$\tgamma(\lambda)=\lambda$. 
We have $\tgamma_{(\id,-I)}=\Ad \lambda$. 
Let $\hat{\mu}$ be the dual action of $\mu$, which is a period two automorphism of $\cM\rtimes_\mu \Z_2$ given by 
$$\hat{\mu}(x)=\left\{
\begin{array}{ll}
x , &\quad x\in \cM \\
-\lambda , &\quad x=\lambda
\end{array}
\right.. 
$$
Then it is straightforward to show 
$$\trho^2(x)=\sum_{g\in G}S_g\talpha_g(x)S_g^*+\sum_{g,r}T_g(e_r)\hat{\mu}\circ \trho(x)T_g(e_r)^*.$$
Now we can see that the category generated by $\hat{\mu}\circ \trho$ is equivalent to $\cC$, and we can set $\cC'$ to be 
the category generated by $\hat{\mu}\circ \trho$ and $\tgamma_{\Gamma_\cC}$. 
\end{remark}

\begin{example} Any solutions of Eq.(\ref{m=n1})-(\ref{m=n5}) given in Section \ref{m=n} have trivial symmetry, 
and the corresponding C$^*$ near-group has trivial automorphism group. 
\end{example}

\begin{example} \label{55}
For $G=\Z_5$ and $\inpr{g}{h}=\zeta_5^{gh}$, there is a unique $a(g)$, which is given by $a(g)=\zeta_5^{2g^2}$. 
There is a unique solution of Eq.(\ref{m=n1})-(\ref{m=n5}) as follows (see \cite[Example A 4]{I01} and \cite[Table 2]{EG14}): 
$c=-1$, $d=(3+3\sqrt{5})/2$, $b(0)=-1/d$, 
$$b(1)=b(4)=\frac{\zeta_5^{-1}}{\sqrt{5}},\quad b(2)=b(3)=\frac{\zeta_5}{\sqrt{5}}.$$
In particular, the solution has a symmetry $a(g)=a(-g)$ and $b(g))=b(-g)$, and therefore we have $\Out(\cC)\cong \Z_2$ 
for the corresponding near-group category $\cC$ for $G=\Z_5$ and $m=5$. 
\end{example} 

\begin{example}\label{22312} 
The solution for $G=\Z_2\times \Z_2\times \Z_3$ with $m=12$ given in Example \ref{Z2Z2Z3} has a $\Z_2$-symmetry, and 
$\Out(\cC)\cong \Z_2$ for the corresponding near-group category $\cC$. 
\end{example}

\begin{example}\label{36aut} As we have seen in Section \ref{2n}, there are exactly two C$^*$ near-group categories for $G=\Z_3$ with $m=6$. 
For each of the two categories, the solutions of Eq.(\ref{p1})-(\ref{p11}) are parametrized by $\{(x,y)\in \R^2;\; x^2+y^2=\sqrt{3}/24\}$. 
The group automorphism $\Z_3\ni g\mapsto-g\in \Z_3$ takes $(x,y)$ to $(-x,-y)$. 
The gauge group $\cG(A,C,J)$ is $O(2)$, and 
$$L=\left(
\begin{array}{cc}
1 &0  \\
0 &-1 
\end{array}
\right)
$$ acts on the solutions as $(x,y)\mapsto (x,-y)$, and the rotation 
$$R(\theta)= \left(
\begin{array}{cc}
\cos\theta &-\sin\theta  \\
\sin\theta &\cos\theta 
\end{array}
\right)$$ 
acts on the solutions as $(x,y)\mapsto (\cos4\theta x+\sin 4\theta y,-\sin 4\theta x+\cos4\theta y)$. 
We fix a solution $B_g$ corresponding to $(\sqrt{\sqrt{3}/24},0)$. 
Identifying $\Aut_{\rho,0}(\cO_9)$ with 
$$\{(\theta,u)\in \Aut(\Z_3)\times O(2);\; B_{\theta(g)}(\xi)=(u\otimes u)B_g(u^*\xi)u^*\},$$ 
we see that it is generated by the two elements 
$$(-1,R(\frac{\pi}{4})),\quad (\id, L).$$
Thus $\Aut_{\rho,0}(\cO_9)$ is isomorphic to the dihedral group $D_{16}$ of order 16, 
and the automorphism group $\Aut(\cC)$, as well as the outer automorphism group $\Out(\cC)$, 
of the corresponding near-group category $\cC$ are isomorphic to $D_8$. 
\end{example}

\subsection{The case of $m>|G|$} 
In this section, we assume that $\Gamma$ is a finite subgroup of the group 
$$\{u\in \cG(A,C,J);\; B_{g}(\xi)=(u\otimes u)B_g(u^*\xi)u^*\}.$$ 
We regard $\gamma$ as an outer action of $\Gamma$ on $\cM$ via $\gamma_u=\gamma_{(1,\id,u)}$. 
Recall that 
$$\cN=\rho(\cM)\vee\{U(g)\}_{g\in G}\subset \cM$$ is a $2^G_l1$ subfactor with $l=m/n$. 
The purpose of this subsection is to determine the structure of the subfactor $\cN\rtimes_\gamma \Gamma\subset \cM\rtimes_\gamma \Gamma$. 
Let $\{\lambda_u\}_{u\in \Gamma}$ be the implementing unitary representation in $\cM\rtimes_\gamma \Gamma$. 
Then we can extend $\alpha$ and $\rho$ to $\cM\rtimes_\gamma \Gamma$ by $\talpha_g(\lambda_u)=\lambda_u$ and $\trho(\lambda_u)=\lambda_u$. 
We would like to describe the structure of the fusion category generated by $\trho$. 

\begin{lemma} Let the notation be as above. 
\begin{itemize}
\item[(1)] The two commuting actions $\alpha$ and $\gamma$ give an outer action of $G\times \Gamma$ on $\cM$. 
\item[(2)] $\trho$ is irreducible. 
\item[(3)] $\talpha$ is an outer action of $G$. 
\end{itemize}
\end{lemma}

\begin{proof} (1) The statement follows from Theorem \ref{Appouter} below. 

(2) Assume $X\in (\trho,\trho)$. 
Then $X$ is uniquely expanded as $X=\sum_{u\in \Gamma}X_u\lambda_u$ with $X_u\in \cM$. 
For any $x\in \cM$, we have 
$$\sum_{u\in \Gamma}\rho(x)X_u\lambda_u=\sum_{u\in \Gamma}X_u\lambda_u\rho(x)=\sum_{u\in \Gamma}X_u\gamma_u\circ \rho(x)\lambda_u,$$
and so $\rho(x)X_u=X_u\gamma_u\circ \rho(x)$. 
The Frobenius reciprocity and (1) show that 
$(\gamma_u\rho,\rho)\cong(\gamma_u,\rho^2)$ is trivial except for $u=I$, and $X\in (\rho,\rho)=\C$, which 
means that $\trho$ is irreducible. 

(3) Assume that $Y\in \cM\rtimes_\gamma \Gamma$ satisfies $Yx=\talpha_g(x)Y$ for any $x\in \cM\rtimes_\gamma \Gamma$. 
Then $Y$ is expanded as $Y=\sum_{u\in \Gamma}Y_u\lambda_u$ with $Y_u\in \cM$, and we can show that 
$Y_u\in (\gamma_u,\alpha_g)=\C\delta_{u,I}\delta_{g,0}$ as before. 
This implies that $Y=0$ unless $g=0$, and so $\talpha$ is an outer action. 
\end{proof}

Since we are working on a non-commutative group $\Gamma$, we recall the dual action of $\gamma$ as a Roberts action of 
the dual $\hat{\Gamma}$ of $\Gamma$. 
Recall that a Hilbert space $\cH$ in $\cM$ is a closed subspace of $\cM$ in the weak operator topology such that $W^*V$ is a scalar 
for any $V,W \in \cH$. 
The space $\cH$ is equipped with the inner product $\inpr{V}{W}$ for $V,W\in \cH$ given by $W^*V$. 
Let $\{V_i\}_{i}$ be an orthonormal basis of a Hilbert space $\cH$ in $\cM$. 
The support of $\cH$ is defined by 
$$\sum_{i}V_iV_i^*,$$
which does not depend on the choice of the orthonormal basis. 
When the support is $I$, we say that $\cH$ has full support. 
For a full support Hilbert space $\cH$ in $\cM$, we can define a unital endomorphism $\rho_\cH$ of $\cM$ by 
$$\rho_\cH(x)=\sum_{i}V_ixV_i^*,$$
whose definition does not depend on the choice of the orthonormal basis either. 
When $\cH$ is globally invariant under the $\Gamma$-action $\gamma$, the endomorphism 
$\rho_\cH$ preserves the fixed point algebra $\cM^\Gamma=\{x\in \cM;\; \gamma_u(x)=x,\;\forall u\in \Gamma\}$, 
and we denote by $\check{\gamma}_\cH$ the restriction of $\rho_{\cH}$ to $\cM^\Gamma$. 
The equivalence class of $\check{\gamma}_{\cH}$ depends only on the unitary equivalence class of $\cH$ as a representation space of $\Gamma$. 
Since $\gamma$ is outer, if the representation $\cH$ is irreducible, so is $\check{\gamma}_{\cH}$. 
For any finite dimensional (not necessary irreducible) unitary representation $\pi$ of $\Gamma$, there exists a full support Hilbert space 
$\cH_\pi$ in $\cM$ so that the $\Gamma$-action on $\cH$ is equivalent to $\pi$. 
We use the notation $\check{\gamma}_{\pi}$ instead of $\check{\gamma}_{\cH_\pi}$ for simplicity. 
Let $\nu:\cM^\Gamma\hookrightarrow \cM$ be the inclusion map. 
Then we have $[\nu\check{\gamma}_{\pi}]=\dim \pi [\nu]$ by definition, and so 
$$[\overline{\nu}\nu]=\bigoplus_{\pi\in \hat{\Gamma}}\dim \pi[\check{\gamma}_{\pi}],$$
by the Frobenius reciprocity. 
We call $\{\check{\gamma}_{\pi}\}_{\pi\in \hat{\Gamma}}$ the pre-dual action of $\gamma$. 

Let $\{V(\pi)_i\}_{i=1}^{\dim \pi}$ be an orthonormal basis of $\cH_\pi$. 
Applying the above argument to the second dual action $\gamma_u\otimes \Ad \varrho_u$ on $\cM\otimes B(\ell^2(\Gamma))$, 
where $\{\varrho_{u}\}_{u\in \Gamma}$ is the right regular representation of $\Gamma$, 
we can define the dual action $\hat{\gamma}_{\pi}\in \End(\cM\rtimes_\gamma \Gamma)$, which is explicitly given by the following formula: 
$$\hat{\gamma}_\pi(x)=\left\{
\begin{array}{ll}
\sum_{i=1}^{\dim \pi}V(\pi)_ixV(\pi)_i^* , &\quad x\in \cM \\
\lambda_u , &\quad x=\lambda_u
\end{array}
\right.. 
$$ 
This can been seen by identifying $\lambda_u\in \cM\rtimes_\gamma \Gamma$ with $I\otimes \lambda_u\in \cM\otimes B(\ell^2(\Gamma))$, 
where $\{\lambda_u\}$ in the latter is the left regular representation of $\Gamma$,  and 
$\cM$ with $\pi_\gamma(\cM)\subset \cM\otimes B(\ell^2(\Gamma))$,  
where 
$$\pi_\gamma(x)\xi(u)=\gamma_{u^{-1}}(x)\xi(u),$$
and by using the Hilbert space $\{V\otimes I;\; V\in \cH_\pi\}$ in $\cM\otimes B(\ell^2)$. 
In fact $V(\pi)_i\otimes I$ commutes with $I\otimes \lambda_u$ and we have
$$\sum_{i=1}^{\dim \pi}(V(\pi)_i\otimes I)\pi_{\gamma}(x)(V(\pi)_i^*\otimes I)=\pi_{\gamma}(\sum_{i=1}^{\dim \pi}V(\pi)_ixV(\pi)_i^*).$$

\begin{remark}\label{dualaction} When $\tau\in \Hom(\Gamma,\T)$ is a 1-dimensional representation, the Hilbert space $\cH_\tau$ is spanned by a single 
unitary $V_\tau$ with $\gamma_u(V_\tau)=\tau(u)V_u$. 
Thus we have 
$$\Ad V_\tau^*\circ \hat{\gamma}_\tau(x)=\left\{
\begin{array}{ll}
x , &\quad x\in \cM \\
\tau(u)\lambda_u , &\quad x=\lambda_u
\end{array}
\right.,
$$
which recovers the usual dual action in the commutative case.  
\end{remark}

\begin{lemma} Let the notation be as above. 
\begin{itemize} 
\item[(1)] $[\hgamma_\pi\trho]=[\trho\hgamma_\pi]$. 
\item[(2)] $[\hgamma_\pi\talpha_g]=[\talpha_g\hgamma_\pi]$. 
\item[(3)] For $\pi_1,\pi_2\in \hat{\Gamma}$ and $g_1,g_2\in G$, we have $\dim(\hgamma_{\pi_1}\talpha_{g_1},\hgamma_{\pi_2}\talpha_{g_2})
=\delta_{\pi_1,\pi_2}\delta_{g_1,g_2}$. 
\end{itemize}
\end{lemma}

\begin{proof} (1) Choosing a basis $\{V(\pi)_i\}_{i=1}^{\dim \pi}$ of $\cH_\pi$, we set 
$$\sum_{i=1}^{\dim \pi}V(\pi)_i\rho(V(\pi)_i^*),$$ 
which is a unitary belonging to $(\trho\hgamma_\pi,\hgamma_\pi\trho)$. (2) can be shown in a similar way. 

(3) By the Frobenius reciprocity, we have 
$$\dim(\hgamma_{\pi_1}\talpha_{g_1},\hgamma_{\pi_2}\talpha_{g_2})=\dim(\hgamma_{\overline{\pi_2}\otimes \pi_1},\talpha_{g_2-g_1}).$$
Thus it suffices to show $\dim(\hgamma_{\tau},\talpha_g)=\delta_{\tau,1}\delta_{g,0}$ for $\tau\in \Hom(\Gamma,\T)$ and $g\in G$. 
Thanks to Remark \ref{dualaction}, we can further replace $\hgamma_\tau$ with $\Ad V_\tau^*\circ \hgamma_\tau$. 
Assume that $X\in \cM\rtimes_\gamma\Gamma$ satisfies $X\Ad V_\tau^*\circ \hgamma_\tau(x)=\talpha_g(x)X$ for any $x\in \cM\rtimes_\gamma\Gamma$. 
Then $X$ is expanded as $X=\sum_{u\in \Gamma}X_u\lambda_u$ with $X_u\in \cM$, and $X_u\in (\gamma_u,\alpha_g)=\C\delta_{u,I}\delta_{g,0}$. 
Therefore $X$ is a scalar multiple of $\delta_{g,e}$.  
Setting $x=\lambda_u$, we get $\tau=1$ if $X\neq 0$. 
\end{proof}

We denote by $\pi_0$ the defining representation of $\Gamma$ on $\cK_0$, which is not necessarily irreducible. 

\begin{theorem} Let the notation be as above. 
\begin{itemize}
\item [(1)] $\trho^2$ is decomposed as 
$$[\trho^2]=\bigoplus_{g\in G}[\talpha_g]\oplus n[\hgamma_{\pi_0}\trho].$$
\item [(2)] Let $\kappa :\cN\rtimes_\gamma \Gamma\hookrightarrow \cM\rtimes_\gamma \Gamma$ be the inclusion map. 
Then 
$$[\kappa\overline{\kappa}]=[\id]\oplus[\hgamma_{\pi_0}\trho].$$
\item [(3)]
For any $\pi_1,\pi_2\in \hat{\Gamma}$, we have $\dim (\hgamma_{\pi_1}\trho,\hgamma_{\pi_2}\trho)=\delta_{\pi_1,\pi_2}$. 
\end{itemize}
\end{theorem}

\begin{proof} (1) 
Since $\gamma_u$ acts on $\{S_g\}_{g\in G}$ trivially, we can easily see $S_g\in (\talpha_g,\trho^2)$. 
As in Section \ref{PEIC}, we choose an orthonormal basis $\{T_g(e_i)\}_{i=1}^l$ of $\cK_0$, and set 
$$e_{h,g}=\sum_{i=1}^l T_g(e_i)T_h(e_i)^*.$$
Then $\{e_{g,h}\}_{g,h\in G}$ is a system of matrix units in $\cM^\gamma$. 
We choose an isometry $R\in \cM^\gamma$ whose range projection is $e_{0,0}$, and set $V(\pi_0)_i=R^*T_0(e_i)$. 
Let $\cH_{\pi_0}$ be the linear span of $\{V(\pi_0)_i\}_{i=1}^l$. 
Then $\cH_{\pi_0}$ is a Hilbert space in $\cM$ with full support such that the restriction of $\gamma$ to $\cH_{\pi_0}$ 
is equivalent to $\pi_0$. 
Therefore we may assume that $\hgamma_{\pi_0}$, which is defined up to equivalence in any case, is given by  
$$\hat{\gamma}_{\pi_0}(x)=\left\{
\begin{array}{ll}
\sum_{i=1}^{\dim \pi}V(\pi_0)_ixV(\pi_0)_i^* , &\quad x\in \cM \\
\lambda_u , &\quad x=\lambda_u
\end{array}
\right.. 
$$ 
We set $R_g=e_{g,0}R$, which is an isometry in $\cM^\gamma$. 
We claim $R_g\in (\hgamma_{\pi_0}\trho,\trho^2)$. 
Indeed, for $x\in \cM$, we have 
\begin{align*}
\lefteqn{\trho^2(x)R_g=\rho^2(x)e_{g,0}R_0=\sum_{i=1}^lT_g(e_i)\trho(x)T_0(e_i)^*R_0} \\
 &=\sum_{i=1}^l e_{g,0}RR^*T_0(e_i)\trho(x)T_0(e_i)^*R_0=R_g\hgamma_{\pi_0}\circ \trho(x),\\
\end{align*}
and 
$$\trho^2(\lambda_u)R_g=\lambda_uR_g=R_g\lambda_u=R_g\hgamma_{\pi_0}\circ \trho(\lambda_u),$$
which shows the claim. 
Since 
$$\sum_{g\in G}R_gR_g^*=\sum_{g\in G}e_{g,0}e_{0,0}e_{0,g}=\sum_{g\in G}e_{g,g}=\sum_{g,i}T_g(e_i)T_g(e_i)^*,$$
we get the statement. 

(2) Let $\iota:\cN \hookrightarrow \cM$ be the inclusion map. Then the proof of Theorem \ref{classify2G1} shows 
$(\iota,\rho\iota)=(j_1\circ j_2)^*\cK_0$. 
This and the above argument shows the statement. 

(3) The statement follows from (1) because 
\begin{align*}
\lefteqn{\dim(\hgamma_{\pi_1}\trho,\hgamma_{\pi_2}\trho)=\dim(\hgamma_{\overline{\pi_2}\otimes \pi_1},\trho^2)} \\
 &=\sum_{g\in G}\dim(\hgamma_{\overline{\pi_2}\otimes \pi_1},\talpha_g)+n\dim(\hgamma_{\overline{\pi_0}\otimes \overline{\pi_2}\otimes \pi_1},\trho)
 =\sum_{g\in G}\dim(\hgamma_{\pi_1},\hgamma_{\pi_2}\talpha_g)=\delta_{\pi_1,\pi_2}.
\end{align*}
\end{proof}

\begin{example}\label{36equi} We apply the above theorem to Example \ref{36aut}. As $\Gamma$, we take the group generated by 
$$\left(
\begin{array}{cc}
1 &0  \\
0 &-1 
\end{array}
\right),\quad 
\left(
\begin{array}{cc}
0 &-1  \\
1 &0 
\end{array}
\right),
$$
which is the dihedral group $D_8$. 
The defining representation $\pi_0$ is real and irreducible, and the tensor product $\pi_0\otimes\pi_0$ is decomposed into 
1-dimensional representations, 
$$\pi_0\otimes \pi_0\cong\bigoplus_{\tau\in \Hom(D_8,\T)}\tau.$$
Since $\tau\otimes \pi_0\cong \pi_0$ for any $\tau\in \Hom(D_8,\T)$, we get 
\begin{align*}
[(\hgamma_{\pi_0}\trho)^2]&=[\hgamma_{\pi_0\otimes \pi_0}\trho^2]
=\sum_{g\in \Z_3,\;\tau\in\Hom(D_8,\T)}[\hgamma_\tau\talpha_g]
+3\sum_{\tau\in \Hom(D_8,\T)}[\hgamma_{\tau\otimes \pi_0}\trho] \\
 &=\sum_{g\in \Z_3,\;\tau\in\Hom(D_8,\T)}[\hgamma_\tau\talpha_g]+12[\hgamma_{\pi_0}\trho].
\end{align*}
This means that the fusion category generated by $\hgamma_{\pi_0}\trho$ is a near-group category with the group $\Z_2\times \Z_2\times \Z_3$ and 
the multiplicity parameter $m=12$. 
\end{example}

\subsection{The case of $m=|G|$}
In this subsection, we assume that $\cC$ is a C$^*$ near-group category with a finite abelian group $G$ and 
$m=n$, and we use the same notation as in Section \ref{m=n}. 
In this case, the group $\Aut(\cC)=\Out(\cC)$ is isomorphic to the set of group automorphisms of $G$ leaving  
the triple $(\inpr{\cdot}{\cdot},a,b)$ invariant. 
For simplicity, we assume that we have such an automorphism $\theta$ of order two, 
which covers all the known examples (see Example \ref{55}, Example \ref{22312}). 
Then $\gamma=\gamma_{(\theta,I)}$ is an outer automorphism of the factor $\cM$ of order two 
commuting with $\rho$, and it satisfies the relation $\gamma\circ \alpha_g=\alpha_{\theta(g)}\circ \gamma$. 
Thus Theorem \ref{Appouter} implies that $\alpha$ and $\gamma$ give an outer action of 
$G\rtimes \Z_2$ on $\cM$. 

We denote by $\lambda$ the implementing unitary in the crossed product $\cM\rtimes_\gamma \Z_2$.  
As before $\rho$ extends to $\trho\in \End(\cM\rtimes_\gamma\Z_2)$ by $\trho(\lambda)=\lambda$. 
Let $G^\theta$ be the set of automorphisms of $G$ fixed by $\gamma$.  
Then $\alpha_g$ with $g\in G^\theta$ also extends to $\talpha_g\in \Aut(\cM\rtimes_\gamma\Z_2)$ by 
$\talpha_g(\lambda)=\lambda$. 
We denote by $\hgamma$ the dual action of $\gamma$, which is identified with the generator of the dual action too. 
As in the previous subsection, we can show that $\trho$ is irreducible with $[\hgamma\trho]=[\trho\hgamma]\neq[\trho]$, 
and that $\talpha$ and $\hgamma$ give an outer action of $G^\theta\times \Z_2$. 
 
\begin{lemma} $(\trho^2,\trho^2)=(\rho^2,\rho^2)^\gamma$. 
\end{lemma}

\begin{proof} Assume $X+Y\lambda\in(\trho^2,\trho^2)$ with $X,Y\in \cM$. 
Then $X\in (\rho^2,\rho^2)^\gamma$ and 
$$Y\in (\rho^2\gamma,\rho^2)\cong (\gamma,\rho^4)=\{0\}.$$ 
\end{proof}

Recall that we have 
$$(\rho^2,\rho^2)=\bigoplus_{g\in G}\C S_gS_g^*\oplus \B(\cK).$$
We choose a subset $\Lambda\subset G\setminus G^\theta$ satisfying $\Lambda\cap \theta(\Lambda)=\emptyset$ 
and $\Lambda\cup \theta(\Lambda)=G\setminus G^\theta$. 
Then the above lemma shows 
$$(\trho^2,\trho^2)=\bigoplus_{g\in G^\theta}\C S_gS_g^*\oplus 
\bigoplus_{g\in \Gamma}\C (S_gS_g^*+S_{\theta(g)}S_{\theta(g)}^*)
\oplus \B(\cK)^\gamma.$$
Thus for $g\in \Lambda$, the projection $S_gS_g^*+S_{\theta(g)}S_{\theta(g)}^*$ 
is minimal in $(\trho^2,\trho^2)$, and it corresponds to an irreducible component of $\trho^2$. 
We choose an isometry $R_g\in \cM^\gamma$ whose range projection is $S_gS_g^*+S_{\theta(g)}S_{\theta(g)}^*$, 
and defined $\pi_g\in \End (\cM\rtimes_\gamma\Z_2)$ by $\pi_g(x)=R_g^*\trho^2(x)R_g$.  
Then  
$$\pi_g(x)=\left\{
\begin{array}{ll}
R_g^*(S_g\alpha_g(x)S_g^*+S_{\theta(g)}\alpha_{\theta(g)}(x)S_{\theta(g)}^*)R_g , &\quad x\in \cM \\
\lambda , &\quad x=\lambda
\end{array}
\right..
$$ 

For $g\in G^\theta$, we have $T_g\in (\trho,\trho^2)$. 
For $g\in G\setminus G^\theta$, we have 
$$\frac{1}{\sqrt{2}}(T_g+T_{\theta(g)})\in (\trho,\trho^2),$$
$$\frac{1}{\sqrt{2}}(T_g-T_{\theta(g)})\in (\hgamma\trho,\trho^2).$$
Thus it is straightforward to show the following. 

\begin{theorem} Let the notation be as above. 
\begin{itemize}
\item[(1)] $\trho^2$ is decomposed as 
$$[\trho^2]=\sum_{g\in G^\theta}[\talpha_g]+\sum_{g\in \Lambda} [\pi_g]+\frac{|G|+|G^\theta|}{2}[\trho]
+\frac{|G|-|G^\theta|}{2}[\hgamma\trho].$$
\item[(2)] Let $\kappa: \cN\rtimes_\gamma \Z_2\hookrightarrow \cM\rtimes_\gamma\Z_2$ be the inclusion map. 
Then 
$$[\kappa\overline{\kappa}]=[\id]+[\trho].$$ 
\end{itemize}
\end{theorem}

About the fusion rules, we have the following. 

\begin{theorem} Let the notation be as above. 
\begin{itemize}
\item[(1)] The tensor category generated by 
$\{\hgamma\}\cup\{\talpha_g\}_{g\in G^\theta}\cup\{\pi_g\}_{g\in \Lambda}$ 
is equivalent to the representation category of $\hG\rtimes_{\theta}\Z_2$. 
\item[(2)] For $g\in G^\theta$, we have 
$$[\talpha_g][\trho]=[\trho][\talpha_g]=[\trho].$$
\item[(3)] For $g\in \Lambda$, we have 
$$[\pi_g][\trho]=[\trho][\pi_g]=[\trho]+[\hgamma\trho].$$
\end{itemize}
\end{theorem} 

\begin{proof} (1) Let $\cM^G$ be the fixed point algebra 
$$\cM^G=\{x\in \cM;\;\alpha_g(x)=x,\;\forall g\in G\},$$
and let $\iota_1:\cM^G \hookrightarrow \cM$ and $\iota_2:\cM\hookrightarrow \cM\rtimes_\gamma\Z_2$ 
be the inclusion maps. 
Then it is easy to show that $\cM^G\subset \cM\rtimes_\gamma \Z_2$ is a crossed product inclusion 
by a $\hG\rtimes_\theta \Z_2$-action, and 
the tensor category generated by $\iota_2\iota_1\overline{(\iota_2\iota_1)}$ is equivalent to 
the representation category of $\hG\rtimes_\theta \Z_2$. 

By construction, we have $[\talpha_g][\iota_2]=[\iota_2][\alpha_g]$ for $g\in G^\theta$, and 
$$[\pi_g][\iota_2]=[\iota_2][\alpha_g]+[\iota_2][\alpha_{\theta(g)}],$$
for $g\in \Lambda$. 
Thus we get the following for $g\in \Lambda$ by the Frobenius reciprocity: 
$$[\pi_g]=[\iota_2\alpha_g\overline{\iota_2}]=[\iota_2\alpha_{\theta(g)}\overline{\iota_2}].$$ 
Now we get 
\begin{align*}
\lefteqn{[\iota_2\iota_1\overline{(\iota_2\iota_1)}]=\sum_{g\in G}[\iota_2\alpha_g\overline{\iota_2}]
=\sum_{g\in G^\theta}[\talpha_g][\iota_2\overline{\iota_2}]+\sum_{g\in \Lambda}2[\pi_g]} \\
 &=\sum_{g\in G^\theta}[\talpha_g]+\sum_{g\in G^\theta}[\talpha_g\hgamma]+\sum_{g\in \Lambda}2[\pi_g],
\end{align*}
which shows the statement.

(2) We have $\talpha_g\circ \trho=\trho$, and taking its adjoint, we get the statement. 

(3) By construction, we have   
$$\frac{1}{\sqrt{2}}R_g^*(S_g+S_{\theta(g)})\in (\trho,\pi_g\trho),$$
$$\frac{1}{\sqrt{2}}R_g^*(S_g-S_{\theta(g)})\in (\hgamma\trho,\pi_g\trho),$$
which shows $[\pi_g\trho]=[\trho]+[\hgamma\trho]$. 
Taking its adjoint, we get the statement.
\end{proof}

\begin{remark} The endomorphism $\iota_2\alpha_g\overline{\iota_2}$ corresponds to the induced representation 
$$\mathrm{Ind}_{\hG}^{\hG\rtimes_\theta \Z_2}g,$$
where $g\in G=\hat{\hG}$ is regarded as a one-dimensional representation of $\hG$. 
\end{remark}

From the above two theorems, it is easy to determine the principal graphs of the subfactor 
$\cN\rtimes_\gamma \Z_2\subset \cM\rtimes_\gamma\Z_2$, in particular, for Example \ref{55} and Example \ref{22312}.
We leave it for the reader.   

\section{Appendix} Let $\cO_{n+m}$ be the Cuntz algebra with canonical generators $\{S_i\}_{i\in J}$, where 
$J=\{1,2,\ldots,n+m\}$.  
Let $J_1=\{1,2,\ldots,m\}$ and $J_2=\{m+1,m+2,\ldots,n+m\}$. 
We consider a weighted gauge action $\gamma :\T \rightarrow \Aut(\cO_{n+m})$ defined by
$$\gamma_{t}(S_i)=\left\{
\begin{array}{ll}
e^{\ii t}S_i , &\quad i\in J_1  \\
e^{2\ii t}S_i , &\quad i\in J_2
\end{array}
\right.
.$$
Then $\gamma$ has a unique KMS-state $\varphi$ with the inverse temperature $\log d$, where $d=\frac{m+\sqrt{m^2+4n}}{2}$ (see \cite{E80}). 
The weak closure $\cM$ of $\cO_{n+m}$ in the GNS representation of $\varphi$ is the Powers factor 
of type III$_{1/d}$. 
With identification $\{S_g\}=\{S_i\}_{i\in I_2}$ and $\cK=\Span\{S_i\}_{i\in I_1}$, 
the Cuntz algebra model $(\alpha,\rho)$ of a C$^*$ near-group category $\cC$ can extend to $\cM$, which 
gives a realization of $\cC$ in $\End(\cM)$ (see \cite{I93}). 
The purpose of this appendix is to show that the symmetry group $\Aut_\rho(\cO_{n+m})$ is also realized in 
$\Aut(\cM)\subset \End(\cM)$, and is even faithfully realized in $\Out(\cM)$. 

Let $\beta \in \Aut(\cO_{n+m})$ be an automorphism globally preserving $\Span\{S_i\}_{i\in I_1}$ and 
$\Span\{S_i\}_{i\in I_2}$. 
Since $\beta$ commutes with $\gamma_t$, it preserves $\varphi$ and extends to $\cM$, which is still denoted by $\beta$. 

\begin{theorem}\label{Appouter} Let $\beta\in \Aut(\cM)$ be as above. 
Then $\beta$ is inner if and only if $\beta=\id$. 
\end{theorem}

\begin{proof} We use the same symbol $\varphi$ and $\gamma$ for their normal extension to $\cM$. 
Then we have $\sigma^\varphi_t=\gamma_{-t\log d}$, and the centralized $\cM_\varphi$ is the hyperfinite 
II$_1$ factor with a unique trace $\tau=\varphi|{\cM_\varphi}$. 
We set 
$$w(i)=\left\{
\begin{array}{ll}
1 , &\quad i\in J_1 \\
2 , &\quad i\in J_2
\end{array}
\right.. 
$$ 
For a $k$-tuple $\xi=(\xi_1,\xi_2,\ldots,\xi_k)\in J^k$, we set $S_\xi=S_{\xi_1}S_{\xi_2}\cdots S_{\xi_k}$, 
$r(\xi)=\xi_k$, and 
$$w(\xi)=\sum_{i=1}^k w(\xi_i).$$
Then we have $\gamma_t(S_\xi)=e^{\ii w(\xi)t}S_\xi$. 
We introduce finite dimensional $*$-subalgebras of $\cM_\varphi$ by 
$$A_{k,1}=\Span\{S_\xi S_\eta^*;\;w(\xi)=w(\eta)=k\},$$
$$A_{k,2}=\Span\{S_\xi S_\eta^*;\; w(\xi)=w(\eta)=k+1,\; r(\xi),r(\eta)\in I_2\},$$
$$A_k=A_{k,1}\oplus A_{k,2}.$$
Then $\{A_k\}_{k=1}^\infty$ is an increasing sequence of finite dimensional $*$-subalgebras of $\cM_\varphi$ 
whose union is dense in $\cM_\varphi$ in the strong $*$-topology (see \cite{I93}). 
We denote by $E_k$ the $\tau$-preserving conditional expectation from $\cM_\varphi$ onto $A_k$. 

Assume that $\beta$ is inner, that is, there exists a unitary $u\in \cM$ satisfying $ux=\beta(x)u$ for any $x\in \cM$. 
Since $\beta$ commutes with $\gamma_t$, its Fourier coefficient 
$$u_k=\frac{1}{2\pi}\int_0^{2\pi}e^{-\ii kt}\gamma_t(u)dt$$
satisfies $u_kx=\beta(x)u_k$ too. 
This implies that $u_n$ is a multiple of a unitary. 
On the other hand, the KMS condition implies 
$$\varphi(u_ku_k^*)=\frac{1}{d^k}\varphi(u_ku_k^*),$$
which implies that $u_k=0$ for any $k\neq 0$. 
Thus $u$ belongs to $\cM_\varphi$. 

Let $v_k=E_k(u)$. 
Then we have $v_kx=\beta(x)v_k$ for any $x\in A_k$, and 
the sequence $\{v_k\}_{k=1}^\infty$ of contractions converges to $u$ in the $L^2$-topology with respect to $\tau$. 
We may assume $\beta(S_i)=\omega_iS_i$ for any $i\in J$ by choosing appropriate orthonormal bases of $\Span\{S_i\}_{i\in J_1}$ and 
$\Span\{S_i\}_{i\in J_2}$. 
For $\xi=(\xi_1,\xi_2,\ldots,\xi_k)\in J^k$, we set 
$$\omega_{\xi}=\omega_{\xi_1}\omega_{\xi_2}\cdots \omega_{\xi_k}.$$
Then there exists two numbers $c_{k,1},c_{k,2}\in \C$ satisfying 
$$v_k=c_{k,1}\sum_{w(\xi)=k}\omega_\xi S_{\xi}S_{\xi}^*+c_{k,2}\sum_{w(\xi)=k-1,\;i\in J_2}\omega_\xi\omega_i S_\xi S_iS_i^*S_\xi^*. $$
On the other hand, we have the martingale condition $E_k(v_{k+1})=v_k$. 
Simple computation shows  
$$E_k(S_{\xi}S_iS_i^*S_\xi^*)=\left\{
\begin{array}{ll}
\frac{1}{d}S_\xi S_xi^* , &\quad w(\xi)=k,\;i\in I_1 \\
S_\xi S_iS_i^*S_\xi^* , &\quad w(\xi)=k-1,\; i\in I_2  \\
\frac{1}{d^2}S_\xi S_{\xi}^* , &\quad w(\xi)=k,\;i\in I_2
\end{array}
\right.
$$
Thus 
\begin{align*}
\lefteqn{v_k=E_k(v_{k+1})} \\
 &=c_{k+1,1}\sum_{w(\xi)=k+1}\omega_\xi E_k(S_{\xi}S_{\xi}^*)+c_{k+1,2}\sum_{w(\xi)=k,\;i\in J_2}\omega_\xi\omega_i E_k(S_\xi S_iS_i^*S_\xi^*), \\
 &=c_{k+1,1}\sum_{w(\xi)=k,\;i\in J_1}\omega_\xi \omega_i E_k(S_{\xi}S_iS_i^*S_{\xi}^*)
 +c_{k+1,1}\sum_{w(\xi)=k-1,\;i\in J_2}\omega_\xi \omega_i E_k(S_{\xi}S_iS_i^*S_{\xi}^*)\\
 &+\frac{c_{k+1,2}}{d^2}\sum_{w(\xi)=k,\;i\in J_2}\omega_\xi\omega_i S_\xi S_\xi^*\\
 &=\frac{c_{k+1,1}}{d}\sum_{w(\xi)=k,\;i\in J_1}\omega_\xi \omega_i S_{\xi}S_{\xi}^*
 +c_{k+1,1}\sum_{w(\xi)=k-1,\;i\in J_2}\omega_\xi \omega_i S_{\xi}S_iS_i^*S_{\xi}^*\\
 &+\frac{c_{k+1,2}}{d^2}\sum_{w(\xi)=k,\;i\in J_2}\omega_\xi\omega_i S_\xi S_\xi^*,
\end{align*}
and so 
$$\left(
\begin{array}{c}
c_{k,1}  \\
c_{k,2} 
\end{array}
\right)
=\left(
\begin{array}{cc}
\frac{mz}{d} &\frac{nw}{d^2}  \\
1 &0 
\end{array}
\right)
\left(
\begin{array}{c}
c_{k+1,1}  \\
c_{k+1,2} 
\end{array}
\right),
$$
where 
$$z=\frac{1}{m}\sum_{i\in J_1}\omega_i,$$
$$w=\frac{1}{n}\sum_{i\in J_2}\omega_i.$$
Let $\Gamma$ be the above matrix. 
Then the two eigenvalues of $\Gamma$ are 
$$\lambda_{\pm}=\frac{mz\pm \sqrt{m^2z^2+4nw}}{2d}.$$
Unless $|z|=1$ and $w=z^2$, we have $|\lambda_{\pm}|<1$, and 
$$\left(
\begin{array}{c}
c_{k,1}  \\
c_{k,2} 
\end{array}
\right)
=\Gamma^l
\left(
\begin{array}{c}
c_{k+l,1}  \\
c_{k+l,2} 
\end{array}
\right)
\to
\left(
\begin{array}{c}
0  \\
0 
\end{array}
\right)\quad (l\to\infty),
$$
for we have $|c_{k+l,1}|\leq 1$ and $|c_{k+l,2}|\leq 1$, which is contradiction. 

Assume $|z|=1$ and $w=z^2$ now. 
Then we have $\omega_i=z$ for any $i\in J_1$ and $\omega_i=z^2$ for any $i\in J_2$, and 
the restriction of $\beta$ to $\cM_\varphi$ is trivial. 
This implies that $u\in \cM_\varphi\cap\cM_\varphi'=\C$, and $\beta=\id$. 
\end{proof} 


\begin{thebibliography}{99}

\bibitem{AMP15} Afzaly N.; Morrison S.; Penneys D. 
\textit{The classification of subfactors with index at most $5\frac{1}{2}$.} 
preprint, arXiv:1509.00038. 

\bibitem{BPMS12} Bigelow, S.; Peters, E.; Morrison, S.; Snyder, N. 
\textit{Constructing the extended Haagerup planar algebra.} 
Acta Math. \textbf{209} (2012), no. 1, 29--82.

\bibitem{BKLR15} Bischoff, M.; Kawahigashi, Y.; Longo, R.; Rehren, K.-H. 
\textit{Tensor categories and endomorphisms of von Neumann algebras (with applications to Quantum Field Theory.} 
Springer Briefs in Mathematical Physics Vol. 3, 2015. 
 
\bibitem{B00} Brugui{\`e}res, A. 
\textit{Cat{\'e}gories pr{\'e}modulaires, modularisations et invariants des vari{\'e}t{\'e}s de dimension 3.} 
Math. Ann. \textbf{316} (2000), 215--236. 

\bibitem{C77} Connes, A. 
\textit{Periodic automorphisms of the hyperfinite factor of type II$_1$.} 
Acta Sci. Math. (Szeged) \textbf{39} (1977), 39--66.

\bibitem{Cu77} Cuntz, J. 
\textit{Simple C$^*$-algebras generated by isometries.}
Comm. Math. Phys. \textbf{57} (1977), no. 2, 173--185. 

\bibitem{DPR90} Dijkgraaf, R.; Pasquier, V.; Roche, P. 
\textit{Quasi Hopf algebras, group cohomology and orbifold models.} 
Recent advances in field theory (Annecy-le-Vieux, 1990). 
Nuclear Phys. B Proc. Suppl. \textbf{18B} (1990), 60--72 (1991). 

\bibitem{DW90}  Dijkgraaf, R.; Witten, E. 
\textit{Topological gauge theories and group cohomology.} 
Comm. Math. Phys. \textbf{129} (1990), 393--429. 

\bibitem{DHR71} Doplicher, S.; Haag, R.; Roberts, J. E. 
\textit{Local observables and particle statistics. I.} 
Comm. Math. Phys. \textbf{23} (1971), 199--230.

\bibitem{EGO04} Etingof, P.; Gelaki, S.; Ostrik, V. 
\textit{Classification of fusion categories of dimension $pq$.} 
Int. Math. Res. Not. 2004, no. 57, 3041--3056.

\bibitem{EGNO15}  Etingof, P.; Gelaki, S.; Nikshych, D.; Ostrik, V. 
\textit{Tensor categories.} Mathematical Surveys and Monographs, 205. 
American Mathematical Society, Providence, RI, 2015. 

\bibitem{ENO05} Etingof, P.; Nikshych, D.; Ostrik, V. 
\textit{On fusion categories.} 
Ann. of Math. \textbf{162} (2005), 581--642. 

\bibitem{ENO10} Etingof, P.; Nikshych, D.; Ostrik, V. 
\textit{Fusion categories and homotopy theory.} 
With an appendix by Ehud Meir. Quantum Topol. \textbf{1} (2010), 209--273.

\bibitem{E80} Evans, D. E. 
\textit{On $O_n$}. Publ. Res. Inst. Math. Sci. \textbf{16} (1980), 915--927. 

\bibitem{EG11} Evans, D. E.; Gannon, T. 
\textit{The exoticness and realisability of twisted Haagerup-Izumi modular data.} 
Comm. Math. Phys. \textbf{307} (2011), 463--512. 

\bibitem{EG14} Evans, D. E.; Gannon, T. 
\textit{Near-group fusion categories and their doubles.} 
Adv. Math. \textbf{255} (2014), 586--640. 

\bibitem{EK94} Evans, D. E.; Kawahigashi, Y. 
\textit{Orbifold subfactors from Hecke algebras.} 
Comm. Math. Phys. \textbf{165} (1994), 445--484. 

\bibitem{EK98} Evans, D. E.; Kawahigashi, Y. 
\textit{Quantum symmetries on operator algebras.} 
Oxford Mathematical Monographs. Oxford Science Publications. T
he Clarendon Press, Oxford University Press, New York, 1998.

\bibitem{EP12} Evans D. E.; Pugh, M. 
\textit{Braided subfactors, spectral measures, planar algebras, and Calabi-Yau algebras 
associated to SU(3) modular invariants.} 
Progress in operator algebras, noncommutative geometry, and their applications, 
17--60, Theta Ser. Adv. Math., 15, Theta, Bucharest, 2012. 

\bibitem{GIS15} Grossman P.; Izumi, M.; Snyder, N. 
\textit{The Asaeda-Haagerup fusion categories.}
to appear in J. Reine Angew. Math. arXiv:1501.07324. 

\bibitem{GS12} Grossman, P.; Snyder, N. 
\textit{Quantum subgroups of the Haagerup fusion categories.} 
Comm. Math. Phys. \textbf{311} (2012), 617--643.

\bibitem{GJS15} Grossman, P.; Jordan, D.; Snyder, N. 
\textit{Cyclic extensions of fusion categories via the Brauer-Picard groupoid.} 
Quantum Topol. \textbf{6} (2015), no. 2, 313--331. 

\bibitem{HH09} Hagge, T.; Hong, Seung-Moon.  
\textit{Some non-braided fusion categories of rank three.} 
Commun. Contemp. Math. \textbf{11} (2009), no. 4, 615--637. 

\bibitem{H} Haagerup, U. 
\textit{Connes' bicentralizer problem and uniqueness of the injective factor of type III$_1$.} 
Acta Math. \textbf{158} (1987), no. 1-2, 95--148. 

\bibitem{HY00} Hayashi, T.; Yamagami, S. 
\textit{Amenable tensor categories and their realizations as AFD bimodules.} 
J. Funct. Anal. \textbf{172} (2000), 19--75. 

\bibitem{HB82} Huppert, B.; Blackburn, N. 
\textit{Finite groups. III.}
Grundlehren der Mathematischen Wissenschaften, \textbf{243}. 
Springer-Verlag, Berlin-New York, 1982. 

\bibitem{Is94} Isaacs, I. M. 
\textit{Character theory of finite groups.} 
Dover Publications, Inc., New York, 1994. 


\bibitem{I93} Izumi, M. 
\textit{Subalgebras of infinite C$^*$-algebras with finite Watatani indices. I. Cuntz algebras.} 
Comm. Math. Phys. \textbf{155} (1993), 157--182. 

\bibitem{I97} Izumi, M. 
\textit{Goldman's type theorems in index theory.} 
Operator algebras and quantum field theory (Rome, 1996), 249--269, Int. Press, Cambridge, MA, 1997. 

\bibitem{I98} Izumi, M. 
\textit{Subalgebras of infinite C$^*$-algebras with finite Watatani indices. II. Cuntz-Krieger algebras.} 
Duke Math. J. \textbf{91} (1998), no. 3, 409--461.

\bibitem{I00} Izumi, M. 
\textit{The structure of sectors associated with Longo-Rehren inclusions. I. General theory.} 
Comm. Math. Phys. \textbf{213} (2000), 127--179.

\bibitem{I01} Izumi, M. 
\textit{The structure of sectors associated with Longo-Rehren inclusions. II. Examples.} 
Rev. Math. Phys. \textbf{13} (2001), 603--674.

\bibitem{IK02} Izumi, M.; Kosaki, H. 
\textit{Kac algebras arising from composition of subfactors: general theory and classification.} 
Mem. Amer. Math. Soc. \textbf{158} (2002), no.750.  

\bibitem{J80}  Jones, V. F. R. 
\textit{Actions of finite groups on the hyperfinite type II$_1$ factor.} 
Mem. Amer. Math. Soc. \textbf{28} (1980), no. 237. 

\bibitem{J83} Jones, V. F. R. 
\textit{Index for subfactors.} 
Invent. Math. \textbf{72} (1983), no. 1, 1--25.

\bibitem{JMS14} Jones, V. F. R.; Morrison, S.; Snyder, N. 
\textit{The classification of subfactors of index at most 5.} 
Bull. Amer. Math. Soc. (N.S.) \textbf{51} (2014), no. 2, 277--327.

\bibitem{KT07} Katayama, Y.; Takesaki, M. 
\textit{Outer actions of a discrete amenable group on approximately finite dimensional factors. II. The III$_{\lambda}$-case, 
$\lambda\neq 0$.} 
Math. Scand. \textbf{100} (2007), no. 1, 75--129. 

\bibitem{K94} Kawahigashi, Y. 
\textit{On flatness of Ocneanu's connections on the Dynkin diagrams and classification of subfactors.} 
J. Funct. Anal. \textbf{127} (1995), 63--107. 

\bibitem{L14} Larson, H. 
\textit{Pseudo-unitary non-self-dual fusion categories of rank 4.} 
J. Algebra \textbf{415} (2014), 184--213.

\bibitem{L15} Liu, Zhengwei. 
\textit{Singly generated planar algebras of small dimension, Part IV.}
preprint, arXiv:1507.06030. 

\bibitem{LMP15} Liu, Zhengwei; Morrison, S.; Penneys, D. 
\textit{1-supertransitive subfactors with index at most $6\frac{1}{5}$.}
Commun. Math. Phys. \textbf{334}, (2015), 889--922.  


\bibitem{L90} Longo, R. 
\textit{Index of subfactors and statistics of quantum fields. II. Correspondences, braid group statistics and Jones polynomial.} 
Comm. Math. Phys. \textbf{130} (1990), 285--309. 

\bibitem{M05} Masuda, Toshihiko,  
\textit{An analogue of Connes-Haagerup approach for classification of subfactors of type III$_1$.} 
J. Math. Soc. Japan \textbf{57} (2005), 959--1001. 

\bibitem{NS07} Ng, Siu-Hung; Schauenburg, P. 
\textit{Higher Frobenius-Schur indicators for pivotal categories.} 
Hopf algebras and generalizations, 63--90, Contemp. Math., 441, Amer. Math. Soc., Providence, RI, 2007. 

\bibitem{NO} Nikshych, D.; Ostrik, V. 
\textit{On the structure of near-group categories.} 
in preparation. 

\bibitem{O03} Ostrik, V. 
\textit{Fusion categories of rank 2.} 
Math. Res. Lett. \textbf{10} (2003), 177--183. 

\bibitem{O13} Ostrik, V. 
\textit{Pivotal fusion categories of rank 3.} 
to appear in Moscow Math J. arXiv:1309.4822. 

\bibitem{P95}  Popa, S. 
\textit{Classification of subfactors and their endomorphisms.} 
CBMS Regional Conference Series in Mathematics, 
86. Published for the Conference Board of the Mathematical Sciences, 
Washington, DC; by the American Mathematical Society, Providence, RI, 1995.

\bibitem{R93} Robinson, D. J. S.
\textit{A course in the theory of groups.} 
Graduate Texts in Mathematics, \textbf{80}. Springer-Verlag, New York, 1993.

\bibitem{S03} Siehler, J. 
\textit{Near-group categories.} 
Algebr. Geom. Topol. \textbf{3} (2003), 719--775. 

\bibitem{S80} Sutherland, C. E. 
\textit{Cohomology and extensions of von Neumann algebras. I, II.} 
Publ. Res. Inst. Math. Sci. \textbf{16} (1980), no. 1, 105--133, 135--174.

\bibitem{TY98} Tambara, D.; and Yamagami, S.
\textit{Tensor categories with fusion rules of self-duality for finite abelian groups.} 
J. Algebra \textbf{209} (1998),  692--707. 

\bibitem{Y02}  Yamagami, S. 
\textit{Group symmetry in tensor categories and duality for orbifolds.} 
J. Pure Appl. Algebra \textbf{167} (2002), 83--128.
\end{thebibliography}
\end{document}